%% file: main.tex
\PassOptionsToPackage{table,dvipsnames}{xcolor}
\documentclass[aos]{imsart}

\RequirePackage{amsthm,amsmath,amsfonts,amssymb}
\RequirePackage[authoryear]{natbib}
\RequirePackage[colorlinks,citecolor=blue,urlcolor=blue]{hyperref}
\RequirePackage{graphicx}

\usepackage{csquotes}
\usepackage{bbm}
\usepackage{xcolor}
\usepackage{tikz}
\usepackage{enumitem}
\usepackage[mode=image]{standalone}
\usepackage{subcaption}

\usepackage{multibib}
\newcites{app}{References}

\startlocaldefs

\theoremstyle{plain}
\newtheorem{dfn}{Definition} 
\newtheorem{lem}[dfn]{Lemma}
\newtheorem{rem}[dfn]{Remark}
\newtheorem{prop}[dfn]{Proposition}
\newtheorem*{rem*}{Remark}

\newtheorem{thm}[dfn]{Theorem}
\newtheorem*{thm*}{Theorem}

\newtheorem{ass}{Assumption}

\numberwithin{dfn}{section}
\numberwithin{ass}{section}

\definecolor{darkgreen}{rgb}{0.0,0.5,0.0}

\DeclareMathOperator*{\argmax}{arg\,max}

\newcommand{\convd}{\overset{d}{\longrightarrow}}
\newcommand{\convp}{\overset{\p}{\longrightarrow}}

\newcommand{\p}{\mathbb{P}}
\newcommand{\e}{\mathbb{E}}
\newcommand{\var}{\mathrm{Var}}
\newcommand{\cov}{\mathrm{Cov}}

\def\gbr#1{\lfloor #1 \rfloor}

\newcommand{\simp}{\overset{\p}{\sim}}

\newcommand{\gcell}[1]{%
  \begingroup
    \pgfmathparse{#1*40}%
    \pgfmathsetmacro{\pct}{min(max(\pgfmathresult,0),40)}%
    \edef\GreyColor{black!\pct}%
    \expandafter\cellcolor\expandafter{\GreyColor}#1%
  \endgroup
}

\endlocaldefs

\begin{document}

\begin{frontmatter}
\title{Practically significant change points in high dimension - measuring signal strength pro active component}
\runtitle{Practically significant change points in high dimension}

\begin{aug}
\author[A]{\fnms{Pascal}~\snm{Quanz} 
\ead[label=e1]{pascal.quanz@rub.de}} \and 
\author[A]{\fnms{Holger}~\snm{Dette}\ead[label=e2]{holger.dette@rub.de}\orcid{0000-0001-7048-474X}
}
\address[A]{Fakultät für Mathematik, Ruhr-Universität Bochum, Germany
\printead[presep={\ \\ }]{e1,e2}}

\end{aug}

\begin{abstract}
    In this paper, we consider the change point testing problem for high-dimensional time series. Unlike conventional approaches, where one tests whether the difference $\delta$ of the mean vectors before and after the change point 
    is equal to zero, we argue that the consideration of the null hypothesis $H_0 : \| \delta \| \le \Delta$, for some norm $\| \cdot \|$ and  a threshold $\Delta > 0$, is better suited. The reason is that in the high-dimensional regime, it is rare, and perhaps impossible, to have a null hypothesis that can be exactly modeled by assuming that all components of the vector $\delta$ are precisely equal to zero. By the formulation of the null hypothesis as a composite hypothesis, the change point testing problem becomes significantly more challenging.      
    We develop pivotal inference for testing hypotheses of this type in the setting of high-dimensional time series, first, measuring deviations from the null vector by the $\ell_2$-norm $\| \cdot \|_2$ normalized by the dimension. Second, by measuring deviations using a {\it sparsity adjusted $\ell_2$-``norm''} $\|\cdot \|_2 /\sqrt{\|\cdot \|_0 } $, where $\|\cdot \|_0$ denotes the $\ell_0$-``norm,'' we propose a pivotal test procedure which intrinsically adapts to sparse alternatives in a data-driven way by pivotally estimating the set of nonzero entries of the vector $\delta$. To establish the statistical validity of our approach, we derive tail bounds of certain classes of distributions that frequently appear as limiting distributions of self-normalized statistics.
    As a theoretical foundation for all results, we develop a general weak invariance principle for the partial sum process $X_1^{\smash{\top}} \xi + \cdots + X_{\smash{\gbr{\lambda n}}}^{\smash{\top}} \xi$ for a time series $(X_j)_{j \in \mathbb{Z}}$ and a contrast vector $\xi \in \mathbb{R}^p$ under increasing dimension $p$, which is of independent interest.
    Finally, we investigate the finite sample properties of the tests by means of a simulation study and illustrate its application in a data example.
\end{abstract}

\begin{keyword}[class=MSC]
\kwd[Primary ]{62M10}
\kwd{62M10}
\kwd[; secondary ]{62G10}
\kwd{62G20}
\end{keyword}

\begin{keyword}
\kwd{high-dimensional time series} 
\kwd{structural breaks}
\kwd{self-normalization}
\kwd{$U$-statistics}
\end{keyword}

\end{frontmatter}


\section{Introduction}
\label{sec1}
\def\theequation{1.\arabic{equation}}	
\setcounter{equation}{0}

In various applications such as finance, genomics or meteorology, one encounters time series data that exhibit structural changes over time, which are commonly referred to as change points. Such a change point can manifest in different ways, e.g. in the sequence of mean values, (auto-)covariances or other higher order parameters. Since the seminal paper of \cite{page1}, who systematically analyzed data for a structural change in the mean value, a considerable amount of literature on statistical inference for change point detection has been published. The review articles of \cite{auehor2013,jandhyala:2013,Niuetal2016,truongetal2020} and more recently \cite{cho:kirch:2024} give a good overview, and the literature on change point analysis continues to grow, driven particularly by its wide range of applications.

Several authors have worked on the problem of detecting change points in multivariate data \citep[see][among others]{ombaoetal2015, aue4moments, horvathhuskova, puchpreudet2015, kirchetal2015, haerancho, SunPour2018, chuchen2019, ZifengJianShao2022}. Meanwhile, in the era of big data, change point detection has been identified as one of the major challenges for the analysis of massive data \citep{NationalResearchCouncil.2013}, and there are many applications in which the dimension, denoted by $p$, is large relative to the sample size, denoted by $n$ \citep[see, for example,][]{fanmackey}. 
 As a result, high-dimensional change point problems have received considerable attention in recent years. 
 Here, one typically assumes that the dimension increases with the sample size to investigate the statistical properties of the proposed inference tools. For example, \cite{jirak2015}, \cite{dettegös} and \cite{chengwangwu} propose change point tests for high-dimensional means, aggregating the maximum of CUSUM statistics,
\cite{wangzouyin2018} consider the change point problem high-dimensional categorical data, and \cite{wangvolgushev} develop pivotal change points test using an $\ell_2$-distance combined with the concept of self-normalization.
Several authors use sparsity assumptions for change point estimation in high-dimensional sequences \citep[see, for example][]{chofryz2015, wangtengyao, faridazaid, liuetal2021, CWS22}, while others propose a framework that targets both sparse and dense alternatives simultaneously \citep[see, for example,][]{liuetal2020, yuchen2020, zhangetal2022}. There also exists a considerable amount of work on change point detection in high-dimensional panel data and on the use of factor models in high-dimensional change point problems; see, for example, \cite{Horvathetal2021, Barigozzietal2018, choetal2023}.

This list is by no means complete and a recent review on change point analysis for high-dimensional data can be found in \cite{liuzhangliu}. However, a common feature of most of the cited literature is the fact that change point methodology is developed for the detection of arbitrary small deviations in the parameters. In the context of testing for a change in a high-dimensional mean vector, one usually considers a null hypothesis of the form $H_0: \| \delta \| = 0$, where $\delta$ is the difference of the mean vectors before and after a potential change point and $\| \cdot \|$ denotes some norm. In the present paper, we take a different stance on this problem. We argue that in many applications, especially in the high-dimensional regime, it is often unlikely that all $p$ components of the vector $\delta$ are exactly zero. This point of view is in line with \cite{berger1987}, who state that it \enquote{\it is rare, and perhaps impossible, to have a null hypothesis that can be exactly modeled by a parameter being exactly $0$} (in our context the parameter is the norm $\|\delta\|$ of the difference $\delta$ of the mean vectors before and after the change point). Similarly, \cite{tukey1991} argues in the context of multiple comparisons of means that
\enquote{\it all we know about the world teaches us that the effects of A and B are always different -- in some decimal place -- for every A and B. Thus, asking \enquote{Are the effects different?} is foolish.}
If one adopts this point of view, testing the hypothesis $H_0: \| \delta \| = 0$ for the difference $\delta$ of the mean vectors before and after a potential change point may not be reasonable, as the null hypothesis of exact equality is not believed to be true.
Such issues are especially pertinent in the era of big data, where sample sizes (and dimensions) are typically large. In these settings, the null hypothesis will often be rejected, even if the norm of the difference vector $\delta$ is small and of no practical or scientific significance.

As an alternative, we propose in this paper change point analysis for {\it practically relevant changes}, which are defined by changes where the norm of the difference $\delta$ between the (high-dimensional) mean vectors before and after the change point is larger than a given threshold, say $\Delta >0$. In the context of hypothesis testing, this means that we are interested in the hypotheses (squaring for algebraic convenience, because we will work with $\ell_2$-type norms)
\begin{align}\label{det11}
    H_0^\Delta: \| \delta \|^2 \leq \Delta ~~~ \text{versus} ~~~ H_1^\Delta: \| \delta \|^2 >\Delta.
\end{align}
Throughout this paper, we will call hypotheses of the form \eqref{det11} \textit{relevant hypotheses}.
Note that for the choice $\Delta = 0$, equation \eqref{det11} yields the hypotheses $H_0: \delta = 0$ versus $H_1: \delta \neq 0$, which we refer to as the \textit{classical hypotheses} throughout this paper. These have been studied by numerous authors \citep[see][among others]{jirak2015,wangvolgushev}. However, following Tukey's paradigm, we argue that in the high-dimensional regime, it is very unlikely that none of the components are changing and therefore strictly focus on the case where $\delta \neq 0$ and $\Delta > 0$.

The consideration of relevant hypotheses raises the question of the choice of the threshold $\Delta$, which is problem-specific and has to be discussed carefully in each application. For applications where such a discussion is too difficult, we propose two solutions: first, we present a data-adaptive rule on how to choose this threshold by testing the hypotheses \eqref{det11} for various values of $\Delta$ and choosing the smallest threshold for which the null hypothesis is rejected, and second, we construct a pivotal asymptotic confidence interval for the squared norm of $\delta$, from which reasonable values for $\Delta$ can be obtained (see Remark \ref{remark01} (c) and (d)).

\smallskip

{\bf Our main contributions. } Testing relevant hypotheses of the form \eqref{det11} in the high-dimensional regime poses many significant challenges from both a practical and a methodological/theoretical perspective. In contrast to classical hypotheses, where all norms define the same null hypothesis, the choice of the norm plays a crucial role in our approach, as different norms can define (sometimes substantially) different null hypotheses. We consider two different measures based on the $\ell_2$-norm $\| \cdot \|_2$, following \cite{chenqin, wangshao, wangvolgushev} with the key adjustment of incorporating a normalization factor.

\begin{itemize}
    \item[1)] To make the threshold in \eqref{det11} meaningful -- especially when $p$ is large -- we first consider the \textit{normalized $\ell_2$-norm}, that is,
    \begin{align}\label{det11a}
        \| \delta \| = \frac{\| \delta \|_2}{\sqrt{p} } .
    \end{align}
    Similar to the aforementioned authors, we use a $U$-statistic for the estimation of the (squared) normalized $\ell_2$-norm, with a trimming parameter to reduce the bias induced by the temporal dependency.
    
    A significant methodological and theoretical challenge in testing the composite hypotheses \eqref{det11} lies in the fact that the nominal level has to be controlled for all vectors $\delta \in \mathbb{R}^p$ satisfying $H_0^{\Delta} : \|\delta \|^2 \leq \Delta $. 
    This requires the study of the (asymptotic) stochastic properties of an estimator for $\| \delta \|^2$ simultaneously for all $\delta \in \mathbb{R}^p$. In particular, we find that the (unknown) variance of the limiting distribution depends on $\delta$ itself, which, as a consequence, leads to an even more complicated variance structure in the case of high-dimensional dependent data.
    We address these issues by developing a pivotal test by a self-normalization approach specifically tailored to the problem of testing the relevant hypotheses in \eqref{det11}. We prove that the resulting test is consistent, has asymptotic level $\alpha$ and is also uniformly consistent against a set of local alternatives. Moreover, we also develop a data-driven way of choosing the trimming parameter of the $U$-statistic. 
  
    Self-normalization has turned out to be a powerful tool for testing classical hypotheses in the high-dimensional regime \citep[see][among others]{wangshao,wangvolgushev}, but none of these techniques can be used for testing relevant hypotheses in the high-dimensional regime as considered here. It is worth mentioning that we do not require the assumption of independence or linearity in our time series, which has been made the latter. All our procedures are developed for general high-dimensional time series. 

    \item[2)] The procedure based on the normalized $\ell_2$-norm \eqref{det11a} is targeted at detecting dense changes, which can be quite common in certain fields, such as medicine or finance \citep[see][]{wangvolgushev}. On the other hand, there exist many other applications, where a change occurs in only a few data streams. 
    In such cases, it is likely that the null hypothesis in \eqref{det11} is not  rejected in high-dimensional settings when one is measuring deviations with the norm \eqref{det11a}. 
    We address this problem by proposing the use of the \textit{sparsity adjusted $\ell_2$-norm}
    \begin{align}\label{det11b}
        \| \delta \|_{2,0} = \frac{ \| \delta \|_2 }{\sqrt{\| \delta \|_0}}
    \end{align}
    in the hypotheses \eqref{det11}. Note that $\| \cdot \|_{2,0}$ does not define a norm on $\mathbb{R}^p$, but is used as a measure of the signal $\| \delta \|_2$ \textit{pro active component}. In Section \ref{secAdaptive}, we argue that this choice achieves a good balance between maximum-type statistics, which are commonly used to detect sparse signals, and sum-type statistics, which are  used to detect dense signals. We develop a pivotal, consistent and asymptotic level $\alpha$ test for the relevant hypotheses using the norm $\| \delta \|_{2,0}$. Note that by the nature of these hypotheses, this test adapts to sparsity of the signal.
    By extending the principle of self-normalization, we are able to construct a pivotal and consistent estimator for the set
    \begin{align}\label{defSetS}
        S = \{ \ell = 1,  \ldots , p \mid \delta_\ell \neq 0 \},
    \end{align}
    which avoids the estimation of long-run variances. Consequently, this eliminates the requirement for the selection of tuning parameters, thus making the estimation procedure relatively simple. We use this result to develop a pivotal test for the relevant hypotheses using the measure $\| \delta \|_{2,0}$ in \eqref{det11b}.

    \item[3)]  By incorporating the concept of self-normalization in the context of detecting relevant change points in high-dimensional time series we  derive several theoretical results, which are of independent interest. First, we provide a weak invariance principle for the partial sum process obtained by an inner product of the time series with a contrast vector of increasing dimension. Second, a further weak invariance principle for a similar object is proved, which is obtained by using the unit vectors as contrasts, and this holds uniformly with respect to the coordinates.
    Third, we derive estimates for the tail probabilities of certain functionals of the Brownian motion $\mathbb{B}$,  which typically appear in the application of self-normalization including (the inverses of)
    \begin{align}\label{det111}
        \begin{aligned}
            \mathbb{V}_\alpha &:= \int_0^1 \lambda^\alpha ( \mathbb{B} (\lambda) - \lambda\mathbb{B} (1) )^2 \mathrm{d}  \lambda, & \alpha &\ge 0, \\
            \mathbb{W}_\alpha &:=  \int_0^1  \lambda^\alpha \big( (\mathbb{B}^2  (\lambda)- \lambda) - \lambda^2 (\mathbb{B} ^2(1)  -1 ) \big)^2   \mathrm{d}  \lambda, & \alpha & \ge 1.
        \end{aligned}
    \end{align}
   Such estimates are required for proving consistency of the pivotal estimator of the set \eqref{defSetS}.
\end{itemize}

\section{Testing for relevant changes in high dimension}
\label{sec2}
\renewcommand{\theequation}{\thesection.\arabic{equation}}
\setcounter{equation}{0}

\subsection{Notations}
We denote by $p$ the dimension of a vector that depends on the sample size $n$ (but we do not reflect this in the notation). For a vector $a \in \mathbb{R}^{p}$, let $\| a \|_2 = (a^\top a)^{1/2}$ be the common $\ell_2$-norm. For a $k \times k$ matrix $A = (A_{ij})_{i,j = 1, \ldots, k}$, we define the Frobenius norm as $\| A \|_F$, the spectral norm by $\| A \|$ and its trace by $\mathrm{tr} (A)$. 
Denote by $\ell^\infty ([0,1]) = \{ f: [0,1] \to \mathbb{R} \mid \sup_{x \in [0,1]} |f (x)| < \infty \}$ the space of all bounded functions defined on the interval $[0,1]$ and when talking about convergence, we denote convergence in probability by \enquote{$\convp$}, convergence in distribution by \enquote{$\convd$} and by \enquote{$\rightsquigarrow$} weak convergence in $\ell^\infty ([0,1])$.  We say $X \simp a_n$ iff $X_n / a_n \convp c \in (0, \infty)$ and denote equality in distribution by \enquote{$=^d$.}

\subsection{Statistical model} 
Let $X_1, \ldots , X_n \in \mathbb{R}^{p}$ denote a sample from a $p$-dimensional time series $(X_j)_{j \in \mathbb{Z}}$ and let $\mu, \delta \in \mathbb{R}^{p}$ denote two deterministic vectors. For some $k_0 \in \{ 1, \ldots , ~{n-1}\}$, we consider the following model
\begin{align}\label{model}
    X_j = \mu + \begin{cases}
        \eta_j, & j \leq k_0,\\
        \rho_j + \delta, &  j > k_0,
    \end{cases} ~~~~~ j = 1,  \ldots , n,
\end{align}
where $(\eta_j)_{j \in \mathbb{Z}}$ and $(\rho_j)_{j \in \mathbb{Z}}$ are centered and jointly second-order stationary error processes defined by $\eta_j = f_1 (\varepsilon_j, \varepsilon_{j-1},  \ldots )$ and $\rho_j = f_2 ( \varepsilon_{j}, \varepsilon_{j-1}, \ldots )$ for two measurable functions $f_1, f_2: \mathcal{S}^\mathbb{N} \longrightarrow \mathbb{R}^{p}$, a measurable space $\mathcal{S}$ and $(\varepsilon_j)_{j \in \mathbb{Z}}$ a sequence of independent identically distributed $\mathcal{S}$-valued random variables.
Throughout this paper, we assume that there exists a constant $\vartheta_0 \in (0,1)$, such that $k_0 = \lfloor n \vartheta_0 \rfloor$ and that $\delta \neq 0$.

The vector $\delta = (\delta _1 , \ldots , \delta_p)^\top $ describes the change in the mean vector (at time $k_0$). Picking up the discussion from Section \ref{sec1}, we assume that at least one coordinate of $\delta$ is nonzero (though it may be small). The change is considered scientifically (or practically) relevant if its norm exceeds a given threshold, say $\Delta > 0$. In this section, we will consider the $\ell_2$-norm normalized by the dimension in \eqref{det11a} to measure whether a change in the mean vector is relevant. As mentioned before, it is used for a better interpretation of the threshold $\Delta$ in the hypotheses \eqref{det11}.

\subsection{Test procedure}
In this section, we develop a test for the hypotheses \eqref{det11} with the norm \eqref{det11a}. A theoretical justification of our approach will be given in Section \ref{sec41}. It is based on the sequential estimation of the size of a potential change between the two samples $X_1, \ldots , X_k$ and $X_{k+1}, \ldots, X_n$ for various values of $k$. For this purpose, we define for $k,m \in \mathbb{N}$ with $k > m$ the normalizing factor
\begin{align}\label{nm}
    N_m (k) := \sum_{\substack{i_1,i_2=1\\|i_1-i_2| > m}}^{k}1 = (k - m) \cdot (k-m-1),
\end{align}
where we put $N_m (k) = 0$ for $k \leq m$, and consider the statistic
\begin{align}\label{statistic1}
   T_n (k,m) := \frac{1}{N_m (k)N_m (n-k) p } \sum_{\substack{i_1,i_2=1\\|i_1-i_2| > m}}^{k} \sum_{\substack{j_1,j_2=k+1\\|j_1-j_2| > m}}^{ n} (X_{i_1} - X_{j_1})^\top (X_{i_2} - X_{j_2}),
\end{align}
where we set $ T_n (k,m)=0$ whenever $k \leq m+1$ or $k + m \geq n-1$. Here, $m$ denotes a trimming parameter to address potential temporal dependencies in the data. To motivate the definition of the statistic $T_n (k, m)$, consider the case of independent data and put $m=0$. In this case, we obtain by a straightforward calculation that (throughout this paper $\mathbbm{1}$ denotes the indicator function)
$$
    \e [T_n (k,0) ] = \| \delta \|^2 \cdot  \Big \{ ~ \frac{N_0(n-k_0)}{N_0(n-k)} \mathbbm{1} \{ k <  k_0 \} + \mathbbm{1} \{ k = k_0 \} + \frac{N_0(k_0)}{N_0(k)} \mathbbm{1} \{ k > k_0 \} \Big \},
$$
which also explains the normalizing factor $N_m (k) N_m (n-k)$ in the definition of the statistic \eqref{statistic1}. In particular, in the independent case, this choice makes $T_n (k_0, 0)$ an unbiased estimator for $\| \delta \|^2$. Moreover, it holds that $ \e [T_n (k,0)] < \| \delta \|^2$ (and $\e [T_n (k,0)] > \| \delta \|^2$) if and only if $k < k_0 $ (and $k > k_0 $). 
This motivates us to choose the maximizer of a function that is  closely related to $T_n (k,0) $  as the estimator of the location $k_0$ of the change point.
While the choice $m = 0$ is sufficient to eliminate the bias for independent data, this is not true in the case of temporal dependence. Here, $T_n (k, 0)$ will exhibit a bias. Consider, for example, the case $k = k_0$ for a general $m$, where a tedious but straightforward calculation yields the approximation
\begin{align}
    \begin{aligned}
        \label{tnbias}
        \e [T_n (k_0, m)] -  \| \delta  \| ^2 & \approx  \frac{2}{N_m (k_0) p} \sum_{h=m+1}^{k_0 + 1} (k_0 - h)  \cdot  \mathrm{tr} (\cov (\eta_0, \eta_h))\\
        & ~~~~~ + \frac{2}{N_m (n-k_0) p} \sum_{h=m+1}^{n-k_0 - 1} (n- k_0 - h)  \cdot  \mathrm{tr} (\cov (\rho_0, \rho_h)).
    \end{aligned}
\end{align}

It turns out that for general dependence structures and dimensions, this bias is not negligible compared to the variance of the statistic $T_n (k_0, m) $. However, if $m$ increases with the sample size (the appropriate order will be specified below), it can be shown that $T_n (k_0, m)$ is an asymptotically unbiased estimator of $\| \delta \|^2 $. Note that the idea of trimming in high-dimensional testing problems has also been used in \cite{ChenWu2019}, \cite{wangshao} and \cite{wangvolgushev}, who consider a different statistic. In the latter work, it is assumed that $m/n$ converges to some small positive constant. 

It will be shown that for a suitable choice of $m$ and an appropriate normalization, the difference $T_n (k_0, m) - \| \delta \|^2$ converges weakly to a centered normal distribution, that is,
\begin{align}\label{det13}
    \frac{\sqrt{n}}{\sigma_n} (T_n (k_0, m)  - \| \delta \|^2 ) \convd \mathcal{N} (0, 1),
\end{align}
as $n \to \infty$, where the normalizing factor $\sigma_n$ is given by
\begin{align}\label{variance}
    \sigma_n^2  = \frac{4 }{p^2 \vartheta_0 (1-\vartheta_0)} \cdot \delta^\top \Gamma \delta \approx \frac{4 n^2}{p^2 k_0 (1-k_0)} \cdot \delta^\top \Gamma \delta ,
\end{align}
and $\Gamma = \sum_{h\in\mathbb{Z}} \cov (X_0, X_h)$ denotes the long-run covariance matrix of the time series $(X_j)_{j \in \mathbb{Z}}$. To make this result useful for testing the hypotheses \eqref{det11}, one would require a precise estimate of the normalizing factor $\sigma_n$ and of the unknown location of the change point $k_0$. For the latter, we use the common estimator $\hat k_n :=  \lfloor n \hat \vartheta_n \rfloor$, where 
\begin{align}\label{cpestimator}
    \hat \vartheta_n := \frac{1}{n} \argmax_{1 \leq k \leq n} \bigg\|  \frac{k(n-k)}{n^2} \bigg( \frac{1}{k} \sum_{j=1}^{k} X_j - \frac{1}{n-k} \sum_{j= k +1}^n X_j \bigg) \bigg\|_2^2
\end{align}
\citep[see, for example,][]{harizcp,jandhyala:2013}.
In Theorem \ref{cpthm} below, it will be shown that $\hat \vartheta_n$ estimates $\vartheta_0$ with an error of order $O_\p ( \log^2 (n) / n) $, which will be sufficient to replace $k_0$ in \eqref{det13} by the estimator $\hat k_n$ without changing the limiting distribution. 
Note that we do not use trimmed estimators in this definition, because in Section \ref{sec3} we will develop a data-adaptive rule for choosing the trimming parameter $m$, which requires an initial estimate of the change point.

The estimation of the normalizing factor $\sigma_n$ in \eqref{variance} is more challenging. First, it requires an estimator of the difference between the mean vectors before and after the change point. A natural estimate is the difference between the mean vectors of the first $\hat k_n$ and the remaining $n - \hat k_n$ observations.
Second, one needs a good estimate of the long-run variance $\Gamma$. This is a particular challenging problem in the high-dimensional regime, even in the one-sample problem \citep[see, for example,][]{chen2013covariance,basu2016regularized}. 
In order to avoid the intrinsic difficulties  that arise in the estimation of high-dimensional autocovariance matrices, we will develop a self-normalization procedure that cancels out the variance in the limit. We emphasize that the common self-normalizing concepts proposed by \cite{shaozhang2010} and \cite{wangvolgushev} cannot be used in this context, as they are constructed for a different weak convergence result, which arises from testing classical hypotheses (in our context $H_0: \delta =0$). In fact, a special self-normalizing technique has to be implemented in order to address the specific structure of the composite hypotheses in \eqref{det11} for $\Delta >0$ in the high-dimensional regime.

To be precise, we consider a sequential version of the statistic $T_n( k, m)$, which is defined by
\begin{align}\label{statistic}
   T_n (k,m;\lambda) := \frac{1}{N_m (k)N_m (n-k) p } \sum_{\substack{i_1,i_2=1\\|i_1-i_2| > m}}^{\gbr{\lambda k}} \sum_{\substack{j_1,j_2=k+1\\|j_1-j_2| > m}}^{k + \gbr{\lambda (n-k)}} (X_{i_1} - X_{j_1})^\top (X_{i_2} - X_{j_2}),
\end{align}
where we set $ T_n (k,m ; \lambda ) \equiv 0$ whenever $k \leq m+1$ or $k + m \geq n-1$. It follows from Theorem \ref{thm2.1} (with $A=\{1, \ldots, p \}$) in Section \ref{sec4} that, if appropriately normalized, this sequential process converges weakly in $\ell^\infty ([0,1])$, that is,
\begin{align} \label{det01}
    \bigg\{  \frac{\sqrt{n}}{\sigma_n} \big(T_n (\hat k_n, m; \lambda) - \Lambda_n (\lambda) \| \delta \|^2 \big)\bigg\}_ { \lambda \in [0,1]} \rightsquigarrow \big\{ \lambda^3 \mathbb{B} (\lambda) \big\}_{ \lambda \in [0,1]} 
\end{align}
as $n,m \to\infty$, $m^{3/2}=o(n)$, where
\begin{align}\label{Lambda_n}
    \Lambda_n (\lambda) := \frac{N_m (\gbr{\lambda \hat k_n})}{N_m (\hat k_n)} \frac{N_m (\gbr{\lambda (n- \hat k_n)})}{N_m (n-\hat k_n)}
\end{align}
and $\{ \mathbb{B} (\lambda) \}_{\lambda \in [0,1]}$ is a standard Brownian motion on the interval $[0,1]$. Theorem \ref{thm2.1} actually provides a weak Gaussian approximation result, which is of independent interest. The weak convergence in \eqref{det01} and the continuous mapping theorem imply
\begin{align} \label{det100}
    \frac{\sqrt{n}}{\sigma_n} 
    \big (  T_n (\hat k_n, m ) - \| \delta \|^2,  V_n \big ) \convd \Big ( \mathbb{B} (1) , \Big ( \int_0^1 \lambda^6 (\mathbb{B} (\lambda) - \lambda \mathbb{B} (1) )^2 \mathrm{d} \nu (\lambda) \Big )^{1/2} \Big) , 
\end{align}
where the statistic $V_n$ is defined by 
\begin{align}\label{vn}
    V_n := \bigg( \int_0^1 (T_n (\hat k_n, m; \lambda) - \Lambda_n (\lambda) T_n (\hat k_n, m  ) )^2 \mathrm{d} \nu (\lambda) \bigg)^{1/2}, 
\end{align}
$\nu$ is a probability measure on the interval $[0,1]$. A further application of the continuous mapping theorem then yields
\begin{align}\label{defG}
    \frac{T_n (\hat k_n, m ) - \| \delta \|^2}{ V_n} \convd 
     \mathbb{G} := \frac{\mathbb{B} (1)}{ \big( \int_0^1 \lambda^6 (\mathbb{B} (\lambda) - \lambda \mathbb{B} (1) )^2 \mathrm{d} \nu (\lambda) \big)^{1/2} } .
\end{align}
We will show in Lemma~\ref{lemmasubexp} in the appendix that the denominator of $\mathbb{G} $ is nonzero with probability one, and therefore the random variable $\mathbb{G} $ is well-defined. Moreover, in the same lemma it is shown that the distribution of $\mathbb{G}$ has subexponential tails when $\nu$ is the Lebesgue measure. Some quantiles of this distribution are listed in Table~\ref{quantiletable}, where the measure $\nu$ is taken as the discrete uniform distribution with different numbers of support points.

\begin{table}[!ht]
    \centering
    \begin{tabular}{c c cccccc }
        \hline
        && \multicolumn{6}{c}{$\alpha$}\\
        \cline{3-8}
        $K$ && 80\% & 90\% & 95\% & 97.5\% & 99\% & 99.5\% \\
        \hline
        $10$ && 8.579 & 14.813 & 21.448 & 28.683 & 40.587 & 51.834 \\
        $15$ && 8.308 & 14.333 & 20.465 & 26.295 & 35.344 & 42.782 \\
        $20$ && 8.660 & 14.063 & 19.839 & 26.764 & 34.151 & 38.287 \\
        $25$ && 8.533 & 13.896 & 19.772 & 25.542 & 33.141 & 37.764 \\
        \hline
    \end{tabular}
    \caption{\it Quantiles  of the random variable $\mathbb{G}$ defined in \eqref{defG}, where $\nu$ is the discrete uniform distribution supported on  $\{ 1/K, \ldots, (K-1)/K \}$ for different values of $K$. The quantiles have been calculated by 10000 replications.}
    \label{quantiletable}
\end{table}

Based on the weak convergence in \eqref{defG}, we propose to reject the null hypothesis of no relevant change in the high-dimensional sequence of mean vectors whenever 
\begin{align}\label{rule}
    T_n (\hat k_n, m ) > \Delta + q_{1-\alpha} V_n,
\end{align}
where $q_{1-\alpha}$ denotes the $(1-\alpha)$-quantile of the distribution of $\mathbb{G}$. In Section \ref{sec4}, we show that the decision rule \eqref{rule} defines a (uniformly) consistent and asymptotic level $\alpha$ test for the hypotheses in \eqref{det11} with the norm \eqref{det11a}. In particular, we prove, under suitable assumptions (see Theorem \ref{testconsistent} for details), that 
\begin{align}\label{teststuff}
    \lim_{n \to \infty  } \p \big( T_n (\hat k_n, m) > \Delta + q_{1-\alpha} V_n  \big) = 
    \begin{cases}   
        0,~~  \text{ if } ~ \frac{\sqrt{n}}{\sigma_n}  ( \Delta - \| \delta \|^2 ) \to \infty,\\
        \alpha, ~~\text{ if } ~ \frac{\sqrt{n}}{\sigma_n}  ( \Delta - \| \delta \|^2 ) \to  0,\\
        1,  ~~\text{ if } ~ \frac{\sqrt{n}}{\sigma_n}  ( \Delta - \| \delta \|^2 ) \to  - \infty.
    \end{cases}
\end{align}

\begin{rem}\label{remark01} ~~
{\rm 
    \begin{itemize}
        \item[(a)] Self-normalization is a well known principle to obtain pivotal statistics, and has been used by several authors \citep[we refer to][among others]{lobato,shaozhang2010} for testing classical hypotheses, which corresponds to the hypothesis $H_0: \delta = 0$ in the present context. 
        For these hypotheses, the concept has recently been applied by \cite{wangshao} and \cite{wangvolgushev} to obtain asymptotically distribution free tests for comparing means and for change point detection in high-dimensional time series. 
        In the present context of testing for relevant change points, we obtain a different limiting distribution as derived in these papers, due to the consideration of relevant hypotheses and the use of a $U$-statistic in combination with the quadruple sum. The result is more closely related to the work of \cite{dettekokot}, who consider the problem of detecting a relevant change in the sequence of functional data. 
        Note that the factor $\Lambda_n (\lambda) \approx \lambda^4$ appears in the normalizing factor \eqref{vn} (instead of its limit $\lambda^4$), to avoid unnecessary restrictions on the growth rate of the dimension $p$ relative to the sample size $n$ (this will be clear from the proofs of our main results in the online supplement).

        \item[(b)] The test \eqref{rule} depends on the measure $\nu$ used in the normalizing statistic $V_n$ defined in \eqref{vn}. Some quantiles of the limiting distribution of the random variable $\mathbb{G}$ in \eqref{defG} are listed in Table~\ref{quantiletable} for various discrete uniform distributions. We observe that these quantiles vary for different measures $\nu$. However, the final test is not very sensitive with respect to the choice of $\nu$ as the test statistic depends on the measure $\nu$ as well. To understand this property heuristically, we introduce the notation
        $$
            \mathbb{V} (\nu) = \Big( \int_0^1 \lambda^6 (\mathbb{B} (\lambda) - \lambda \mathbb{B} (1) )^2 \mathrm{d} \nu (\lambda) \Big)^{1/2} 
        $$
        and $q_{1- \alpha} (\mathbb{B}(1) / \mathbb{V} (\nu))$ to express the dependence of the denominator in \eqref{defG} and the quantile $q_{1-\alpha}$ on the measure $\nu$ on explicitly. With these notations we obtain from \eqref{defG} for the probability of rejection by the test \eqref{rule} the approximation \begin{align*}
            \p \big( T_n (\hat k_n, m) > \Delta + q_{1-\alpha} V_n \big) \approx \p \bigg( \mathbb{B} (1) > \sqrt{n} \frac{\Delta - \| \delta \|^2}{\sigma_n} + \frac{\sqrt{n}}{\sigma_n} q_{1-\alpha} \bigg ( \frac{ \mathbb{B} (1) }{ \mathbb{V} (\nu) } \bigg ) \mathbb{V} (\nu) \bigg),
        \end{align*}
        where we used the fact that $\frac{\sqrt{n}}{\sigma_n} \big( T_n (\hat k_n, m) - \|\delta \|^2 , V_n \big) \convd  \big( \mathbb{B}(1),\mathbb{V} (\nu) \big)$, which follows from \eqref{det100}.
        As for a fixed constant $c q_{1-\alpha} (\mathbb{B} (1) / c) = q_{1-\alpha} (\mathbb{B} (1))$ it can be argued heuristically that the rejection probabilities are not very sensitive to the choice of the measure $\nu$. We also demonstrate this robustness by a small simulation study in Section S1 of the online supplement.

        \item[(c)] Obviously the hypotheses in \eqref{det11} are nested and the decision rule in \eqref{rule} is monotone with respect to $\Delta$. Therefore, by the sequential rejection principle, we can simultaneously test the hypotheses $H_0^{\Delta}: \| \delta \|^2 \leq \Delta$ for all $\Delta \geq 0$ and find the smallest $\Delta$ for which $H_0^\Delta $ cannot be rejected, that is
        \begin{align*}
            \Delta_\alpha = \inf \{ \Delta > 0 \mid T_n (\hat k_n, m) \leq \Delta + q_{1-\alpha} V_n \}.
        \end{align*}
    We then reject all hypotheses with $\Delta < \Delta_\alpha$ and accept all hypotheses with $\Delta \geq \Delta_\alpha$ (controlling the level simultaneously). Therefore, the value $\Delta_\alpha$ can be interpreted as a measure of evidence against the existence of a change point and the question about a reasonable choice can be postponed until seeing the data.
        
        \item[(d)]
        From the weak convergence \eqref{defG} and symmetry of the distribution of the random variable $\mathbb{G}$, it follows that a one-sided (upper) asymptotic $(1-\alpha)$ confidence interval for the squared normalized norm $\| \delta \|^2 = \|\delta \|^2_2/p> 0$ is given by
        \begin{align*}
            \big[ 0, T_n (\hat k_n, m) + q_{1-\alpha} V_n \big].
        \end{align*}
        Similarly, a two-sided asymptotic $(1-\alpha)$ confidence interval is given by
        \begin{align*}
            \big[ \max \big\{ 0,  T_n  (\hat k_n, m) - q_{1-\alpha / 2} V_n \big\} , T_n (\hat k_n, m) + q_{1-\alpha / 2} V_n \big].
        \end{align*}
    \end{itemize}
    }
\end{rem}

\section{Relevant hypotheses with a sparsity adjusted $\ell_2$-norm} \label{secAdaptive}
\renewcommand{\theequation}{\thesection.\arabic{equation}}
\setcounter{equation}{0}

In Section \ref{sec2}, we use the normalized $\ell_2$-norm \eqref{det11a} to measure the signal strength $\|\delta \|_2$ pro component, which is reasonable if the signal $\delta =  (\delta_1, \ldots , \delta_p)^\top$ before and after the change point is dense in the sense that the cardinality $s := |S| = \| \delta \|_0$ of the set $S$ of nonvanishing coordinates of $\delta$ in \eqref{defSetS} is close to the dimension $p$. However, with the norm \eqref{det11a}, it is not possible to distinguish vectors with the same $\ell_2$-norm but a different number of nonvanishing components. In particular, if one fixes $s$ components and fills the remaining $p-s$ components with $0$'s, the norm of the resulting vector (for example, a unit vector) converges to $0$ with increasing dimension. In such cases, a signal might not be detected, even if some of the components substantially exceed the threshold $\Delta$. A natural norm addressing such problems is the maximum norm $\|\delta \|_\infty = \max_{i=1}^p |\delta_i|$. However, this norm is not very sensitive to dense signals, where many small deviations accumulate to a relevant deviation from $\delta=0$. To ensure a balance between sparsity and signal strength, we propose to measure {\it the signal strength $\|\delta \|_2$ pro active component} and use the norm \eqref{det11b} instead of \eqref{det11a} in the definition of the hypotheses \eqref{det11}. Note that
\begin{align*}
    \frac{\| \delta \|_2}{\sqrt{p}} \leq \frac{\| \delta \|_2}{\sqrt{ \|\delta \|_0} }  \leq \| \delta \|_\infty,
\end{align*}
and that there is equality in the first inequality for vectors with only nonvanishing components and in the second inequality for vectors of the form $\delta = (c, \ldots, c, 0,  \ldots, 0)^\top$ with $c \neq 0$ (including permutations).

While the ratio of the $\ell_2$- and $\ell_1$-norm has been considered as sparsity-promoting objective by several authors \citep[see, for example,][]{pmlr-v28-lopes13,Linagetal2013,hoyer,XU2021486}, the ratio $\|\delta \|_{2,0}$ of the $\ell_2$- and $\ell_0$-norm has not been widely used so far \citep[an exception is the work of][who used this concept for learning the structure of sparse Ising models]{Dedieu.2021}. In order to construct a test for the relevant hypotheses with this norm
\begin{align}\label{hypotheses_new}
    H_0^{\Delta} : \| \delta \|^2_{2,0} \leq \Delta ~~~~~ \text{ versus } ~~~~~ H_1^{\Delta} : \| \delta \|^2_{2,0} > \Delta,
\end{align}
we recall the definition $N_m$ from \eqref{nm}, denote by $X_{i,\ell}$ the $\ell$th entry of the vector $X_i =(X_{i,1} , \ldots , X_{i,p})^\top$ 
and use the statistic
\begin{align}\label{det89}
    & T_{n, S} (\hat k_n,m;\lambda) \\
    & \nonumber
    := \frac{1}{N_m (\hat k_n)N_m (n - \hat k_n) s } \sum_{\substack{i_1,i_2=1\\|i_1-i_2| > m}}^{\gbr{\lambda \hat k_n}} \sum_{\substack{j_1,j_2= \hat k_n + 1\\|j_1-j_2| > m}}^{\hat k_n + \gbr{\lambda (n - \hat k_n)}} \sum_{\ell \in S}  (X_{i_1, \ell} - X_{j_1, \ell}) (X_{i_2, \ell} - X_{j_2, \ell})
\end{align}
as a sequential estimator of the sparsity adjusted $\ell_2$-norm $\| \delta \|_{2,0}^2$ of the difference between the mean vectors before and after the change point. We also set $T_{n, S} (\hat k_n,m;\lambda) = 0$ whenever $\hat k_n \leq m + 1$ or $\hat k_n + m \geq n - 1$ and $T_{n,S} (\hat k_n, m) = T_{n,S} (\hat k_n, m; 1)$. 
Restricting the sum to indices in $S$ has the advantage that we are able to increase the power of the resulting test by avoiding the summation of error terms corresponding to the components with no signal.

To obtain an implementable estimator of $\|\delta \|_{2,0}^2$, it is necessary to estimate the set $S$ in the statistic \eqref{det89}, and in the following discussion we develop such an estimator, which does not rely on estimates of the long-run variance. For this purpose, we propose for $\ell = 1, \ldots, p$
\begin{align*}
    \hat \delta_\ell^2 &=
    T_{n, \{ \ell \} } (\hat k_n, m ; 1) 
\end{align*}
as an estimator of the $\ell$th (squared) coordinate $\delta_\ell^2$ of the vector $\delta$. To account for the variability of $\hat \delta^2_\ell $ and to avoid long-run variance estimation, we employ the self-normalizing principle again and consider a sequential version of the estimator $\hat \delta_\ell^2$, that is,
\begin{align}\label{hat_delta_ell}
  \hat \delta_\ell^2 (\lambda )=     \frac{1}{N_m (\hat k_n ) N_m (n-\hat k_n )} \sum_{\substack{i_1, i_2 = 1 \\ |i_1 - i_2| > m}}^{\gbr{\lambda \hat k_n }}\sum_{\substack{j_1, j_2 = \hat k_n  + 1 \\ |j_1 - j_2| > m}}^{\hat k_n  + \gbr{\lambda (n-\hat k_n )}} (X_{i_1, \ell} - X_{j_1, \ell} ) (X_{i_2, \ell} - X_{j_2, \ell})~. 
\end{align}
We note that $\hat \delta_\ell^2 = \hat \delta_\ell^2 (1)$ and finally define 
\begin{align}\label{hat_v_ell}
    \hat v_\ell = \bigg( \int_0^1 \big( \hat \delta_\ell^2 (\lambda) - \lambda^4 \hat \delta_\ell^2  (1) \big)^2 \mathrm{d} \lambda \bigg)^{1/2}
\end{align}
as the self-normalizing statistic for $\hat \delta_\ell ^2$ ($\ell = 1, \ldots, p$). 
The estimator of the set $S$ in \eqref{defSetS} then takes the form
\begin{align}\label{def_S_hat}
    \hat S_n = \{ \ell = 1,  \ldots , p \mid \hat \delta_\ell^2  > \hat v_\ell \cdot \log^{3/2} (p) \}.
\end{align}
Under suitable assumptions, we will see in Theorem \ref{snconsistent} below that $\p (\hat S_n = S) \to 1$ as $n \to \infty$. With the help of Theorem \ref{thm2.1} (also stated below), this allows us to prove the weak convergence
\begin{align*}
    \Big \{  \frac{\sqrt{n}}{\sigma_{n,S}} \big(T_{n, \hat S_n} (\hat k_n, m; \lambda) - \Lambda_n (\lambda) \| \delta \|^2_{2,0} \big)\Big \}_ { \lambda \in [0,1]} \rightsquigarrow \big\{ \lambda^3 \mathbb{B} (\lambda) \big\}_{ \lambda \in [0,1]}
\end{align*}
for the estimator \eqref{det89} with $S$ replaced by $\hat S_n$, where $\Lambda_n$ is defined in \eqref{Lambda_n} and 
\begin{align*}
    \sigma_{n,S}^2 = \frac{4}{s^2 \vartheta_0 (1-\vartheta_0)} \sum_{i,j \in S} \delta_i \Gamma_{ij} \delta_j.
\end{align*}
Consequently, by defining the statistics $ \hat T_{n, \hat S_n} (\hat k_n, m) = T_{n,\hat S_n} (\hat k_n, m; 1)$ and  
\begin{align*}
    V_{n,\hat S_n} = \bigg( \int_0^1 (T_{n,\hat S_n} (\hat k_n, m; \lambda) - \Lambda_n (\lambda) \cdot T_{n,\hat S_n} (\hat k_n, m  ) )^2 \mathrm{d} \nu (\lambda) \bigg)^{1/2},
\end{align*}
it then follows by the continuous mapping theorem that 
$$
    \frac{ T_{n,\hat S_n}  (\hat k_n, m ) - \| \delta \|^2_{2,0} }{   V_{n,\hat S_n} } \convd \mathbb{G}, 
$$
where the random variable $\mathbb{G} $ is defined in \eqref{defG}. 
Therefore, we obtain a fully data-adaptive test for the hypotheses in \eqref{hypotheses_new} by rejecting the null hypothesis in \eqref{hypotheses_new} whenever
\begin{align}\label{fully_adaptive_test}
    \hat T_{n, \hat S_n} (\hat k_n, m )  > \Delta + q_{1-\alpha} V_{n, \hat S_n},
\end{align}
where $q_{1 - \alpha}$ is the $(1-\alpha)$-qantile of the distribution of $\mathbb{G}$. We will show in Theorem \ref{adaptive_consistency_thm} below that this decision rule defines an asymptotic level $\alpha$ test for the hypotheses \eqref{hypotheses_new}, that is,
\begin{align}\label{teststuff_S}
    \lim_{n \to \infty  } \p \big( T_{n, \hat S_n} (\hat k_n, m) > \Delta + q_{1-\alpha} V_{n, \hat S_n} \big) = 
    \begin{cases}   
        0,~~  \text{ if } ~ \frac{\sqrt{n}}{\sigma_{n, S}}  ( \Delta - \| \delta \|^2_{2,0} ) \to \infty,\\
        \alpha, ~~\text{ if } ~ \frac{\sqrt{n}}{\sigma_{n, S}}  ( \Delta - \| \delta \|^2_{2,0} ) \to  0,\\
        1,  ~~\text{ if } ~ \frac{\sqrt{n}}{\sigma_{n, S}}  ( \Delta - \| \delta \|^2_{2,0} ) \to  - \infty.
    \end{cases}
\end{align}

A proof of these results is complicated and requires several steps, which are described in Section S6 of the online supplement. Here, the main step lies in the proof of the consistency of the set estimator \eqref{def_S_hat}, see Theorem \ref{snconsistent} for a precise statement.
To control the probability $\p (S \not \subset \hat S_n)$, which is crucial for signal detection, it is necessary to derive bounds for the tail probabilities of the maximum deviation $M_S= \max_{\ell \in S} ({ \hat \delta_\ell^2 - \delta_\ell^2 ) / \hat v_\ell }$. Similarly, the probability $\p (\hat S_n \not \subset S)$ can be controlled by a tail probability bound for $M_{S^C}= \max_{\ell \in S^C} ({ \hat \delta_\ell^2 - \delta_\ell^2 ) / \hat v_\ell }$. The quantities $( \hat \delta_\ell^2 - \delta_\ell^2 )/ \hat v_\ell$ have different asymptotic properties depending on whether $\ell \in S$ or $\ell \in S^C$. For $\ell \in S$, the statistic $(\hat \delta_\ell^2 - \delta_\ell^2 ) / \hat v_\ell$ can be (stochastically) approximated by a random variable $\mathbb{G}_\ell$ with the same distribution as the random variable $\mathbb{G}$ in \eqref{defG}. In the case $\ell \in S^C$, there exists an approximation by the random variables $\mathbb{H}_\ell$ with distribution 
\begin{align*}
   \mathbb{H} := \frac{\mathbb{B}^2 (1) - 1}{\big ( \int_0^1 \lambda^4 \big( (\mathbb{B}^2 (\lambda) -\lambda) - \lambda^2 (\mathbb{B}^2 (1) - 1) \big)^2 \mathrm{d}\lambda \big )^{1/2}} .
\end{align*}
Note that by incorporating self-normalization, we obtain the random variables $\mathbb{G}$ and $\mathbb{H}$, whose distributions are independent of the dependence structure of the underlying time series.
The upper bounds for the tail probabilities of the maximum deviations $M_S$ and $M_{S^C}$ depend on the error rates made in these approximations and the tail behavior of the maxima of the corresponding limiting distributions. The tail decay rates for $\mathbb{G}$ and $\mathbb{H}$ are of independent interest and derived as part of Lemmas \ref{lemmasubexp} and \ref{lemmasubsubexp} in the appendix.

\begin{rem} ~
{\rm 
\begin{itemize}
    \item[(a)] 
    Note that we explicitly use the Lebesgue-measure in   \eqref{hat_v_ell} to define the self-normalizing statistic $\hat v_\ell$. This choice is made by a technical argument in the derivation of the tail bounds for the random variables $M_S$ and $M_{S^C}$. Here, we need to derive tail bounds for the random variables $\mathbb{V}_\alpha $, $ \mathbb{V}_\alpha^{-1}$ and $\mathbb{W}_\alpha^{-1}$, which are defined in \eqref{det111}. Our argument relies on an explicit representation of the Laplace transform of $\mathbb{V}_\alpha$ and $\mathbb{W}_\alpha$, which is not available for general measures. We refer to the proof of Lemmas~\ref{lemmasubexp} and \ref{lemmasubsubexp} in the online supplement for more details. 

    \item[(b)] Note that our approach of estimating the set of nonvanishing components differs from a method introduced by \cite{powerenhancement} in a different context. These authors develop a technique to boost the power of testing the classical hypothesis $H_0: \theta = 0$ for a high-dimensional vector $\theta=(\theta_1, \ldots, \theta_p)^\top$ by combining an asymptotically pivotal statistic with a \enquote{power enhancement component.} Their method is based on a screening technique for the nonvanishing components, which basically compares the statistic $\hat \theta_\ell / \hat \omega _\ell$ with a threshold, where $\hat \theta_\ell$ denotes an estimator of the $\ell$th component of $\theta$ and $\hat \omega _\ell$ is a consistent estimator of its corresponding (asymptotic) long-run variance. In the high-dimensional regime, estimating these long-run variances can be challenging, as it requires selecting $p$ tuning parameters. In contrast to this work, the set estimator $\hat S_n$ in \eqref{def_S_hat} for the set $S$ is based on self-normalized statistics and therefore avoids the problem of selecting many tuning parameters.

    \item[(c)] In Section \ref{sec4}, it will be demonstrated that the test \eqref{fully_adaptive_test} is consistent and has asymptotic level $\alpha$ for a dimension $p$ increasing with an exponential rate of the sample size, which depends on the number $s$ of nonvanishing components of the signal $\delta$. For example, in the case of dense data, that is, $s \sim \sqrt{p}$ or even $s \sim p$, it can be of the form $p = e^{c \cdot n^{1/5}}$ for some constant $c > 0$.
    In the case of very sparse data, such as $s \sim p^{n^{-a}}$, we obtain $p = e^{c \cdot n^{1/5 + a/3}}$. In general, the actual upper bound for the growth rate on $p$ depends on many components, which are described in detail in Remark \ref{growth_rate} in the following section.
\end{itemize}
}
\end{rem}

\section{Main results}
\label{sec4}
\renewcommand{\theequation}{\thesection.\arabic{equation}}
\setcounter{equation}{0}

In this section, we provide a theoretical justification of the validity of the tests proposed in Sections \ref{sec2} and \ref{secAdaptive}. All statements rely on two weak invariance principles, which are stated first in Section \ref{sec42} and of own interest.
In Section \ref{sec41}, we derive statistical guarantees for the test \eqref{rule} for the hypotheses \eqref{det11} with the norm \eqref{det11a}, which compares the norm of the difference $\delta$ before and after the change point without adjusting for sparsity.
The corresponding results for the validity of the test \eqref{fully_adaptive_test} for the hypotheses \eqref{hypotheses_new} are given in Section \ref{subsec_theoretical_prop_unif}.

\subsection{Two weak invariance principles}\label{sec42}

\noindent
In this section, let $(Y_j)_{j \in \mathbb{Z}}$ be a stationary sequence of centered $\mathbb{R}^{p}$-valued random variables with $Y_j = g(\varepsilon_j, \varepsilon_{j-1},  \ldots )$, $j \in \mathbb{Z}$, for some measurable function $g : \mathcal{S}^\mathbb{N} \to \mathbb{R}^p$ and a sequence $(\varepsilon_j)_{j \in \mathbb{Z}}$ of independent $\mathcal{S}$-valued random variables, where $\mathcal{S}$ is a measurable space. We will denote the long-run $p \times p$ covariance matrix of $(Y_j)_{j \in \mathbb{Z}}$ by
\begin{align}\label{det50h}
    \Gamma_Y = (\Gamma_{Y, ij})_{i,j = 1}^p := \sum_{h \in \mathbb{Z}} \cov (Y_0, Y_h).
\end{align}
For a vector $\xi \in \mathbb{R}^{p}$, we consider the quantity $\frac{1}{n} \sum_{j=1}^{k} Y_j^\top \xi$, where $k = 1, \ldots, n$, and develop an invariance principle for these partial sums by making use of the martingale decomposition $Y_j^\top \xi = E_j^\top \xi - F_j^\top \xi$, where
\begin{align}
    \begin{aligned}\label{martingaledecomp}
        E_j &= \sum_{k=0}^{\infty} (\e [Y_{j+k} \mid \mathcal{F}_j] - \e [Y_{j+k} \mid \mathcal{F}_{j-1}]), \quad\quad\quad F_j = G_j - G_{j-1},
    \end{aligned}
\end{align}
$G_j = ( G_{j,1}, \ldots, G_{j,p} )^\top = \sum_{k=1}^\infty \e [Y_{j+k} \mid \mathcal{F}_j]$ and $\mathcal{F}_j = \sigma (\varepsilon_j, \varepsilon_{j-1},  \ldots )$ denotes the sigma field generated by the random variables $(\varepsilon_k)_{k=j,j-1, \ldots } $. 
Note that this decomposition was also used in \cite{strongwu} to derive a strong invariance principle similar to the one we present below. However, the (weak) invariance principles stated in Theorems \ref{thm2.0} and \ref{thm2.0_unif_neu} below, and proved in the online supplement, are tailored to the high-dimensional scenario with a different set of assumptions.
To be precise, for $j \in \mathbb{N}$ we introduce the matrices $\Phi_j = (\Phi_{j,\ell_1 \ell_2} )_{\ell_1, \ell_2=1}^p \in \mathbb{R}^{p \times p}$ with entries 
\begin{align}\label{Phi}
    \Phi_{j,\ell_1 \ell_2} := \e [E_{j,\ell_1} E_{j,\ell_2} \mid \mathcal{F}_{j-1}],
\end{align}
where $E_{j, \ell}$ denotes the $\ell$th component of the vector $E_j$ in \eqref{martingaledecomp}. For the first result of this section, we state the following set of assumptions.

\begin{ass}\label{ass1}
    Let $0 < \alpha < 1$. We assume for any $j_1, j_2 \in \mathbb{N}$ as $n \to \infty$
    \begin{enumerate}[label=(A\arabic*)]
        \item $(\xi^\top \Gamma_Y \xi)^{-1} \sum_{\ell_1, \ell_2 = 1}^{p} \xi_{\ell_1} \xi_{\ell_2} \cdot \mathrm{Cov} (G_{j_1, \ell_1}, G_{j_2, \ell_2}) = O(n^{-\alpha})$, \label{ass_thm_3}
        
        \item $(\xi^\top \Gamma_Y \xi)^{-2} \sum_{\ell_1,  \ldots , \ell_4 = 1}^{p} \xi_{\ell_1} \cdots \xi_{\ell_4} \cdot \mathrm{Cov} (\Phi_{j_1, \ell_1 \ell_2}, \Phi_{j_2, \ell_3 \ell_4}) = O(n^{-\alpha / 2})$, \label{ass_thm_1}

        \item $(\xi^\top \Gamma_Y \xi)^{-2} \sum_{\ell_1,  \ldots , \ell_4 = 1}^{p} \xi_{\ell_1} \cdots \xi_{\ell_4} \cdot \mathrm{cum} (E_{0, \ell_1},  \ldots , E_{0, \ell_4}) = O(1)$.  \label{ass_thm_2}

    \end{enumerate}
\end{ass}

Condition \ref{ass_thm_3} ensures that the remainder term of the martingale approximation becomes small, whereas conditions \ref{ass_thm_1} and \ref{ass_thm_2} ensure that the martingale part is controlled and converges weakly.

\begin{thm}\label{thm2.0}
    Let Assumption \ref{ass1} be satisfied. Then, on a possibly richer probability space, there exists a sequence  of random variables $(\tilde{Y}_j)_{j \in \mathbb{Z}}$ with the same distribution as $(Y_j)_{j \in \mathbb{Z}}$ and a standard Brownian motion $\mathbb{B}$ such that
    \begin{align*}
         \frac{1}{n \sqrt{\xi^\top \Gamma_Y \xi}} \sum_{j=1}^{k} Y_j^\top \xi  ~ =^d ~  \tilde{S}_n (k) := \frac{1}{n \sqrt{\xi^\top \Gamma_Y \xi}} \sum_{j=1}^k \tilde{Y}_j^\top \xi,
    \end{align*}
    for any $k \leq n$, where the partial sums $\tilde{S}_n (k) $ satisfy for any $0 \leq \beta < \alpha / 4$
    \begin{align*}
        \sup_{0 \leq \lambda \leq 1} \big| \sqrt{n} \tilde{S}_n (\gbr{n \lambda}) - \mathbb{B} (\lambda) \big| = o_\p (n^{-\beta}), ~~~~~ \text{ as } n \to \infty.
    \end{align*}
\end{thm}

Next, we provide a uniform weak invariance principle for the partial sums corresponding to the coordinates $(Y_{j,\ell} )_ {\ell \in A} $ of the vectors $Y_j =( Y_{j,1}, \ldots ,Y_{j,p} )^\top  \in \mathbb{R}^p$ for a set $A \subset \{1, \ldots, p\}$, which may depend on $n$. For this stronger result, the conditions stated previously have to be adapted.

\begin{ass}\label{ass1_uniform_i_neu}
    Let $0 < \alpha < 1$ and $f : \mathbb{N} \to (0, \infty ) $ be a nondecreasing function. For any $j_1, j_2 \in \mathbb{N}_0$ and $\ell \in A$, let the following conditions be satisfied.
    \begin{enumerate}[label=(B\arabic*)]
        \item $\sum_{\ell \in A} \Gamma_{Y, \ell\ell} \lesssim n^{(1 - \alpha)/2 } \log^\gamma (n) \cdot f (|A|)$ for some $\gamma \in \mathbb{N}$,\label{unif_ass_thm_0_neu}
        \item $\mathrm{Var} (G_{j_1, \ell}) \lesssim \Gamma_{Y,\ell\ell}^2 \cdot n^{\alpha / 2 - 1}$, \label{unif_ass_thm_3_neu}
        \item $| \cov (\Phi_{j_1, \ell\ell}, \Phi_{j_2, \ell\ell}) | \lesssim \Gamma_{Y, \ell\ell}^2 \cdot \rho^{|j_1 - j_2|}$ for some $0 < \rho < 1$, \label{unif_ass_thm_1_neu}
        \item $|\mathrm{cum} (E_{0, \ell}, E_{0, \ell}, E_{0, \ell}, E_{0, \ell})| \lesssim  \Gamma_{Y, \ell\ell}^2$.  \label{unif_ass_thm_2_neu}
    \end{enumerate}
\end{ass}

\noindent
Condition \ref{unif_ass_thm_0_neu} introduces a rate function $f$, which links the rate of summability of the long-run variances of each coordinate (as a function of the number of coordinates) to the convergence rate of the approximation. Similarly as before, condition \ref{unif_ass_thm_0_neu} is used to bound the remainder term in the martingale approximation, while conditions \ref{unif_ass_thm_1_neu} and \ref{unif_ass_thm_2_neu} ensure that the martingale part converges weakly.

\begin{thm}\label{thm2.0_unif_neu}
    Let Assumption \ref{ass1_uniform_i_neu} be satisfied. Then, on a possibly richer probability space, there exists a sequence of random variables $(\tilde{Y}_{j, \ell})_{j \in \mathbb{Z}}$ with the same distribution as $(Y_{j, \ell})_{j \in \mathbb{Z}}$ for any $\ell \in A$ and a standard Brownian motion $\mathbb{B}$ such that
    \begin{align*}
         \frac{1}{n} \sum_{j=1}^{k} Y_{j,\ell}  ~ =^d ~  \tilde{S}_{n, \ell} (k) := \frac{1}{n} \sum_{j=1}^k \tilde{Y}_{j,\ell},
    \end{align*}
    for any $k \leq n$ and $\ell \in A$, where the partial sums satisfy for any $0 \leq \beta < \alpha / 4$
    \begin{align*}
        \max_{\ell \in A} \sup_{0 \leq \lambda \leq 1} \big| \sqrt{n} \tilde{S}_{n, \ell} (\gbr{n \lambda}) - \mathbb{B} (\lambda \Gamma_{Y,\ell\ell} )  \big| = o_\p \big( n^{-\beta} f(|A|) \big), ~~~~~ \text{ as } n \to \infty.
    \end{align*}
\end{thm}

\noindent
In the following, we will use Theorems \ref{thm2.0} and \ref{thm2.0_unif_neu} frequently for the vectors $Y_j =  X_j - \e [X_j] $ in model \eqref{model}.

\subsection{Statistical guarantees for the methods developed in Section \ref{sec2}}
\label{sec41}

In order to prove the validity of the methodology developed in Section \ref{sec2}, we require some assumptions and additional notation which are introduced first.
Let $\Sigma_h := \cov (X_0, X_h) = (\Sigma_{h,ij}) _{i,j=1, \ldots , p} $  denote the autocovariance matrix between the vectors $X_0$ and $X_h$ with entries $\Sigma_{h,ij} $ and define $\bar \Gamma = \sum_{h \in \mathbb{Z}} |\Sigma_h|$, where, using a slightly unconventional notation, $|M| = (|m_{ij}|)_{i,j=1}^p$ denotes the matrix which is obtained by taking the absolute value of all elements of the matrix $M = (m_{ij})_{i,j=1}^p$. For the asymptotic analysis, we recall that $\delta$ denotes the difference before and after the change point in model \eqref{model} and that $S = \{ \ell = 1,  \ldots , p \mid \delta_\ell \neq 0 \}$ is defined as the set of nonvanishing components of the vector $\delta$. Furthermore, for a set $A \subset \{ 1, \ldots, p \}$ denote by $\bar \Gamma_A = (\bar \Gamma_{\ell_1 \ell_2})_{\ell_1, \ell_2 \in A}$ the submatrix of $\bar \Gamma_A$  defined by the index set $A$.

\begin{ass}\label{asses}
    Assume that the following conditions are satisfied for some set $A \subset \{ 1, \ldots, p\}$ with $A \cap S \neq \emptyset$.
    \begin{enumerate}[label=(C\arabic*)]
        \item $\| \bar \Gamma_A \|_{F}^2 = O(\delta^\top \Gamma_A \delta )$, \label{assA1undA2}
    
        \item $\mathrm{tr}^2(\bar \Gamma_A) = O(\| \bar \Gamma_A \|_F^2)$, \label{assA3}

        \item $\sum_{\ell_1, \ell_2 \in A} | \mathrm{cum} ({X_{0, \ell_1}, X_{h_1, \ell_1}, X_{h_2, \ell_2}, X_{h_3, \ell_2}}) | \lesssim \| \bar \Gamma_A \|_F^2 \cdot \rho^{\max_i h_i - \min_i h_i}$ for some constant $\rho \in (0,1)$, \label{assA4}

        \item $\sum_{h = 0}^m |\delta^\top \Sigma_h \delta| = O(\delta^\top \Gamma_A \delta)$, \label{assA100}

        \item $m \to \infty$ and $m^{3/2} = o(n)$, \label{assA5}

        \item $\mathrm{tr} (\bar \Gamma) = O( \| \delta \|^2_2 )$, \label{assA6}

        \item Assumption \ref{ass1} is satisfied for $\xi = (\delta_\ell)_{\ell \in A}$ and $(X_{j, \ell} - \e [X_{j, \ell}])_{j \in \mathbb{Z}, \ell \in A}$.\label{assA7}

    \end{enumerate}
\end{ass}

Before we continue, we give a brief discussion of these assumptions.
If $A = S \neq \emptyset$ and for at least for one lag $h \in \mathbb{Z}$ and a pair $(i,j) $ the element $\Sigma_{h,ij}$ of the auto-covariance matrix $\Sigma_h$ does not vanish, we obtain by condition \ref{assA1undA2} the estimate $0 < \| \bar \Gamma_S \|^2_F \lesssim s^2 \sigma_{n, S}^2 $ for the normalizing long-run variance in \eqref{variance}. This will be required for the weak convergence in \eqref{det100}, which is the basic result to make self-normalization possible.
Condition \ref{assA6} is required to prove the consistency of the change point estimator in \eqref{cpestimator}. Because this estimator takes into account all coordinates, this condition involves the matrix $\bar \Gamma$ instead of $\bar \Gamma_S$.

Condition \ref{assA5} refers to the trimming parameter $m$ and a similar assumption has been made by \cite{wangshao}. However, an important difference to the the work of these authors consists in  the fact that our approach does not directly rely on physical dependence assumptions such as the Uniform Geometric Moment Condition (UGMC). In fact, condition \ref{assA4} can be verified for UGMC(4) processes if $|A|^2 \lesssim \| \bar \Gamma_A \|_F^2$ \citep[see Remark 9.6 in the Supplementary Material of][]{wangshao}. 
Moreover, a condition similar to our condition \ref{assA4} is stated in \cite{wangvolgushev} but for simpler processes.
These authors consider only linear time series and despite the temporal dependency, this yields a simpler condition involving the cumulant. To control the nonlinear dependence structure in our case, we incorporate the constant $\rho$, which is common for this type of setting as can be seen with the UGMC.
Finally, it is worth mentioning that we only require control of the cumulants up to order four.

\begin{thm}\label{cpthm}
   If condition \ref{assA6} of Assumption \ref{asses} holds, then the estimator $\hat \vartheta_n$  in \eqref{cpestimator} of the change point $k_0 = \lfloor n \vartheta_0 \rfloor$ in model \eqref{model} satisfies
    \begin{align*}
        |\hat \vartheta_n - \vartheta_0| = O_\p \bigg( \frac{\log^2 (n)}{n} \cdot \frac{\mathrm{tr} (\bar \Gamma)}{\| \delta \|_2^2 } \bigg).
    \end{align*}
    In particular, $\hat \vartheta_n$ is consistent for $\vartheta_0$.
\end{thm}

Our main result in this section provides a weak invariance principle for the stochastic process defined in \eqref{det89}. In fact, we prove a slightly more general version, where we replace the set $S$ in \eqref{defSetS} by a an arbitrary set $A \subset \{ 1, \ldots, p \}$ satisfying $A \cap S \neq \emptyset$. As mentioned before, in the case $A = \{ 1 , \ldots, p \}$, this result implies the weak convergence statement in \eqref{det01}, which is the essential component in the proof of the validity of the test defined by \eqref{rule}. Recall that $\{\mathbb{B} (\lambda) \}_{\lambda \in [0,1]}$ defines a standard Brownian motion on the interval $[0,1]$.

\begin{thm}\label{thm2.1}
    If Assumption \ref{asses} is satisfied, then, on a possibly richer probability space, for each $n \in \mathbb{N}$, there exists a Gaussian process $\{ G_{n,A} (\lambda) \}_{\lambda \in [0,1]}$ with the same distribution as $\{ \lambda^3 \mathbb{B} (\lambda) \}_{\lambda \in [0,1]}$, which satisfies
    \begin{align*}
        \sup_{\lambda \in [0,1]} \bigg| \frac{\sqrt{n}}{\sigma_{n, A}} \big(T_{n,A} (\hat k_n, m; \lambda) - \Lambda_0 (\lambda)\| \delta \|^2_{2,A} \big) - G_{n, A} (\lambda ) \bigg| = o_\p (1)
    \end{align*}
    as $n\to \infty$, where
    \begin{align*}
        \sigma_{n,A}^2 = \frac{4}{|A|^2 \vartheta_0 (1-\vartheta_0)} \sum_{i,j \in A} \delta_i \Gamma_{ij} \delta_j,
    \end{align*}
    $\| \delta \|_{2,A}^2 = |A|^{-1} \sum_{\ell \in A} \delta_\ell^2$ and 
    \begin{align}\label{Lambda_0}
        \Lambda_0 (\lambda) := \frac{N_m (\gbr{\lambda k_0})}{N_m (k_0)} \frac{N_m (\gbr{\lambda (n- k_0)})}{N_m (n - k_0)}.
    \end{align}
\end{thm}

\noindent 
An important step in the proof of Theorem \ref{thm2.1} is to establish the stochastic approximation
\begin{align}\label{det20a}
    \begin{split}
       & \sup_{\lambda \in [0,1]} \bigg| \frac{\sqrt{n}}{\sigma_{n, A}} \big(T_{n, A} (\hat k_n, m; \lambda) - \Lambda_0 (\lambda) \| \delta \|^2_{2,A} \big) \\
       & ~~~~~~~~~~~~~~~~~~~~~~~~~~~~~~~~ - \frac{\sqrt{n}}{\sigma_{n, A}}
       \frac{2 \lambda^2 }{N_m (n-\hat k_n) |A|}  \sum_{\substack{j_1,j_2=\hat k_n+1\\|j_1-j_2| > m}}^{\hat k_n + \gbr{\lambda (n-\hat k_n)}} \sum_{\ell \in A} (X_{j_1,\ell} - \e [X_{j_1, \ell}]) \delta_\ell \\
       & ~~~~~~~~~~~~~~~~~~~~~~~~~~~~~~~~ -  \frac{\sqrt{n}}{\sigma_{n, A}} \frac{2 \lambda^2}{N_m (\hat k_n) |A| } \sum_{\substack{i_1,i_2=1\\|i_1-i_2| > m}}^{\gbr{\lambda \hat k_n}} \sum_{\ell \in A} (X_{i_1, \ell} - \e[X_{i_1, \ell}]) \delta_\ell \bigg| = o_\p (1) , 
    \end{split}
\end{align}
for which we require conditions \ref{assA1undA2} - \ref{assA4}, as well as conditions \ref{assA5} and \ref{assA6}. Next, we use Theorem \ref{thm2.0} for the sequence $(X_j - \e [X_j])_{j \in \mathbb{Z}}$ to derive a weak invariance principle for the two sums in equation \eqref{det20a}. For its application in the present context, among other conditions, we require condition \ref{assA7} for the centered random variables $(X_j - \e [X_j])_{j \in \mathbb{Z}}$.

\begin{thm}\label{testconsistent}
    If Assumption \ref{asses} is satisfied for the set $A=\{1, \ldots , p \} $, then the decision rule \eqref{rule} defines an asymptotic level $\alpha$ test. 
    Furthermore, the properties of equation \eqref{teststuff} are satisfied. In particular, the test is consistent if $\frac{\sqrt{n}}{\sigma_n}  ( \Delta - \| \delta \|^2 ) \to  - \infty$.
\end{thm}
Note that Theorem \ref{testconsistent} refers to a fixed model. The final result of this subsection shows that the type I and type II errors of the test \eqref{rule} converge to $0$, uniformly with respect to a certain class of local null and alternative hypotheses, respectively. For its precise statement, we define the sets (for a constant $C>0$) 
 \begin{align} \label{det20b}
    \mathcal{A}_n &:= \{ \delta \in \mathbb{R}^p \mid \| \delta \|^2 > \Delta +  \| \bar \Gamma \|_F / p  \},\\
    \label{det20c}
    \mathcal{B}_n &:= \{ \delta \in \mathbb{R}^p \mid C \leq \| \delta \|^2 < \Delta - \| \bar \Gamma \|_F /p \}.
\end{align}

\begin{thm}\label{newuniformconsistency}
    Let $\sum_{h=0}^\infty \| \Sigma_h \| = o (\|\bar\Gamma \|_F)$, and conditions \ref{assA3}, \ref{assA4} and \ref{assA5} be satisfied for the set $A=\{1, \ldots , p \}$. Then, we have
    \begin{align*}
        \liminf_{n \to \infty} \inf_{\delta \in \mathcal{A}_n} \p \big( T_n (\hat k_n , m) > \Delta +  q_{1-\alpha} V_n \big) = 1,\\
        \limsup_{n \to \infty} \sup_{\delta \in \mathcal{B}_n} \p \big( T_n (\hat k_n , m) > \Delta +  q_{1-\alpha} V_n \big) = 0,
    \end{align*}
    where the sets ${\cal A}_n$ and ${\cal B}_n$ are defined in \eqref{det20b} and \eqref{det20c}, respectively.
\end{thm}

\noindent
Note that Theorem~\ref{newuniformconsistency} requires fewer conditions than Theorem~\ref{testconsistent}, because it does not cover the case $\big| \| \delta \|^2 - \Delta \big| \le \| \bar \Gamma \|_F / p$.

\subsection{Statistical guarantees for the methods developed in Section \ref{secAdaptive}}
\label{subsec_theoretical_prop_unif}

We recall $\Sigma_h = \cov (X_0, X_h)$ and its $\ell$th diagonal entry $\Sigma_{h, \ell\ell}$. Moreover, recall that $s = |S|$ and denote by $s_c = |S^C|$. To prove the consistency of the estimator \eqref{def_S_hat} for the set $S$, we require the following assumptions.

\begin{ass}\label{asses_S_equal_Shat}
    Let $0 < \alpha < 1$ and $(c_\ell)_{\ell =1, \ldots ,p}$ be a nonnegative sequence such that $\sum_{\ell \in S} c_\ell \lesssim n^{1/4} f(s)$ and $\sum_{\ell \in S^C} c_\ell \lesssim n^{\alpha / 8} f_c(s_c)$ for some nondecreasing functions $f, f_c: \mathbb{N} \to (1, \infty ) $. Assume that the following conditions are satisfied.
    \begin{enumerate}[label=(D\arabic*)]
        \item $\max_{\ell \in S} |\delta_\ell| = O(1)$ and $\max_{\ell = 1, \ldots, p} \Gamma_{\ell \ell} = O(1)$, \label{ass_shat_1}
        
        \item $\max_{\ell \in S^C} \big| \sum_{|h| > m} \Sigma_{h, \ell \ell}  \big| = O(n^{- \alpha / 4})$, \label{ass_shat_2}

        \item $\sum_{|h| \leq n } |\Sigma_{h+k, \ell \ell}| \lesssim c_\ell \cdot \gamma_k$ for some sequence $(\gamma_k)_{k \in \mathbb{N}} \in \ell^2 (\mathbb{N})$, \label{ass_shat_4}

        \item there exists a constant  $\rho \in (0,1)$ such that $|\mathrm{cum} ( X_{0, \ell},  X_{h_1, \ell},  X_{h_2, \ell},  X_{h_3, \ell})| \lesssim c_\ell^2 \cdot \rho^{\max_i h_i - \min_i h_i}$ for all $\ell = 1, \ldots, p$, \label{ass_shat_5}
        
        \item $m \to \infty$ and $m^{2}/n = O(n^{-\alpha / 2})$, \label{ass_shat_6}
        
        \item $\mathrm{tr} (\bar \Gamma) = O(\| \delta \|^2_2)$, \label{ass_shat_7}

        \item Assumption \ref{ass1_uniform_i_neu} holds with $\alpha$ and for $(X_{j} - \e [X_{j}])_{j \in \mathbb{Z} }$, the sets $A=S$ and $A=S^C$ and the corresponding functions $f$ and $f_c$, \label{ass_shat_8}

        \item $n^{-\alpha / 8} f_c(s_c) = o(1)$ and $\min_{\ell \in S^C} \Gamma_{\ell \ell} \geq \log^{3} (s_c) / \log^{3/2} (p)$.\label{ass_shat_3}

    \end{enumerate}
\end{ass}

\noindent

Again, we give a brief discussion of these assumptions before continuing. Condition \ref{ass_shat_1} ensures that the variances $\sigma_\ell^2 = 4 \delta_\ell^2 \Gamma_{\ell\ell} / ( \vartheta_0 (1-\vartheta_0))$ for $\ell \in S$ are uniformly bounded.
Condition \ref{ass_shat_2} guarantees that the differences 
$$
    \lambda \vartheta_0 \Gamma_{\ell \ell}  - \frac{1}{n} \sum_{|h| \leq m} \sum_{i = 1}^{\gbr{\lambda \hat k_n} - |h|} ( X_{i, \ell} - \e [X_{i, \ell}])  (X_{i + |h|, \ell} - \e [X_{i + |h|, \ell}])
$$
converge to $0$ in probability, uniformly with respect to $\ell \in S^C$ and $\lambda \in (0,1) $. This is an important step in the derivation of stochastic approximations for the estimators $\hat \delta_\ell^2 (\lambda)$ in \eqref{hat_delta_ell}, which hold uniformly with respect to $\ell \in S^C$ and $\lambda \in [0,1] $, see Theorem S6.2 in the online supplement for details.
Condition \ref{ass_shat_4} gives a bound on the diagonal elements of the autocovariance matrices in terms of the temporal and spatial component.
Note that $\sum_{k=1}^\infty \gamma_k^2 < \infty $ and $\sum_{\ell \in S} c_\ell \lesssim n^{1/4} f(s)$ for a nondecreasing and positive function $f$, where $s$ is the number of nonvanishing coefficients of the vector $\delta$. This function determines the decay of the variance in the spatial component restricted to the set $S$.
Condition \ref{ass_shat_5} is used to control the temporal dependency, while \ref{ass_shat_6} is used to control the bias. It is slightly stronger than the corresponding condition \ref{assA5}, because we require uniform statements. Conditions \ref{ass_shat_7} and \ref{ass_shat_8} are analogs of \ref{assA6} and \ref{assA7}, respectively.
Finally, the technical condition \ref{ass_shat_3} is required in an intermediate step of the proof of the consistency of the estimator in \eqref{def_S_hat}, that is, $\p (\hat S_n \subset S) \to 1$. In particular, it requires a bound on the size of the set $S^C$ of the vanishing components of the signal $\delta$.

\begin{thm}\label{snconsistent}
    Let Assumption \ref{asses_S_equal_Shat} be satisfied and assume that the vector $\delta = (\delta_1, \ldots , \delta_p)^\top $ satisfies 
    \begin{align}\label{fndelta}
        F_n (\delta) := \{  \ell = 1,  \ldots , p \mid 0 < \delta_\ell^2  \leq \zeta_n \}=\emptyset
    \end{align}
    for sufficiently large $n$,
    where \begin{align} \label{det30a}
        \zeta_n = \bigg( 2(n^{-\alpha / 8} f(s) \vee 1) \log (s)  + \max_{\ell \in S} \sigma_\ell \bigg) \cdot\frac{2\log^{3/2} (p)}{ \sqrt{n}} \cdot \sqrt{\log (s)} .
    \end{align}
    Then, the set estimator $\hat S_n$ defined in \eqref{def_S_hat} is consistent, that is, $\p (\hat S_n = S) \to 1$ as $n \to \infty$.
\end{thm}

\begin{thm}\label{adaptive_consistency_thm}
    Let Assumption \ref{asses} for the set $A=S$ and Assumption \ref{asses_S_equal_Shat} be satisfied. If condition \eqref{fndelta} holds, the decision rule \eqref{fully_adaptive_test} defines an asymptotic level $\alpha$ test for the hypotheses \eqref{hypotheses_new}. 
    Furthermore, the statements in equation \eqref{teststuff_S} hold, and in particular, the test is consistent if $\frac{\sqrt{n}}{\sigma_{n, S}} (\Delta - \| \delta \|_{2,0}^2) \to - \infty$.
\end{thm}

\begin{rem} \label{growth_rate}
~~
{\rm
\begin{enumerate}[label=(\alph*)]
    \item The term $\zeta_n$ in \eqref{fndelta} defines a threshold for the detectable entries in the vector $\delta$, and a lower value generally means that the signal can be detected more easily. 
    $\zeta_n$ aggregates various sources of noise, such as the asymptotic variances $\sigma_\ell^2$ of the statistic $\hat \delta_\ell^2$, the long-run variance $\Gamma_{\ell\ell}$ for $\ell \in S$ of the sample (captured by the function $f(s)$ through the condition \ref{ass_shat_4}), but also the dimension $p$ and the size of the set $S$ of nonvanishing components of the vector $\delta$. Increasing the sample size $n$ reduces the impact of these components, but there remains a trade-off between $n$ and the dimension $p$. 
    The term $\sqrt{\log (s)}$ in \eqref{fndelta} stems from the self-normalization approach, reflecting a sub-Gaussian tail behavior of the distribution of $\mathbb{V}= (\int_0^1 \lambda^6 (\mathbb{B} (\lambda) - \lambda \mathbb{B} (1) )^2 \mathrm{d} \lambda)^{1/2}$, which appears in denominator of the limit in \eqref{defG}. The quantity $2\log^{3/2} (p)$ depends on the threshold used in the estimator $\hat S_n$ in the sense that it dominates it.

    \item For a better understanding of the delicate relation between $s$, $p$, $(\delta_\ell)_{\ell \in S}$ and $(\sigma_\ell)_{\ell \in S}$ consider the case $s = p^\theta$ (note that $p$ depends on $n$) and define $\sigma_{S, \infty} := \max_{\ell \in S} \sigma_\ell$. Obviously,
    \begin{align*}
        \zeta_n' := 2 ( n^{-\alpha / 8} f(s) \vee 1 + \sigma_{S,\infty}) \log^{3/2} (s) \frac{2\log^{3/2} (p)}{\sqrt{n}}
    \end{align*}   
    is an upper bound for the quantity $\zeta_n$ defined in \eqref{det30a}. Consequently, \eqref{fndelta} holds if $\zeta_n' < \delta_{\text{min}} := \min_{\ell \in S} \delta_\ell$. A sufficient condition for this inequality is
    $$
        4 \max \{ n^{-\alpha / 8} f(s) \vee 1, \sigma_{S,\infty}  \} \frac{2\theta^{3/2} \log^{3} (p)}{\sqrt{n}} < \delta_{\text{min}},
    $$
    or equivalently
    \begin{align*}
        \begin{split}
            p &< \exp \bigg\{ \bigg( \frac{\delta_{\text{min}} \sqrt{n}}{4 \theta^{3/2} \cdot \max \{ n^{-\alpha / 8} f(s) \vee 1, \sigma_{S, \infty} \} } \bigg)^{1/3} \bigg\} \\
            &= \min \bigg\{ \exp \bigg\{ \bigg( \frac{\delta_{\text{min}} \sqrt{n}}{8 \theta^{3/2} \cdot  (n^{-\alpha / 8} f(s) \vee 1)} \bigg)^{1/3} \bigg\} , \exp \bigg\{ \bigg( \frac{\delta_{\text{min}} \sqrt{n}}{8 \theta^{3/2} \cdot \sigma_{S,\infty} } \bigg)^{1/3} \bigg\} \bigg\}.
        \end{split}
    \end{align*}
    As $\theta = \log (s) / \log (p)$, this means that a higher level of sparsity (smaller values of $s$) yields a larger bound for the dimension for which the set $S$ can be identified asymptotically. The bound also increases with increasing $\delta_{\min}$ and decreasing $\sigma_{S,\infty}$.
    \item A careful inspection of the proofs in the online supplement shows that the results of this section remain true if the exponent $3/2$ in the definition of $\hat S_n$ in \eqref{def_S_hat} is replaced by a constant $\kappa > 1$ and the assumptions are sightly modified.
    \item The comments made in Remark \ref{remark01} (b) and (c) remain valid for the decision rule \eqref{fully_adaptive_test}.
\end{enumerate}
}
\end{rem}

\section{Finite sample study}
\label{sec3} 
\renewcommand{\theequation}{\thesection.\arabic{equation}}
\setcounter{equation}{0}

In this section, we investigate the finite sample properties of the testing procedures developed in Sections~\ref{sec2} and \ref{secAdaptive} by means of a simulation study and a data example.  Further results of the simulation study can be found in the online supplement.

\subsection{Simulation study}
We start by developing an estimator for the trimming parameter $m$ in the statistics $T_n (\hat k_n, m)$ and $T_{n,S} (\hat k_n, m)$ and investigate its finite sample performance. Afterwards, we present some results for the estimator $\hat S_n$ of the set of nonvanishing coordinates of the signal $\delta$. Finally, we illustrate the finite sample properties of the test \eqref{rule} and the fully adaptive test \eqref{fully_adaptive_test} for testing the relevant hypotheses \eqref{det11} with the norms \eqref{det11a} and \eqref{det11b}, respectively.

Throughout this section, we consider an independent (IND), moving average (MA) and autoregressive (AR) model for the sequences $(\eta_j)_{j \in \mathbb{Z}}$ and $(\rho_j)_{j \in \mathbb{Z}}$ in model \eqref{model}. To be precise, let $\{ \varepsilon_j \}_{j \in \mathbb{Z}}$ be an i.i.d sequence of random variables distributed as $\mathcal{N}(0, \mathrm{diag}_p (0.5))$ and $\{ \tilde \varepsilon_j \}_{j \in \mathbb{Z}}$ be an i.i.d sequence of random variables distributed as $\mathcal{N}(0, \Sigma)$, where the $ij$th entry of the matrix  $\Sigma$ is given by $\Sigma_{ij} = 0.5 \cdot 0.9^{|i-j|}$.  Set $f := f_1 = f_2$ (hence $\eta_j =^d \rho_j$ for any $j \in \mathbb{Z}$) in \eqref{model} and for the independent model, let $f (\varepsilon_j, \varepsilon_{j-1},  \ldots ) = \varepsilon_j$.
The MA$(q)$ model for $q = 2,4,6$, is defined as
\begin{align*}
    f(\tilde \varepsilon_{j}, \tilde  \varepsilon_{j-1}, \ldots) = \tilde \varepsilon_j + \sum_{k = 1}^q c_k \tilde \varepsilon_{j-k}, ~~~~~ ~~~~~ j \in \mathbb{Z},
\end{align*}
where $(c_1, \ldots, c_6) = (0.5, 0.25, 0.2, 0.1, 0.05, 0.025)$.
Finally, the AR$_c (1)$ model is defined for $c=0.5, 0.6$ as
\begin{align*}
    \eta_j = c \cdot \eta_{j-1} +  \tilde \varepsilon_j, ~~~~~ ~~~~~ j \in \mathbb{Z}.
\end{align*}

All simulations are based on $1000$ replications and for the tests \eqref{rule} and \eqref{fully_adaptive_test}, we choose the nominal level $\alpha = 0.05$. Moreover,  we consider the following parameters:
the common mean, is set to $\mu = (10, \ldots, 10)^\top$;
the true change point location is $\vartheta_0 = 0.6$;
the entries of the vector $\delta = (\delta_1, \ldots, \delta_s, 0, \ldots, 0)^\top$ are given by  $\delta_1 = \ldots = \delta_s$ depending on the size of $\| \delta \|^2_{2,0}$  (recall that $|S| = s$);
the threshold in the tests \eqref{rule} and \eqref{fully_adaptive_test} are given by $\Delta = 2.0$;
the probability measure $\nu$ in the definition of the self-normalizer in \eqref{vn} is  the discrete uniform distribution $\nu_K$ with support $\{ 1/K, \ldots, (K-1) / K \}$ with $K=20$.

\begin{figure}
    \centering
    \footnotesize
    \subfloat{\includegraphics[width=7cm]{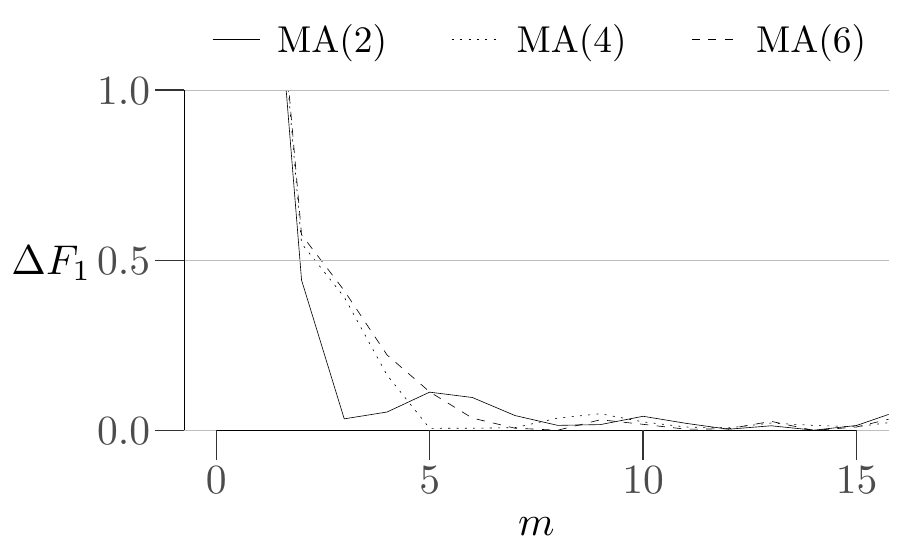}}
    \subfloat{\includegraphics[width=7cm]{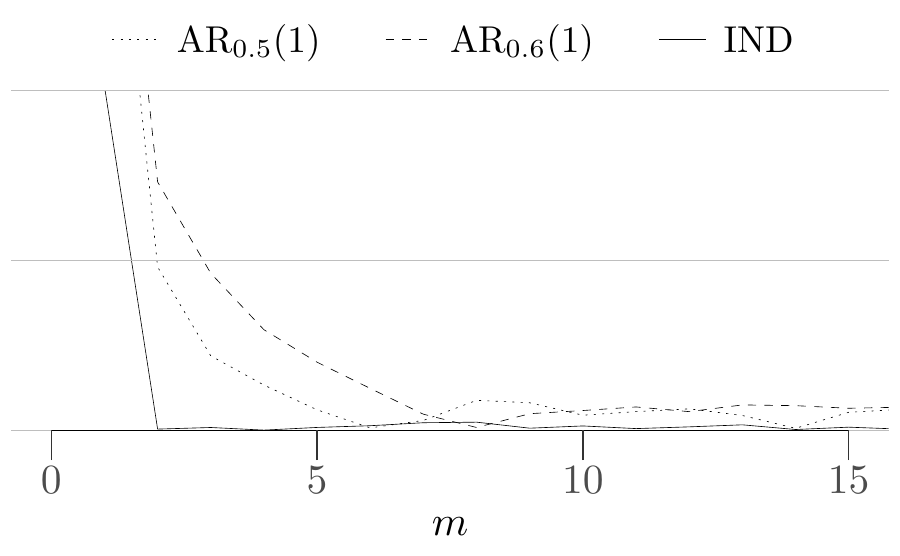}}

    \caption{\it 
        Each curve shows a single realization for the difference function $\Delta F_1$. The parameters $n$, $p$ and $S$ are set to $n = 200$, $p=400$, $S = \{ 1, \ldots, p \}$ and $\| \delta \|^2_{2,0} = 2.0$.
    }
    \label{fig_m}
\end{figure}

\subsection*{Estimation of $m$} 
The performance of the test \eqref{rule} depends on the choice of the trimming parameter $m$, and in real applications it is indispensable to give a data-adaptive rule on how to choose it. We describe one way of doing this in the slightly more general case, where the covariance matrices before and after the change point are not necessarily identical. For this purpose, define the matrices $\Sigma_h^{(1)} = \cov (\eta_0, \eta_h)$ and $\Sigma_h^{(2)} = \cov (\rho_0, \rho_h)$, $h \in \mathbb{Z}$. In order to quantify the sizes of the sums of the traces, recall the approximation of the bias derived in \eqref{tnbias}. We consider the following estimators for these sums of the traces
\begin{align*}
    \tilde F_{1} (m) &:=\frac{1}{N_m (\hat k_n)}\sum_{\substack{j_1, j_2 = 1 \\ |j_1 - j_2| > m}}^{\hat k_n} (X_{j_1} - \bar X_{n,1})^\top (X_{j_2} - \bar X_{n,1}),\\
    \tilde F_{2} (m) &:=  \frac{1}{N_m (n - \hat k_n)} \sum_{\substack{j_1, j_2 = \hat k_n + 1 \\ |j_1 - j_2| > m}}^{n} (X_{j_1} - \bar X_{n,2})^\top (X_{j_2} - \bar X_{n,2}),
\end{align*}
where $\bar X_{n,1}$  and  $\bar X_{n,2}$ denotes the empirical means of $X_1,  \ldots , X_{ \hat k_n}$  and  $X_{\hat k_n + 1},  \ldots , X_n$, respectively. Usually, the terms $\mathrm{tr} (\Sigma_{h}^{(i)})$, $i=1,2$, are larger for small values of $h$ , resulting in a larger impact of the bias, when $m$ is small. Increasing $m$ yields a progressively smaller decrease in bias. In other words, values of $m$, for which $\Delta F_i (m) := F_i(m) - F_i(m - 1)$, $i=1,2$, is close to zero, will significantly reduce the bias. In some cases, a good choice of $m$ can be obtained by visually inspecting the functions $\Delta F_1$ and $\Delta F_2$ as illustrated in Figure~\ref{fig_m}.
For a nonvisual approach, we propose to use
\begin{align}\label{mhat}
    \hat m_n = \max \big\{ \hat m_n^{(1)}, \hat m_n^{(2)} \big\},
\end{align}
where for $i=1,2$
\begin{align*}
    \hat m_n^{(i)} := \min \{ m \leq M_i \mid |\Delta F_i (m)| \leq T \} ~ - ~ 1,
\end{align*}
$M_1, M_2$ are some maximal values for $m$ and $T$ is a cutoff value. In practice, $T = 0.01$ works well for us. We suggest that the values for $M_1$ and $M_2$ should not be too large compared to $\hat k_n$ and $n - \hat k_n$, respectively; otherwise, power loss could be too high. We chose them in such a way that they cut off at most one third of the respective part of the sample, that is, $M_1 = \hat k_n/3$, $M_2 = (n-\hat  k_n)/3$. Note that in some cases, it might be advantageous to use two separate trimming parameters $\hat m_n^{(1)}$ and $\hat m_n^{(2)}$ instead of the maximum of these two, in particular, if they differ substantially.

\begin{table}[!ht]
    \centering
    \begin{tabular}{rc c ccccccccc c c}
        \hline
        &  && \multicolumn{9}{c }{\it Percentage of estimated $m$ larger than} && \\
        \cline{4-12}
        Model & $s/p$&& $\ge2$ & $\ge4$ & $\ge6$ & $\ge8$ & $\ge10$ & $\ge12$ & $\ge14$ & $\ge16$ & $\ge18$ && Mean \\
        \hline
IND            & 1.00 && 0.56 & 0.10 & 0.02 & 0.00 & 0.00 & 0.00 & 0.00 & 0.00 & 0.00 && 1.94  \\
               & 0.05 && 0.52 & 0.09 & 0.01 & 0.00 & 0.00 & 0.00 & 0.00 & 0.00 & 0.00 && 1.89  \\
MA$(2)$        & 1.00 && 1.00 & 0.89 & 0.71 & 0.54 & 0.41 & 0.31 & 0.23 & 0.17 & 0.13 && 9.76  \\
               & 0.05 && 1.00 & 0.88 & 0.71 & 0.55 & 0.43 & 0.31 & 0.23 & 0.16 & 0.12 && 9.72  \\
MA$(4)$        & 1.00 && 1.00 & 1.00 & 0.91 & 0.75 & 0.62 & 0.49 & 0.39 & 0.30 & 0.24 && 13.13 \\
               & 0.05 && 1.00 & 1.00 & 0.90 & 0.76 & 0.62 & 0.49 & 0.39 & 0.30 & 0.24 && 12.99 \\
MA$(6)$        & 1.00 && 1.00 & 1.00 & 0.98 & 0.83 & 0.70 & 0.57 & 0.45 & 0.36 & 0.29 && 14.32 \\
               & 0.05 && 1.00 & 1.00 & 0.98 & 0.85 & 0.73 & 0.61 & 0.50 & 0.40 & 0.32 && 14.88 \\
AR$_{0.5}(1)$  & 1.00 && 1.00 & 1.00 & 0.95 & 0.82 & 0.69 & 0.53 & 0.41 & 0.33 & 0.26 && 13.59 \\
               & 0.05 && 1.00 & 1.00 & 0.96 & 0.84 & 0.68 & 0.53 & 0.40 & 0.32 & 0.24 && 13.53 \\
AR$_{0.6}(1)$  & 1.00 && 1.00 & 1.00 & 1.00 & 0.96 & 0.86 & 0.74 & 0.62 & 0.52 & 0.44 && 17.49 \\
               & 0.05 && 1.00 & 1.00 & 1.00 & 0.96 & 0.88 & 0.75 & 0.65 & 0.54 & 0.45 && 17.75 \\
        \hline
    \end{tabular}
    \caption{\it
        Percentage of replications, where the estimator $\hat m_n$ in \eqref{mhat} is larger than the stated number in the top row. The value in each cell is rounded to two decimals and the parameters are set to $n = 200, p = 400$ and $\| \delta \|^2_{2,0} = 2.0$.
    }
    \label{table:hist_m}
\end{table}

In Table~\ref{table:hist_m}, we show a tabulated version of the distribution function and the mean value of $\hat m_n$, where we consider a sparse ($s/p = 0.05$) and a dense ($s/p = 1.0$) situation. We observe that the estimator $\hat m_n$ is not very sensitive with respect to these two different scenarios. Note that the estimator is not designed to estimate the order of a model, but it is worth mentioning that in all simulation runs $\hat m_n $ never produced a value smaller than $q$ in the MA($q$) models. In general, the distribution of this estimator seems to be right-skewed.

\begin{table}
    \centering
    \begin{tabular}{c c ccccc cccc c c}
        \hline
        & \multicolumn{9}{c}{$m$} && \\
        \cline{3-11}
        && 0 & 1 & 2 & 3 & 4 & 5 & 10 & 15 & 20 && $\hat m_n$ \\
        \hline
        IND            && 0.06 & 0.04 & 0.05 & 0.05 & 0.05 & 0.05 & 0.05 & 0.06 & 0.00 && 0.06 \\
        MA(2)          && 0.08 & 0.07 & 0.06 & 0.06 & 0.07 & 0.06 & 0.08 & 0.06 & 0.00 && 0.05 \\
        MA(4)          && 0.10 & 0.10 & 0.09 & 0.08 & 0.08 & 0.09 & 0.09 & 0.07 & 0.00 && 0.06 \\
        MA(6)          && 0.09 & 0.08 & 0.07 & 0.09 & 0.09 & 0.08 & 0.09 & 0.07 & 0.00 && 0.05 \\
        AR$_{0.5}(1)$  && 0.11 & 0.09 & 0.08 & 0.07 & 0.09 & 0.07 & 0.09 & 0.07 & 0.00 && 0.07 \\
        AR$_{0.6}(1)$  && 0.17 & 0.15 & 0.12 & 0.13 & 0.13 & 0.12 & 0.13 & 0.07 & 0.00 && 0.05 \\
        \hline
    \end{tabular}
    \caption{\it
        Empirical rejection probabilities of the test \eqref{fully_adaptive_test} with $S= \{ 1, \ldots, p \}$ for different choices of $m$, where $n=200, p=400$ and $\| \delta \|^2_{2,0} = 2.0$ (boundary of $H_0$).
 }
    \label{table:m_trim}
\end{table}

In Table~\ref{table:m_trim}, we display the rejection probabilities of the test \eqref{fully_adaptive_test} at the boundary of the hypotheses, that is, $\| \delta \|^2_{2,0} = \Delta$, for different (fixed) values of the trimming parameter $m$ and the estimator $\hat m_n$ (right column), where the nominal level is $\alpha = 0.05$ and the sample size and dimension are given by $n=200$ and $p=400$, respectively. 
Note that for the independent case, the nominal level is well approximated for all values of $m$, as $m=0$ suffices in this case.
From a theoretical point of view, $m=q$ should be sufficient for the MA($q$) processes. However, for the MA($4$) and MA($6$) model larger values are required to obtain a very precise approximation of the nominal level. A similar comment applies to the AR models. If $m$ is too large, the test is conservative in all cases.
In the last column, we show the simulated rejection probabilities of the test \eqref{fully_adaptive_test} with the estimator $\hat m_n$ of the trimming parameter defined in \eqref{mhat}. We observe that the data-adaptive choice of the trimming parameter yields the best results in all cases under consideration and generally improves the performance of
the test.

\subsection*{Estimation of $\vartheta_0$} 
In the online supplement, we provide simulation results for the change point estimator \eqref{cpestimator}. In particular, we investigate the accuracy of $\hat \vartheta_n$ at the boundary of the null hypothesis. Our results show that this estimator performs reasonably well in the case of a sparse signal $\delta$ for small values of $\|\delta \|^2_{2,0}$. The accuracy increases if either $\|\delta \|^2_{2,0}$ or $s/p$ increases. We observe that the dependence of the underlying model has an effect on the performance of the estimator, though only minor.

\begin{table}[!ht]
    \centering
    \begin{tabular}{cc c ccc c ccc c ccc}
        \hline
         & && \multicolumn{3}{c}{Precision} && \multicolumn{3}{c}{Recall} && \multicolumn{3}{c}{F-score} \\
        \cline{4-6} \cline{8-10} \cline{12-14}
         & && \multicolumn{3}{c}{$\| \delta \|_{2,0}^2$} && \multicolumn{3}{c}{$\| \delta \|_{2,0}^2$} && \multicolumn{3}{c}{$\| \delta \|_{2,0}^2$} \\
        \cline{4-6} \cline{8-10} \cline{12-14}
        Model & $s/p$ && 1 & 2 & 3 && 1 & 2 & 3 && 1 & 2 & 3\\
        \hline
        IND & 0.25 && 0.97 & 0.97 & 0.97 && 1.00 & 1.00 & 1.00 && 0.98 & 0.98 & 0.98 \\
            & 0.50 && 0.99 & 0.99 & 0.99 && 1.00 & 1.00 & 1.00 && 0.99 & 0.99 & 0.99 \\
            & 0.75 && 1.00 & 1.00 & 1.00 && 1.00 & 1.00 & 1.00 && 1.00 & 1.00 & 1.00 \\
            & 1.00 && 1.00 & 1.00 & 1.00 && 1.00 & 1.00 & 1.00 && 1.00 & 1.00 & 1.00 \\
        \hline
        MA(2) & 0.25 && 0.95 & 0.96 & 0.96 && 0.79 & 0.92 & 0.93 && 0.86 & 0.94 & 0.94 \\
              & 0.50 && 0.98 & 0.98 & 0.99 && 0.79 & 0.93 & 0.94 && 0.88 & 0.96 & 0.96 \\
              & 0.75 && 0.99 & 1.00 & 1.00 && 0.80 & 0.93 & 0.94 && 0.89 & 0.96 & 0.97 \\
              & 1.00 && 1.00 & 1.00 & 1.00 && 0.79 & 0.93 & 0.94 && 0.89 & 0.97 & 0.97 \\
        \hline
        MA(6) & 0.25 && 0.93 & 0.95 & 0.94 && 0.62 & 0.82 & 0.83 && 0.74 & 0.88 & 0.88 \\
              & 0.50 && 0.97 & 0.98 & 0.98 && 0.62 & 0.78 & 0.82 && 0.76 & 0.87 & 0.90 \\
              & 0.75 && 0.99 & 0.99 & 0.99 && 0.63 & 0.80 & 0.83 && 0.77 & 0.89 & 0.90 \\
              & 1.00 && 1.00 & 1.00 & 1.00 && 0.61 & 0.79 & 0.85 && 0.76 & 0.89 & 0.92 \\
        \hline
        AR$_{0.5}(1)$ & 0.25 && 0.93 & 0.95 & 0.95 && 0.67 & 0.85 & 0.85 && 0.78 & 0.90 & 0.89 \\
                      & 0.50 && 0.98 & 0.98 & 0.98 && 0.65 & 0.83 & 0.85 && 0.78 & 0.90 & 0.91 \\
                      & 0.75 && 0.99 & 0.99 & 0.99 && 0.67 & 0.83 & 0.82 && 0.80 & 0.91 & 0.90 \\
                      & 1.00 && 1.00 & 1.00 & 1.00 && 0.67 & 0.84 & 0.86 && 0.80 & 0.91 & 0.92 \\
        \hline
        AR$_{0.6}(1)$ & 0.25 && 0.91 & 0.93 & 0.94 && 0.47 & 0.68 & 0.76 && 0.62 & 0.78 & 0.84 \\
                      & 0.50 && 0.97 & 0.98 & 0.98 && 0.47 & 0.70 & 0.73 && 0.63 & 0.82 & 0.83 \\
                      & 0.75 && 0.99 & 0.99 & 0.99 && 0.48 & 0.72 & 0.74 && 0.65 & 0.83 & 0.85 \\
                      & 1.00 && 1.00 & 1.00 & 1.00 && 0.49 & 0.70 & 0.74 && 0.65 & 0.82 & 0.85 \\
        \hline
    \end{tabular}
    \caption{\it 
        Precision measures \eqref{det1000} for the estimated set $\hat S_n$ for different values of $\| \delta \|^2_{2,0}$, where $n=200, p= 400$.
    }
    \label{tab:Shat_estimation}
\end{table}

\subsection*{Estimation of the set $S$}
A crucial component of the test \eqref{fully_adaptive_test} is the estimator $\hat S_n$ of the set $S$.
In Table~\ref{tab:Shat_estimation}, we display three metrics to assess the performance of $\hat S_n$, that is,
\begin{align} \label{det1000}
    \mathrm{Precision} = \frac{|S \cap \hat S_n|}{|\hat S_n|}, ~~~~~ \mathrm{Recall} = \frac{|S \cap \hat S_n|}{|S|}, ~~~~~ \text{F-score} = \frac{2 \cdot \mathrm{Precision} \cdot \mathrm{Recall}}{\mathrm{Precision} + \mathrm{Recall}}
\end{align}
(note that in some cases, precision is referred to as \textit{positive predictive value} and recall is referred to as \textit{sensitivity}).
These results suggest that the estimator for $S$ is rather conservative, since the precision is very high in most cases, while the recall remains moderate. Moreover, a stronger dependence inhibits the overall estimation, as the F-score decreases, for example, for the $\mathrm{AR}_{0.6} (1)$ model. In the independent case $\hat S_n$ performs very well even for small values of $\| \delta \|^2_{2,0}$.

\begin{table}[!ht]
    \centering
    \scriptsize
    \begin{tabular}{ccc c cccc c cccc c cc}
        \hline
        &&&& \multicolumn{9}{c}{\eqref{fully_adaptive_test}}  && \multicolumn{2}{c}{} \\
        \cline{5-13} 
        &&&& \multicolumn{4}{c}{$H_0$} && \multicolumn{4}{c}{$H_1$} && \multicolumn{2}{c}{} \\
        \cline{5-8} \cline{10-13}
        &&&& \multicolumn{4}{c}{$s/p$} && \multicolumn{4}{c}{$s/p$} && \multicolumn{2}{c}{\eqref{rule}} \\
        \cline{5-8} \cline{10-13} \cline{15-16}
        Model & $n$ & $p$ && 0.25 & 0.50 & 0.75 & 1.00 && 0.25 & 0.50 & 0.75 & 1.00 && $H_0$ & $H_1$\\
        \hline
        IND & 200 & 200 && 0.00 & 0.01 & 0.03 & 0.05 && 0.66 & 0.99 & 1.00 & 1.00 && 0.05 & 1.00 \\
            &     & 400 && 0.00 & 0.02 & 0.03 & 0.04 && 0.98 & 1.00 & 1.00 & 1.00 && 0.05 & 1.00 \\
            &     & 800 && 0.00 & 0.02 & 0.02 & 0.06 && 1.00 & 1.00 & 1.00 & 1.00 && 0.04 & 1.00 \\
            & 400 & 200 && 0.00 & 0.01 & 0.02 & 0.06 && 0.81 & 1.00 & 1.00 & 1.00 && 0.05 & 1.00 \\
            &     & 400 && 0.00 & 0.01 & 0.02 & 0.05 && 1.00 & 1.00 & 1.00 & 1.00 && 0.04 & 1.00 \\
            &     & 800 && 0.00 & 0.01 & 0.03 & 0.05 && 1.00 & 1.00 & 1.00 & 1.00 && 0.04 & 1.00 \\
        \hline
        MA$(2)$ & 200 & 200 && 0.03 & 0.05 & 0.04 & 0.05 && 0.11 & 0.19 & 0.28 & 0.38 && 0.06 & 0.39 \\
                &     & 400 && 0.03 & 0.04 & 0.08 & 0.05 && 0.17 & 0.33 & 0.46 & 0.57 && 0.05 & 0.63 \\
                &     & 800 && 0.04 & 0.06 & 0.07 & 0.07 && 0.29 & 0.53 & 0.63 & 0.74 && 0.05 & 0.82 \\
                & 400 & 200 && 0.02 & 0.03 & 0.05 & 0.04 && 0.16 & 0.31 & 0.47 & 0.58 && 0.06 & 0.58 \\
                &     & 400 && 0.03 & 0.05 & 0.05 & 0.05 && 0.28 & 0.51 & 0.69 & 0.81 && 0.06 & 0.78 \\
                &     & 800 && 0.03 & 0.04 & 0.05 & 0.05 && 0.44 & 0.75 & 0.88 & 0.95 && 0.05 & 0.94 \\
        \hline
        MA$(6)$ & 200 & 200 && 0.03 & 0.05 & 0.05 & 0.06 && 0.09 & 0.17 & 0.21 & 0.29 && 0.07 & 0.33 \\
                &     & 400 && 0.03 & 0.05 & 0.06 & 0.05 && 0.14 & 0.24 & 0.35 & 0.36 && 0.08 & 0.50 \\
                &     & 800 && 0.03 & 0.05 & 0.05 & 0.05 && 0.20 & 0.31 & 0.38 & 0.40 && 0.07 & 0.69 \\
                & 400 & 200 && 0.03 & 0.03 & 0.06 & 0.06 && 0.12 & 0.23 & 0.39 & 0.51 && 0.06 & 0.46 \\
                &     & 400 && 0.04 & 0.04 & 0.05 & 0.07 && 0.20 & 0.40 & 0.58 & 0.68 && 0.06 & 0.71 \\
                &     & 800 && 0.03 & 0.04 & 0.06 & 0.06 && 0.35 & 0.65 & 0.78 & 0.91 && 0.05 & 0.87 \\
        \hline
        AR$_{0.5}(1)$ & 200 & 200 && 0.03 & 0.05 & 0.04 & 0.05 && 0.09 & 0.17 & 0.24 & 0.30 && 0.07 & 0.37 \\
                      &     & 400 && 0.03 & 0.05 & 0.05 & 0.06 && 0.15 & 0.26 & 0.34 & 0.42 && 0.06 & 0.54 \\
                      &     & 800 && 0.03 & 0.06 & 0.05 & 0.05 && 0.24 & 0.37 & 0.43 & 0.46 && 0.06 & 0.76 \\
                      & 400 & 200 && 0.03 & 0.04 & 0.06 & 0.05 && 0.12 & 0.28 & 0.39 & 0.50 && 0.05 & 0.50 \\
                      &     & 400 && 0.04 & 0.04 & 0.05 & 0.06 && 0.24 & 0.43 & 0.60 & 0.73 && 0.06 & 0.74 \\
                      &     & 800 && 0.03 & 0.05 & 0.05 & 0.06 && 0.39 & 0.71 & 0.82 & 0.92 && 0.06 & 0.91 \\
        \hline
        AR$_{0.6}(1)$ & 200 & 200 && 0.03 & 0.04 & 0.06 & 0.04 && 0.08 & 0.15 & 0.20 & 0.24 && 0.07 & 0.32 \\
                      &     & 400 && 0.04 & 0.05 & 0.06 & 0.06 && 0.13 & 0.18 & 0.23 & 0.26 && 0.07 & 0.42 \\
                      &     & 800 && 0.03 & 0.03 & 0.03 & 0.03 && 0.12 & 0.19 & 0.19 & 0.19 && 0.08 & 0.66 \\
                      & 400 & 200 && 0.03 & 0.04 & 0.06 & 0.06 && 0.23 & 0.26 & 0.30 & 0.40 && 0.06 & 0.42 \\
                      &     & 400 && 0.03 & 0.04 & 0.06 & 0.06 && 0.17 & 0.34 & 0.48 & 0.59 && 0.05 & 0.58 \\
                      &     & 800 && 0.04 & 0.06 & 0.07 & 0.07 && 0.32 & 0.60 & 0.74 & 0.81 && 0.06 & 0.81 \\
        \hline
    \end{tabular}
    \caption{\it
        Rejection probabilities of the tests \eqref{rule} and \eqref{fully_adaptive_test} at the boundary of the hypotheses, that is, $\| \delta \|^2 = \Delta = 2.0$ and $\| \delta \|^2_{2,0} = \Delta = 2.0$, respectively (denoted by $H_0$) and under an alternative, that is  $\| \delta \|^2 = 2.25$ and $\| \delta \|^2_{2,0} = 2.25$, respectively (denoted by $H_1$).
        The columns titled $s/p$ specify the ratio of nonzero entries of $\delta$.
    }
    \label{table:erp_fullyadaptive}
\end{table}

\subsection*{Empirical rejection probabilities}

Next, we investigate the empirical rejection probabilities of the tests \eqref{rule} and \eqref{fully_adaptive_test}, which are displayed in Table~\ref{table:erp_fullyadaptive}. For the null hypothesis, we consider $\| \delta \|^2 = \Delta = 2.0$ and $\| \delta \|^2 = 2.25$ as alternative for the test \eqref{rule}. For the test \eqref{fully_adaptive_test}, we consider $\| \delta \|^2_{2,0} = \Delta = 2.0$ as null hypothesis and $\| \delta \|^2_{2,0} = 2.25$ as an alternative. For both tests, we use the rule \eqref{mhat} to estimate the trimming parameter $m$. In the left (main) part of the table, we show the rejection probabilities of the test \eqref{fully_adaptive_test} for the hypotheses \eqref{det11} with the measure \eqref{det11b}. In the two right columns, we display the results for the test \eqref{rule} using \eqref{det11a}. Note that the results of the tests are not directly comparable, as they refer to different hypotheses (except in the case $s/p=1)$. The table consists of different models with varying sample sizes $n$ and dimensions $p$. 

The overall approximation of the nominal level by the test \eqref{fully_adaptive_test} is good, and in most cases, where the set $\hat S_n$ was estimated, the simulated level is very close to the nominal level $\alpha = 0.05$.
The test is conservative in the independent model if the ratio $s/p$ is small. Observing the right part of the table, we see that the test \eqref{rule} is approximating its nominal level in all cases under consideration.

Both tests show a reasonable performance under the alternative.
In almost all cases, the power increases with the dimension $p$. When the dependence is too strong, this effect can diminish as can be observed for the AR$_{0.6}(1)$ model, when $n=200$ and $p=800$. Here, the signal-to-noise ratio is too weak for $\hat S_n$ to be a good estimate of $S$ (cf. also Table~\ref{tab:Shat_estimation}). This is also confirmed by the performance of the test \eqref{rule}, where this set is not estimated. The power of both tests decreases with a stronger dependency.
If it is known that the change vector $\delta$ is dense, i.e. $s/p=1$, both tests refer to the same hypotheses and in this case the test \eqref{rule} should be preferred, since it has more power than the test \eqref{fully_adaptive_test}.

\subsection{Application to pollution data}\label{secData_example}

We conclude this section by presenting a real-world data example obtained via the OpenAQ application programming interface (API, see \url{openaq.org}). 
The data set we use (accessed on 24 August 2025) consists of measurements of the pollutants $\mathrm{PM}_{10}$ and $\mathrm{PM}_{2.5}$ (in $\mu g / m^3$) from different stations in southern Europe, that is, the Iberian Peninsula and most of France. The time series data are measured from 17 November 2016 to 21 August 2025 in a weekly interval, which results in $n = 458$ measurements. For the first pollutant, our data include the data of 800 measurement stations and 581 stations measuring the second pollutant (each station contributes one measurement per week), resulting in a dimension $p = 1381$. Missing data points were interpolated and faulty (negative) values were replaced by zeros. 
As mentioned in Section \ref{sec1}, we square the hypotheses in \eqref{det11} for algebraic convenience. In this section, we replace $\Delta$ by $\Delta^2$ and test the hypotheses
\begin{align*}
    H_0^\Delta: \| \delta \|_{2,0} \leq \Delta ~~~ \text{versus} ~~~ H_1^\Delta: \| \delta \|_{2,0} > \Delta
\end{align*}
(and those for $\| \delta \|$ analogously). Table~\ref{tab:dataex_estimators} lists all relevant parameters and estimators of this data set, where we note that the values for $\hat m_n^{(1)}$ and $\hat m_n^{(2)}$ have been chosen by visual inspection of the difference functions $\Delta F_1$ and $\Delta F_2$, respectively. The location of the change point $\hat k_n = 175$ corresponds to 19 March 2020.
The small value of $\hat s_n$ indicates that the significant contribution to the change in the mean function is due to only a small number of stations. 
Since $\hat S_n$ is a conservative estimator for $S$, we note that one cannot infer that the stations included in the set $\hat S_n$ are the only stations that exhibit a change.

\begin{table}
    \centering
    \begin{tabular}{ cc c cccc c cc c cc }
        \hline
        $n$ & $p$ && $\hat k_n$ & $\hat s_n$ & $\hat m_n^{(1)}$ & $\hat m_n^{(2)}$ && $T_{n, \hat S_n}$ & $V_{n, \hat S_n}$ && $T_n$ & $V_n$ \\
        \hline
        458 & 1381 && 175  & 61 &  15 & 18 &&  2364.36 & 11.31 && 134.31 & 0.50 \\
        \hline
    \end{tabular}
    \caption{\it 
        Parameters and estimators of the pollution data set, related to the application of the tests \eqref{rule} and \eqref{fully_adaptive_test}.
    }
    \label{tab:dataex_estimators}
\end{table}

We applied both tests, that is, \eqref{rule} and \eqref{fully_adaptive_test} to this data set and obtained significant results in both cases. The results consist of $\Delta_{\max}^2$ and two confidence intervals (for $\| \delta \|^2_{2,0}$ and $\| \delta \|^2$) as described in Remark~\ref{remark01}. We display their square roots in Table~\ref{tab:dataex_testresults} to obtain results for $\| \delta \|_{2,0}$ and $\| \delta \|$. The positive values for $\Delta_{\max}$ indicate that the values for $\| \delta \|_{2,0}$ and $\| \delta \|$ are significant and any null hypothesis where $\Delta < \Delta_{\max}$ will be rejected.
These results are in line with the significant increase in pollutants $\mathrm{PM}_{10}$ and $\mathrm{PM}_{2.5}$ as reported by governmental sources, which is due to a Sahara desert dust wave that occurred in March 2020 (see \url{www.miteco.gob.es} and \url{www.atmo-france.org}).

\begin{table}
    \centering
    \begin{tabular}{ll c ccc }
        \hline
        & & & \multicolumn{3}{c}{$\alpha$} \\
        \cline{4-6} 
        Test & Quantity && 10\% & 5\% & 1\% \\
        \hline
        \eqref{fully_adaptive_test} & $\Delta_{\max}$ && 46.96          & 46.26          & 44.48 \\
                                  & upper CI          && [0, 50.23]     & [0, 50.88]     & [0, 52.45] \\
                                  & two-sided CI      && [46.26, 50.88] & [45.41, 51.64] & [43.95, 52.89] \\
        \hline
        \eqref{rule} & $\Delta_{\max}$ && 11.28          & 11.15          & 10.83 \\
                     & upper CI        && [0, 11.89]     & [0, 12.01]     & [0, 12.30] \\
                     & two-sided CI    && [11.15, 12.01] & [11.00, 12.15] & [10.73, 12.39] \\
        \hline
    \end{tabular}
    \caption{\it
        Quantities based on the pollution data set for the tests \eqref{rule} and \eqref{fully_adaptive_test} at various nominal levels.
    }
    \label{tab:dataex_testresults}
\end{table}






\begin{funding}
The work of Pascal Quanz and Holger Dette has been partially supported by the Deutsche Forschungsgemeinschaft (DFG), project number 45723897 and by TRR 391 \textit{Spatio-temporal Statistics for the Transition of Energy and Transport}, project number 520388526 (DFG), respectively. 
We also acknowledge support for computational resources from the DFG under \textit{Germany’s Excellence Strategy – EXC 2092 CASA – 390781972}.
\end{funding}

\begin{appendix}

\renewcommand{\theequation}{A.\arabic{equation}}
\setcounter{equation}{0}

\section*{Tail probabilities}

\begin{lem}\label{lemmasubexp}
    Let $\mathbb{B}$ be a standard Brownian motion, $\alpha \geq 0$ and consider the random variable
    \begin{align}\label{V_alpha_def}
        \mathbb{V}_\alpha = \bigg( \int_0^1 \lambda^{\alpha} (\mathbb{B} (\lambda) - \lambda \mathbb{B} (1))^2 \mathrm{d}\lambda \bigg)^{1/2}
    \end{align}
    \begin{enumerate}[label=(\roman*)]
        \item The characteristic function of $\mathbb{V}_\alpha^2$ is given by
        \begin{align*}
            \varphi_{\mathbb{V}_\alpha^2} (t) = \bigg( \frac{(1 + i)\sqrt{t}}{\alpha + 2}  \bigg)^{\frac{1}{2(\alpha + 2)}} \bigg( \Gamma \bigg( \frac{\alpha + 3}{\alpha + 2} \bigg) J_{1/(\alpha + 2)} \bigg( \frac{2(1+i) \sqrt{t}}{\alpha + 2} \bigg)  \bigg)^{-1/2},
        \end{align*}
        where $J_\gamma$ is the Bessel function of the first kind of order $\gamma$. Moreover, $\mathbb{V}_\alpha$ has a Lebesgue-density and is positive with probability $1$.
        \item The random variables $\mathbb{V_{\alpha}}^{-1}$ and $\mathbb{G}_\alpha = \mathbb{B}(1) / \mathbb{V}_\alpha$ have subexponential distributions.
        \item Let $\mathbb{V}_\alpha^{(1)}, \ldots, \mathbb{V}_\alpha^{(p)}$ have the same distribution as $\mathbb{V}_\alpha$. Then, as $p \to \infty$
        \begin{align*}
            \max_{\ell = 1, \ldots, p} \mathbb{V}_\alpha^{(\ell)} = O_\p (\sqrt{\log (p)}).
        \end{align*}
    \end{enumerate}
\end{lem}

\begin{lem}\label{lemmasubsubexp}
    Let $\mathbb{B}$ be a Brownian motion, $\mathbb{M} (\lambda) = \mathbb{B} (\lambda)^2 - \lambda$, $\alpha \geq 1$ and consider the random variables
    \begin{align}\label{WA}
        \mathbb{W}_\alpha = \bigg( \int_0^1 \lambda^\alpha \big( \mathbb{M} (\lambda) - \lambda^2 \mathbb{M} (1) \big)^2 \mathrm{d}\lambda \bigg)^{1/2}
    \end{align}
    and $\mathbb{H}_\alpha = \mathbb{M} (1) / \mathbb{W}_\alpha$. Then, for sufficiently large $t$ we have the following bounds
    \begin{enumerate}[label=(\roman*)]
        \item $\e [\exp (-t \mathbb{W}_\alpha^2)] \le 9 \exp (-C_\alpha t^{1/4})$,
        \item $\p (|\mathbb{H}_\alpha| > t) \le 11 \exp (-D_\alpha t^{2/5})$,
    \end{enumerate}
    where $C_\alpha = (16 (\alpha + 3))^{-1}$ and $D_\alpha = 3 \cdot (C_\alpha / 4)^{4/3}$.
\end{lem}

\end{appendix}

\bibliographystyle{apalike}
\bibliography{lit}

\begin{supplement}

\stitle{Supplement to \enquote{Practically significant change points in high dimension - measuring signal strength pro active component}}


\end{supplement}
\setcounter{section}{0}

\input{supplement_content}





\end{document}

%% file: supplement_content.tex
\renewcommand\thesection{S\arabic{section}}
\renewcommand\thesubsection{\thesection.\arabic{subsection}}
\renewcommand\thesubsubsection{\thesubsection.\arabic{subsubsection}}

\section{Additional simulation results}

In this section, we present additional simulation results, complementing the results from Section \ref{sec3}.

\subsection*{Accuracy of the change point estimator $\hat \vartheta_n$}

In Tables~\ref{table:cp_heatmap_ind}, \ref{table:cp_heatmap_ma2} and \ref{table:cp_heatmap_ma4}, we provide accuracy results for the change point estimator for various models. The value in each cell represents the proportion
\begin{align*}
    | \{ | \hat \vartheta_n^{(i)} - 0.6 | \le 0.05  \mid i = 1, \ldots, 1000 \}| / 1000
\end{align*}
of replications, where the change point $\vartheta_0 = 0.6$ is estimated (almost) correctly. We consider a deviation by 0.05 a precise estimate, as minor changes of the estimate for $\vartheta_0$ do not significantly influence the outcome of the tests. Moreover, we emphasize that for almost all combinations of $\| \delta \|_{2,0}^2$ and $s/p$ in the tables, the value $\vartheta_0 = 0.6$ is attained most frequently and, as mentioned in Section \ref{sec3}, these results show that for small values of $\| \delta \|^2_{2,0}$ and $s/p$, the estimate $\hat \vartheta_n$ for $\vartheta_0$ is very precise. This justifies the use of this estimator even in the case $s \ll p$. As can be observed, the accuracy is better for the independent case and becomes slightly worse for the dependent models. Note that the values for $\| \delta \|_{2,0}^2$ and $s/p$ are still quite small.

\begin{table}[!ht]
    \centering
    $s/p$
    \begin{tabular}{c|*{13}{c} }
        0.60 & \gcell{0.04} & \gcell{0.37} & \gcell{0.98} & \gcell{1.00} & \gcell{1.00} & \gcell{1.00} & \gcell{1.00} & \gcell{1.00} & \gcell{1.00} & \gcell{1.00} & \gcell{1.00} & \gcell{1.00} & \gcell{1.00}\\
        0.55 & \gcell{0.04} & \gcell{0.37} & \gcell{0.98} & \gcell{1.00} & \gcell{1.00} & \gcell{1.00} & \gcell{1.00} & \gcell{1.00} & \gcell{1.00} & \gcell{1.00} & \gcell{1.00} & \gcell{1.00} & \gcell{1.00}\\
        0.50 & \gcell{0.04} & \gcell{0.32} & \gcell{0.94} & \gcell{1.00} & \gcell{1.00} & \gcell{1.00} & \gcell{1.00} & \gcell{1.00} & \gcell{1.00} & \gcell{1.00} & \gcell{1.00} & \gcell{1.00} & \gcell{1.00}\\
        0.45 & \gcell{0.05} & \gcell{0.32} & \gcell{0.92} & \gcell{1.00} & \gcell{1.00} & \gcell{1.00} & \gcell{1.00} & \gcell{1.00} & \gcell{1.00} & \gcell{1.00} & \gcell{1.00} & \gcell{1.00} & \gcell{1.00}\\
        0.40 & \gcell{0.04} & \gcell{0.26} & \gcell{0.88} & \gcell{1.00} & \gcell{1.00} & \gcell{1.00} & \gcell{1.00} & \gcell{1.00} & \gcell{1.00} & \gcell{1.00} & \gcell{1.00} & \gcell{1.00} & \gcell{1.00}\\
        0.35 & \gcell{0.05} & \gcell{0.23} & \gcell{0.84} & \gcell{1.00} & \gcell{1.00} & \gcell{1.00} & \gcell{1.00} & \gcell{1.00} & \gcell{1.00} & \gcell{1.00} & \gcell{1.00} & \gcell{1.00} & \gcell{1.00}\\
        0.30 & \gcell{0.06} & \gcell{0.24} & \gcell{0.77} & \gcell{0.98} & \gcell{1.00} & \gcell{1.00} & \gcell{1.00} & \gcell{1.00} & \gcell{1.00} & \gcell{1.00} & \gcell{1.00} & \gcell{1.00} & \gcell{1.00}\\
        0.25 & \gcell{0.05} & \gcell{0.17} & \gcell{0.67} & \gcell{0.97} & \gcell{1.00} & \gcell{1.00} & \gcell{1.00} & \gcell{1.00} & \gcell{1.00} & \gcell{1.00} & \gcell{1.00} & \gcell{1.00} & \gcell{1.00}\\
        0.20 & \gcell{0.04} & \gcell{0.16} & \gcell{0.52} & \gcell{0.93} & \gcell{1.00} & \gcell{1.00} & \gcell{1.00} & \gcell{1.00} & \gcell{1.00} & \gcell{1.00} & \gcell{1.00} & \gcell{1.00} & \gcell{1.00}\\
        0.15 & \gcell{0.06} & \gcell{0.10} & \gcell{0.34} & \gcell{0.80} & \gcell{0.98} & \gcell{1.00} & \gcell{1.00} & \gcell{1.00} & \gcell{1.00} & \gcell{1.00} & \gcell{1.00} & \gcell{1.00} & \gcell{1.00}\\
        0.10 & \gcell{0.04} & \gcell{0.08} & \gcell{0.24} & \gcell{0.56} & \gcell{0.88} & \gcell{0.98} & \gcell{1.00} & \gcell{1.00} & \gcell{1.00} & \gcell{1.00} & \gcell{1.00} & \gcell{1.00} & \gcell{1.00}\\
        0.05 & \gcell{0.03} & \gcell{0.08} & \gcell{0.11} & \gcell{0.26} & \gcell{0.52} & \gcell{0.79} & \gcell{0.92} & \gcell{0.98} & \gcell{1.00} & \gcell{1.00} & \gcell{1.00} & \gcell{1.00} & \gcell{1.00}\\
        0.00 & \gcell{0.05} & \gcell{0.05} & \gcell{0.04} & \gcell{0.06} & \gcell{0.04} & \gcell{0.05} & \gcell{0.05} & \gcell{0.05} & \gcell{0.04} & \gcell{0.06} & \gcell{0.05} & \gcell{0.04} & \gcell{0.05}\\
          \hline
        \multicolumn{1}{c|}{} & 0.00 & 0.05 & 0.10 & 0.15 & 0.20 & 0.25 & 0.30 & 0.35 & 0.40 & 0.45 & 0.50 & 0.55 &  0.60 \\
        \multicolumn{1}{c}{} & \multicolumn{13}{c}{$\| \delta \|^2_{2,0}$}
    \end{tabular}
    \caption{\it
        Accuracy of the change point estimator (for sparse data) in the independent model, where $n = 200$ and $p = 400$.
    }
    \label{table:cp_heatmap_ind}
\end{table}

\begin{table}[!ht]
    \centering
    $s/p$
    \begin{tabular}{c|*{13}{c} }
        0.60 & \gcell{0.08} & \gcell{0.09} & \gcell{0.18} & \gcell{0.33} & \gcell{0.53} & \gcell{0.69} & \gcell{0.83} & \gcell{0.91} & \gcell{0.94} & \gcell{0.97} & \gcell{0.99} & \gcell{1.00} & \gcell{1.00}\\
        0.55 & \gcell{0.07} & \gcell{0.09} & \gcell{0.16} & \gcell{0.29} & \gcell{0.50} & \gcell{0.64} & \gcell{0.80} & \gcell{0.94} & \gcell{0.95} & \gcell{0.96} & \gcell{0.98} & \gcell{0.99} & \gcell{1.00}\\
        0.50 & \gcell{0.08} & \gcell{0.09} & \gcell{0.14} & \gcell{0.27} & \gcell{0.44} & \gcell{0.64} & \gcell{0.80} & \gcell{0.87} & \gcell{0.94} & \gcell{0.95} & \gcell{0.98} & \gcell{0.99} & \gcell{0.99}\\
        0.45 & \gcell{0.06} & \gcell{0.09} & \gcell{0.17} & \gcell{0.25} & \gcell{0.40} & \gcell{0.60} & \gcell{0.75} & \gcell{0.83} & \gcell{0.90} & \gcell{0.94} & \gcell{0.95} & \gcell{0.99} & \gcell{0.99}\\
        0.40 & \gcell{0.08} & \gcell{0.08} & \gcell{0.15} & \gcell{0.24} & \gcell{0.36} & \gcell{0.57} & \gcell{0.67} & \gcell{0.79} & \gcell{0.87} & \gcell{0.94} & \gcell{0.95} & \gcell{0.98} & \gcell{0.98}\\
        0.35 & \gcell{0.09} & \gcell{0.11} & \gcell{0.14} & \gcell{0.22} & \gcell{0.34} & \gcell{0.48} & \gcell{0.63} & \gcell{0.75} & \gcell{0.83} & \gcell{0.90} & \gcell{0.92} & \gcell{0.97} & \gcell{0.98}\\
        0.30 & \gcell{0.06} & \gcell{0.10} & \gcell{0.12} & \gcell{0.21} & \gcell{0.31} & \gcell{0.43} & \gcell{0.62} & \gcell{0.70} & \gcell{0.80} & \gcell{0.88} & \gcell{0.91} & \gcell{0.94} & \gcell{0.96}\\
        0.25 & \gcell{0.05} & \gcell{0.07} & \gcell{0.10} & \gcell{0.17} & \gcell{0.26} & \gcell{0.39} & \gcell{0.51} & \gcell{0.63} & \gcell{0.74} & \gcell{0.83} & \gcell{0.87} & \gcell{0.93} & \gcell{0.96}\\
        0.20 & \gcell{0.08} & \gcell{0.10} & \gcell{0.12} & \gcell{0.12} & \gcell{0.21} & \gcell{0.32} & \gcell{0.42} & \gcell{0.55} & \gcell{0.64} & \gcell{0.75} & \gcell{0.80} & \gcell{0.85} & \gcell{0.92}\\
        0.15 & \gcell{0.08} & \gcell{0.07} & \gcell{0.09} & \gcell{0.13} & \gcell{0.17} & \gcell{0.24} & \gcell{0.32} & \gcell{0.42} & \gcell{0.57} & \gcell{0.64} & \gcell{0.74} & \gcell{0.79} & \gcell{0.85}\\
        0.10 & \gcell{0.07} & \gcell{0.06} & \gcell{0.08} & \gcell{0.10} & \gcell{0.12} & \gcell{0.17} & \gcell{0.24} & \gcell{0.28} & \gcell{0.37} & \gcell{0.47} & \gcell{0.59} & \gcell{0.63} & \gcell{0.70}\\
        0.05 & \gcell{0.07} & \gcell{0.09} & \gcell{0.08} & \gcell{0.09} & \gcell{0.12} & \gcell{0.14} & \gcell{0.15} & \gcell{0.18} & \gcell{0.22} & \gcell{0.25} & \gcell{0.32} & \gcell{0.37} & \gcell{0.44}\\
        0.00 & \gcell{0.08} & \gcell{0.07} & \gcell{0.07} & \gcell{0.08} & \gcell{0.08} & \gcell{0.07} & \gcell{0.08} & \gcell{0.07} & \gcell{0.07} & \gcell{0.05} & \gcell{0.07} & \gcell{0.06} & \gcell{0.07}\\
        \hline
        \multicolumn{1}{c|}{} & 0.00 & 0.05 & 0.10 & 0.15 & 0.20 & 0.25 & 0.30 & 0.35 & 0.40 & 0.45 & 0.50 & 0.55 &  0.60 \\
        \multicolumn{1}{c}{} & \multicolumn{13}{c}{$\| \delta \|^2_{2,0}$}
    \end{tabular}
    \caption{\it
        Accuracy of the change point estimator (for sparse data) in the $\mathrm{MA}(2)$ model, where $n = 200$ and $p = 400$.
    }
    \label{table:cp_heatmap_ma2}
\end{table}

\begin{table}[!ht]
    \centering
    $s/p$
    \begin{tabular}{c|*{13}{c} }        
        0.60 & \gcell{0.08} & \gcell{0.09} & \gcell{0.17} & \gcell{0.27} & \gcell{0.42} & \gcell{0.58} & \gcell{0.72} & \gcell{0.82} & \gcell{0.90} & \gcell{0.94} & \gcell{0.97} & \gcell{0.99} & \gcell{0.99}\\
        0.55 & \gcell{0.08} & \gcell{0.09} & \gcell{0.15} & \gcell{0.27} & \gcell{0.40} & \gcell{0.56} & \gcell{0.72} & \gcell{0.82} & \gcell{0.89} & \gcell{0.94} & \gcell{0.96} & \gcell{0.97} & \gcell{0.99}\\
        0.50 & \gcell{0.10} & \gcell{0.09} & \gcell{0.16} & \gcell{0.25} & \gcell{0.34} & \gcell{0.52} & \gcell{0.64} & \gcell{0.78} & \gcell{0.86} & \gcell{0.92} & \gcell{0.95} & \gcell{0.97} & \gcell{0.98}\\
        0.45 & \gcell{0.07} & \gcell{0.09} & \gcell{0.14} & \gcell{0.24} & \gcell{0.35} & \gcell{0.49} & \gcell{0.60} & \gcell{0.74} & \gcell{0.83} & \gcell{0.89} & \gcell{0.94} & \gcell{0.95} & \gcell{0.97}\\
        0.40 & \gcell{0.10} & \gcell{0.10} & \gcell{0.12} & \gcell{0.19} & \gcell{0.33} & \gcell{0.45} & \gcell{0.59} & \gcell{0.69} & \gcell{0.78} & \gcell{0.88} & \gcell{0.92} & \gcell{0.95} & \gcell{0.97}\\
        0.35 & \gcell{0.09} & \gcell{0.08} & \gcell{0.14} & \gcell{0.19} & \gcell{0.27} & \gcell{0.39} & \gcell{0.50} & \gcell{0.66} & \gcell{0.73} & \gcell{0.83} & \gcell{0.88} & \gcell{0.93} & \gcell{0.95}\\
        0.30 & \gcell{0.07} & \gcell{0.09} & \gcell{0.11} & \gcell{0.18} & \gcell{0.25} & \gcell{0.34} & \gcell{0.46} & \gcell{0.60} & \gcell{0.71} & \gcell{0.78} & \gcell{0.85} & \gcell{0.90} & \gcell{0.92}\\
        0.25 & \gcell{0.07} & \gcell{0.09} & \gcell{0.11} & \gcell{0.14} & \gcell{0.22} & \gcell{0.31} & \gcell{0.38} & \gcell{0.52} & \gcell{0.63} & \gcell{0.72} & \gcell{0.78} & \gcell{0.86} & \gcell{0.88}\\
        0.20 & \gcell{0.07} & \gcell{0.10} & \gcell{0.10} & \gcell{0.16} & \gcell{0.17} & \gcell{0.28} & \gcell{0.33} & \gcell{0.44} & \gcell{0.52} & \gcell{0.62} & \gcell{0.71} & \gcell{0.79} & \gcell{0.84}\\
        0.15 & \gcell{0.08} & \gcell{0.09} & \gcell{0.10} & \gcell{0.12} & \gcell{0.18} & \gcell{0.19} & \gcell{0.31} & \gcell{0.37} & \gcell{0.43} & \gcell{0.53} & \gcell{0.59} & \gcell{0.68} & \gcell{0.75}\\
        0.10 & \gcell{0.10} & \gcell{0.08} & \gcell{0.08} & \gcell{0.12} & \gcell{0.14} & \gcell{0.16} & \gcell{0.20} & \gcell{0.27} & \gcell{0.30} & \gcell{0.36} & \gcell{0.44} & \gcell{0.51} & \gcell{0.59}\\
        0.05 & \gcell{0.10} & \gcell{0.06} & \gcell{0.08} & \gcell{0.09} & \gcell{0.10} & \gcell{0.13} & \gcell{0.14} & \gcell{0.17} & \gcell{0.18} & \gcell{0.21} & \gcell{0.25} & \gcell{0.28} & \gcell{0.36}\\
        0.00 & \gcell{0.08} & \gcell{0.09} & \gcell{0.09} & \gcell{0.09} & \gcell{0.07} & \gcell{0.07} & \gcell{0.08} & \gcell{0.09} & \gcell{0.10} & \gcell{0.08} & \gcell{0.07} & \gcell{0.09} & \gcell{0.09}\\
          \hline
        \multicolumn{1}{c|}{} & 0.00 & 0.05 & 0.10 & 0.15 & 0.20 & 0.25 & 0.30 & 0.35 & 0.40 & 0.45 & 0.50 & 0.55 &  0.60 \\
        \multicolumn{1}{c}{} & \multicolumn{13}{c}{$\| \delta \|^2_{2,0}$}
    \end{tabular}
    \caption{\it
        Accuracy of the change point estimator (for sparse data) in the $\mathrm{MA}(4)$ model, where $n = 200$ and $p = 400$.
    }
    \label{table:cp_heatmap_ma4}
\end{table}

\subsection*{Sensitivity with respect to the choice of the measure $\nu$ in the self-normalizing statistic}

In the definition of the self-normalizing statistic $V_n = V_n (\nu)$, we choose $\nu$ as the discrete uniform probability measure on the set $\{ 1/K, \ldots,~{(K-1) / K} \}$. Table~\ref{tab:sensitivityK} shows that the choice of $K$ has a very small impact on the the performance of the test \eqref{rule}. Note that for a different choice of $K$, the quantile $q_{1-\alpha}$ in the decision rules \eqref{rule} and \eqref{fully_adaptive_test} must also be adapted (see Table~\ref{quantiletable}).

\begin{table}[!ht]
    \centering
    \begin{tabular}{ccc c cccc c cccc }
        \hline
        & \multicolumn{3}{c}{} & \multicolumn{4}{c}{$H_0$} && \multicolumn{4}{c}{$H_1$} \\
        \cline{5-8} \cline{10-13}
        & \multicolumn{3}{c}{} & \multicolumn{4}{c}{$K$} && \multicolumn{4}{c}{$K$} \\
        \cline{5-8} \cline{10-13}
        Model & $n$ & $p$ && 10 & 15 & 20 & 25 && 10 & 15 & 20 & 25 \\
        \hline
        IND & 200 & 200 && 0.05 & 0.04 & 0.05 & 0.05 && 1.00 & 1.00 & 1.00 & 1.00 \\
            &     & 400 && 0.05 & 0.04 & 0.05 & 0.06 && 1.00 & 1.00 & 1.00 & 1.00 \\
            &     & 800 && 0.06 & 0.05 & 0.04 & 0.05 && 1.00 & 1.00 & 1.00 & 1.00 \\
            & 400 & 200 && 0.06 & 0.05 & 0.05 & 0.05 && 1.00 & 1.00 & 1.00 & 1.00 \\
            &     & 400 && 0.05 & 0.04 & 0.04 & 0.05 && 1.00 & 1.00 & 1.00 & 1.00 \\
            &     & 800 && 0.05 & 0.05 & 0.04 & 0.05 && 1.00 & 1.00 & 1.00 & 1.00 \\
        \hline
        MA$(2)$ & 200 & 200 && 0.06 & 0.07 & 0.06 & 0.06 && 0.36 & 0.40 & 0.39 & 0.38 \\
                &     & 400 && 0.07 & 0.06 & 0.05 & 0.06 && 0.55 & 0.57 & 0.63 & 0.62 \\
                &     & 800 && 0.06 & 0.06 & 0.05 & 0.05 && 0.75 & 0.80 & 0.82 & 0.81 \\
                & 400 & 200 && 0.05 & 0.04 & 0.06 & 0.06 && 0.52 & 0.58 & 0.58 & 0.57 \\
                &     & 400 && 0.05 & 0.05 & 0.06 & 0.05 && 0.75 & 0.77 & 0.78 & 0.78 \\
                &     & 800 && 0.05 & 0.06 & 0.05 & 0.05 && 0.91 & 0.94 & 0.94 & 0.95 \\
        \hline
        MA$(6)$ & 200 & 200 && 0.08 & 0.06 & 0.07 & 0.08 && 0.35 & 0.31 & 0.33 & 0.38 \\
                &     & 400 && 0.07 & 0.07 & 0.08 & 0.07 && 0.47 & 0.49 & 0.50 & 0.50 \\
                &     & 800 && 0.07 & 0.08 & 0.07 & 0.07 && 0.67 & 0.69 & 0.69 & 0.71 \\
                & 400 & 200 && 0.07 & 0.05 & 0.06 & 0.07 && 0.47 & 0.47 & 0.46 & 0.50 \\
                &     & 400 && 0.06 & 0.06 & 0.06 & 0.06 && 0.66 & 0.69 & 0.71 & 0.69 \\
                &     & 800 && 0.07 & 0.07 & 0.05 & 0.07 && 0.85 & 0.86 & 0.87 & 0.88 \\
        \hline
        AR$_{0.5}(1)$ & 200 & 200 && 0.08 & 0.07 & 0.07 & 0.07 && 0.33 & 0.36 & 0.37 & 0.36 \\
                      &     & 400 && 0.08 & 0.07 & 0.06 & 0.08 && 0.50 & 0.50 & 0.54 & 0.58 \\
                      &     & 800 && 0.06 & 0.07 & 0.06 & 0.07 && 0.73 & 0.72 & 0.76 & 0.74 \\
                      & 400 & 200 && 0.06 & 0.05 & 0.05 & 0.05 && 0.47 & 0.51 & 0.50 & 0.52 \\
                      &     & 400 && 0.07 & 0.06 & 0.06 & 0.06 && 0.70 & 0.69 & 0.74 & 0.74 \\
                      &     & 800 && 0.06 & 0.04 & 0.06 & 0.05 && 0.87 & 0.91 & 0.91 & 0.91 \\
        \hline
        AR$_{0.6}(1)$ & 200 & 200 && 0.08 & 0.07 & 0.07 & 0.08 && 0.28 & 0.30 & 0.32 & 0.30 \\
                      &     & 400 && 0.07 & 0.08 & 0.07 & 0.07 && 0.40 & 0.41 & 0.42 & 0.46 \\
                      &     & 800 && 0.07 & 0.06 & 0.08 & 0.08 && 0.60 & 0.58 & 0.66 & 0.62 \\
                      & 400 & 200 && 0.05 & 0.07 & 0.06 & 0.07 && 0.40 & 0.36 & 0.42 & 0.40 \\
                      &     & 400 && 0.08 & 0.06 & 0.05 & 0.07 && 0.53 & 0.58 & 0.58 & 0.60 \\
                      &     & 800 && 0.05 & 0.05 & 0.06 & 0.06 && 0.76 & 0.78 & 0.81 & 0.79 \\
        \hline
    \end{tabular}
    \caption{\it
        Sensitivity of the test \eqref{rule} for the hypotheses \eqref{det11} with the norm \eqref{det11a} with respect to the choice of $\nu$ in the self-normalizing statistic. Different discrete uniform distributions on $\{ 1 / K, \ldots, (K-1)/K \}$ are considered. Under $H_0$ the data was generated with $\| \delta \|^2 = 2.0$ (boundary of the hypotheses) and under the alternative $H_1$ the data was generated with $\| \delta \|^2 = 2.25$.
    }
    \label{tab:sensitivityK}
\end{table}

Moving to the second test \eqref{fully_adaptive_test}, we require the Lebesgue measure in the definition of $\hat v_\ell$ for each $\ell = 1, \ldots, p$. However, in practice, one must use an approximation for this, similar to the discrete uniform probability measure that we use for $\nu$ in $V_n$. Hence, one is left with the choice for $K = K_{\hat S_n}$ as well. In Table~\ref{tab:sensitivityK2}, we display results for the test \eqref{fully_adaptive_test} for some choices of $K_{\hat S_n}$ and independently some choices of $K = K_{\hat V_n}$ in the definition of $V_n$. In general, the final test performance is not much affected from this either. Minor differences can be seen between the cases $K_{\hat S_n} = 10$ and $20$, where the larger of these two values mainly performs better, since the approximation to the Lebesgue measure is better.

\begin{table}
    \centering
    \begin{tabular}{ccc c cc c cc c cc c cc}
        \hline
        & & && \multicolumn{5}{c}{$H_0$} && \multicolumn{5}{c}{$H_1$} \\
        \cline{5-9} \cline{11-15}
        & & && \multicolumn{5}{c}{$K_{\hat V_n}$} && \multicolumn{5}{c}{$K_{\hat V_n}$} \\
        \cline{5-9} \cline{11-15}
        & & && \multicolumn{2}{c}{$10$} && \multicolumn{2}{c}{$20$} && \multicolumn{2}{c}{$10$} && \multicolumn{2}{c}{$20$} \\
        \cline{5-6} \cline{8-9} \cline{11-12} \cline{14-15}
        & & && \multicolumn{2}{c}{$K_{\hat S_n}$} && \multicolumn{2}{c}{$K_{\hat S_n}$} && \multicolumn{2}{c}{$K_{\hat S_n}$} && \multicolumn{2}{c}{$K_{\hat S_n}$} \\
        \cline{5-6} \cline{8-9} \cline{11-12} \cline{14-15}
        Model & $n$ & $p$ && 10 & 20 && 10 & 20 && 10 & 20 && 10 & 20 \\
        \hline
        IND & 200 & 200 && 0.06 & 0.05 && 0.05 & 0.04 && 1.00 & 1.00 && 1.00 & 1.00 \\
            &     & 400 && 0.05 & 0.05 && 0.05 & 0.04 && 1.00 & 1.00 && 1.00 & 1.00 \\
            &     & 800 && 0.07 & 0.07 && 0.05 & 0.06 && 1.00 & 1.00 && 1.00 & 1.00 \\
            & 400 & 200 && 0.05 & 0.06 && 0.06 & 0.06 && 1.00 & 1.00 && 1.00 & 1.00 \\
            &     & 400 && 0.05 & 0.05 && 0.05 & 0.05 && 1.00 & 1.00 && 1.00 & 1.00 \\
            &     & 800 && 0.07 & 0.06 && 0.07 & 0.06 && 1.00 & 1.00 && 1.00 & 1.00 \\
        \hline
        MA(2) & 200 & 200 && 0.07 & 0.06 && 0.07 & 0.05 && 0.32 & 0.36 && 0.36 & 0.38 \\
              &     & 400 && 0.06 & 0.05 && 0.05 & 0.05 && 0.49 & 0.50 && 0.52 & 0.57 \\
              &     & 800 && 0.06 & 0.08 && 0.08 & 0.07 && 0.66 & 0.70 && 0.73 & 0.74 \\
              & 400 & 200 && 0.06 & 0.06 && 0.06 & 0.04 && 0.53 & 0.50 && 0.56 & 0.58 \\
              &     & 400 && 0.04 & 0.05 && 0.05 & 0.05 && 0.73 & 0.75 && 0.80 & 0.81 \\
              &     & 800 && 0.06 & 0.05 && 0.06 & 0.05 && 0.90 & 0.92 && 0.95 & 0.95 \\
        \hline
        MA(6) & 200 & 200 && 0.06 & 0.05 && 0.06 & 0.06 && 0.26 & 0.28 && 0.25 & 0.29 \\
              &     & 400 && 0.06 & 0.06 && 0.06 & 0.05 && 0.31 & 0.35 && 0.36 & 0.36 \\
              &     & 800 && 0.04 & 0.06 && 0.05 & 0.05 && 0.36 & 0.39 && 0.39 & 0.40 \\
              & 400 & 200 && 0.06 & 0.06 && 0.06 & 0.06 && 0.41 & 0.43 && 0.49 & 0.51 \\
              &     & 400 && 0.06 & 0.06 && 0.06 & 0.07 && 0.68 & 0.64 && 0.69 & 0.68 \\
              &     & 800 && 0.08 & 0.07 && 0.07 & 0.06 && 0.85 & 0.86 && 0.89 & 0.91 \\
        \hline
        AR$_{0.5}(1)$ & 200 & 200 && 0.06 & 0.05 && 0.06 & 0.05 && 0.26 & 0.27 && 0.30 & 0.30 \\
                      &     & 400 && 0.06 & 0.07 && 0.06 & 0.06 && 0.38 & 0.37 && 0.39 & 0.42 \\
                      &     & 800 && 0.06 & 0.07 && 0.05 & 0.05 && 0.42 & 0.43 && 0.46 & 0.46 \\
                      & 400 & 200 && 0.07 & 0.06 && 0.07 & 0.05 && 0.52 & 0.47 && 0.52 & 0.50 \\
                      &     & 400 && 0.07 & 0.06 && 0.06 & 0.06 && 0.67 & 0.70 && 0.71 & 0.73 \\
                      &     & 800 && 0.07 & 0.07 && 0.07 & 0.06 && 0.87 & 0.88 && 0.92 & 0.92 \\
        \hline
        AR$_{0.6}(1)$ & 200 & 200 && 0.06 & 0.06 && 0.06 & 0.04 && 0.20 & 0.19 && 0.22 & 0.24 \\
                      &     & 400 && 0.05 & 0.04 && 0.06 & 0.06 && 0.25 & 0.23 && 0.23 & 0.26 \\
                      &     & 800 && 0.02 & 0.03 && 0.03 & 0.03 && 0.17 & 0.19 && 0.17 & 0.19 \\
                      & 400 & 200 && 0.06 & 0.07 && 0.06 & 0.06 && 0.38 & 0.39 && 0.38 & 0.40 \\
                      &     & 400 && 0.07 & 0.07 && 0.06 & 0.06 && 0.55 & 0.56 && 0.61 & 0.59 \\
                      &     & 800 && 0.08 & 0.09 && 0.07 & 0.07 && 0.76 & 0.74 && 0.80 & 0.81 \\
        \hline
    \end{tabular}
    \caption{\it 
        Sensitivity of the test \eqref{fully_adaptive_test} for the hypotheses \eqref{hypotheses_new} with respect to the choice of $\nu = \nu_{V_n}$ in the definition of $V_n$ and with respect to the choice of $\nu = \nu_{\hat S_n}$ in the definition of each $\hat v_\ell$ for $\ell = 1, \ldots, p $. Different discrete uniform distributions on $\{ 1 / K_{V_n}, \ldots, (K_{V_n}-1)/K_{V_n} \}$ are considered and the Lebesgue measure in the definition of each $\hat v_\ell$ is approximated by different discrete distributions on $\{ 1 / K_{\hat S_n}, \ldots, (K_{\hat S_n}-1)/K_{\hat S_n} \}$. Under $H_0$ the data was generated with $\| \delta \|^2_{2,0} = 2.0$ (boundary of the hypotheses) and under the alternative $H_1$ the data was generated with $\| \delta \|_{2,0}^2 = 2.25$.
    }
    \label{tab:sensitivityK2}
\end{table}

\newpage
\subsection*{Empirical rejection probabilities}

Recall that $\{ \varepsilon_j \}_{j \in \mathbb{Z}}$ is an i.i.d sequence of $\mathcal{N}(0, \mathrm{diag}_p (0.5))$-distributed random variables. For the remainder of this section, we assume that the $\mathrm{MA}^* (6)$ model is defined as
\begin{align*}
    f(\varepsilon_{j},  \varepsilon_{j-1}, \ldots) = \varepsilon_j + \sum_{k = 1}^6 c_k  \varepsilon_{j-k}, ~~~~~ ~~~~~ j \in \mathbb{Z},
\end{align*}
where $(c_1, \ldots, c_6) = (0.5, 0.25, 0.2, 0.1, 0.05, 0.025)$ and the $\mathrm{AR}_c^* (1)$ model is defined for $c=0.5, 0.6$ as
\begin{align*}
    \eta_j = c \cdot \eta_{j-1} +  \varepsilon_j, ~~~~~ ~~~~~ j \in \mathbb{Z}.
\end{align*}

For these models, we display additional simulation results for the performance of the estimator $\hat S_n$ in Table~\ref{tab:additional_shat_estimation} and empirical rejection probabilities of the tests \eqref{rule} and \eqref{fully_adaptive_test} in Table~\ref{table:erp_fullyadaptive_additional}.
Generally, these results look very similar to the ones presented in Section \ref{sec3}.
The most noticable difference is the improved recall of $\mathrm{MA}^*(6)$ compared to the $\mathrm{MA}(6)$ model, which is due to spatial independence. A similar observation can be made for the AR models. The precision is almost identical to the models that exhibit spatial dependence.

\begin{table}[]
    \centering
    \begin{tabular}{cc c ccc c ccc c ccc}
        \hline
         & && \multicolumn{3}{c}{Preicison} && \multicolumn{3}{c}{Recall} && \multicolumn{3}{c}{F1-Score} \\
        \cline{4-6} \cline{8-10} \cline{12-14}
         & && \multicolumn{3}{c}{$\| \delta \|_{2,0}^2$} && \multicolumn{3}{c}{$\| \delta \|_{2,0}^2$} && \multicolumn{3}{c}{$\| \delta \|_{2,0}^2$} \\
        \cline{4-6} \cline{8-10} \cline{12-14}
        Model & $s/p$ && 1 & 2 & 3 && 1 & 2 & 3 && 1 & 2 & 3\\
        \hline
        $\mathrm{MA}^* (6)$ & 0.25 && 0.93 & 0.95 & 0.95 && 0.76 & 0.97 & 0.99 && 0.84 & 0.96 & 0.97 \\
                            & 0.50 && 0.98 & 0.98 & 0.98 && 0.76 & 0.97 & 0.99 && 0.86 & 0.98 & 0.98 \\
                            & 0.75 && 0.99 & 0.99 & 0.99 && 0.76 & 0.97 & 0.99 && 0.86 & 0.98 & 0.99 \\
                            & 1.00 && 1.00 & 1.00 & 1.00 && 0.76 & 0.97 & 0.99 && 0.86 & 0.99 & 1.00 \\
        \hline
        $\mathrm{AR}^*_{0.5}(1)$ & 0.25 && 0.94 & 0.95 & 0.95 && 0.79 & 0.98 & 0.99 && 0.86 & 0.96 & 0.97 \\
                                 & 0.50 && 0.98 & 0.98 & 0.98 && 0.79 & 0.98 & 0.99 && 0.88 & 0.98 & 0.99 \\
                                 & 0.75 && 0.99 & 0.99 & 0.99 && 0.79 & 0.98 & 0.99 && 0.88 & 0.99 & 0.99 \\
                                 & 1.00 && 1.00 & 1.00 & 1.00 && 0.88 & 0.99 & 0.99 && 0.88 & 0.99 & 1.00 \\
        \hline
        $\mathrm{AR}^*_{0.6}(1)$ & 0.25 && 0.91 & 0.94 & 0.94 && 0.60 & 0.90 & 0.96 && 0.73 & 0.92 & 0.95 \\
                                 & 0.50 && 0.97 & 0.98 & 0.98 && 0.60 & 0.89 & 0.95 && 0.74 & 0.93 & 0.96 \\
                                 & 0.75 && 0.99 & 0.99 & 0.99 && 0.60 & 0.90 & 0.95 && 0.75 & 0.95 & 0.97 \\
                                 & 1.00 && 1.00 & 1.00 & 1.00 && 0.60 & 0.89 & 0.95 && 0.75 & 0.94 & 0.97 \\
        \hline
    \end{tabular}
    \caption{\it 
        Precision measures \eqref{det1000} for the estimated set $\hat S_n$ for different values of $\| \delta \|^2_{2,0}$, where $n=200, p= 400$.
    }
    \label{tab:additional_shat_estimation}
\end{table}

Although the precision metrics in Table~\ref{tab:additional_shat_estimation} show some differences when it comes to dependence that arises along the coordinates, the test \eqref{rule} does not seem to be impacted. This indicates that the main impact of spatial dependence comes from the estimation of the set $\hat S_n$, which is relevant for the test \eqref{fully_adaptive_test}. When comparing the rejection probabilities from Table~\ref{table:erp_fullyadaptive_additional} to those from Table~\ref{table:erp_fullyadaptive}, most results look similar, but the test \eqref{fully_adaptive_test} is more conservative and has more power when there is no dependence along the coordinates.

\begin{table}
    \centering
    \scriptsize
    \begin{tabular}{ccc c cccc c cccc c cc}
        \hline
        &&&& \multicolumn{9}{c}{\eqref{fully_adaptive_test}}  && \multicolumn{2}{c}{} \\
        \cline{5-13} 
        &&&& \multicolumn{4}{c}{$H_0$} && \multicolumn{4}{c}{$H_1$} && \multicolumn{2}{c}{} \\
        \cline{5-8} \cline{10-13}
        &&&& \multicolumn{4}{c}{$s/p$} && \multicolumn{4}{c}{$s/p$} && \multicolumn{2}{c}{\eqref{rule}} \\
        \cline{5-8} \cline{10-13} \cline{15-16}
        Model & $n$ & $p$ && 0.25 & 0.50 & 0.75 & 1.00 && 0.25 & 0.50 & 0.75 & 1.00 && $H_0$ & $H_1$\\
        \hline
        $\mathrm{MA}^*(6)$ & 200 & 200 && 0.01 & 0.02 & 0.04 & 0.09 && 0.30 & 0.77 & 0.91 & 0.94 && 0.05 & 0.99 \\
                           &     & 400 && 0.01 & 0.03 & 0.06 & 0.09 && 0.55 & 0.88 & 0.90 & 0.94 && 0.06 & 1.00 \\
                           &     & 800 && 0.01 & 0.01 & 0.01 & 0.01 && 0.68 & 0.84 & 0.88 & 0.88 && 0.08 & 1.00 \\
                           & 400 & 200 && 0.00 & 0.01 & 0.04 & 0.06 && 0.42 & 0.94 & 1.00 & 1.00 && 0.06 & 1.00 \\
                           &     & 400 && 0.01 & 0.01 & 0.04 & 0.07 && 0.84 & 1.00 & 1.00 & 1.00 && 0.07 & 1.00 \\
                           &     & 800 && 0.01 & 0.02 & 0.06 & 0.08 && 1.00 & 1.00 & 1.00 & 1.00 && 0.06 & 1.00 \\
        \hline
        $\mathrm{AR}^*_{0.5}(1)$ & 200 & 200 && 0.00 & 0.01 & 0.05 & 0.08 && 0.33 & 0.79 & 0.94 & 0.96 && 0.08 & 1.00 \\
                                 &     & 400 && 0.01 & 0.03 & 0.06 & 0.09 && 0.63 & 0.90 & 0.94 & 0.96 && 0.07 & 1.00 \\
                                 &     & 800 && 0.01 & 0.01 & 0.01 & 0.00 && 0.68 & 0.86 & 0.90 & 0.90 && 0.05 & 1.00 \\
                                 & 400 & 200 && 0.00 & 0.02 & 0.03 & 0.06 && 0.46 & 0.95 & 1.00 & 1.00 && 0.07 & 1.00 \\
                                 &     & 400 && 0.00 & 0.02 & 0.03 & 0.06 && 0.88 & 1.00 & 1.00 & 1.00 && 0.07 & 1.00 \\
                                 &     & 800 && 0.01 & 0.02 & 0.05 & 0.09 && 1.00 & 1.00 & 1.00 & 1.00 && 0.06 & 1.00 \\
        \hline
        $\mathrm{AR}^*_{0.6}(1)$ & 200 & 200 && 0.01 & 0.02 & 0.06 & 0.09 && 0.23 & 0.57 & 0.72 & 0.79 && 0.09 & 0.98 \\
                                 &     & 400 && 0.01 & 0.01 & 0.02 & 0.02 && 0.30 & 0.52 & 0.61 & 0.64 && 0.08 & 1.00 \\
                                 &     & 800 && 0.00 & 0.00 & 0.00 & 0.00 && 0.15 & 0.25 & 0.27 & 0.31 && 0.08 & 1.00 \\
                                 & 400 & 200 && 0.00 & 0.01 & 0.04 & 0.07 && 0.37 & 0.89 & 0.98 & 1.00 && 0.07 & 1.00 \\
                                 &     & 400 && 0.00 & 0.02 & 0.04 & 0.09 && 0.76 & 0.99 & 1.00 & 1.00 && 0.06 & 1.00 \\
                                 &     & 800 && 0.01 & 0.04 & 0.04 & 0.06 && 0.97 & 1.00 & 1.00 & 1.00 && 0.06 & 1.00 \\
        \hline
    \end{tabular}
    \caption{\it
        Rejection probabilities of the tests \eqref{rule} and \eqref{fully_adaptive_test} at the boundary of the hypotheses, that is, $\| \delta \|^2 = \Delta = 2.0$ and $\| \delta \|^2_{2,0} = \Delta = 2.0$, respectively (denoted by $H_0$) and under an alternative, that is  $\| \delta \|^2 = 2.25$ and $\| \delta \|^2_{2,0} = 2.25$, respectively (denoted by $H_1$).
        The columns titled $s/p$ specify the ratio of nonzero entries of $\delta$.
    }
    \label{table:erp_fullyadaptive_additional}
\end{table}

\clearpage
\newpage

\section{Auxiliary results and maximal inequalities}

\subsection{Auxiliary results}
For ease of reference, we will state the following simple but useful results. The proof of Lemma~\ref{lem1} is elementary and is therefore omitted.
\begin{lem}\label{lem1}
    Let $(s_i)_{i=1, \ldots ,n}$ be a sequence, $0 \leq k < n - m$ and $m \geq 0$, then
    \begin{align*}
        \sum_{\substack{i,j=k + 1 \\ |i-j| > m}}^n s_i = \sum_{i=k + 1}^{n-m} (n-m-i) \cdot s_i ~~+~~  \sum_{i=k + m+1}^{n} (i-k - m-1) \cdot s_i.
    \end{align*}
\end{lem}

Setting $s_i = 1$ in Lemma~\ref{lem1} for all $i$ yields
\begin{align*}
    \sum_{\substack{i,j= k+1 \\ |i-j| > m}}^n 1 = N_m (n -k),
\end{align*}
where $N_m$ is defined in \eqref{nm}.

\begin{lem}\label{lemNmUpperbound}
    Let $(r_n)_{n \geq 1}$ be some nonnegative sequence. Then, for the quantity $N_m$ defined in \eqref{nm}, we have
    \begin{align}\label{p1}
        \max_{\substack{1 \leq a,b \leq n \\ |a-b| \lesssim r_n}} | N_m (a) - N_m (b) | \lesssim r_n (n+m).
    \end{align}
    Moreover, if $m = O(n)$,
    \begin{align*}
        N_m (a_n) = O(n^2)
    \end{align*}
    for any sequence $(a_n)_{n \in \mathbb{N}}$ with $a_n = O(n)$.
\end{lem}
\begin{proof}
    By the definition of $N_m$, we have for $1 \leq a,b \leq n$
    \begin{align*}
        |N_m (a) - N_m (b) | &= |(a-m)^2 - (b-m)^2 - (a-b)|\\
         &= |a-b||a+b-2m-1| \\
        &\lesssim r_n ( |a+b - 2m| + 1)\\
        &\lesssim r_n (n + m).
    \end{align*}
    Taking the maximum on both sides proves \eqref{p1}. For the second part, note that we have $N_m (m+1) = 0$ and
    \begin{align*}
        |a_n - (m+1)| \lesssim n+m
    \end{align*}
    by the fact that $a_n = O(n)$. Now, we apply \eqref{p1} and obtain
    \begin{align*}
        N_m (a_n) = N_m (a_n) - N_m (m+1) \leq \max_{\substack{1 \leq a,b \leq 1 \\ |a-b| \lesssim n+m}} |N_m (a) - N_m (b)| \lesssim (n+m)^2 = O(n^2),
    \end{align*}
    since $m = O(n)$, which completes the proof.
\end{proof}

\begin{lem}[Uniform continuity of the Brownian motion]\label{double_unif_cont_BM}
    Let $n \in \mathbb{N}$, $A \subset \{1, \ldots , p\}  $ and $(\alpha_{n,\ell} (\lambda) ~  | ~ \lambda \in [0,1] )_{\ell \in A}$ and $(\beta_{n,\ell} (\lambda) ~  | ~ \lambda \in [0,1] )_{\ell \in A }$ be random sequences such that
    \begin{align*}
        \p \Big( \max_{\ell \in A} \sup_{0 \leq \lambda \leq 1} \big| \alpha_{n,\ell} (\lambda) - \beta_{n, \ell} (\lambda) \big| > n^{-2\gamma} \Big) = o(1) 
    \end{align*}
    for some $\gamma > 0$. Then, for a standard Brownian motion $\mathbb{B}$ we have
    \begin{align*}
        \max_{\ell \in A} \sup_{0 \leq \lambda \leq 1} \big  |\mathbb{B} (\alpha_{n,\ell} (\lambda)) - \mathbb{B} (\beta_{n,\ell}  (\lambda)) \big  | = o_\p (n^{-\rho})
    \end{align*}
    for any $0 < \rho < \gamma$.
\end{lem}
\begin{proof}
    By assumption, the event
    \begin{align*}
        \Big \{ \forall \ell \in A: \sup_{0 \leq \lambda \leq 1} \big| \alpha_{n,\ell} (\lambda) - \beta_{n, \ell} (\lambda) \big| \leq n^{-2\gamma}  \Big \}
    \end{align*}
    has asymptotic probability one for some $\gamma > 0$. Therefore, it follows for any $0 <\rho < \gamma$
    \begin{align*}
        &\p \bigg( \max_{\ell \in A} \sup_{0 \leq \lambda \leq 1} |\mathbb{B} (\alpha_{n,\ell} (\lambda)) - \mathbb{B} (\beta_{n,\ell}  (\lambda)) | > n^{-\rho} \bigg)\\
        & ~~~~~ ~~~~~ ~~~~~ ~~~~~ ~~~~~ \leq \p \bigg( \sup_{0 \leq \lambda \leq 1} \sup_{\lambda' : |\lambda' - \lambda| \leq n^{-2\gamma}} |\mathbb{B} (\lambda') - \mathbb{B} (\lambda) | > n^{-\rho} \bigg) + o(1).
    \end{align*}
    By Lemma 3 in \cite{fischer}, we obtain
    \begin{align*}
        \p \bigg( \sup_{0 \leq \lambda \leq 1} \sup_{\lambda' : |\lambda' - \lambda| \leq n^{-2\gamma}} |\mathbb{B} (\lambda') - \mathbb{B} (\lambda) | > n^{-\rho} \bigg) &\lesssim  n^{\rho} \big( n^{-2\gamma} \log (2 n^{2\gamma}) \big)^{1/2} \\
        &= n^{\rho - \gamma} \big( \log (2 n^{2\gamma}) \big)^{1/2} = o(1),
    \end{align*}
    where the last estimate follows, because $0 < \rho < \gamma$.
\end{proof}

\subsection{Maximal inequalities}

In this section, we establish some maximal inequalities, which will be used throughout the upcoming proofs. These maximal inequalities and also some that will be established later are proved using either Proposition 1 in \cite{strongwu} or Proposition \ref{maxmax}, which refers to a maximum of $n = 2^d$ random variables.

We will apply these results for a maximum taken over $1 \leq k \leq n$, where $n \to \infty $. The following remark shows that it is no restriction to consider only the case $n=2^d$.

\begin{rem}\label{wu_n2hochd}
{\rm 
    Let $(h_{n, j})_{j \in \mathbb{N}}$ be any sequence of random variables for $n \in \mathbb{N}$ and define $H_{k, n} =\sum_{j =1}^k h_{n, j}  $. For any $n \in \mathbb{N} $ there exists a $d=d(n) \in \mathbb{N}$ such that  $2^d \leq n \leq 2^{d+1}$ and (obviously)  
    \begin{align*}
        \e \Big[ \max_{1 \leq k \leq 2^d} |H_{k, n} | \Big] \leq \e \Big[ \max_{1 \leq k \leq n} |H_{k, n}| \Big] \leq \e \Big[ \max_{1 \leq k \leq 2^{d+1}} |H_{k, n}| \Big].
    \end{align*}
    If we find an upper bound for the term on the right hand side that is of order $O(f(d+1))$ for a nonnegative, nondecreasing function $f$, then the left hand side is of order $O(f(d))$. If these rates coincide, then we can conclude that the term in the middle, has a rate of $O(f(d))$ as well. To see this, note that, since then
    \begin{align*}
        O(f(d+1)) = O(f(d)) \leq \e \Big[  \max_{1 \leq k \leq n} |H_{k, n}| \Big] \leq O(f(d+1)) = O(f(d)).
    \end{align*}
    Using monotonicity of $f$ yields
    \begin{align*}
        O\big (f(\log (n))\big ) \leq \e \Big[  \max_{1 \leq k \leq n} |H_{k, n}| \Big] \leq O\big (f(\log (n))\big ).
    \end{align*}
    In our cases the condition $O(f(d)) = O(f(d+1))$ will always be satisfied and we will only show that the term on the left-hand side is of order $O(f(d))$.
    }
\end{rem}

\begin{prop}\label{max2}
    Let $n \in \mathbb{N}$ and $(a_{j,n})_{j=1, \ldots ,n}$ be an arbitrary deterministic sequence. Then, we have
    \begin{align*}
        \e \bigg[ \max_{1 \leq k \leq n} \bigg\| \sum_{j=1}^k a_{j,n} (X_j - \e [X_j]) \bigg\|_2^2 \bigg] \lesssim {\max_{j = 1, \ldots ,n} a_{j,n}^2} \cdot n \log (n)^2 \mathrm{tr} (\bar \Gamma),
    \end{align*}
    where $\mathrm{tr}( \bar \Gamma ) = \sum_{\ell = 1}^{p} \sum_{h=0}^{\infty} |\cov ({X_{0, \ell}, X_{h, \ell}})|$.
\end{prop}

\begin{proof}
    Let $\tilde{X}_j = X_j - \e [X_j]$, then
    \begin{align*}
        &\e \bigg [  \max_{1 \leq k \leq 2^d} \bigg\| \sum_{j=1}^k a_{j,n} \tilde{X}_j \bigg\|_2^2 \bigg] \\
        & \quad\quad\quad\quad\quad \leq \sum_{\ell =1}^{p} \e \bigg( \max_{1 \leq k \leq 2^d} \bigg| \sum_{j=1}^k a_{j,n} \tilde{X}_{j, \ell} \bigg| \bigg)^2\\
        & \quad\quad\quad\quad\quad
        \leq \sum_{\ell=1}^{p} \bigg( \sum_{r=0}^d \bigg( 
            \sum_{u = 1}^{2^{d-r}} \e \bigg(
                \sum_{j=2^r (u - 1) + 1}^{2^r u} a_{j,n} \tilde{X}_{j, \ell}
            \bigg)^2    
        \bigg)^{1/2} \bigg)^2 \\
        & \quad\quad\quad\quad\quad\leq \sum_{\ell = 1}^{p} \bigg( \sum_{r=0}^d \bigg( 
            \sum_{u = 1}^{2^{d-r}} \sum_{j_1, j_2 =2^r (u - 1) + 1}^{2^r u} a_{j_1,n} a_{j_2, n} \e [{\tilde{X}_{j_1, \ell} \tilde{X}_{j_2, \ell}}]
        \bigg)^{1/2} \bigg)^2\\
        & \quad\quad\quad\quad\quad\lesssim \sum_{\ell = 1}^{p} \bigg( \sum_{r=0}^d \bigg( 
            \sum_{u = 1}^{2^{d-r}} \sum_{h=0}^{2^r - 1} \sum_{j = 2^r (u - 1) + 1}^{2^r u - h} a_{j,n} a_{j + h, n} \e [{\tilde{X}_{0, \ell} \tilde{X}_{h, \ell}}]
        \bigg)^{1/2} \bigg)^2\\
        & \quad\quad\quad\quad\quad\leq \big( \max_{j = 1, \ldots ,n} |a_{j,n}| \big)^2 \sum_{\ell = 1}^{p} \bigg( \sum_{r=0}^d \bigg( 
            2^{d-r} \sum_{h=0}^{2^r - 1} (2^r - h) |\e [{\tilde{X}_{0, \ell} \tilde{X}_{h, \ell}}]| \bigg)^{1/2} 
        \bigg)^2\\
        & \quad\quad\quad\quad\quad\leq \max_{j = 1, \ldots ,n} a_{j,n}^2  \cdot 2^d d^2  \sum_{\ell = 1}^{p} \sum_{h=0}^{2^d - 1}  |\e [{\tilde{X}_{0, \ell} \tilde{X}_{h, \ell}}]|\\
        & \quad\quad\quad\quad\quad= \max_{j = 1, \ldots ,n} a_{j,n}^2  \cdot 2^d d^2 \mathrm{tr} (\bar \Gamma).
    \end{align*}
\end{proof}

\begin{prop}\label{max3}
    Let $n \in \mathbb{N}$ and $(a_{j,n})_{j=1, \ldots ,n}$ be an arbitrary deterministic sequence. For $k = 1,  \ldots , n$ let
    \begin{align*}
        W_n (k) := \frac{1}{n} \sum_{j=1}^{k} a_{j,n} (X_j - \e [X_j]).
    \end{align*}
    Then, we have for $0<\kappa  \leq 1 $ and $\varphi \in [0,1]$
    \begin{align*}
        \e \bigg[ \sup_{\vartheta \in (0,1) : |\vartheta - \varphi| < \kappa} \big\|  W_n (\gbr{n \vartheta}) - W_n (\gbr{n \varphi}) \big\|_2^2  \bigg] \lesssim \max_{j = 1,  \ldots , 2n} a_{j,n}^2 \frac{\kappa}{n} \log^2 (n) \cdot \mathrm{tr} (\bar \Gamma).
    \end{align*}
\end{prop}

\begin{proof}
    We use Proposition \ref{max2} to obtain
    \begin{align*}
        &\e \bigg[ \sup_{\vartheta: |\vartheta -  \varphi| < \kappa} \big\| W_n (\gbr{n \vartheta}) - W_n (\gbr{n \varphi}) \big\|_2^2  \bigg]\\
        & ~~~~~ ~~~~~ ~~~~~ ~~~~~ ~~~~~ = \frac{1}{n^2} \e \bigg[ \sup_{\vartheta: |\vartheta -  \varphi| < \kappa} \bigg\| \sum_{j= \gbr{n \varphi} + 1}^{\gbr{n \vartheta}} \!\!\! a_{j,n} (X_j - \e [X_j]) \bigg\|_2^2  \bigg]\\
        & ~~~~~ ~~~~~ ~~~~~ ~~~~~ ~~~~~ \leq \frac{1}{n^2} \e \bigg[ \sup_{\vartheta: |\vartheta -  \varphi| < \kappa} \bigg\| \sum_{j=1}^{\gbr{n \vartheta}- \gbr{n \varphi} } \!\!\! a_{j + \gbr{n \varphi},n}  (X_{j+\gbr{n \varphi}} - \e [X_{j + \gbr{n \varphi}}]) \bigg\|_2^2 \bigg]\\
        & ~~~~~ ~~~~~ ~~~~~ ~~~~~ ~~~~~ \leq \frac{1}{n^2} \e \bigg[ \max_{1 \leq k \leq \kappa n + 1} \bigg\| \sum_{j=1}^{k} a_{j + \gbr{n \varphi},n}  (X_{j+\gbr{n \varphi}} - \e [X_{j + \gbr{n \varphi}}]) \bigg\|_2^2 \bigg]\\
        & ~~~~~ ~~~~~ ~~~~~ ~~~~~ ~~~~~ \lesssim \max_{j = 1,  \ldots , \gbr{\kappa n}} a_{j+\gbr{n \varphi}, n}^2 \frac{1}{n^2}  (\kappa n + 1) \log^2 (\kappa n + 1) \cdot \mathrm{tr} (\bar \Gamma) \\
        & ~~~~~ ~~~~~ ~~~~~ ~~~~~ ~~~~~ \leq \max_{j = 1,  \ldots , 2n} a_{j, n}^2 \frac{\kappa}{n} \log^2 (n) \mathrm{tr} (\bar \Gamma), 
    \end{align*}
    where $\bar  \Gamma = \sum_{h \in \mathbb{Z}} |\cov (X_0, X_h)|$.
\end{proof}

\begin{prop}\label{maxmax}
    Let $q > 1$ and $Z_{j,\ell}$ be random variables in $L^q$, $1 \leq j \leq 2^d$, $\ell \in A$, where $d \in \mathbb{N}$ and $A$ is a finite set. Then, for $S_{k, \ell} = Z_{1, \ell} + \cdots + Z_{k, \ell}$ we have
    \begin{align*}
        \e \Big[ \Big( \max_{\ell \in A} \max_{1 \leq k \leq 2^d} |S_{k, \ell}| \Big)^q \Big]^{1/q} \leq \sum_{r = 0}^d \bigg( \sum_{m=1}^{2^{d-r}} \sum_{\ell \in A} \e [| S_{2^rm, \ell} - S_{2^r (m-1), \ell} |^q] \bigg)^{1/q}
    \end{align*}
\end{prop}

\begin{proof}
    Define $S_k := \max_{\ell \in A} |S_{k, \ell} |$ and note that 
    \begin{align*}
        \e \Big[ \Big( \max_{\ell \in A} \max_{1 \leq k \leq 2^d} |S_{k, \ell}| \Big)^q \Big]^{1/q} =\e \Big[ \Big( \max_{1 \leq k \leq 2^d} S_{k}  \Big)^q \Big]^{1/q}.
    \end{align*}
    Using a telescopic sum, we can represent $S_k$ as a sum up to $k$, and therefore Proposition 1 of \cite{strongwu} yields 
    \begin{align*}
        \e \Big[ \Big( \max_{\ell \in A} \max_{1 \leq k \leq 2^d} |S_{k, \ell}| \Big)^q \Big]^{1/q} &\leq \sum_{r = 0}^d \bigg( \sum_{m=1}^{2^{d-r}} \e [(|S_{2^rm} - S_{2^r (m-1)}|)^q] \bigg)^{1/q}.
    \end{align*}
    Observing that 
    \begin{align*}
        |S_{2^rm} - S_{2^r (m-1)}| \leq \max_{\ell \in A} \big| |S_{2^rm, \ell}| - |S_{2^r(m -1), \ell}| \big| \leq \max_{\ell \in A} | S_{2^rm, \ell} - S_{2^r(m - 1), \ell} |,
    \end{align*}
    we obtain
    \begin{align*}
        \e \Big[ \Big( \max_{\ell \in A} \max_{1 \leq k \leq 2^d} |S_{k, \ell}| \Big)^q \Big]^{1/q} &\leq \sum_{r = 0}^d \bigg( \sum_{m=1}^{2^{d-r}} \e \Big[ \Big(\max_{\ell \in A} |S_{2^rm, \ell} - S_{2^r (m-1), \ell}| \Big)^q \Big] \bigg)^{1/q}\\
        &\leq \sum_{r = 0}^d \bigg( \sum_{m=1}^{2^{d-r}} \sum_{\ell \in A}  \e [|S_{2^rm, \ell} - S_{2^r (m-1), \ell}|^q] \bigg)^{1/q},
    \end{align*}
    where we have used the fact that $\max_{\ell \in A} |x_\ell|^q = ( \max_{\ell \in A}  |x_\ell| )^q$, since the set $A$ is finite.
\end{proof}

\begin{prop}\label{prop_maxmax_m} 
    Let $A$ be a finite set and for each $k = 1,  \ldots , n$ let $a_{k+1} , \ldots , a_{k+m}$ be constants, then
    \begin{align*}
        \e \bigg[ \max_{\ell \in A} \max_{1 \leq k \leq n} \bigg| \sum_{j = k + 1}^{k + m} a_{j,k}  \big(X_{j, \ell} - \e [X_{j, \ell}] \big) \bigg| \bigg] \lesssim  A_n \sqrt{m n} \bigg( \sum_{\ell \in A} \sum_{h = 0}^{m-1} \big| \cov ( X_{0, \ell}, X_{h, \ell} ) \big| \bigg)^{1/2},
    \end{align*}
    where $A_n:= \max_{k = 1}^n \max_{j = k+1}^{k+m} |a_{j, k}|$.
\end{prop}

\begin{proof}
    With the notation $\tilde X_j = X_j - \e [X_j]$, we have by Proposition \ref{maxmax} (see also Remark \ref{wu_n2hochd}) for $S_{k,\ell} = \sum_{j = k + 1}^{k + m} a_{j, k} \tilde X_{j,\ell}$ (which can be written as a telescopic sum)
    \begin{align*}
        \e \bigg[ \max_{\ell \in A} \max_{1 \leq k \leq 2^d} \big| S_{k,\ell} \big| \bigg] &\leq \sum_{r = 0}^d \bigg( \sum_{u = 1}^{2^{d-r}} \e \big[  (S_{2^r u, \ell} - S_{2^r (u-1), \ell})^2\big] \bigg)^{1/2}\\
        &\lesssim \sum_{r = 0}^d \bigg( \sum_{u = 1}^{2^{d-r}} \e \big[  S_{2^r u, \ell}^2\big] \bigg)^{1/2} + \sum_{r = 0}^d \bigg( \sum_{u = 1}^{2^{d-r}} \e \big[  S_{2^r (u-1), \ell}^2 \big] \bigg)^{1/2}.
    \end{align*}
    For $\ell \in A$, we have by stationarity
    \begin{align*}
         \e \big[  S_{2^r u, \ell}^2 \big] = \sum_{j_1, j_2 = 2^r u + 1}^{2^r u + m} a_{j_1, k} a_{j_2, k} \e [\tilde X_{j_1, \ell} \tilde X_{j_2, \ell} ] &\lesssim A_n^2 \sum_{h = 0}^{m-1} (m - h) \cdot |\cov (X_{0, \ell}, X_{h, \ell})|\\
         &\leq A_n^2 m \sum_{h = 0}^{m-1} |\cov (X_{0, \ell}, X_{h, \ell})|
    \end{align*}
    and therefore, we obtain
    \begin{align*}
        \sum_{r = 0}^d \bigg( \sum_{\ell \in A} \sum_{u = 1}^{2^{d-r}} \e \big[  S_{2^r u, \ell}^2\big] \bigg)^{1/2} &\lesssim \sum_{r = 0}^d \bigg( \sum_{\ell \in A} \sum_{u = 1}^{2^{d-r}} A_n^2 m \sum_{h = 0}^{m-1} |\cov (X_{0, \ell}, X_{h, \ell})| \bigg)^{1/2}\\
        &\lesssim A_n \sqrt{m  2^{d}} \bigg( \sum_{\ell \in A} \sum_{h = 0}^{m-1} |\cov (X_{0, \ell}, X_{h, \ell})| \bigg)^{1/2}. 
    \end{align*}
    A similar bound can be derived for the second term involving $S_{2^r(u-1), \ell}$ which completes the proof.
\end{proof}

\subsection{Proof of Lemma~\ref{lemmasubexp}} 
\begin{enumerate}[label=(\roman*)]
    \item We show that the random variable $\mathbb{V}_\alpha^2$ has an integrable characteristic function. From this fact it follows that $\mathbb{V}_\alpha^2$ and therefore also $\mathbb{V}_\alpha$ have a Lebesgue density \citep[see pp. 347f. in][]{billingsley1995probability}.

    Define the Brownian bridge as $\mathbb{U} (\lambda) := \mathbb{B} (\lambda) - \lambda \mathbb{B} (1)$ and $\mathbb{U}_\nu (\lambda) := \mathbb{U} (\lambda^\nu)$ for $\nu > 0$. Then, by substituting $\rho = \lambda^{\alpha + 1}$, we have
    \begin{align}\label{relVM}
        \mathbb{V}_\alpha^2 = \int_0^1 \lambda^{\alpha} \mathbb{U}(\lambda)^2 \mathrm{d}\lambda = \frac{1}{\alpha + 1} \int_0^1 \mathbb{U} (\rho^{1/(\alpha + 1)})^2 \mathrm{d} \rho = \frac{1}{\alpha + 1} \| \mathbb{U}_\nu \|_{L^2}^2
    \end{align}
    where $\nu = 1/(\alpha + 1)$. 
    Similar to the discussion after Corollary 1 in \cite{laplacetransform}, we therefore obtain for the characteristic function of $\mathbb{V}_\alpha^2$
    \begin{align}\label{charf_0}
        \e [\exp (it \mathbb{V}_\alpha^2) ] =\e [\exp (\tfrac{it}{\alpha + 1} \|\mathbb{U}_\nu \|_2^2) ] = \bigg( \prod_{n=1}^\infty (1 - \tfrac{2 it }{\alpha + 1} a_n)  \bigg)^{-1/2},
    \end{align}
    where the quantities $a_n$ are the coefficients in the Karhunen-Loève expansion
    \begin{align*}
        \| \mathbb{U}_\nu \|_2^2 = \sum_{n=1}^\infty a_n \xi_n^2
    \end{align*}
    (that is the sequence of eigenvalues repeated according to their multiplicity of the covariance operator of $\mathbb{U}_\nu$) and $\{\xi_n\}_{n \geq 1}$ is a sequence of i.i.d standard normal random variables. Analogously to their Corollary 2, we have for an entire function $f$ of order $\lambda < 1$, such that $1/a_n$, $n\geq 1$ are the only zeros, counting multiplicities, of $f$ that
    \begin{align}\label{charf}
        \e [\exp (it \mathbb{V}_\alpha^2) ] = \bigg( \frac{f(2it / (\alpha + 1))}{f(0)} \bigg)^{-1/2},
    \end{align}
    where $f(z) / f(0) = \prod_{n=1}^\infty (1 - a_nz)$ (note that this function converges for every $z = 2it / (\alpha + 1)$ with $t \in \mathbb{R}$ by \eqref{charf_0}).
    The eigenvalues of the covariance operator of $\mathbb{U}_\nu$ are defined by the equation
    \begin{align}\label{eigenvaleq}
        \int_0^1 (t^\nu \wedge s^\nu - t^\nu s^\nu) y_i(t) \mathrm{d}t = a_i y_i(s), ~~~~ i=1,2, \ldots 
    \end{align}
    Finally, in Theorem 5 in \cite{laplacetransform}, it is shown that the eigenvalues of \eqref{eigenvaleq} are given by the roots of the entire function (with order $1/2$)
    \begin{align*}
        f ( \rho ) = \frac{J_{1/(\alpha + 2)}(c_\alpha \rho^{1/2} )}{\rho^{1/(2\alpha + 4)}},
    \end{align*}
    where $c_\alpha = 2\sqrt{\alpha + 1} / (\alpha + 2)$. Note that $f(\rho) \to (c_\alpha/2)^{1/(\alpha + 2)} \Gamma(\tfrac{\alpha + 3}{\alpha + 2} )^{-1}$ when $\rho \to 0$ \citep[cf. (B.22) in Appendix B in][]{mainardifrancesco}.
    Hence, we have with \eqref{charf} that 
    \begin{align*}
        \varphi_{\mathbb{V}_\alpha^2}(t) &= \e [\exp (it \mathbb{V}_\alpha^2) ]\\
        &= \bigg( \frac{(1 + i)\sqrt{t}}{\alpha + 2} \bigg)^{\frac{1}{2(\alpha + 2)}} \bigg( \Gamma \bigg( \frac{\alpha + 3}{\alpha + 2} \bigg) J_{1/(\alpha + 2)} \bigg( \frac{2(1+i) \sqrt{t}}{\alpha + 2} \bigg)  \bigg)^{-1/2}.
    \end{align*}

    Next, we show that $\varphi_{\mathbb{V}_\alpha^2}$ is integrable. Using $|\varphi_{\mathbb{V}_\alpha^2} (t)| \leq 1$ for each $t \in \mathbb{R}$, we have
    \begin{align}\label{twointegrals}
        \int_{\mathbb{R}} |\varphi_{\mathbb{V}_\alpha^2} (t)| \mathrm{d}t \leq 2 + \int_{1}^\infty |\varphi_{\mathbb{V}_\alpha^2} (t)| \mathrm{d}t + \int_{1}^\infty |\varphi_{\mathbb{V}_\alpha^2} (-t)| \mathrm{d}t.
    \end{align}
    We will show that both integrals on the right-hand side are finite and we will start with the first one.
    By (B.24) in Appendix B in \cite{mainardifrancesco}, we have 
    \begin{align}\label{asmyptoticequiv}
        J_{1/(\alpha + 2)} (z) \sim \sqrt{\frac{2}{\pi z}} \cos \bigg( z - \frac{\pi}{2 (\alpha + 2)} - \frac{\pi}{4} \bigg), ~~~~~ \text{ as } z \to \infty
    \end{align}
    where $z = \frac{2(1+i) \sqrt{t}}{\alpha + 2}$, which satisfies $|\arg (z)| = \pi / 4 < \pi$. Therefore, it follows $|\varphi_{\mathbb{V}_\alpha^2} (t)| \sim \tilde \varphi_{\mathbb{V}_\alpha^2} (t)$ as $t \to \infty$, where
    \begin{align*}
        \tilde \varphi_{\mathbb{V}_\alpha^2} (t) = C_\alpha \cdot \sqrt{\frac{t^{\frac{\alpha + 4}{4(\alpha + 2)}}}{|\cos (\frac{2(1+i)\sqrt{t}}{\alpha + 2} - \frac{\pi}{2(\alpha + 2)} - \frac{\pi}{4}) |}}, ~~~~~ t \geq 0
    \end{align*}
    with $C_\alpha = (\frac{\sqrt{2}}{\alpha + 2})^{\frac{\alpha + 4}{4 (\alpha + 2)}} \pi^{1/4} \Gamma(\frac{\alpha + 3}{\alpha + 2})^{-1/2}$.
    Hence, by the limit comparison test for integrals, it suffices to show that
    \begin{align*}
        \int_{1}^\infty \tilde \varphi_{\mathbb{V}_\alpha^2} (t) \mathrm{d}t < \infty.
    \end{align*}
    For this, we note that a simple calculation shows that 
    \begin{align*}
        |\cos(x+iy)|^2 \geq \sinh^2 (y)
    \end{align*}
    for any $x,y \in \mathbb{R}$, which yields the estimate 
    \begin{align*}
        \int_{1}^\infty \tilde \varphi_{\mathbb{V}_\alpha^2} (t) \mathrm{d}t \leq C_\alpha \int_{1}^\infty \sqrt{\frac{(\sqrt{t})^{\frac{\alpha + 4}{2(\alpha + 2)}}}{\sinh (\frac{2}{\alpha + 2} \sqrt{t})} } \mathrm{d}t \leq C_\alpha \int_{1}^\infty \sqrt{\frac{\sqrt{t}}{\sinh (\frac{2}{\alpha + 2} \sqrt{t})}} \mathrm{d}t,
    \end{align*}
    because $\frac{\alpha + 4}{2 (\alpha + 2)} \leq 1$ and 
    $t \geq 1$.
    Since $t \mapsto \sinh (t)$ grows exponentially, the integral on the right hand side is convergent. Therefore, the integral on the left hand side converges as well.

    Next, we turn to the second integral in \eqref{twointegrals} and have similarly as above, using \eqref{asmyptoticequiv} with $z = \frac{2(i-1)\sqrt{t}}{\alpha + 2}$ (note that $|\arg (z) | = \frac{3\pi}{4} < \pi$)
    \begin{align*}
        |\varphi_{\mathbb{V}_\alpha^2} (-t)| \sim C_\alpha \cdot \sqrt{\frac{t^{\frac{\alpha + 4}{4(\alpha + 2)}}}{|\cos (\frac{2(i - 1)\sqrt{t}}{\alpha + 2} - \frac{\pi}{2(\alpha + 2)} - \frac{\pi}{4}) |}}, ~~~~~ \text{ as } t \to \infty \text{ with } t \geq 0.
    \end{align*}
    Now, we can argue exactly as above and obtain the same bound for the integral. Hence, we have shown that the characteristic function of $\mathbb{V}_\alpha^2$ is integrable and the proof is finished.

    \item We start by computing the Laplace transformation of the square of the denominator. Recall that $\mathbb{U}_{\nu} (\lambda) = \mathbb{U} (\lambda^\nu)$. Then, by Theorem 5 in \cite{laplacetransform},
    \begin{align*}
        \e \big[ \exp \big( \negthickspace - \negthickspace t \| \mathbb{U}_\nu \|_{L^2}^2 \big) \big] = \bigg( \frac{\sqrt{\alpha + 1}}{\alpha + 2} \sqrt{2t}  \bigg)^{\frac{1}{2(\alpha + 2)}} \negthickspace \bigg( \Gamma \bigg(\frac{\alpha + 3}{\alpha + 2} \bigg) I_{1/(\alpha + 2)} \bigg( \negthickspace \frac{\sqrt{\alpha + 1}}{\alpha + 2} \sqrt{8t} \bigg) \bigg)^{-1/2}
    \end{align*}
    for the Laplace transformation of  $\| \mathbb{U}_\nu \|_{L^2}^2$, where $t > 0$ and  $I_s$ is the modified Bessel function of the first kind of order $s$. Using this and \eqref{relVM} we obtain
    \begin{align}\label{laplacetrafo}
        \e \big[ \exp \big( - t \mathbb{V}_\alpha^2  \big) \big] &= \bigg( \frac{\sqrt{2t}}{\alpha + 2} \bigg)^{1/(2(\alpha + 2))} \bigg( \Gamma \bigg(\frac{\alpha + 3}{\alpha + 2} \bigg) I_{1/(\alpha + 2)} \bigg( \frac{\sqrt{8t}}{\alpha + 2}  \bigg) \bigg)^{-1/2}
    \end{align}
    for  the Laplace transformation of $\mathbb{V}_\alpha^2$.
    Introducing the notation $x = \sqrt{2t} / (\alpha + 2)$ and using Theorem 2.1 (iv) in \cite{Ifantis1991} yields
    \begin{align*}
        \bigg( \Gamma \bigg(\frac{\alpha + 3}{\alpha + 2} \bigg) I_{1/(\alpha + 2)} ( 2x ) \bigg)^{-1/2} \negthickspace\negthickspace < e^{-x} x^{- \frac{1}{2(\alpha + 2)} } \bigg( \frac{2 \frac{\alpha + 3}{\alpha + 2}}{2x + 2 \frac{\alpha + 3}{\alpha + 2}} \bigg)^{\frac{\alpha + 3}{\alpha + 2}} \negthickspace < e^{-x} x^{-1/(2(\alpha + 2))}.
    \end{align*}
    Combining this estimate with \eqref{laplacetrafo} gives
    \begin{align}\label{finalbound}
        \e \big[ \exp \big( - t \mathbb{V}_\alpha^2 \big) \big] < e^{-x} = \exp \bigg( - \frac{\sqrt{2t}}{\alpha + 2} \bigg).
    \end{align}
    
    Next, turning to $\mathbb{G}_\alpha = \mathbb{B}(1) / \mathbb{V}_\alpha$, we note that the numerator and the denominator are independent (this can easily be seen by expanding the Brownian bridge into its Karhunen-Loève expansion). If we denote by $c = \| \mathbb{B} (1) \|_{\psi_2}$ the sub-Gaussian norm of $\mathbb{B} (1)$, then we have $\e [\exp (\mathbb{B}(1)^2 / c^2)] \leq 2$ and obtain with \eqref{finalbound} for $t > 0$
    \begin{align*}
        \p \bigg( \bigg| \frac{\mathbb{B} (1)}{\mathbb{V}_\alpha} \bigg| \geq t \bigg) &\leq \p \big( \exp (\mathbb{B} (1)^2 / c^2) \geq \exp \big( (t/c)^2 \mathbb{V}_\alpha^2 \big)  \big) \\
        & \leq \e \bigg[ \exp (\mathbb{B} (1)^2 / c^2) \exp \big( - (t/c)^2 \mathbb{V}_\alpha^2  \big) \bigg]\\
        & = \e \big[ \exp (\mathbb{B} (1)^2 / c^2) \big] \e \big[ \exp \big( - (t/c)^2 \mathbb{V}_\alpha^2 \big) \big]\\
        & \leq 2 \exp \bigg( - \frac{\sqrt{2} t }{c(\alpha + 2)}  \bigg) \\
        &= 2 \exp ( - K_\alpha t ),
    \end{align*}
    where $K_\alpha = \frac{\sqrt{2}}{c(\alpha + 2)} $. By a similar argument, one easily sees that $\mathbb{V}_\alpha^{-1}$ has a subexponential distribution as well.

    \item Observing equation (49) in \cite{tolmatz2002distribution} and $\mathbb{V}_\alpha \leq \mathbb{V}_0$ (a.s.), we have for sufficiently large $t > 0$
    \begin{align}
        \p (\mathbb{V}_\alpha > t) \leq \p (\mathbb{V}_0 > t) 
        = \frac{2}{\pi^{3/2} t} \exp \big( - \pi^2 t^2 / 2 \big) \big( 1 + O(t^{-2}) \big)\leq \exp (-C t^2),
        \label{det20d}
    \end{align}
    where $C>0$ is a positive constant. 
    Let $\varepsilon > 0$ and choose $M_\varepsilon^2 = \frac{\log (p_0 / \varepsilon)}{C \log (p_0)}$ for $p_0$ sufficiently large, such that \eqref{det20d}  is satisfied for $t = \sqrt{C\log (p_0)}$). Then, for $\mathbb{V}_\alpha^{(1)}, \ldots, \mathbb{V}_\alpha^{(p)}$ having the same distribution as $\mathbb{V}_\alpha$ and $p \geq p_0$
    \begin{align*}
        \p \Big( \max_{\ell = 1, \ldots, p}  \mathbb{V}_\alpha^{(\ell)} > \sqrt{\log (p)} \cdot M_\varepsilon \Big) \leq p \p (\mathbb{V}_\alpha > \sqrt{\log (p)} M_\varepsilon) \leq p\exp(-C \log (p) M_\varepsilon^2) \leq \varepsilon.
    \end{align*}
\end{enumerate}

\subsection{Proof of Lemma~\ref{lemmasubsubexp}}

\begin{enumerate}[label=(\roman*)]
    \item
        Recall that $\mathbb{M}(\lambda) = \mathbb{B} (\lambda)^2 - \lambda$, and note that $\{ \mathbb{M}(\lambda) \}_{\lambda \in [0,1]}$ is a continuous martingale. Therefore, by the Dambis-Dubins-Schwarz Theorem \citep[see Theorem 5.13 in][]{LeGall2016}, there exists a Brownian motion $\{ W(\lambda) \}_{\lambda \geq 0}$, such that $\mathbb{M} (\lambda) = W(\langle \mathbb{M} \rangle_\lambda)$, where $\{ \langle \mathbb{M} \rangle_\lambda \}_{\lambda \in [0,1]}$ is the quadratic variation process of $\mathbb{M}$, that is
        \begin{align*}
            \langle \mathbb{M} \rangle_\lambda = 4 \int_0^\lambda \mathbb{B}^2 (t) \mathrm{d}t.
        \end{align*}
        With the notation $\bar W(\lambda) = W(2 \lambda^2) - \lambda^2 W(2)$, $D(\lambda) = W(\langle \mathbb{M} \rangle_\lambda) - W(2 \lambda^2)$ and $E (\lambda) = D(\lambda) - \lambda^2 D(1)$, we obtain the representation $\mathbb{M} (\lambda) - \lambda^2 \mathbb{M} (1) = \bar W(\lambda) + E(\lambda)$. Now, Young's inequality (with $\varepsilon = 2$) yields
        \begin{align*}
            (\mathbb{M} (\lambda) - \lambda^2 \mathbb{M} (1) )^2 &= \bar W (\lambda)^2 + 2 \bar W (\lambda) E(\lambda) + E(\lambda)^2\\
            &\geq \bar W (\lambda)^2 - 2 |\bar W(\lambda)| |E (\lambda)| + E(\lambda)^2\\
            &\geq \frac{1}{2} \bar W (\lambda)^2 - E (\lambda)^2.
        \end{align*}
        Hence, by the definition of $\mathbb{W}_\alpha$ in \eqref{WA}, we obtain
        \begin{align}\label{W_and_E}
            \mathbb{W}_\alpha^2 \geq \frac{1}{2} Z_\alpha - \int_0^1 \lambda^\alpha E(\lambda)^2 \mathrm{d} \lambda,
        \end{align}
        with
        \begin{align*}
            Z_\alpha  = \int_0^1 \lambda^\alpha \bar W (\lambda)^2 \mathrm{d} \lambda 
            &= \int_0^1 \lambda^\alpha \big(W (2 \lambda^2) - \lambda^2 W(2) \big)^2 \mathrm{d} \lambda  \\
            & = \frac{1}{2} \int_0^1 t^{\alpha / 2 - 1/2} \big( W(2t) - t W(2) \big)^2 \mathrm{d}t\\
            &= \int_0^1 t^{\alpha / 2 - 1/2} \big( \beta (t) - t \beta(t) \big)^2 \mathrm{d}t\\
            &=^d \mathbb{V}_{\alpha / 2 - 1/2}^2,
        \end{align*}
        where we used the substitution $t = \lambda^2$ in the second line,  $\beta(t) = \frac{1}{\sqrt{2}} W(2t)$ is a standard Brownian motion and $\mathbb{V}_{\alpha / 2 - 1 / 2}$ is defined in \eqref{V_alpha_def}. Moreover, let $\eta = \eta (t) \geq 2$ and note that on the event $A_\eta := \{ \sup_{0 \leq \lambda \leq 1}  |D (\lambda)| \leq \eta \}$, we obtain
        \begin{align}\label{elambda_eta}
            \int_0^1 \lambda^\alpha E(\lambda)^2 \mathrm{d} \lambda \leq 4 \sup_{0 \leq \lambda \leq 1} D(\lambda)^2 \int_0^1 \lambda^\alpha \mathrm{d} \lambda = c_\alpha \eta^2 ,
        \end{align}
        where $c_\alpha = 4/(\alpha + 1)$. Hence, by \eqref{finalbound}, \eqref{W_and_E} and \eqref{elambda_eta} we obtain (since $\alpha \ge 1$)
        \begin{align}\label{first_result}
            \begin{aligned}
                \e [\exp (-t\mathbb{W}_\alpha^2)  \mathbbm{1}_{A_\eta}] &\le \e \bigg[\exp \bigg(-\frac{t}{2}  \mathbb{V}_{\alpha / 2 - 1/2}^2 \bigg) \bigg] \exp \big( c_\alpha \eta^2 \big) \\
                &\le \exp \big( - d_\alpha \sqrt{t} + c_\alpha \eta^2 \big),
            \end{aligned}
        \end{align}
        where $d_\alpha = 2/(\alpha + 3)$.\\
        For the complementary event $A_\eta^C$, we first define the event $B_\eta := \{ \sup_{0 \leq \lambda \leq 1} \langle \mathbb{M} \rangle_\lambda \leq \eta \}$, on which we have (note that $\eta \ge 2$)
        \begin{align*}
            \sup_{0 \leq \lambda \leq 1} |D (\lambda)| = \sup_{0 \leq \lambda \leq 1} |W(\langle \mathbb{M} \rangle_\lambda) - W(2 \lambda^2)| \leq 2 \sup_{0 \leq \lambda \leq \eta} |W(\lambda)|.
        \end{align*}
        This implies
        \begin{align}\label{second_result}
            \begin{aligned}
                \p \big( A_\eta^C \cap B_\eta \big) = \p \Big( \sup_{0 \le \lambda \le 1} |D(\lambda)| > \eta, B_\eta \Big) &\le \p \Big( \sup_{0 \le \lambda \le \eta} |W(\lambda)| > \eta / 2 \Big)\\
                & \le 4 \exp \Big( - \frac{\eta}{8} \Big),
            \end{aligned}
        \end{align}
        where in the last inequality we used that for any Brownian motion $\{ B(\lambda) \}_{0 \le \lambda \le r}$ on the interval $[0,r]$ with $r>0$ it holds (as a consequence of Doob's martingale inequality) that for $a > 0$
        \begin{align*}
            \p \Big( \sup_{0 \leq \lambda \leq r} |B (\lambda)| > a \Big) &\leq \p \Big( \sup_{0 \leq \lambda \leq r} B(\lambda) > a \Big) + \p \Big( \inf_{0 \leq \lambda \leq r}  B(\lambda) < -a \Big)\\
            &= 2 \p \Big( \sup_{0 \leq \lambda \leq r}  B(\lambda) > a \Big)\\
            &= 4 \p (B(r) > a) \\
            &\leq 4 \exp \Big( - \frac{a^2}{2r} \Big).
        \end{align*}
        From this it also follows that
        \begin{align}\label{third_result}
            \p \big( B_\eta^C \big) = \p \Big( \sup_{0 \le \lambda \le 1} 4 \int_0^\lambda \mathbb{B} (t)^2 \mathrm{d} t > \eta \Big)
            &\le \p \bigg( \int_0^1 \mathbb{B} (t)^2 \mathrm{d}t > \eta / 4 \bigg) \nonumber\\
            &\le \p \Big( \sup_{0 \le \lambda \le 1} |\mathbb{B} (\lambda)| > \sqrt{\eta} / 2 \Big)\nonumber\\
            &\le 4 \exp \Big( - \frac{\eta}{8}  \Big).
        \end{align}
        Therefore, combining \eqref{first_result}, \eqref{second_result} and \eqref{third_result} yields
        \begin{align*}
            \e [\exp (-t \mathbb{W}_\alpha^2)] &\le \e [\exp (-t \mathbb{W}_\alpha^2) \mathbbm{1}_{A_\eta}] + \p (A_\eta^C \cap B_\eta) + \p (B_\eta^C) \\
            &\le \exp (- d_\alpha \sqrt{t} + c_\alpha \eta^2) + 8 \exp \Big( - \frac{\eta}{8} \Big).
        \end{align*}
        Choosing $\eta = \sqrt{\frac{d_\alpha}{2c_\alpha}} t^{1/4}$ and $t$ large enough (such that $\eta \geq 2$), yields
        \begin{align}\label{finalbound_W}
            \e [\exp (-t \mathbb{W}_\alpha^2)] &\le \exp \Big(- \frac{d_\alpha}{2} \sqrt{t} \Big) + 8 \exp (- \eta / 8) \le 9 \exp (- C_\alpha t^{1/4}),
        \end{align}
        where $C_\alpha = (16 (\alpha + 3))^{-1}$.

    \item 
        For $\varepsilon = \varepsilon (t) > 0$, we split the probability as follows
        \begin{align}\label{eps_split}
            \p ( | \mathbb{H}_\alpha| > t) \le \p (\mathbb{W}_\alpha \le \varepsilon) + \p (|\mathbb{M} (1)| > t \varepsilon).
        \end{align}
        By \eqref{finalbound_W} and Markov's inequality, we have for $a > 0$
        \begin{align*}
            \p (\mathbb{W}_\alpha \le \varepsilon) \le \exp (a \varepsilon^2) \cdot \e [\exp (-a \mathbb{W}_\alpha^2)] \le 9 \exp (a \varepsilon^2 - C_\alpha a^{1/4}).
        \end{align*}
        Minimizing with respect to $a$, one gets $a =  C_\alpha^{4/3} (2 \varepsilon)^{-8/3}$ leading to
        \begin{align*}
            \p (\mathbb{W}_\alpha \le \varepsilon) \le 9 \exp (-\tilde C_\alpha \varepsilon^{-2/3}),
        \end{align*}
        where $\tilde C_\alpha = 3 \cdot (C_\alpha / 4)^{4/3}$. For the second term in \eqref{eps_split} note that $\mathbb{M} (1) = \mathbb{B}^2 (1) - 1$ is a centered $\chi^2_1$ random variable and therefore has a subexponential distribution by the centering Lemma. We will denote its subexponential norm by $\| \mathbb{M} (1) \|_{\psi_1}$. Therefore,
        \begin{align*}
            \p (|\mathbb{M} (1)| > t \varepsilon) &\le \exp \bigg(- \frac{t \varepsilon}{\| \mathbb{M} (1) \|_{\psi_1}} \bigg) \cdot \e \bigg[ \exp \bigg( \frac{|\mathbb{M} (1)|}{\| \mathbb{M} (1) \|_{\psi_1}} \bigg) \bigg]\\
            &\le 2  \exp \bigg(- \frac{t \varepsilon}{\| \mathbb{M} (1) \|_{\psi_1}} \bigg).
        \end{align*}
        An easy calculation shows that $\varepsilon = t^{-3/5}$ equalizes the exponents in \eqref{eps_split}. Hence,
        \begin{align*}
            \p (|\mathbb{H}_\alpha| > t) \le 9 \exp (- \tilde C_\alpha t^{2/5}) + 2 \exp \bigg( - \frac{t^{2/5}}{\| \mathbb{M} (1) \|_{\psi_1}} \bigg) \le 11 \exp (-\tilde D_\alpha t^{2/5}),
        \end{align*}
        where $\tilde D_\alpha = \min \{ 3 \cdot (C_\alpha/4)^{4/3}, \| \mathbb{M} (1) \|^{-1}_{\psi_1} \}$. By a simple calculation, one can show $\e [\exp(|\mathbb{M} (1) | / 4)] \le 2$, implying $\| \mathbb{M} (1) \|_{\psi_1} \le 4$. Since $3 \cdot (C_\alpha / 4)^{4/3} \le 1/4$, we have $\tilde D_\alpha \ge \min \{ 3 \cdot (C_\alpha / 4)^{4/3}, 1/4 \} = 3 \cdot (C_\alpha / 4)^{4/3} = D_\alpha$.
\end{enumerate}

\section{Proofs of the results in Section \ref{sec42}}

\subsection{Proof of Theorem \ref{thm2.0}}

We assume without loss of generality \citep[cf.][p. 333]{strassen} that $Y_1, Y_2,  \ldots $ are defined on a probability space that is rich enough to support a Brownian motion, and consider the decomposition \eqref{martingaledecomp}
\begin{align}\label{det6}
    Y_j^\top \xi = E_j^\top \xi - F_j^\top \xi,
\end{align}
where $E_j$ and $F_j$ are defined therein. 
We also introduce the notation
\begin{align*}
    D_{j,n} := \frac{1}{\sqrt{n} \sigma_n} E_j^\top \xi ~~~~~~~~ \text{and} ~~~~~~~~ R_{j,n} := \frac{1}{\sqrt{n} \sigma_n} F_j^\top \xi
\end{align*}
with $\sigma_n^2 = \xi^\top \Gamma_Y \xi$, where $\Gamma_Y$ is defined in \eqref{det50h}.
Straightforward calculations show that $\{ D_{j,n} \}_{j=1, \ldots ,n}$ is a martingale difference sequence with respect to the filtration $\{ \mathcal{F}_j \}_{j=1, \ldots ,n}$, where $\mathcal{F}_j = \sigma (\varepsilon_j, \varepsilon_{j-1},  \ldots )$ is the sigma field generated by $\varepsilon_{j}, \varepsilon_{j-1}, \ldots$. This implies that $\{ M_{k,n} \}_{k=1, \ldots ,n}$ with $M_{k,n} := \sum_{j=1}^k D_{j,n}$ is a martingale with respect to $\{ \mathcal{F}_j \}_{j=1, \ldots ,k}$ for any $n \in \mathbb{N}$. 
Therefore, it follows from the Skorokhod Representation Theorem (Theorem 4.3 in \citealp{strassen}; Chapter 5 in \citealp{scott1973central}; Theorem A.1. in \citealp{hallheyde}) that for a Brownian motion $\mathbb{B}$ and for any $n$ there exist nonnegative random variables $\tau_{1,n}, \tau_{2,n},  \ldots $ with partial sums 
$$
T_{k,n} = \sum_{j=1}^k \tau_{j,n},
$$
such that for $k \geq 1$ and $j \leq k$
\begin{align}\label{det1} 
    & M_{k,n}  = \mathbb{B} (T_{k,n}), \\
    & \e [{\tau_{j,n} \mid \tilde{\mathcal{F}}_{j-1, n}}]  = \e [{D_{j,n}^2  \mid {\mathcal{F}}_{j-1}}], \label{det2} \\
    & \e [{\tau_{j,n}^q \mid \tilde{\mathcal{F}}_{j-1, n}}] \leq C_q \e [{|D_{j,n} |^{2q} \mid {\mathcal{F}}_{j-1}}], ~~~~~ \text{ for any } q \geq 1 \label{det3}
\end{align}
almost surely, where $\tilde{\mathcal{F}}_{j, n}$ denotes the $\sigma$-field generated by $\varepsilon_j, \varepsilon_{j-1},  \ldots $ and $\{ \mathbb{B} (t) \} _{0 \leq t \leq T_{j,n}}$. 
Moreover,  $\tau_{j,n}$ is measurable with respect to $\tilde{\mathcal{F}}_{j,n}$.

From the decomposition \eqref{det6}, we obtain
\begin{align*}
    \sup_{0 \leq \lambda \leq 1} |\sqrt{n} \tilde{S}_n (\lfloor n \lambda \rfloor) - \mathbb{B} (\lambda)| \leq \sup_{0 \leq \lambda \leq 1} |M_{\gbr{n \lambda}, n} - \mathbb{B} (\lambda)| + \max_{1 \leq k \leq n} \bigg| \sum_{j=1}^k R_{j,n} \bigg|,
\end{align*}
and we will now show that both terms on the right-hand side converge to $0$ at the desired order, that is,
\begin{align}\label{det50f}
    \sup_{0 \leq \lambda \leq 1} |M_{\gbr{n \lambda}, n} - \mathbb{B} (\lambda)| &= o_\p (n^{-\beta}), \\
    \label{det50g}
    \max_{1 \leq k \leq n} \bigg| \sum_{j=1}^k R_{j,n} \bigg| &= o_\p (n^{-\beta}).
\end{align}

{\bf Proof of \eqref{det50f}.}
We obtain from \eqref{det1}
\begin{align*}
    \sup_{0 \leq \lambda \leq 1} |M_{\gbr{n \lambda}, n} - \mathbb{B} (\lambda)| = \sup_{0 \leq \lambda \leq 1} \big|\mathbb{B} \big(T_{\gbr{n \lambda}, n} \big) - \mathbb{B} (\lambda) \big|,
\end{align*}
and show below that
\begin{align}
    \label{det4}
    \p \Big( \sup_{0 \leq \lambda \leq 1} |T_{\gbr{n \lambda}, n} - \lambda| > n^{-\alpha / 2} \Big) = o(1), ~~~~~ \text{ as }n \to \infty .
\end{align}
Combining this result with Lemma~\ref{double_unif_cont_BM} yields
\begin{align*}
    \sup_{0 \leq \lambda \leq 1} \big|\mathbb{B} \big(T_{\gbr{n \lambda}, n} \big) - \mathbb{B} (\lambda) \big| = o_\p (n^{-\beta})
\end{align*}
whenever $\beta < \alpha / 4$, and therefore \eqref{det50f} follows.

For a proof of \eqref{det4} we recall the definition of $T_{\gbr{n \lambda}, n}$ and use \eqref{det2} to obtain the inequality
\begin{align}\label{det5}
    |T_{\gbr{n \lambda}, n} - \lambda| \leq |Q_n (\lambda)| + |J_n (\lambda)|
\end{align}
almost surely, where
\begin{align*}
    Q_n (\lambda) &:= \sum_{j=1}^{\gbr{n \lambda}} (\tau_{j,n} - \e [{\tau_{j,n} \mid \tilde{\mathcal{F}}_{j-1, n}}]) \quad\text{and}\quad J_n (\lambda) := \sum_{j=1}^{\gbr{n \lambda}} \e [{D_{j,n}^2 \mid {\mathcal{F}}_{j-1}}]  - \lambda .
\end{align*}
In order to show that
\begin{align}\label{det8}
    \p \Big( \sup_{0 \leq \lambda \leq 1 } |Q_n (\lambda)| > n^{-\alpha / 2} \Big) = o(1), ~~~~~ \text{ as } n \to \infty,
\end{align}
we will first define $\rho_{j,n} := n \cdot \tau_{j,n}$ and use \eqref{det3} and Hölder's inequality to obtain for any $1 < x < 2$
\begin{align*}
    \e \big[\rho_{j,n}^x \big] &= n^x \e \big[ \tau_{j,n}^x \big] \lesssim n^x \e \big[ |D_{j,n} |^{2x} \big] = \frac{1}{\sigma_n^{2x}} \e \big[|E_j^\top \xi|^{2x} \big]
    \leq \frac{1}{\sigma_n^{2x}} \e \big[ (E_j^\top \xi)^4 \big]^{x/2}\\
    &= \frac{1}{\sigma_n^{2x}} \bigg( \sum_{i_1,  \ldots , i_4 = 1}^{p} \xi_{i_1} \cdots \xi_{i_4} \big( 3\e\big[ E_{j,i_1} E_{j, i_2} \big] \e \big[ E_{j,i_3} E_{j, i_4} \big]\\
    & \quad\quad\quad\quad\quad\quad\quad\quad\quad\quad\quad\quad\quad\quad + \mathrm{cum} ( E_{j,i_1},  \ldots , E_{j,i_4} ) \big) \bigg)^{x/2}.
\end{align*}
By Proposition 9.7 in the Supplementary Material of \cite{wangshao}, it follows that $\e [E_{j,i_1} E_{j,i_2}] = (\Gamma_{Y})_{i_1,i_2}$ for any $j$. Therefore, by condition \ref{ass_thm_2} and the fact that $\sigma_n^{2} = \xi^\top \Gamma_{Y} \xi$,
\begin{align*}
    \e \big[ \rho_{j,n}^x \big] &\lesssim \frac{1}{\sigma_n^{2x} } \bigg( 3 (\xi^\top \Gamma_{Y} \xi)^2 + \sum_{i_1,  \ldots , i_4 = 1}^{p} \xi_{i_1} \cdots \xi_{i_4} \mathrm{cum} (E_{0,i_1},  \ldots , E_{0, i_4}) \bigg)^{x/2} = O(1).
\end{align*}
This implies by Minkowski's and Hölder's inequality that for any $j \leq n$ and any $1 < x < 2$
\begin{align*}
    V_{j,n} := \rho_{j,n} - \e [{ \rho_{j,n} \mid  \tilde{\mathcal{F}}_{j-1, n}}] ~~ \in ~~ L^x.
\end{align*}
Consequently,  the triangular array  $\{ |V_{j,n}|^x \mid j \leq n, n \geq 1 \}$ is uniformly integrable for any $1 < x < 2$ and we can apply the main theorem from \cite{gut} (note that in our notation $x$ refers to the constant $p$ in this paper), which gives
\begin{align*}
    \frac{1}{n} \e \bigg[ \bigg| \sum_{j=1}^{ n} (
    \rho_{j,n} - \e [{\rho_{j,n} \mid \tilde{\mathcal{F}}_{j-1, n} }] ) \bigg|^x \bigg] = \frac{1}{n} \e \bigg[ \bigg| \sum_{j=1}^{ n} (V_{j,n} - \e \big[ V_{j,n} \mid \mathcal{G}_{j-1, n} \big]) \bigg|^x \bigg] = o(1),
\end{align*}
where the sigma field $\mathcal{G}_{j,n}$ is defined by $\mathcal{G}_{j, n} = \sigma (V_{i,n} \mid i \leq j)$ for $j \leq n$, and the first identity follows from the Tower property (note that $\mathcal{G}_{j, n} \subseteq \tilde{\mathcal{F}}_{j, n}$), which shows 
\begin{align*}   
   V_{j,n} - 
      \e \big[ V_{j,n} \mid \mathcal{G}_{j-1, n} \big] = \rho_{j,n} - \e [{\rho_{j,n} \mid \tilde{\mathcal{F}}_{j-1, n} }].
\end{align*}
Finally, note that $\{ Q_n(k/n) \}_{k=1,\ldots,n}$ is a martingale with respect to $\{ \tilde{\mathcal{F}}_{j} \}_{j=1,\ldots,n}$ (and hence $\{ |Q_n(k/n)|^x \}_{k=1,\ldots,n}$ is a submartingale). Therefore, we obtain by Doob's martingale inequality that
\begin{align*}
    \p \Big(\sup_{0 \leq \lambda \leq 1} |Q_n (\lambda) | > n^{-\alpha / 2} \Big) &\le n^{x \cdot \alpha / 2} \cdot \e [|Q_n (1)|^x]\\
    &= \frac{n^{x \cdot \alpha  /2 - x + 1}}{n} \e \bigg[ \bigg| \sum_{j=1}^n (\rho_{j,n} - \e [\rho_{j,n} \mid \tilde {\mathcal{F}}_{j-1, n}]) \bigg|^x \bigg] = o(1),
\end{align*}
since $n^{x \cdot \alpha  /2 - x + 1} = O(1)$, if we choose $x = 2 / (2- \alpha) \in (1,2)$. This proves \eqref{det8}.

\medskip

For a proof of the estimate
\begin{align}\label{jn2020}
    \p \Big( \sup_{0 \leq \lambda \leq 1} |J_n (\lambda) | > n^{-\alpha / 2} \Big) = o(1), ~~~~~ \text{ as } n \to \infty,
\end{align}
we use Proposition 1 from \cite{strongwu}, which gives
\begin{align*}
    \e \Big[ \max_{1 \leq k \leq n} | J_n (k / n) | \Big] \le \sum_{r=0}^d \bigg( \sum_{m=1}^{2^{d - r} } \e \Big[ \Big( J_n (2^{r-d}m) - J_n (2^{r-d} (m-1)) \Big)^2  \Big] \bigg)^{1/2}.
\end{align*}
Rewriting the inner term, we have
\begin{align*}
    &\e \Big[ \Big( J_n (2^{r-d}m) - J_n (2^{r-d} (m-1)) \Big)^2  \Big]\\
    & ~~~~~ ~~~~~ ~~~~~ ~~~~~ ~~~~~ ~~~~~ ~~~~~ = \e \Big[ \Big( \sum_{j = 2^r (m-1) + 1}^{2^r m} \e [D_{j,n}^2 \mid \mathcal{F}_{j-1} ] - 2^{r-d} \Big)^2 \Big]\\
    & ~~~~~ ~~~~~ ~~~~~ ~~~~~ ~~~~~ ~~~~~ ~~~~~ = 2^{-2d} \cdot \e \Big[ \Big( \sum_{j = 2^r (m-1) + 1}^{2^r m} \frac{1}{\sigma_n^2} \e [(E_j^\top \xi)^2 \mid \mathcal{F}_{j-1} ] - 2^{r} \Big)^2 \Big]\\
    & ~~~~~ ~~~~~ ~~~~~ ~~~~~ ~~~~~ ~~~~~ ~~~~~ = 2^{-2d} \cdot \e \Big[ \Big( \sum_{j = 2^r (m-1) + 1}^{2^r m}  \frac{\xi^\top \Phi_j \xi}{\sigma_n^2}  - 2^r \Big)^2 \Big],
\end{align*}
where the matrix $\Phi_j$ is defined in \eqref{Phi}. Then, using the definition of the covariance and Proposition 9.7 of the Supplementary Material of \cite{wangshao} yields
\begin{align*}
    &\e \Big[ \Big( \sum_{j = 2^r (m-1) + 1}^{2^r m}   \frac{\xi^\top \Phi_j \xi}{\sigma_n^2}  - 2^r \Big)^2 \Big]\\
    & \quad\quad\quad\quad\quad = \sum_{j_1,j_2 = 2^r (m-1) + 1}^{2^r m}  \frac{\e [\xi^\top \Phi_{j_1} \xi \cdot \xi^\top \Phi_{j_2} \xi]}{\sigma_n^4} - 2 \cdot 2^r \sum_{j=2^r (m-1) + 1}^{2^r m} \frac{\e [\xi^\top \Phi_j \xi]}{\sigma_n^2} + 2^{2r}\\
    & \quad\quad\quad\quad\quad = \frac{1}{\sigma_n^4} \sum_{j_1, j_2 = 2^r (m-1) + 1}^{2^r m} \sum_{i_1, \ldots, i_4 = 1}^p  \xi_{i_1} \cdots \xi_{i_4} \cov (\Phi_{j_1, i_1 i_2}, \Phi_{j_2, i_3 i_4})\\
    & \quad\quad\quad\quad\quad \lesssim 2^{2r} \cdot 2^{-d \alpha / 2},
\end{align*}
where the last inequality follows from condition \ref{ass_thm_1}. Therefore,
\begin{align*}
    \p \Big( \sup_{0 \leq \lambda \leq 1} |J_n (\lambda) | > n^{-\alpha / 2} \Big) &\lesssim 2^{-d \alpha / 2} \cdot \sum_{r=0}^d \bigg( \sum_{m=1}^{2^{d - r} } 2^{-2d} \cdot 2^{2r} \cdot 2^{-d \alpha / 2} \bigg)^{1/2}\\
    &= 2^{-d/2 - \alpha / 4} \sum_{r=0}^d 2^{r/2}\\
    &= o(1)
\end{align*}
and \eqref{jn2020} follows.
Combining \eqref{det5}, \eqref{det8} and \eqref{jn2020} proves \eqref{det4}.
\medskip

{\bf Proof of \eqref{det50g}.}
We apply Proposition 1 from \cite{strongwu} for $n = 2^d$ (see Remark \ref{wu_n2hochd}) and use that $F_j = G_j - G_{j-1}$ (see \eqref{martingaledecomp}) to obtain
\begin{align}
    \e \bigg[ \max_{1 \leq k \leq 2^d} \bigg| \sum_{j=1}^k R_{j,n}  \bigg| \bigg] &\leq \frac{2^{-d/2}}{\sqrt{\xi^\top \Gamma_Y \xi}} \sum_{r=0}^d \bigg( \sum_{m=1}^{2^{d-r}} \e \bigg( \sum_{j=2^r (m-1) + 1}^{2^r m} F_j^\top \xi \bigg)^2 \bigg)^{1/2}\nonumber\\
    &=  \frac{2^{-d/2}}{\sqrt{\xi^\top \Gamma_Y \xi}} \sum_{r=0}^d \bigg( \sum_{m=1}^{2^{d-r}} \e \big[(G_{2^r m }^\top \xi - G_{2^r (m-1)}^\top \xi)^2 \big] \bigg)^{1/2}\nonumber\\
    &\lesssim \frac{2^{-d/2}}{\sqrt{\xi^\top \Gamma_Y \xi}} \sum_{r=0}^d \bigg( \sum_{m=1}^{2^{d-r}} \xi^\top \e \big[G_{2^r m } G_{2^r m }^\top  \big] \xi \bigg)^{1/2} \label{derersteterm}\\
    & ~~~~~ ~~~~~ + \frac{2^{-d/2}}{\sqrt{\xi^\top \Gamma_Y \xi}} \sum_{r=0}^d \bigg( \sum_{m=1}^{2^{d-r}} \xi^\top \e \big[G_{2^r (m-1) } G_{2^r (m-1) }^\top  \big] \xi \bigg)^{1/2}\label{derzweiteterm}
\end{align}
From condition \ref{ass_thm_3} it follows for the term in \eqref{derersteterm}
\begin{align*}
    &\frac{2^{-d/2}}{\sqrt{\xi^\top \Gamma_Y \xi}} \sum_{r=0}^d \bigg( \sum_{m=1}^{2^{d-r}} \xi^\top \e \big[G_{2^r m } G_{2^r m }^\top  \big] \xi \bigg)^{1/2}\\
    & ~~~~~ ~~~~~ ~~~~~ ~~~~~ ~~~~~ ~~~~~ = 2^{-d/2} \sum_{r=0}^d \bigg( \frac{1}{\xi^\top \Gamma_Y \xi}  \sum_{m=1}^{2^{d-r}} \xi^\top \e \big[G_{2^r m } G_{2^r m }^\top  \big] \xi \bigg)^{1/2}\\
    & ~~~~~ ~~~~~ ~~~~~ ~~~~~ ~~~~~ ~~~~~ = O (2^{-d/2} n^{-\alpha} ) \sum_{r=0}^d 2^{(d-r) / 2} \\
    & ~~~~~ ~~~~~ ~~~~~ ~~~~~ ~~~~~ ~~~~~ = O(n^{-\alpha/2})\\
    & ~~~~~ ~~~~~ ~~~~~ ~~~~~ ~~~~~ ~~~~~ = o (n^{-\beta}).
\end{align*}
An identical calculation shows that the same rate holds for the term in \eqref{derzweiteterm}. The estimate \eqref{det50g} now follows by Markov's inequality.

\subsection{Proof of Theorem \ref{thm2.0_unif_neu}}

For $\ell \in A$ let $\xi = (0, \ldots ,0, 1, 0,  \ldots , 0)^\top$ denote the $\ell$th unit vector in $\mathbb{R}^p$, then the decomposition in \eqref{det6} becomes
\begin{align*}
    Y_j^\top \xi = Y_{j,\ell} = E_{j, \ell} - F_{j, \ell}.
\end{align*}
In this case, we define for $\ell \in A$
\begin{align*}
    D_{j,n, \ell} = \frac{E_{j, \ell}}{\sqrt{n} }  ~~~~~ \text{ and } ~~~~~ R_{j,n, \ell} = \frac{F_{j, \ell}}{\sqrt{n} } ,
\end{align*}
and $\{ M_{k, n, \ell}\} _{k=1, \ldots ,n} = \{ \sum_{j=1}^k D_{j,n,\ell }\}_{k=1, \ldots ,n}$  is a martingale with respect to the filtration $\{\mathcal{F}_j\}_{j=1, \ldots ,n}$. For a Brownian motion $\mathbb{B}$ the Skorokhod Representation Theorem will be applied to these martingales to obtain nonnegative random variables $\tau_{1,n, \ell}, \tau_{2,n, \ell},  \ldots $ for each $\ell \in A$ that satisfy
\begin{align}\label{Det1} 
    & M_{k,n, \ell}  = \mathbb{B} (T_{k, n, \ell}), \\
    & \e [{\tau_{j,n, \ell} \mid \tilde{\mathcal{F}}_{j-1, n, \ell}}]  = \e [{D_{j,n, \ell}^2  \mid {\mathcal{F}}_{j-1}}], \label{Det2} \\
    & \e [{\tau_{j,n, \ell}^q \mid \tilde{\mathcal{F}}_{j-1, n, \ell}}] \leq C_q \e \big[ {|D_{j,n, \ell} |^{2q} \mid {\mathcal{F}}_{j-1}} \big], ~~~~~ \text{ for any } q \geq 1 \label{Det3}
\end{align}
almost surely, where $\tilde{\mathcal{F}}_{j, n, \ell}$ is the sigma field generated by $\varepsilon_j, \varepsilon_{j-1}, \ldots $ and $\{ \mathbb{B} (t) \} _{0 \leq t \leq T_{j,n, \ell}}$.

Using \eqref{Det1}, we now need to show that
\begin{align}\label{show_unif_op_neu}
    \max_{\ell \in A} \sup_{0 \leq \lambda \leq 1} |\mathbb{B} (T_{\gbr{n \lambda}, n, \ell}) - \mathbb{B} (\lambda {\Gamma_{Y, \ell\ell}} )| + \max_{\ell \in A} \max_{1 \leq k \leq n} \bigg| \sum_{j=1}^k R_{j,n, \ell} \bigg| = o_\p (n^{-\beta} f(|A|)).
\end{align}
For the first term, recall that $\beta < \alpha / 4$. Pick a constant $\rho > 0$ such that $\beta < \rho < \alpha / 4$ and denote by $\gamma_n = n^{-\rho} \cdot f(|A|) = o \big( n^{-\beta} f(|A|) \big)$. Then, by Lemma~\ref{double_unif_cont_BM} (note that we modified this lemma slightly by multiplying $f(|A|)$ to $n^{-\rho}$, which does not change the result), it suffices to show that
\begin{align}\label{qn_pre}
     \p \Big( \exists \ell \in A : \sup_{0 \leq \lambda \leq 1} \big| T_{\gbr{n \lambda}, n, \ell} - \lambda {\Gamma_{Y, \ell\ell}}  \big| > \gamma_n^2  \Big) = o(1)
\end{align}
for the first term in \eqref{show_unif_op_neu} to converge to zero. For this purpose, we use \eqref{Det2}, which gives
\begin{align}
    \p \Big( \exists \ell \in A:& ~ \sup_{0 \leq \lambda \leq 1}  \big|T_{\gbr{n \lambda}, n, \ell} - \lambda \Gamma_{Y, \ell \ell} \big|  > \gamma_n^2  \Big) \nonumber \\
    &\leq \p \Big( \max_{\ell \in A} \sup_{0 \leq \lambda \leq 1} |Q_{n, \ell} ( \lambda ) | > \gamma_n^2 / 2 \Big) + \p \Big( \max_{\ell \in A} \sup_{0 \leq \lambda \leq 1} |J_{n, \ell} ( \lambda ) | > \gamma_n^2 / 2 \Big), \label{qnjngamma_neu}
\end{align}
where for $\ell \in A$
\begin{align*}
    Q_{n, \ell} (\lambda) &:= \sum_{j=1}^{\gbr{n \lambda}} (\tau_{j,n, \ell} - \e [{\tau_{j,n, \ell} \mid \tilde{\mathcal{F}}_{j-1, n, \ell}}]),\\
    J_{n, \ell} (\lambda) &:= \sum_{j=1}^{\gbr{n \lambda}} \e [{D_{j,n, \ell}^2 \mid {\mathcal{F}}_{j-1}}]  - \lambda {\Gamma_{Y, \ell\ell}}.
\end{align*}

Using Doob's martingale inequality and the union bound, we obtain for the first term in \eqref{qnjngamma_neu}
\begin{align}\label{maxsupgammaq1}
    \p \Big( \max_{\ell \in A} \sup_{0 \leq \lambda \leq 1} |Q_{n, \ell} ( \lambda ) | > \gamma_n^2 / 2 \Big) \lesssim \gamma_n^{-4} \sum_{\ell \in A} \e [Q_{n, \ell} (1)^2].
\end{align}

Observing that for $j_1 < j_2$
\begin{align*}
    \e [\tau_{j_1,n, \ell} \tau_{j_2, n, \ell}] &= \e \big[\tau_{j_1,n, \ell} \e [\tau_{j_2, n, \ell} \mid \tilde{\mathcal{F}}_{j_2 - 1, n, \ell}] \big] \\
    \e \big[ \e [\tau_{j_1, n, \ell} \mid \tilde{\mathcal{F}}_{j_1 - 1, n, \ell} ] \e [\tau_{j_2, n, \ell} \mid \tilde{\mathcal{F}}_{j_2 - 1, n, \ell} ] \big] &= \e \big[ \tau_{j_2, n, \ell}  \e [\tau_{j_1, n, \ell} \mid \tilde{\mathcal{F}}_{j_1 - 1, n, \ell} ]  \big]
\end{align*}
and using \eqref{Det3}, we obtain
\begin{align*}
    \sum_{\ell \in A} \e \big[ Q_{n, \ell}^2 (1) \big] &= \sum_{\ell \in A} \sum_{j = 1}^{n} \e \big[ \big(\tau_{j,n, \ell} - \e [ \tau_{j,n, \ell} \mid \tilde{\mathcal{F}}_{j - 1, n, \ell} ] \big)^2 \big] \\
    &\leq \sum_{\ell \in A} \sum_{j = 1}^{n}  \e \big[ \tau_{j,n,\ell}^2 \big] \\
    &\lesssim \sum_{\ell \in A}\sum_{j = 1}^{n} \e [D_{j,n, \ell}^4]\\
    &\lesssim  \frac{3}{n} \sum_{\ell \in A} \Big(  \Gamma_{Y, \ell\ell}^2  + \mathrm{cum} (E_{0, \ell}, E_{0, \ell},E_{0, \ell},E_{0, \ell})  \Big)\\
    &\lesssim  \frac{1}{n} \bigg( \sum_{\ell \in A} \Gamma_{Y, \ell\ell} \bigg)^2\\
    &\lesssim n^{-\alpha} \log^{2\gamma} (n) \cdot f^2(|A|)
\end{align*}
by conditions \ref{unif_ass_thm_0_neu} and \ref{unif_ass_thm_2_neu}. Observing \eqref{maxsupgammaq1} and the definition of $\gamma_n$, it follows that the first term of \eqref{qnjngamma_neu} converges to zero.

\medskip

Turning to the $J_{n, \ell} (\lambda)$ term in \eqref{qnjngamma_neu}, we have by Proposition \ref{maxmax}
\begin{align*}
    &\p \Big( \max_{\ell \in A} \sup_{0 \leq \lambda \leq 1} |J_{n, \ell} (\lambda) | > \gamma_n^2 / 2 \Big)\\
    & \quad\quad\quad\quad\quad\quad\quad \lesssim \gamma_n^{-2} \sum_{r=0}^d \bigg( \sum_{m=1}^{2^{d-r}}  \sum_{\ell \in A} \e \Big[ \Big(J_{n, \ell} (2^{r-d}m) - J_{n, \ell} (2^{r-d}(m - 1)) \Big)^2 \Big] \bigg)^{1/2}.
\end{align*}
Recall from \eqref{Phi} that $\Phi_{j, \ell\ell} = \e [E_{j, \ell}^2 \mid \mathcal{F}_{j - 1}]$, which gives $\e [\Phi_{j, \ell\ell}] = \Gamma_{Y, \ell\ell}$ \citep[again by Proposition 9.7 of the Supplementary Material of ][]{wangshao}. Moreover, by the definition of $D_{j,n,\ell}$ we have $\e [D_{j,n,\ell}^2 \mid \mathcal{F}_{j-1}] = \Phi_{j, \ell\ell} / n $ and using \ref{unif_ass_thm_1_neu} implies for the inner term
\begin{align*}
    &\e \Big[ \Big(J_{n, \ell} (2^{r-d}m) - J_{n, \ell} (2^{r-d}(m - 1)) \Big)^2 \Big]\\
    & ~~~~~~~ = \e \Big[  \Big(  \sum_{j=2^r (m-1) + 1}^{2^r m} \e [D_{j,n, \ell}^2 \mid \mathcal{F}_{j-1} ] - 2^{r-d} \Gamma_{Y, \ell\ell}   \Big)^2   \Big]\\
    & ~~~~~~~ = 2^{-2d} \bigg( \sum_{j_1, j_2 = 2^r (m-1) + 1}^{2^r m} \e \big[ \Phi_{j_1, \ell \ell} \Phi_{j_2, \ell\ell}  \big] - 2 \cdot 2^{r} \Gamma_{Y, \ell \ell} \sum_{j=2^r (m-1) + 1}^{2^r m} \e [\Phi_{j, \ell \ell}] + 2^{2r} \Gamma_{Y,\ell\ell}^2 \bigg)\\
    & ~~~~~~~ = 2^{-2d} \sum_{j_1, j_2 = 2^r (m-1) + 1}^{2^r m} \cov (\Phi_{j_1, \ell\ell}, \Phi_{j_2, \ell \ell})\\
    & ~~~~~~~ \lesssim 2^{-2d} \cdot \Gamma_{Y,\ell\ell}^2 \sum_{j_1, j_2 = 2^r (m-1) + 1}^{2^r m} \rho^{|j_1 - j_2|}\\
    & ~~~~~~~ \lesssim 2^{r - 2d} \cdot \Gamma_{Y,\ell\ell}^2.
\end{align*}
Therefore, using condition \ref{unif_ass_thm_0_neu}, we obtain for the second term in \eqref{qnjngamma_neu}
\begin{align*}
    \p \Big( \max_{\ell \in A} \sup_{0 \leq \lambda \leq 1} |J_{n, \ell} (\lambda) | > \gamma_n^2 / 2 \Big) &\lesssim \gamma_n^{-2} \sum_{r=0}^d \bigg( \sum_{m=1}^{2^{d-r}}  \sum_{\ell \in A} 2^{r - 2d} \cdot \Gamma_{Y,\ell\ell}^2 \bigg)^{1/2}\\
    &\le \gamma_n^{-2} \cdot 2^{-d/2} d \cdot \sum_{\ell \in A} \Gamma_{Y, \ell\ell}\\
    & \lesssim \gamma_n^{-2}\cdot 2^{-d /2} d \cdot 2^{d(1-\alpha) / 2} d^\gamma \cdot f(|A|)\\
    &= o(1),
\end{align*}
because $\gamma_n^2 = n^{-2\rho} f^2(|A|)$ and $\rho < \alpha / 4$. This shows that the second term in \eqref{qnjngamma_neu} also converges to zero and \eqref{qn_pre} follows.

For the second term in \eqref{show_unif_op_neu}, we use Proposition \ref{maxmax} with $n = 2^d$ (see Remark \ref{wu_n2hochd}) and that $F_{j,\ell} = G_{j, \ell} - G_{j-1, \ell}$ (see \eqref{martingaledecomp}) to obtain
\begin{align}
    \e \bigg[ \max_{\ell \in A} \max_{1 \leq k \leq 2^d} \bigg| \sum_{j = 1}^k R_{j,n, \ell} \bigg| \bigg] & \leq 2^{-d/2}  \e \bigg[ \max_{1 \leq k \leq 2^d} \max_{\ell \in A}  \bigg| \sum_{j = 1}^k F_{j, \ell} \bigg| \bigg] \nonumber \\
    & \lesssim 2^{-d/2} \sum_{r=0}^d \bigg( \sum_{m=1}^{2^{d-r}} \sum_{\ell \in A} \e \bigg[ \bigg( \sum_{j=2^r (m-1) + 1}^{2^r m} F_{j,\ell} \bigg)^2 \bigg] \bigg)^{1/2} \nonumber\\
    & \leq 2^{-d/2} \sum_{r=0}^d \bigg( \sum_{m=1}^{2^{d-r}} \sum_{\ell \in A} \e \big[ \big( G_{2^rm, \ell} - G_{2^r (m-1), \ell} \big)^2 \big] \bigg)^{1/2} \nonumber\\
    & \lesssim 2^{-d/2} \sum_{r=0}^d \bigg( \sum_{m=1}^{2^{d-r}} \sum_{\ell \in A} \var (G_{2^rm, \ell}) \bigg)^{1/2} \label{gterm1_neu}\\
    & ~~~~~ ~~~~~ + 2^{-d/2} \sum_{r=0}^d \bigg( \sum_{m=1}^{2^{d-r}} \sum_{\ell \in A} \var (G_{2^r (m - 1), \ell}) \bigg)^{1/2}. \label{gterm2_neu}
\end{align}
Using condition \ref{unif_ass_thm_3_neu}, we obtain for the term in \eqref{gterm1_neu} that
\begin{align*}
    2^{-d/2} \sum_{r=0}^d \bigg( \sum_{m=1}^{2^{d-r}} \sum_{\ell \in A} \var (G_{2^rm, \ell}) \bigg)^{1/2} &\lesssim 2^{-d/2} \sum_{r=0}^d \bigg( \sum_{m=1}^{2^{d-r}} 2^{d(\alpha / 2 - 1)} \sum_{\ell \in A} \Gamma_{Y,\ell\ell}^2 \bigg)^{1/2}\\
    &\lesssim 2^{- d/2} \cdot 2^{-d \alpha / 4}  \sum_{r = 0 }^d 2^{-r / 2} \sum_{\ell \in A} \Gamma_{Y, \ell\ell} \\
    &\lesssim 2^{-d \alpha / 4} \cdot d^\gamma \cdot f(|A|) \\
    &= o(n^{-\beta} f(|A|)),
\end{align*}
since $0 < \beta < \alpha / 4$.
A similar bound can be found for the term \eqref{gterm2_neu}, which completes the proof.

\section{Proofs of the results in Section \ref{sec41}}

\subsection{Proof of Theorem \ref{cpthm}}

We will use similar arguments as given in the proof of Theorem 1 from \cite{harizcp}. However, several changes are necessary to adapt to the high-dimensional regime. Define
\begin{align*}
    D_n (k) := \frac{k(n-k)}{n^2} \bigg( \frac{1}{k} \sum_{j=1}^k X_j - \frac{1}{n-k} \sum_{j= k + 1}^{n} X_j \bigg),
\end{align*}
which appeared in \eqref{cpestimator} inside the norm
and let $B_n (k) := W_n (k) - \frac{k}{n} W_n (n)$, where
\begin{align*}
    W_n (k) = \frac{1}{n} \sum_{j=1}^{k} (X_j - \e [X_j])
\end{align*}
and
\begin{align*}
    h(k) = \frac{k}{n} \bigg(1- \frac{k_0}{n} \bigg) \mathbbm{1} \{ k \leq k_0 \} + \frac{k_0}{n} \bigg(1 - \frac{k}{n} \bigg) \mathbbm{1} \{ k > k_0 \}.
\end{align*}
Then, we have by the triangle inequality and Proposition \ref{max2}
\begin{align}\label{bnmaxeq}
    \e \Big[ \max_{1 \leq k \leq n} \| B_n (k) \|_2 \Big] \leq C \cdot \frac{\log (n)}{\sqrt{n}} \sqrt{\mathrm{tr} (\bar \Gamma)}.
\end{align}
Next, we rewrite $D_n$ as
\begin{align*}
    D_n (k) = B_n (k) + h(k) \cdot \delta , 
\end{align*}
then, straightforward calculations \citep[see][]{harizcp} show that
\begin{align}\label{bneq}
    \| B_n (\hat k_n) - B_n (k_0) \|_2 \geq \bigg(1 - \frac{h(\hat k_n)}{h(k_0)} \bigg) \big( h(k_0) \| \delta \|_2 - 2 \| B_n (k_0) \|_2 \big),
\end{align}
where we have used that $\hat k_n$ is the maximizer of the function $k \mapsto \| D_n (k) \|_2$. In addition, it is easy to see that
\begin{align*}
    1 - \frac{h(\hat k_n)}{h(k_0)} \geq C \cdot \frac{|k_0 - \hat k_n|}{n \cdot h(k_0)}, 
\end{align*}
where the constant $C$ can be chosen independently of $n$ if $n$ is sufficiently large. By \eqref{bneq} and $h(k_0) \leq 1$, we have (note that $h(k_0) > C $ for sufficiently large $n$ as $\vartheta_0 \in (0, 1)$)
\begin{align}\label{bneq2}
    \| B_n (\hat k_n) - B_n (k_0) \|_2 \geq C \frac{|k_0 - \hat k_n|}{n h(k_0)} \big( \| \delta \|_2 h(k_0) - 2 \| B_n (k_0) \|_2 \big).
\end{align}
First, consider $\p ({ |\hat \vartheta_n - \vartheta_0 | > \alpha})$, for which we have
\begin{align*}
    \p (|\vartheta_0 - \hat \vartheta_n| > \alpha) &= \p \big( |\vartheta_0 - \hat \vartheta_n| > \alpha, ~~ \| B_n (k_0) \|_2  \leq \| \delta \|_2 h(k_0) / 4 \big)\\
    & ~~~~~  + \p \big( |\vartheta_0 - \hat \vartheta_n| > \alpha, ~~ \| B_n (k_0) \|_2  > \| \delta \|_2  h(k_0) / 4 \big)\\
    & \leq \p \big( |\vartheta_0 - \hat \vartheta_n| > \alpha, ~~ \| B_n (k_0) \|_2  \leq \| \delta \|_2  h(k_0) / 4 \big)\\
    & ~~~~~  + \p \big( \| B_n (k_0) \|_2  > \| \delta \|_2  h(k_0)  / 4 \big).
\end{align*}
For the first probability, we have by \eqref{bneq2}, Markov's inequality and the fact that $h(k_0) > C$ (follows from $\vartheta_0 \in (0,1)$)
\begin{align*}
   \p \big( |\vartheta_0 - \hat \vartheta_n| > \alpha, ~~ & \| B_n (k_0) \|_2  \leq \| \delta \|_2  h(k_0) / 4 \big)\\
    & \leq \p \Big( \| B_n (\hat k_n) - B_n (k_0) \|_2 \geq C \cdot \alpha \| \delta \|_2  h(k_0) / 2 \Big) \\
    &\leq  \frac{C}{ \| \delta \|_2 } \e \Big[ \max_{1 \leq k \leq n} \| B_n (k) \|_2 \Big]\\
    &\leq C \cdot \frac{\log(n)}{\sqrt{n}} \frac{\sqrt{\mathrm{tr} (\bar \Gamma)}}{ \| \delta \|_2 },
\end{align*}
where in the last line, we have used \eqref{bnmaxeq}. For the second probability, we have
\begin{align}\label{bndelta4}
\p \big( \| B_n (k_0) \|_2  > \tfrac{\| \delta \|_2 h(k_0) }{4} \big) \leq C \cdot \frac{\log(n)}{\sqrt{n}} \frac{\sqrt{\mathrm{tr} (\bar \Gamma)}}{\| \delta \|_2}
\end{align}
using Markov's inequality and \eqref{bnmaxeq} again. Combining these estimates gives
\begin{align}\label{firstalpha}
    \p ({ |\hat \vartheta_n - \vartheta_0 | > \alpha}) \leq C \cdot \frac{\log (n)}{\sqrt{n}} \frac{\sqrt{\mathrm{tr} (\bar \Gamma)}}{\| \delta \|_2}.
\end{align}

Next, let $r_n = \frac{n}{\log (n)^2} \frac{\| \delta \|_2^2}{\mathrm{tr} (\bar \Gamma)}$ be the inverse of the convergence rate proposed in Theorem \ref{cpthm} and consider for $M \in \mathbb{N}$ and some $\alpha$ that will be chosen later the probability
\begin{align*}
    \p (r_n |\vartheta_0 - \hat \vartheta_n| > 2^M) &= \p \big( r_n^{-1} 2^M < |\vartheta_0 - \hat \vartheta_n| \leq \alpha, ~ \| B_n (k_0) \|_2 \leq \| \delta \|_2 h(k_0) / 4 \big) \\
    & ~~~~~ + \p \big( r_n^{-1} 2^M < |\vartheta_0 - \hat \vartheta_n| \leq \alpha, ~ \| B_n (k_0) \|_2 > \| \delta \|_2 h(k_0) / 4 \big)\\
    & ~~~~~ + \p \big({|\vartheta_0 - \hat \vartheta_n| > \alpha } \big)\\
    &\leq \p \big( r_n^{-1} 2^M < |\vartheta_0 - \hat \vartheta_n| \leq \alpha, ~ \| B_n (k_0) \|_2 \leq \| \delta \|_2 h(k_0) / 4 \big) \\
    & ~~~~~ + \p \big( \| B_n (k_0) \|_2 > \| \delta \|_2 h(k_0) / 4 \big) + \p \big({|\vartheta_0 - \hat \vartheta_n| > \alpha } \big).
\end{align*}
By \eqref{bneq2} and $h(k_0) > C$, we have for the first term 
\begin{align*}
    & \p \big(r_n^{-1} 2^M < |\vartheta_0 - \hat \vartheta_n| \leq \alpha, ~ \| B_n (k_0) \|_2 \leq \| \delta \|_2 h(k_0) / 4 \big) \\
    & ~~~~~~~~~~ \leq E_1 := \p \big( r_n^{-1} 2^M < |\vartheta_0 - \hat \vartheta_n| \leq \alpha, ~  \| B_n (\hat k_n) - B_n (k_0) \|_2 \geq C \tfrac{|k_0 - \hat k_n|}{n} \cdot \tfrac{ \| \delta \|_2 }{2} \big).
\end{align*}
Define $S_{n,j} := \{ t \mid 2^j < r_n |\vartheta_0 - t| \leq 2^{j+1} \}$ for $j \in \mathbb{N}$ and now choose $\alpha$ such that $0 < \alpha < \min (\vartheta_0, 1-  \vartheta_0)/2$.
Then, for $J = J(n, \alpha)$, we have $2^J < r_n \alpha \leq 2^{J+1}$ (note that $J > M$, since $r_n \alpha > 2^M$)
\begin{align*}
    E_1 \leq \sum_{j=M}^{J} \p \big( \hat \vartheta_n \in S_{n,j}, ~  \| B_n (\hat k_n) - B_n (k_0) \|_2 \geq C \tfrac{|k_0 - \hat k_n|}{n} \cdot \tfrac{\| \delta \|_2 }{2} \big).
\end{align*}
Note that $\hat \vartheta_n \in S_{n,j}$ implies that $|\vartheta_0 - \hat \vartheta_n| \geq 2^j r_n^{-1}$, i.e.
\begin{align*}
    \bigg\{ \| B_n (\hat k_n) - B_n (k_0) \|_2 \geq C  \frac{|k_0 - \hat k_n|}{n}  \frac{ \| \delta \|_2 }{2} \bigg\} \subseteq \big\{ \| B_n (\hat k_n) - B_n (k_0) \|_2 \geq C \cdot 2^j r_n^{-1} \| \delta \|_2 \big\}.
\end{align*}
Therefore, and also by Proposition \ref{max3}
\begin{align*}
    E_1 &\leq \sum_{j=M}^{J} \p \big( |\vartheta_0 - \hat \vartheta_n| \leq r_n^{-1} 2^{j+1}, ~  \| B_n (\hat k_n) - B_n (k_0) \|_2 \geq C \cdot 2^j r_n^{-1} \| \delta \|_2 \big)\\
    & {\leq \sum_{j=M}^{J} \frac{1}{C \cdot 2^j r_n^{-1} \| \delta \|_2 } \e \bigg[ \sup_{\vartheta: |\vartheta - \vartheta_0| \leq r_n^{-1}2^{j+1} } \| B_n (\gbr{n \vartheta}) - B_n (\gbr{n \vartheta_0}) \|_2 \bigg] }\\
    &   \leq C \frac{r_n}{ \| \delta \|_2 } \sum_{j=M}^{J} 2^{-j} \\
    & ~~~~~ ~~~~~ ~~~~~ \times \e \bigg[  \sup_{\vartheta: |\vartheta - \vartheta_0| \leq r_n^{-1}2^{j+1} } \bigg( \| W_n (\gbr{n \vartheta}) - W_n (\gbr{n \vartheta_0}) \|_2 + \bigg| \frac{\gbr{n \vartheta} - \gbr{n \vartheta_0}}{n} \bigg| \bigg) \bigg] \\
    & \leq C \cdot \sqrt{r_n} \frac{\log (n)}{\sqrt{n}} \frac{\sqrt{\mathrm{tr} (\bar \Gamma) }}{ \| \delta \|_2 } \sqrt{2} \sum_{j=M}^J (\sqrt{2})^{-j}.
\end{align*}
By \eqref{bndelta4}, \eqref{firstalpha} and the choice of $r_n$, we finally obtain
\begin{align*}
    \p \big(r_n |\vartheta_0 - \hat \vartheta_n| > 2^M \big) \leq C \cdot \bigg( \sum_{j=M}^J (\sqrt{2})^{-j} + 2 \frac{\log(n)}{\sqrt{n}} \frac{\sqrt{\mathrm{tr} (\bar \Gamma)}}{\| \delta \|_2 } \bigg)
\end{align*}
and, if first $n \to \infty$ and then $M \to \infty$, we obtain the desired result.

\section{Proof of Theorem \ref{thm2.1}}

We only prove the statement in the case $A = \{ 1, \ldots, p \}$. The general result follows by exactly the same arguments. We start by showing the following statement recalling the definition of $\sigma_n^2$ in \eqref{variance}.

\begin{prop}\label{lem65}
    Let conditions \ref{assA3} and \ref{assA4} of Assumption \ref{asses} be satisfied. If $m$ is either fixed or satisfies condition \ref{assA5}, we have
    \begin{align*}
        \e \bigg[ \max_{1 \leq k \leq n} \bigg| \frac{1}{\| \bar \Gamma \|_F} \sum_{\substack{i,j = 1 \\ |i-j| > m}}^{k} \tilde{X}_i^\top \tilde{X}_j \bigg| \bigg] \lesssim n \log (n).
    \end{align*}
\end{prop}

\begin{proof}
    The double sum can be expanded as follows:
    \begin{align*}
        S_{n,k} := \sum_{\substack{i,j = 1 \\ |i-j| > m}}^{k} \tilde{X}_i^\top \tilde{X}_j = 2 \sum_{h = m+1}^{k-1} \sum_{j = 1}^{k-h} \tilde{X}_j^\top \tilde{X}_{j + h}
    \end{align*}
    (note that $S_{n,k}=0$  if $k \leq m+1$).
    Rewriting $S_{n,k}$ as a telescopic sum, we obtain by Proposition 1 (with $q=2$, see also Remark \ref{wu_n2hochd}) in \cite{strongwu} that
    \begin{align}\label{gleichung10hoch6}
        \begin{aligned}
            \e \bigg[ { \max_{1 \leq k \leq 2^d} \bigg| \sum_{\substack{i,j = 1 \\ |i - j| > m}}^k \tilde{X}_{i}^\top \tilde{X}_{j} \bigg|} \bigg] &= \e \Big[  \max_{1 \leq k \leq {2^d}} 
            \big | S_{n,k} \big | \Big] \\
            &\leq 2 \sum_{r = 0}^d \bigg(  \sum_{\ell = 1}^{2^{d-r}}  \e \big[ ( S_{n, 2^r\ell}  - S_{2^r (\ell -1 )} )^2 \big]  \bigg)^{1/2}\\
            &= 2 \sum_{r = 0}^d \bigg(  \sum_{\ell = \lceil 2^{-r} (m+2) \rceil + 1}^{2^{d-r}}  \e \big[ ( S_{n, 2^r\ell}  - S_{2^r (\ell -1 )} )^2 \big]  \bigg)^{1/2},
        \end{aligned}
    \end{align}
    since $S_{n,k} = 0$ for $k \leq m+1$. 
    Observing that 
    \begin{align*}
        S_{n, 2^r \ell} - S_{n, 2^r (\ell - 1)} &= 2 \sum_{h = m+1}^{2^r \ell-1} \sum_{j = 1}^{2^r \ell-h} \tilde{X}_j^\top \tilde{X}_{j + h} - 2 \sum_{h = m+1}^{2^r (\ell - 1)-1} \sum_{j = 1}^{2^r (\ell - 1)-h} \tilde{X}_j^\top \tilde{X}_{j + h} \\
        &= 2 \sum_{h = m+1}^{2^r (\ell - 1) - 1} \bigg( \sum_{j = 1}^{2^r\ell - h} \tilde{X}_j^\top \tilde{X}_{j + h} -  \sum_{j = 1}^{2^r (\ell - 1) - h} \tilde{X}_j^\top \tilde{X}_{j + h}  \bigg)\\
        & ~~~~~ ~~~~~ + 2 \sum_{h = 2^r (\ell - 1)}^{2^r \ell - 1} \sum_{j = 1}^{2^r \ell - h}  \tilde{X}_j^\top \tilde{X}_{j + h}\\
        &=2 \sum_{h = m+1}^{2^r (\ell - 1) - 1} \sum_{j = 2^r (\ell - 1) - h + 1}^{2^r\ell - h} \tilde{X}_j^\top \tilde{X}_{j + h}  + 2 \sum_{h = 2^r (\ell - 1)}^{2^r \ell - 1} \sum_{j = 1}^{2^r \ell - h}  \tilde{X}_j^\top \tilde{X}_{j + h}\\
        &=2 \sum_{h = m+1}^{2^r (\ell - 1) - 1} \sum_{j = 2^r (\ell - 1) + 1}^{2^r\ell} \tilde{X}_j^\top \tilde{X}_{j - h}  + 2 \sum_{h = 2^r (\ell - 1)}^{2^r \ell - 1} \sum_{j = 1}^{2^r \ell - h}  \tilde{X}_j^\top \tilde{X}_{j + h}
    \end{align*}
    gives for the right-hand side of \eqref{gleichung10hoch6}
    \begin{align}
        &\e \bigg[ { \max_{1 \leq k \leq {2^d}} \bigg| \sum_{\substack{i,j = 1 \\ |i - j| > m}}^k \tilde{X}_{i}^\top \tilde{X}_{j} \bigg|} \bigg]  \lesssim  R_{n1} + R_{n  2},
        \label{term0}
    \end{align}
    where the terms $R_{n1}$ and $R_{n2} $ are defined by 
    \begin{align}      
        R_{n1}  &=  \sum_{r=0}^d \bigg(
            \sum_{\ell= \lceil 2^{-r} (m+2) \rceil + 1}^{2^{d-r}} \sum_{h_1, h_2 = m + 1}^{2^r (\ell - 1) - 1} \sum_{j_1, j_2 = 2^r(\ell - 1) + 1}^{2^r \ell}  \e \big[ \tilde{X}_{j_1}^\top \tilde{X}_{j_1 - h_1} \tilde{X}_{j_2}^\top \tilde{X}_{j_2 - h_2} \big]
        \bigg)^{1/2} \label{term1}\\
        R_{n2}  & = \sum_{r=0}^d \bigg(
            \sum_{\ell=\lceil 2^{-r} (m+2) \rceil + 1}^{2^{d-r}} \sum_{h_1, h_2 = 2^r(\ell - 1)}^{2^r \ell - 1} \sum_{j_1 = 1}^{2^r \ell - h_1} \sum_{j_2 = 1}^{2^r \ell - h_2} \e \big[ \tilde{X}_{j_1}^\top \tilde{X}_{j_1 + h_1} \tilde{X}_{j_2}^\top \tilde{X}_{j_2 + h_2} \big]
        \bigg)^{1/2}. \label{term2}
    \end{align}
    We consider these two new terms separately and denote in the upcoming calculations the entry in the $i$th row and $j$th column of the matrix $\Sigma_h$ by $\Sigma_{h, ij}$.
    \medskip
    
    \noindent
    {\bf Estimation of the term $ R_{n1}$ in \eqref{term1}. } With the notation of cumulants \citep[see, for example][for a definition and simple properties]{brillinger} we have 
    \begin{align}
        \e [{\tilde{X}_{j_1}^\top \tilde{X}_{j_1 - h_1} \tilde{X}_{j_2}^\top \tilde{X}_{j_2 - h_2}}] &= \sum_{i_1, i_2 = 1}^{p} \Big \{ 
            \e [{\tilde{X}_{j_1,i_1} \tilde{X}_{j_1 - h_1, i_1}}] \e [{\tilde{X}_{j_2,i_2} \tilde{X}_{j_2 - h_2, i_2}}]\nonumber\\
            & ~~~~~~~~~~~~~~~~~~ + 
            \e[{\tilde{X}_{j_1,i_1} \tilde{X}_{j_2, i_2}}] \e [{\tilde{X}_{j_1 - h_1,i_1} \tilde{X}_{j_2 - h_2, i_2}}]\nonumber\\
            & ~~~~~~~~~~~~~~~~~~ + \e[{\tilde{X}_{j_1,i_1} \tilde{X}_{j_2 - h_2, i_2}}] \e [{\tilde{X}_{j_2, i_2} \tilde{X}_{j_1 - h_1, i_1}}]\nonumber\\
            & ~~~~~~~~~~~~~~~~~~ + \mathrm{cum} ({\tilde{X}_{j_1, i_1}, \tilde{X}_{j_1 - h_1, i_1}, \tilde{X}_{j_2}, \tilde{X}_{j_2 - h_2, i_2}})
        \Big \} \nonumber\\
        &\begin{aligned}\label{det12}
            &=\sum_{i_1, i_2 = 1}^{p} \Big \{
            \Sigma_{h_1, i_1 i_1} \Sigma_{h_2, i_2 i_2} \\
            & ~~~~~~~~~~~~~~~~~~ + 
            \Sigma_{j_1 - j_2, i_1 i_2} \Sigma_{j_1 - j_2 + h_2 - h_1, i_1 i_2}\\
            & ~~~~~~~~~~~~~~~~~~ +
            \Sigma_{j_1 - j_2 + h_2, i_1 i_2} \Sigma_{j_1 - j_2 - h_1, i_1 i_2}\\
            & ~~~~~~~~~~~~~~~~~~ + \mathrm{cum} ({\tilde{X}_{j_1, i_1}, \tilde{X}_{j_1 - h_1, i_1}, \tilde{X}_{j_2}, \tilde{X}_{j_2 - h_2, i_2}})
       \Big \}.
        \end{aligned}
    \end{align}
    Using this decomposition gives a representation with four summands for $R_{n1}$ in \eqref{term1}, which are evaluated separately. For the first one, we have
    \begin{align*}
        \sum_{\ell = \lceil 2^{-r} (m+2) \rceil + 1}^{2^{d-r}} &\sum_{h_1, h_2 = m + 1}^{2^r (\ell - 1) - 1} \sum_{j_1, j_2 = 2^r(\ell - 1) + 1}^{2^r \ell} \sum_{i_1, i_2 = 1}^{p}   \Sigma_{h_1, i_1 i_1} \Sigma_{h_2, i_2 i_2} \\
        & ~~~~~ ~~~~~ ~~~~~ ~~~~~ ~~~~~ = 2^{2r} \sum_{\ell = \lceil 2^{-r} (m+2) \rceil + 1}^{2^{d-r}} \bigg( \sum_{h = m + 1}^{2^r (\ell - 1) - 1}  \mathrm{tr} (\Sigma_h) \bigg)^2\\
        & ~~~~~ ~~~~~ ~~~~~ ~~~~~ ~~~~~ \leq 2^{d} 2^r \bigg( \sum_{h \in \mathbb{Z}} \mathrm{tr} (| \Sigma_h |) \bigg)^2\\ 
        & ~~~~~ ~~~~~ ~~~~~ ~~~~~ ~~~~~ = 2^d 2^r \mathrm{tr} ( \bar \Gamma )^2\\
        & ~~~~~ ~~~~~ ~~~~~ ~~~~~ ~~~~~ \lesssim 2^d 2^r \| \bar \Gamma \|_F^2
    \end{align*}
    by condition \ref{assA3}. Therefore, we can bound the first term of $R_{n1}$ (corresponding to the first summand in \eqref{det12}) by
    \begin{align*}
        \frac{1}{\| \bar \Gamma \|_F} \sum_{r=0}^d \big( 2^{d} 2^r \| \bar \Gamma \|_F^2 \big)^{1/2} \lesssim 2^{d}.
    \end{align*}
    Turning to the second term in \eqref{det12}, we use the estimate $\mathrm{tr} (AB) \leq \| A \|_F \| B \|_F$ and obtain
    \begin{align*}
        &\sum_{\ell = \lceil 2^{-r} (m+2) \rceil + 1}^{2^{d-r}} \sum_{h_1, h_2 = m + 1}^{2^r (\ell - 1) - 1} \sum_{j_1, j_2 = 2^r(\ell - 1) + 1}^{2^r \ell}  \sum_{i_1, i_2 = 1}^{p} \Sigma_{j_1 - j_2, i_1 i_2} \Sigma_{j_1 - j_2 + h_2 - h_1, i_1 i_2}\\
        & ~~~~~~~~~~~~ = \sum_{\ell = \lceil 2^{-r} (m+2) \rceil + 1}^{2^{d-r}} \mathrm{tr} \bigg(  \sum_{h_1, h_2 = m + 1}^{2^r (\ell - 1) - 1} \sum_{j_1, j_2 = 2^r(\ell - 1) + 1}^{2^r \ell} \Sigma_{j_1 - j_2} \Sigma_{j_1 - j_2 + h_2 - h_1 }^\top \bigg)\\
        & ~~~~~~~~~~~~ \leq \sum_{\ell = \lceil 2^{-r} (m+2) \rceil + 1}^{2^{d-r}} \mathrm{tr} \bigg( \sum_{|j| \leq 2^r - 1} (2^r - |j|) \cdot |\Sigma_{j}|\\
        & ~~~~~~~~~~~~ ~~~~~~~~~~~~ ~~~~~~~~~~~~ ~~~~~~~~~~~~ \times \sum_{|h| \leq 2^r (\ell -1) -m - 1} (2^r (\ell -1) -m - 1 - |h|) \cdot |\Sigma_{j + h}^\top| \bigg)\\
        & ~~~~~~~~~~~~ \leq 2^r \sum_{\ell = \lceil 2^{-r} (m+2) \rceil + 1}^{2^{d-r}} 2^r (\ell -1) \cdot \mathrm{tr} \bigg( \sum_{|j| \leq 2^r - 1} |\Sigma_{j}| \sum_{|h| \leq 2^r (\ell -1) -m - 1} |\Sigma_{j + h}^\top| \bigg)\\
        & ~~~~~~~~~~~~ \leq 2^d 2^r \sum_{\ell = \lceil 2^{-r} (m+2) \rceil + 1}^{2^{d-r}}  \bigg\| \sum_{|j| \leq 2^r - 1} |\Sigma_{j}| \bigg\|_F \bigg\| \sum_{|h| \leq 2^r (\ell -1) -m - 1} |\Sigma_{j + h}^\top| \bigg\|_F \\
        & ~~~~~~~~~~~~ \leq  2^{2d} \bigg\| \sum_{h \in \mathbb{Z}} |\Sigma_h| \bigg\|_F^2.
    \end{align*}
    Therefore, by the definition of $\bar \Gamma$ we obtain for the second term in the representation of $R_{n1}$ the bound
    \begin{align*}
        \frac{1}{\| \bar \Gamma \|_F} \sum_{r=0}^d \bigg( 2^{2d} \bigg\| \sum_{h \in \mathbb{Z}} |\Sigma_h| \bigg\|_F^2 \bigg)^{1/2} \lesssim d \cdot 2^{d}.
    \end{align*}
    We continue with the third term, which can be treated by the same arguments, that is,\begin{align*}
        &\sum_{\ell = \lceil 2^{-r} (m+2) \rceil + 1}^{2^{d-r}} \sum_{h_1, h_2 = m + 1}^{2^r (\ell - 1) - 1} \sum_{j_1, j_2 = 2^r(\ell - 1) + 1}^{2^r \ell }  \sum_{i_1, i_2 = 1}^{p} \Sigma_{j_1 - j_2 + h_2, i_1 i_2} \Sigma_{j_1 - j_2 - h_1, i_1 i_2}\\
        & ~~~~~~~~~~~~ \leq 2^r \sum_{\ell = \lceil 2^{-r} (m+2) \rceil + 1}^{2^{d-r}} \mathrm{tr} \bigg( 
            \sum_{h_1, h_2 = m + 1}^{2^r (\ell - 1) - 1} \sum_{|j| \leq 2^r - 1} | \Sigma_{j + h_2} | \cdot | \Sigma_{j - h_1}^\top | 
        \bigg)\\
        & ~~~~~~~~~~~~ \leq 2^r \sum_{\ell = \lceil 2^{-r} (m+2) \rceil + 1}^{2^{d-r}} \sum_{|j| \leq 2^r - 1} \bigg\| \sum_{h=m+1}^{2^r (\ell - 1) - 1} |\Sigma_{j + h}| \bigg\|_F \bigg\| \sum_{h=m+1}^{2^r (\ell - 1) - 1} |\Sigma_{j - h}| \bigg\|_F \\
        & ~~~~~~~~~~~~ \lesssim 2^d2^{r} \bigg\| \sum_{h \in \mathbb{Z}} |\Sigma_{h}| \bigg\|_F^2.
    \end{align*}
    As before, we obtain
    \begin{align*}
        \frac{1}{\| \bar \Gamma \|_F} \sum_{r=0}^d \bigg( 2^d 2^{r} \bigg\| \sum_{h \in \mathbb{Z}} |\Sigma_{h}| \bigg\|_F^2 \bigg)^{1/2} \lesssim 2^{d}.
    \end{align*}
    Finally, turning to the term involving the cumulant, we use condition \ref{assA4} to obtain
    \begin{align*}
        &\sum_{\ell = \lceil 2^{-r} (m+2) \rceil + 1}^{2^{d-r}} \sum_{h_1, h_2 = m + 1}^{2^r (\ell - 1) - 1} \sum_{j_1, j_2 = 2^r(\ell - 1) + 1}^{2^r \ell }  \sum_{i_1, i_2 = 1}^{p} \mathrm{cum} ({\tilde{X}_{j_1, i_1}, \tilde{X}_{j_1 - h_1, i_1}, \tilde{X}_{j_2, i_2}, \tilde{X}_{j_2 - h_2, i_2}})\\
        & ~~~~~~~~~ \lesssim \sum_{\ell = \lceil 2^{-r} (m+2) \rceil + 1}^{2^{d-r}} \sum_{h_1, h_2 = m + 1}^{2^r (\ell - 1) - 1} \sum_{j=0}^{2^r -1} (2^r - j )   \sum_{i_1, i_2 = 1}^{p} | \mathrm{cum} ({\tilde{X}_{0, i_1}, \tilde{X}_{-h_1, i_1}, \tilde{X}_{j, i_2}, \tilde{X}_{j - h_2, i_2}}) |\\
        & ~~~~~~~~~ \lesssim 2^r \| \bar \Gamma \|_F^2 \sum_{\ell = \lceil 2^{-r} (m+2) \rceil + 1}^{2^{d-r}} \sum_{h_1, h_2 = m + 1}^{2^r (\ell - 1) - 1} \sum_{j=0}^{2^r -1} \rho^{j + h_1} \\
        & ~~~~~~~~~ \lesssim 2^d 2^r \| \bar \Gamma \|_F^2 \sum_{\ell = \lceil 2^{-r} (m+2) \rceil + 1}^{2^{d-r}} \sum_{h_1 = m + 1}^{\infty} \rho^{h_1} \sum_{j=0}^{\infty} \rho^{j}\\
        & ~~~~~~~~~ \lesssim 2^{2d} \| \bar \Gamma \|_F^2,
    \end{align*}
    since $\rho < 1$. Therefore,
    \begin{align*}
        &\frac{1}{\| \bar \Gamma \|_F} \sum_{r=0}^d \bigg( \sum_{\ell = \lceil 2^{-r} (m+2) \rceil + 1}^{2^{d-r}} \sum_{h_1, h_2 = m + 1}^{2^r (\ell - 1) - 1} \sum_{j_1, j_2 = 2^r(\ell - 1) + 1}^{2^r \ell}   \sum_{i_1, i_2 = 1}^{p} \\
        & ~~~~~ ~~~~~ ~~~~~ ~~~~~ ~~~~~ ~~~~~ ~~~~~ ~~~~~ ~~~~~ ~~~~~ ~~~~~ ~~~~~\mathrm{cum} ({\tilde{X}_{j_1, i_1}, \tilde{X}_{j_1 - h_1, i_1}, \tilde{X}_{j_2, i_2}, \tilde{X}_{j_2 - h_2, i_2}}) \bigg)^{1/2}\\
        & ~~~~~ ~~~~~ \lesssim d 2^{d}.
    \end{align*}
Summarizing these calculations yields
\begin{align} \label{rn1}
    \frac{1}{\| \bar \Gamma \|_F } R_{n1} \lesssim d \cdot 2^{d}.
\end{align}

\noindent {\bf Estimation of the term $ R_{n2}$ in \eqref{term2}. }
Note that the crucial argument for deriving the estimate for  \eqref{rn1} is that the summation with respect to $j_1,j_2$ has at most $2^r$ summands (i.e. it does not depend on $\ell$). For the term in $R_{n2}$ in \eqref{term2}, this is the case as well, since the number of terms in the sum is at most $2^r \ell - h_i \leq 2^r \ell - 2^r (\ell - 1) = 2^r$. Consequently, the same arguments as used for $R_{n1}$ yields $\| \bar \Gamma \|_F^{-1} R_{n2} \lesssim d 2^{d}$, and combining this estimate with \eqref{term0} and \eqref{rn1} gives
\begin{align*}
    &\e \bigg[ { \max_{1 \leq k \leq n} \bigg| \frac{1}{\| \bar \Gamma \|_F } \sum_{\substack{i,j = 1 \\ |i - j| > m}}^k \tilde{X}_{i}^\top \tilde{X}_{j} \bigg|} \bigg]  \lesssim n \log (n),
\end{align*}
which completes the proof of Proposition \ref{lem65}.
\end{proof}

With Proposition \ref{lem65} at hand, we continue with the proof of Theorem \ref{thm2.1}. We start with a decomposition of the statistic $T_n ( \hat k_n, m; \lambda)$ in \eqref{statistic}, that is,
\begin{align}\label{decomp1}
    T_n ( \hat k_n,m; \lambda) &= T_n^{(1)} (\lambda) + T_n^{(2)}(\lambda) + T_n^{(3)}(\lambda) + T_n^{(4)}(\lambda),
\end{align}
where
\begin{align}
    \label{det2a}
    T_n^{(1)} (\lambda)  &= 
        \frac{1}{N_m (\hat k_n) N_m (n-\hat k_n) p }  \sum_{\substack{i_1,i_2=1\\|i_1-i_2| > m}}^{\gbr{\lambda \hat k_n}} \sum_{\substack{j_1,j_2=\hat k_n+1\\|j_1-j_2| > m}}^{\hat k_n + \gbr{\lambda (n-\hat k_n)}} (\tilde{X}_{i_1} - \tilde{X}_{j_1})^\top (\tilde{X}_{i_2} - \tilde{X}_{j_2}),\\
        \label{det2b}
    T_n^{(2)} (\lambda)  &= 
        \frac{1}{N_m (\hat k_n) N_m (n-\hat k_n) p }  \sum_{\substack{i_1,i_2=1\\|i_1-i_2| > m}}^{\gbr{\lambda \hat k_n}} \sum_{\substack{j_1,j_2=\hat k_n+1\\|j_1-j_2| > m}}^{\hat k_n + \gbr{\lambda (n-\hat k_n)}} (\e [X_{i_1}] - \e [X_{j_1}])^\top (\tilde{X}_{i_2} - \tilde{X}_{j_2}),\\
        \label{det2c}
    T_n^{(3)} (\lambda)  &= 
        \frac{1}{N_m (\hat k_n) N_m (n-\hat k_n) p }  \sum_{\substack{i_1,i_2=1\\|i_1-i_2| > m}}^{\gbr{\lambda \hat k_n}} \sum_{\substack{j_1,j_2=\hat k_n+1\\|j_1-j_2| > m}}^{\hat k_n + \gbr{\lambda (n-\hat k_n)}} (\tilde{X}_{i_1} - \tilde{X}_{j_1})^\top (\e [X_{i_2}] - \e [X_{j_2}]), \\
        \label{det2d}
    T_n^{(4)} (\lambda) &= \frac{1}{N_m (\hat k_n) N_m (n-\hat k_n) p } \\
        & ~~~~~ ~~~~~ ~~~~~ ~~~~~ \times \sum_{\substack{i_1,i_2=1\\|i_1-i_2| > m}}^{\gbr{\lambda \hat k_n}} \sum_{\substack{j_1,j_2=\hat k_n+1\\|j_1-j_2| > m}}^{\hat k_n + \gbr{\lambda (n-\hat k_n)}} (\e [X_{i_1}] - \e [X_{j_1}])^\top (\e [X_{i_2}] - \e [X_{j_2}]). \nonumber
\end{align}
In the following, we will show
\begin{align}
    \sup_{0 \leq \lambda \leq 1} \bigg| \frac{\sqrt{n}}{\sigma_n} T_n^{(1)} (\lambda)  \bigg| &= o_\p (1), \label{tn1_eq} \\
    \sup_{0 \leq \lambda \leq 1} \bigg| \frac{\sqrt{n}}{\sigma_n} \big( T_n^{(2)} (\lambda) + T_n^{(3)} (\lambda) \big) - G_n (\lambda) \bigg| &= o_\p (1), \label{tn23_eq} \\
    \sup_{0 \leq \lambda \leq 1} \big| T_n^{(4)} (\lambda) - \Lambda_n (\lambda) \| \delta \|^2 \big| &= o_\p (\sigma_n / \sqrt{n}), \label{tn4_eq}\\
    \sup_{0 \leq \lambda \leq 1} \big| \Lambda_n (\lambda) - \Lambda_0 (\lambda) \big| \| \delta \|^2 &= o_\p (\sigma_n / \sqrt{n} ), \label{lambdalambda}
\end{align}
where $\Lambda_0$ and $\Lambda_n$ have been defined in \eqref{Lambda_0} and \eqref{Lambda_n}, respectively, and $G_n$ is a centered Gaussian process with covariance function $(\lambda_1, \lambda_2 ) \mapsto \lambda_1^3 \lambda_2^3 (\lambda_1 \wedge \lambda_2)$.  From these results and the decomposition \eqref{decomp1} the statement of Theorem \ref{thm2.1} follows. 

\noindent
{\bf Proof of \eqref{tn1_eq}.}
Note that $T_n^{(1)}$ can be decomposed further into terms
\begin{align}
    T_n^{(1)} (\lambda) &= \frac{N_m (\gbr{\lambda (n-\hat k_n)})}{N_m (n-\hat k_n)} P_n (\lambda) - Q_n (\lambda) + \frac{N_m (\gbr{\lambda \hat k_n})}{N_m (\hat k_n)} R_n (\lambda),\label{eq_1}
\end{align}
where 
\begin{align}\label{det1a}
    P_n (\lambda) &:= \frac{1}{N_m (\hat k_n) p } \sum_{\substack{i_1,i_2=1\\|i_1-i_2| > m}}^{\gbr{\lambda \hat k_n}} \tilde{X}_{i_1}^\top \tilde{X}_{i_2}, \\
    Q_n (\lambda) &:= \frac{1}{N_m (\hat k_n) N_m (n-\hat k_n) p } \sum_{\substack{i_1,i_2=1\\|i_1-i_2| > m}}^{\gbr{\lambda \hat k_n}} \sum_{\substack{j_1,j_2=\hat k_n+1\\|j_1-j_2| > m}}^{\hat k_n + \gbr{\lambda (n-\hat k_n)}} (\tilde{X}_{j_1}^\top \tilde{X}_{i_2} + \tilde{X}_{i_1}^\top \tilde{X}_{j_2}), \label{det1b} \\
    R_n (\lambda) &:= \frac{1}{N_m (n-\hat k_n) p } \sum_{\substack{j_1,j_2=\hat k_n+1\\|j_1-j_2| > m}}^{\hat k_n + \gbr{\lambda (n-\hat k_n)}} \tilde{X}_{j_1}^\top \tilde{X}_{j_2}. \label{det1c}
\end{align}
By Proposition \ref{lem65} (together with condition \ref{assA1undA2}) and Theorem \ref{cpthm} (note that the error term in this result is of order $o_\p (1)$ by condition \ref{assA6}), we obtain
\begin{align}\label{comb1}
    \sup_{0 \leq \lambda \leq 1} \bigg| \frac{\sqrt{n}}{\sigma_n} P_n (\lambda) \bigg| = o_\p (1) ~~~~~ \text{and} ~~~~~ \sup_{0 \leq \lambda \leq 1} \bigg| \frac{\sqrt{n}}{\sigma_n} R_n (\lambda) \bigg| = o_\p (1).
\end{align}
Furthermore, applying Lemma~\ref{lem1} twice yields for $Q_n (\lambda)$
\begin{align}\label{eq_2}
    N_m (\hat k_n)& N_m (n- \hat k_n) p \cdot Q_n (\lambda)\\
    &= 2 \sum_{i=1}^{\gbr{\lambda \hat k_n} - m} \sum_{j=\hat k_n+1}^{\hat k_n + \gbr{\lambda (n-\hat k_n)} -m} (\gbr{\lambda \hat k_n} - m - i) (\hat k_n + \gbr{\lambda (n-\hat k_n)} - m - j) \cdot \tilde{X}_i^\top \tilde{X}_j\nonumber\\
    &~~~~~+ 2 \sum_{i=1}^{\gbr{\lambda \hat k_n} - m} \sum_{j=\hat k_n+m+1}^{\hat k_n + \gbr{\lambda (n-\hat k_n)}} (\gbr{\lambda \hat k_n} - m - i) (j - \hat k_n - m - 1) \cdot \tilde{X}_i^\top \tilde{X}_j\nonumber\\
    &~~~~~+ 2 \sum_{i=m+1}^{\gbr{\lambda \hat k_n}} \sum_{j=\hat k_n+1}^{\hat k_n + \gbr{\lambda (n-\hat k_n)} - m} (i-m-1) (\hat k_n + \gbr{\lambda (n-\hat k_n)} - m - j) \cdot \tilde{X}_i^\top \tilde{X}_j\nonumber\\
    &~~~~~+ 2 \sum_{i=m+1}^{\gbr{\lambda \hat k_n}} \sum_{j=\hat k_n+m+1}^{\hat k_n + \gbr{\lambda (n-\hat k_n)}} (i-m-1) (j - \hat k_n - m - 1) \cdot \tilde{X}_i^\top \tilde{X}_j.\nonumber
\end{align}
These four terms are now estimated separately. As a representative example consider the first term. We have using 
the Cauchy-Schwarz and Young's inequality that
\begin{align}
    &\Bigg | \sum_{i=1}^{\gbr{\lambda \hat k_n} - m} \sum_{j=\hat k_n+1}^{\hat k_n + \gbr{\lambda (n-\hat k_n)} -m} (\gbr{\lambda \hat k_n} - m - i) (\hat k_n + \gbr{\lambda (n-\hat k_n)} - m - j) \cdot \tilde{X}_i^\top \tilde{X}_j    \Bigg |\nonumber\\
    & ~~~~~ ~~~~~  = \Bigg | \bigg( \sum_{i=1}^{\gbr{\lambda \hat k_n} - m}  (\gbr{\lambda \hat k_n} - m - i) \cdot \tilde{X}_i \bigg)^\top \nonumber\\
    & ~~~~~ ~~~~~ ~~~~~ ~~~~~ ~~~~~ ~~~~~ ~~~~~  \times \bigg( \sum_{j=\hat k_n+1}^{\hat k_n + \gbr{\lambda (n-\hat k_n)} -m} (\hat k_n + \gbr{\lambda (n-\hat k_n)} - m - j) \cdot \tilde{X}_j \bigg)    \Bigg |\nonumber\\
    & ~~~~~  \leq \bigg\| \sum_{i=1}^{\gbr{\lambda \hat k_n} - m}  (\gbr{\lambda \hat k_n} - m - i) \cdot \tilde{X}_i \bigg\|_2 \nonumber \\
    & ~~~~~ ~~~~~ ~~~~~ ~~~~~ ~~~~~ ~~~~~ ~~~~~  \times \bigg\| \sum_{j=\hat k_n+1}^{\hat k_n + \gbr{\lambda (n-\hat k_n)} -m} (\hat k_n + \gbr{\lambda (n-\hat k_n)} - m - j) \cdot \tilde{X}_j  \bigg\|_2 \nonumber\\
    & ~~~~~  \lesssim \bigg\| \sum_{i=1}^{\gbr{\lambda \hat k_n} - m}  (\gbr{\lambda \hat k_n} - m - i) \cdot \tilde{X}_i \bigg\|_2^2 \label{eq_3} \\
    & ~~~~~ ~~~~~ ~~~~~ ~~~~~ ~~~~~ ~~~~~ ~~~~~ + \bigg\| \sum_{j=\hat k_n+1}^{\hat k_n + \gbr{\lambda (n-\hat k_n)} -m} (\hat k_n + \gbr{\lambda (n-\hat k_n)} - m - j) \cdot \tilde{X}_j  \bigg\|_2^2.\nonumber
\end{align}
Since we are interested in an estimate for $Q_n (\lambda)$, which is multiplied by the factor $N_m (\hat k_n) N_m (n-\hat k_n) p$ of order $n^4 p$ in \eqref{eq_2}, we consider 
\begin{align*}
    &\frac{1}{n^4  p } \sup_{0 \leq \lambda \leq 1} \bigg\| \sum_{i=1}^{\gbr{\lambda \hat k_n} - m}  (\gbr{\lambda \hat k_n} - m - i) \cdot \tilde{X}_i \bigg\|_2^2 \\
    & ~~~~~~~~ \lesssim \sup_{0 \leq \lambda \leq 1} \frac{(\gbr{\lambda \hat k_n} - m)^2}{n^4  p } \bigg\| \sum_{i=1}^{\gbr{\lambda \hat k_n} - m} \tilde{X}_i \bigg\|_2^2 + \frac{1}{n^2 p } \sup_{0 \leq \lambda \leq 1} \bigg\| \sum_{i=1}^{\gbr{\lambda \hat k_n} - m} \frac{i}{n} \tilde{X}_i \bigg\|_2^2 \\
    & ~~~~~~~~ \leq\frac{1}{n^2 p } \bigg( \sup_{0 \leq \lambda \leq 1} \bigg\| \sum_{i=1}^{\gbr{\lambda \hat k_n} - m} \tilde{X}_i \bigg\|_2^2 + \sup_{0 \leq \lambda \leq 1} \bigg\| \sum_{i=1}^{\gbr{\lambda \hat k_n} - m} \frac{i}{n} \tilde{X}_i \bigg\|_2^2 \bigg) 
    \\ & ~~~~~~~~ \lesssim \frac{\log^2 (n)}{n p }  \mathrm{tr} (\bar \Gamma),
\end{align*}
where the last inequality follows from Proposition \ref{max2}. The same argument shows for the second term in \eqref{eq_3} that 
\begin{align*}
    \frac{1}{n^4 p } \e \bigg[ \sup_{0 \leq \lambda \leq 1} \bigg\| \sum_{j=\hat k_n+1}^{\hat k_n + \gbr{\lambda (n-\hat k_n)} -m} (\hat k_n + \gbr{\lambda (n-\hat k_n)} - m - j) \cdot \tilde{X}_j  \bigg\|_2^2 \bigg] \lesssim \frac{\log^2 (n)}{n p }  \mathrm{tr} (\bar \Gamma).
\end{align*}
Since all other terms that appear in the decomposition of $Q_n (\lambda)$ have a similar structure, the same arguments show that their (absolute) expected values are of the order $\frac{\log^2 (n)}{n p }  \mathrm{tr} (\bar \Gamma)$, as well. Consequently, with conditions \ref{assA1undA2} and \ref{assA3} it follows that 
\begin{align}\label{comb2}
    \e \bigg [ \sup_{0 \leq \lambda \leq 1} \bigg| \frac{\sqrt{n}}{\sigma_n} Q_n (\lambda) \bigg| \bigg] \lesssim \frac{\log^2 (n)}{\sqrt{n}} \frac{\mathrm{tr}(\bar \Gamma)}{ \| \bar \Gamma \|_F } \frac{\| \bar \Gamma \|_F }{\sqrt{\delta^\top \Gamma \delta}} = o (1).
\end{align}
Combining \eqref{comb1} and \eqref{comb2} with 
\begin{align*}
    \sup_{0 \leq \lambda \leq 1} \bigg| \frac{N_m (\gbr{\lambda \hat k_n})}{N_m (\hat k_n)} \bigg| = O_\p(1) ~~~~~ \text{and} ~~~~~  \sup_{0 \leq \lambda \leq 1} \bigg| \frac{N_m (\gbr{\lambda (n-\hat k_n)})}{N_m (n-\hat k_n)} \bigg| = O_\p(1)
\end{align*}
finally yields the estimate \eqref{tn1_eq}.

\noindent
\textbf{Proof of \eqref{tn23_eq}.}
Note that
\begin{align}\label{exex}
    \e[X_i] - \e [X_j] = \begin{cases}
        - \mathbbm{1} \{ i \leq k_0 \} \cdot \delta, & \text{if } \hat k_n > k_0, \\
        - \mathbbm{1} \{ j \geq k_0 + 1 \} \cdot \delta, &\text{if } \hat k_n \leq k_0.
    \end{cases}
\end{align}
From now on, we will consider the case $\hat k_n > k_0$.  The other case can be treated in the same way. Hence, using \eqref{exex} we can rearrange the terms inside the two quadruple sums in $T_n^{(2)} (\lambda) + T_n^{(3)} (\lambda)$, henceforth $(T_n^{(2)} (\lambda) + T_n^{(3)} (\lambda)) \mathbbm{1} \{ \hat k_n > k_0 \} $, as
\begin{align}\label{tn23_kn_ln}
    (T_n^{(2)} (\lambda) + T_n^{(3)} (\lambda)) \mathbbm{1} \{ \hat k_n > k_0 \} = ( K_n (\lambda) - L_n (\lambda) )\mathbbm{1} \{ \hat k_n > k_0 \},
\end{align}
where
\begin{align*}
    K_n (\lambda) &:= \frac{1}{N_m (\hat k_n) N_m (n-\hat k_n) p }\sum_{\substack{i_1,i_2=1\\|i_1-i_2| > m}}^{\gbr{\lambda \hat k_n}}   \sum_{\substack{j_1,j_2=\hat k_n+1\\|j_1-j_2| > m}}^{\hat k_n + \gbr{\lambda (n-\hat k_n)}} \negthickspace\negthickspace\negthickspace \big( \mathbbm{1} \{ i_1 \leq k_0 \} \tilde{X}_{j_2}^\top \delta + \mathbbm{1} \{ i_2 \leq k_0 \} \tilde{X}_{j_1}^\top \delta \big)\\
    L_n (\lambda) &:= \frac{1}{N_m (\hat k_n) N_m (n-\hat k_n) p }  \sum_{\substack{i_1,i_2=1\\|i_1-i_2| > m}}^{\gbr{\lambda \hat k_n}} \sum_{\substack{j_1,j_2=\hat k_n+1\\|j_1-j_2| > m}}^{\hat k_n + \gbr{\lambda (n-\hat k_n)}} \negthickspace \negthickspace \negthickspace\big( \mathbbm{1} \{ i_1 \leq k_0 \} \tilde{X}_{i_2}^\top \delta + \mathbbm{1} \{ i_2 \leq k_0 \} \tilde{X}_{i_1}^\top \delta \big)\\
    &= \frac{N_m (\gbr{\lambda (n - \hat k_n)})}{N_m (n-\hat k_n)} \frac{1}{N_m (\hat k_n) p}  \sum_{\substack{i_1,i_2=1\\|i_1-i_2| > m}}^{\gbr{\lambda \hat k_n}} \big( \mathbbm{1} \{ i_1 \leq k_0 \} \tilde{X}_{i_2}^\top \delta + \mathbbm{1} \{ i_2 \leq k_0 \} \tilde{X}_{i_1}^\top \delta \big).
\end{align*}
The term $K_n (\lambda)$ can be rewritten as 
\begin{align}
    K_n (\lambda)
    &= \frac{2}{N_m (\hat k_n)  }\sum_{\substack{i_1,i_2=1\\|i_1-i_2| > m}}^{\gbr{\lambda \hat k_n}} \mathbbm{1} \{ i_1 \leq k_0 \} \cdot \frac{1}{N_m (n-\hat k_n) p}  \sum_{\substack{j_1,j_2=\hat k_n+1\\|j_1-j_2| > m}}^{\hat k_n + \gbr{\lambda (n-\hat k_n)}} \tilde{X}_{j_2}^\top \delta.\label{kn}
\end{align}
We will show below that 
\begin{align}\label{op1}
    \sup_{0 \leq \lambda \leq 1} \bigg| N_m (\hat k_n)^{-1} \negthickspace \negthickspace \sum_{\substack{i_1,i_2=1\\|i_1-i_2| > m}}^{\gbr{\lambda \hat k_n}} \negthickspace\negthickspace \mathbbm{1} \{ i_1 \leq k_0 \} -  \lambda^2 \bigg| = O_\p \bigg ( \frac{m}{n} \bigg  )
\end{align}
and 
\begin{align}
    \begin{aligned}\label{convergence1}
        & \frac{\sqrt{n}}{\sigma_n} \frac{1}{N_m (n-\hat k_n) p } \negthickspace \negthickspace \sum_{\substack{j_1,j_2=\hat k_n+1\\|j_1-j_2| > m}}^{\hat k_n + \gbr{\lambda (n-\hat k_n)}} \tilde{X}_{j_1}^\top \delta\\
        & ~~~~~ ~~~~~ = \frac{\lambda}{1 - \vartheta_0} \frac{p^{-1} \sqrt{\delta^\top \Gamma \delta}}{\sigma_n} \big(
            \mathbb{B} (\vartheta_0 + \lambda (1 - \vartheta_0)) - \mathbb{B} (\vartheta_0)
        \big) + o_\p (1)
    \end{aligned}
\end{align}
uniformly with respect to $\lambda \in [0,1]$.
Combining these estimates  with \eqref{kn} yields
\begin{align}\label{kn_en}
    \sup_{0 \leq \lambda \leq 1} \bigg| \frac{\sqrt{n}}{\sigma_n} K_n (\lambda) - E_n (\lambda) \bigg| = o_\p (1),
\end{align}
where 
\begin{align*}
    E_n (\lambda) := \frac{2 \lambda^3}{1 - \vartheta_0} \frac{p^{-1} \sqrt{\delta^\top \Gamma \delta}}{\sigma_n} \big(
        \mathbb{B} (\vartheta_0 + \lambda (1 - \vartheta_0)) - \mathbb{B} (\vartheta_0)
    \big).
\end{align*}
To derive a corresponding result for the term $L_n (\lambda)$ we use the following two estimates
\begin{align}\label{secondterm1}
    \sup_{0 \leq \lambda \leq 1} \bigg| \frac{N_m (\gbr{\lambda (n - \hat k_n)})}{N_m (n- \hat k_n)} - \lambda^2 \bigg| = O_\p (m/n)
\end{align}
and
\begin{align}
    \begin{aligned}\label{secondterm2}
        &\frac{\sqrt{n}}{\sigma_n} \frac{1}{N_m (\hat k_n) p } \sum_{\substack{i_1,i_2=1\\|i_1-i_2| > m}}^{\gbr{\lambda \hat k_n}} \big( 
        \mathbbm{1} \{ i_1 \leq k_0 \} \tilde{X}_{i_2}^\top \delta + \mathbbm{1} \{ i_2 \leq k_0 \} \tilde{X}_{i_1}^\top \delta
    \big)
    \\
    & ~~~~~ ~~~~~ ~~~~~ ~~~~~  = \frac{2\lambda}{\vartheta_0} \frac{p^{-1} \sqrt{\delta^\top \Gamma \delta}}{\sigma_n} \mathbb{B} \big( \lambda \vartheta_0 \big) + o_\p (1)
    \end{aligned}
\end{align}
(uniformly with respect to $\lambda \in [0,1]$). 
The proof of \eqref{secondterm1} follows by the same arguments as given in the proof of \eqref{op1}; the proof of \eqref{secondterm2} will be provided below.

With \eqref{secondterm1} and \eqref{secondterm2}, we thus obtain
\begin{align}\label{ln_dn}
    \sup_{0 \leq \lambda \leq 1} \bigg| \frac{\sqrt{n}}{\sigma_n} L_n (\lambda) - D_n (\lambda) \bigg| = o_\p (1),
\end{align}
where
\begin{align*}
    D_n (\lambda) := \frac{2\lambda^3}{\vartheta_0} \frac{p^{-1} \sqrt{\delta^\top \Gamma \delta}}{\sigma_n} \mathbb{B} \big( \lambda \vartheta_0 \big).
\end{align*}

Combining \eqref{tn23_kn_ln}, \eqref{kn_en} and \eqref{ln_dn} yields
\begin{align*}
    \sup_{0 \leq \lambda \leq 1} \bigg| \frac{\sqrt{n}}{\sigma_n} \big( T_n^{(2)} (\lambda) + T_n^{(3)} (\lambda) \big) - \big( E_n (\lambda) - D_n (\lambda) \big)
   \bigg| \mathbbm{1} \{ \hat k_n > k_0 \}  = o_\p (1),
\end{align*}
where $\{ E_n (\lambda) - D_n (\lambda) \}_{\lambda \in [0,1]}$ is a Gaussian process. Similar arguments show that the estimate holds if the indicator $\{ \hat k_n > k_0 \} $ is replaced by the indicator $\{ \hat k_n \leq k_0 \} $.

The statement in \eqref{tn23_eq} now follows if we show that the covariance function of the process $\{ E_n (\lambda) - D_n (\lambda) \}_{\lambda \in [0,1]}$ is given by $(\lambda_1, \lambda_2) \mapsto \lambda_1^3 \lambda_2^3 (\lambda_1 \wedge \lambda_2)$. Indeed, we have
\begin{align*}
    \cov \big( D_n (\lambda_1), D_n (\lambda_2) \big) &= \frac{4 \lambda_1^3 \lambda_2^3}{\vartheta_0} \bigg( \frac{p^{-1} \sqrt{\delta^\top \Gamma \delta}}{\sigma_n} \bigg)^2 (\lambda_1 \wedge \lambda_2),\\
    \cov \big( E_n (\lambda_1), E_n (\lambda_2) \big) &= \frac{4 \lambda_1^3 \lambda_2^3}{1-\vartheta_0} \bigg( \frac{p^{-1} \sqrt{\delta^\top \Gamma \delta}}{\sigma_n} \bigg)^2 (\lambda_1 \wedge \lambda_2)
\end{align*}
and $\cov \big( D_n (\lambda_1), E_n (\lambda_2) \big) = 0$ for any $\lambda_1, \lambda_2 \in [0,1]$. This yields
\begin{align*}
    \cov \big( E_n (\lambda_1) - D_n (\lambda_1) , E_n (\lambda_2) - D_n (\lambda_2) \big) &=  \frac{\lambda_1^3 \lambda_2^3 (\lambda_1 \wedge \lambda_2)}{\sigma_n^2} \cdot \frac{4}{p^2} \bigg( \frac{\delta^\top \Gamma \delta}{\vartheta_0} + \frac{\delta^\top \Gamma \delta}{1 - \vartheta_0} \bigg)\\
    &= \lambda_1^3 \lambda_2^3 (\lambda_1 \wedge \lambda_2)
\end{align*}
by the definition of $\sigma_n^2$ in \eqref{variance}.

\bigskip
\noindent
\textbf{Proof of \eqref{tn4_eq}.}
Using \eqref{exex}, we obtain
\begin{align}\nonumber
    T_n^{(4)} (\lambda) &= \frac{ \mathbbm{1} \{ \hat k_n > k_0 \} \| \delta \|_2^2}{N_m (\hat k_n)N_m (n- \hat k_n) p } \sum_{\substack{i_1,i_2=1\\|i_1-i_2| > m}}^{\gbr{\lambda \hat k_n}}  \sum_{\substack{j_1,j_2=\hat k_n+1\\|j_1-j_2| > m}}^{\hat k_n + \gbr{\lambda (n-\hat k_n)}} \mathbbm{1} \{ i_1 \leq k_0 \} \mathbbm{1} \{ i_2 \leq k_0 \}\\
    \nonumber 
    & ~~~~~  + \frac{ \mathbbm{1} \{ \hat k_n \leq k_0 \} \| \delta \|_2^2}{N_m (\hat k_n)N_m (n- \hat k_n) p } \sum_{\substack{i_1,i_2=1\\|i_1-i_2| > m}}^{\gbr{\lambda \hat k_n}}  \sum_{\substack{j_1,j_2=\hat k_n+1\\|j_1-j_2| > m}}^{\hat k_n + \gbr{\lambda (n-\hat k_n)}} \negthickspace \negthickspace \mathbbm{1} \{ j_1 \geq k_0 + 1\} \mathbbm{1} \{ j_2 \geq k_0 + 1 \}\\
    &= \| \delta \|^2 \mathbbm{1} \{ \hat k_n > k_0 \}  \frac{ N_m (\gbr{\lambda (n - \hat k_n)}) }{N_m (n- \hat k_n)} C_n^{(1)} (\lambda) \label{hd1} \\
    & ~~~~~ ~~~~~ ~~~~~ ~~~~~ ~~~~~ ~~~~~ ~~~~~ ~~~~~ ~~~~~ ~~~~~  + \| \delta \|^2 \mathbbm{1} \{ \hat k_n \leq k_0 \} \frac{ N_m (\gbr{\lambda \hat k_n}) }{N_m (\hat k_n)}  C_n^{(2)} (\lambda), \nonumber
\end{align}
where
\begin{align}
    \begin{aligned}\label{gnhn}
        C_n^{(1)} (\lambda) &:=  \frac{1}{N_m (\hat k_n)} \sum_{\substack{i_1,i_2=1\\|i_1-i_2| > m}}^{\gbr{\lambda \hat k_n}} \mathbbm{1} \{ i_1 \leq k_0 \} \mathbbm{1} \{ i_2 \leq k_0 \},\\
        C_n^{(2)} (\lambda) &:= \frac{ 1 }{N_m (n- \hat k_n)}  \sum_{\substack{j_1,j_2=\hat k_n+1\\|j_1-j_2| > m}}^{\hat k_n + \gbr{\lambda (n-\hat k_n)}} \mathbbm{1} \{ j_1 \geq k_0 + 1\} \mathbbm{1} \{ j_2 \geq k_0 + 1 \}.
    \end{aligned}
\end{align}
Observing that $C_n^{(1)}$  in \eqref{hd1} is multiplied by the factor $\mathbbm{1} \{ \hat k_n > k_0 \} $, we have 
\begin{align*}
    C_n^{(1)} (\lambda) &= \frac{\mathbbm{1} \{ \gbr{\lambda \hat k_n} \leq k_0 \} }{N_m (\hat k_n)} \sum_{\substack{i_1,i_2=1\\|i_1-i_2| > m}}^{\gbr{\lambda \hat k_n}} \mathbbm{1} \{ i_1 \leq k_0 \} \mathbbm{1} \{ i_2 \leq k_0 \}\\
    &~~~~~~~~~~ + \frac{\mathbbm{1} \{ \gbr{\lambda \hat k_n} > k_0 \} }{N_m (\hat k_n)} \sum_{\substack{i_1,i_2=1\\|i_1-i_2| > m}}^{\gbr{\lambda \hat k_n}} \mathbbm{1} \{ i_1 \leq k_0 \} \mathbbm{1} \{ i_2 \leq k_0 \}\\
    &= \mathbbm{1} \{ \gbr{\lambda \hat k_n} \leq k_0 \}  \frac{N_m (\gbr{\lambda \hat k_n})}{N_m (\hat k_n)} + \mathbbm{1} \{ \gbr{\lambda \hat k_n} > k_0 \} \frac{N_m (k_0)}{N_m (\hat k_n)}\\
    &= \frac{N_m (\gbr{\lambda \hat k_n} \wedge k_0)}{N_m (\hat k_n)} .
\end{align*}

Below, we will prove the following statement:
\begin{align}\label{napproxes1}
    \sup_{\substack{0 \leq \lambda \leq 1 \\ \gbr{\lambda \hat k_n} > k_0  }} \bigg| \frac{N_m (k_0)}{N_m (\hat k_n)} - \frac{N_m (\gbr{\lambda \hat k_n})}{N_m (\hat k_n)} \bigg| = O_\p \bigg(  \frac{\mathrm{tr} (\bar \Gamma)}{ \| \delta \|_2^2} \frac{\log^2 (n)}{n} \bigg),
\end{align}
which yields for the case $\hat k_n > k_0$ that
\begin{align}\label{gn_eq}
    \sup_{0 \leq \lambda \leq 1} \bigg| C_n^{(1)} (\lambda) - \frac{N_m (\gbr{\lambda \hat k_n})}{N_m (\hat k_n)} \bigg| = O_\p \bigg(  \frac{\mathrm{tr} (\bar \Gamma)}{ \| \delta \|_2^2} \frac{\log^2 (n)}{n} \bigg).
\end{align}
Define
\begin{align}\label{hd5}
    \mathcal{L} := \{ \lambda \in [0,1] \mid k_0 \leq \hat k_n + \gbr{\lambda (n - \hat k_n)} \}. 
\end{align}
Noting that $C_n^{(2)}$ is multiplied by the factor $\mathbbm{1} \{ \hat k_n \leq k_0\}$, then this term can be rewritten as
\begin{align*}
    C_n^{(2)} (\lambda) =  \frac{ \mathbbm{1}_\mathcal{L} }{N_m (n- \hat k_n)}   \sum_{\substack{j_1,j_2=k_0+1\\|j_1-j_2| > m}}^{\hat k_n + \gbr{\lambda (n-\hat k_n)}} 1 + 0 \cdot  \mathbbm{1}_{\mathcal{L}^C}    &=  \frac{ N_m (\hat k_n - k_0 + \gbr{\lambda(n - \hat k_n)}) }{N_m (n- \hat k_n)} \mathbbm{1}_\mathcal{L}.
\end{align*}
We will prove below that for $\hat k_n \le k_0$
\begin{align}\label{napproxes2}
    \sup_{\lambda \in \mathcal{L}} \bigg| \frac{ N_m (\hat k_n - k_0 + \gbr{\lambda(n - \hat k_n)}) }{N_m (n- \hat k_n)} - \frac{ N_m (\gbr{\lambda(n - \hat k_n)}) }{N_m (n- \hat k_n)} \bigg| &= O_\p \bigg(  \frac{\mathrm{tr} (\bar \Gamma)}{ \| \delta \|_2^2} \frac{\log^2 (n)}{n} \bigg),
    \\
    \label{napproxes3}
    \sup_{\lambda \in \mathcal{L}^C} \bigg| \frac{ N_m (\gbr{\lambda(n - \hat k_n)}) }{N_m (n- \hat k_n)}\bigg| &= O_\p \bigg(  \frac{\mathrm{tr} (\bar \Gamma)}{ \| \delta \|_2^2} \frac{\log^2 (n)}{n} \bigg),
\end{align}
which yields for the case $\hat k_n \leq k_0$
\begin{align}\label{hn_eq}
    \sup_{0 \leq \lambda \leq 1} \bigg| C_n^{(2)} (\lambda) - \frac{ N_m (\gbr{\lambda(n - \hat k_n)}) }{N_m (n- \hat k_n)} \bigg| = O_\p \bigg(  \frac{\mathrm{tr} (\bar \Gamma)}{ \| \delta \|_2^2} \frac{\log^2 (n)}{n} \bigg).
\end{align}
Combining \eqref{gn_eq} and \eqref{hn_eq} and observing the definition of $\Lambda_n (\lambda)$ in \eqref{Lambda_n}
gives
\begin{align}
    \begin{aligned}\label{gnhnlambda}
        \sup_{0 \leq \lambda \leq 1 } \bigg| \mathbbm{1} \{ \hat k_n > k_0 \}  \frac{ N_m (\gbr{\lambda (n - \hat k_n)}) }{N_m (n- \hat k_n)} C_n^{(1)} (\lambda) + \mathbbm{1} \{ \hat k_n \leq k_0 \} \frac{ N_m (\gbr{\lambda \hat k_n}) }{N_m (\hat k_n)} C_n^{(2)} (\lambda) - \Lambda_n (\lambda) \bigg|\\
         = O_\p \bigg(  \frac{\mathrm{tr} (\bar \Gamma)}{ \| \delta \|_2^2} \frac{\log^2 (n)}{n} \bigg).
    \end{aligned}
\end{align}
Finally, we obtain for the  term $T_n^{(4)}$ 
in \eqref{hd1} that 
\begin{align*}
    \sup_{0 \leq \lambda \leq 1} \big| T_n^{(4)} - \Lambda_n (\lambda) \| \delta \|^2 \big| = O_\p \bigg( \frac{\mathrm{tr} (\bar \Gamma)}{ p} \frac{\log^2 (n)}{n} \bigg) = o_\p (\sigma_n / \sqrt{n}), 
\end{align*}
where the last estimate holds by conditions \ref{assA1undA2} and \ref{assA3}, which imply
\begin{align}\label{opinclusion}
     \frac{\mathrm{tr} (\bar \Gamma)}{p} \frac{\log^2 (n)}{n} \cdot \frac{\sqrt{n}}{\sigma_n} = \frac{\vartheta_0 (1-\vartheta_0)}{4} \frac{\mathrm{tr} (\bar \Gamma)}{ \sqrt{\delta^\top \Gamma \delta} } \frac{\log^2 (n)}{\sqrt{n}}= o(1).
\end{align}
This finishes the proof of equality \eqref{tn4_eq}.
\medskip

\noindent
\textbf{Proof of \eqref{lambdalambda}.}
By the triangle inequality, we obtain
\begin{align}
    \big| \Lambda_n (\lambda) - \Lambda_0 (\lambda) \big| &\leq \bigg| \frac{N_m (\gbr{\lambda \hat k_n})}{N_m (\hat k_n)}  \frac{ N_m (\gbr{\lambda(n - \hat k_n)}) }{N_m (n- \hat k_n)} - \frac{N_m (\gbr{\lambda k_0})}{N_m (\hat k_n)}  \frac{ N_m (\gbr{\lambda(n - k_0)}) }{N_m (n- \hat k_n)} \bigg|\label{gleichung1}\\
    & ~~~ + \bigg| \frac{N_m (\gbr{\lambda k_0})}{N_m (\hat k_n)}  \frac{ N_m (\gbr{\lambda(n - k_0)}) }{N_m (n- \hat k_n)}  - \frac{N_m (\gbr{\lambda k_0})}{N_m (k_0)} \frac{N_m (\gbr{\lambda (n - k_0)})}{N_m (n-k_0)}  \bigg|. \label{gleichung2}
\end{align}
To derive a bound for \eqref{gleichung1}, we proceed as follows
\begin{align*}
    &\bigg| \frac{N_m (\gbr{\lambda \hat k_n})}{N_m (\hat k_n)}  \frac{ N_m (\gbr{\lambda(n - \hat k_n)}) }{N_m (n- \hat k_n)} - \frac{N_m (\gbr{\lambda k_0})}{N_m (\hat k_n)}  \frac{ N_m (\gbr{\lambda(n - k_0)}) }{N_m (n- \hat k_n)} \bigg|\\
    &~~~~~ \leq \frac{1}{N_m (\hat k_n) N_m (n- \hat k_n)} \Big\{ \Big| N_m (\gbr{\lambda \hat k_n}) \big( N(\gbr{\lambda (n - \hat k_n)}) - N_m (\gbr{\lambda (n - k_0)}) \big) \Big|\\
    & ~~~~~ ~~~~~ ~~~~~ ~~~~~ ~~~~~ ~~~~~ ~~~~~ ~~~~~ + \Big|N_m (\gbr{\lambda (n - k_0)}) \big( N_m (\gbr{\lambda \hat k_n}) - N_m (\gbr{\lambda k_0}) \big) \Big|\Big\}\\
    & ~~~~~ \leq \frac{N_m (\gbr{\lambda \hat k_n}) +  N_m (\gbr{\lambda (n - k_0)}) }{N_m (\hat k_n) N_m (n- \hat k_n)} \max_{\substack{1 \leq a,b \leq n  \\ |a-b| \lesssim |\hat k_n - k_0|}} \big| N_m (a) - N_m (b) \big|,
\end{align*}
since by the definition of $\gbr{\cdot}$
\begin{align*}
    | \gbr{\lambda \hat k_n} - \gbr{\lambda k_0} | \leq |\hat k_n - k_0 | ~~~~~ \text{and} ~~~~~ | \gbr{\lambda (n - \hat k_n)} - \gbr{\lambda (n - k_0)} | \leq |\hat k_n - k_0|.    
\end{align*}
Using similar arguments as before, we have
\begin{align*}
    \sup_{0 \leq \lambda \leq 1} \bigg| \frac{N_m (\gbr{\lambda \hat k_n}) +  N_m (\gbr{\lambda (n - k_0)}) }{N_m (\hat k_n) N_m (n- \hat k_n)} n^2 \bigg| = O_\p (1)
\end{align*}
and by Lemma~\ref{lemNmUpperbound}
\begin{align*}
    \frac{1}{n^2}\max_{\substack{1 \leq a,b \leq n  \\ |a-b| \lesssim |\hat k_n - k_0|}} \big| N_m (a) - N_m (b) \big| \lesssim \frac{(m+n) |\hat k_n - k_0|}{n^2} = O_\p \bigg(\frac{\log^2 (n)}{n} \frac{\mathrm{tr} (\bar \Gamma)}{\| \delta \|_2^2} \bigg).
\end{align*}
For \eqref{gleichung2}, we add and subtract $N_m (\hat k_n)^{-1} N_m (n - k_0)^{-1}$ and obtain
\begin{align*}
    &\bigg| \frac{N_m (\gbr{\lambda k_0})}{N_m (\hat k_n)}  \frac{ N_m (\gbr{\lambda(n - k_0)}) }{N_m (n- \hat k_n)}  - \frac{N_m (\gbr{\lambda k_0})}{N_m (k_0)} \frac{N_m (\gbr{\lambda (n - k_0)})}{N_m (n-k_0)}  \bigg|\\
    & ~~~~~ \leq N_m (\gbr{\lambda k_0}) N _m (\gbr{\lambda (n - k_0)}) \bigg\{ N_m(\hat k_n)^{-1} \Big| N_m (n - \hat k_n)^{-1} - N_m (n - k_0)^{-1} \Big| \\
    & ~~~~~ ~~~~~ ~~~~~ ~~~~~ ~~~~~ ~~~~~ ~~~~~ ~~~~~ ~~~~~ ~~~~~   + N_m (n - k_0)^{-1} \Big| N_m (\hat k_n)^{-1} - N_m (k_0)^{-1} \Big| \bigg\}.
\end{align*}
Using $\lambda \leq 1$, we have for the first term
\begin{align*}
    &N_m (\gbr{\lambda k_0}) N _m (\gbr{\lambda (n - k_0)})  N_m(\hat k_n)^{-1} \Big| N_m (n - \hat k_n)^{-1} - N_m (n - k_0)^{-1} \Big|\\
    & ~~~~~ ~~~~~ ~~~~~ ~~~~~ ~~~~~ ~~~~~ \leq \frac{N_m (k_0)}{N_m(\hat k_n)}   \bigg| \frac{ N_m (n - k_0) - N_m (n - \hat k_n)}{N_m (n - \hat k_n)}  \bigg|\\
    & ~~~~~ ~~~~~ ~~~~~ ~~~~~ ~~~~~ ~~~~~ \leq \frac{N_m (k_0) n^2}{N_m(\hat k_n) N_m (n - \hat k_n)} \frac{1}{n^2} \max_{\substack{1 \leq a,b \leq n \\ |a-b| \leq |k_0 - \hat k_n| }} \big|N_m (a) - N_m (b) \big|\\
    & ~~~~~ ~~~~~ ~~~~~ ~~~~~ ~~~~~ ~~~~~ = O_\p \bigg( \frac{\log^2 (n)}{n} \frac{\mathrm{tr} (\bar \Gamma)}{\| \delta \|_2^2} \bigg)
\end{align*}
by the same arguments used before. Similarly, we obtain for the second term
\begin{align*}
    &N_m (\gbr{\lambda k_0}) N _m (\gbr{\lambda (n - k_0)})  N_m(n - \hat k_0)^{-1} \Big| N_m (\hat k_n)^{-1} - N_m (k_0)^{-1} \Big|\\
    & ~~~~~ ~~~~~ ~~~~~ ~~~~~ ~~~~~ ~~~~~ ~~~~~ ~~~~~ ~~~~~ ~~~~~ ~~~~~ ~~~~~ ~~~~~ ~~~~~ ~~~~~ ~~~~~ ~~~~~ ~~~~~ ~~~~~ = O_\p \bigg( \frac{\log^2 (n)}{n} \frac{\mathrm{tr} (\bar \Gamma)}{\| \delta \|_2^2} \bigg).
\end{align*}
Combining \eqref{gleichung1} and \eqref{gleichung2} yields
\begin{align}\label{fürspäter}
    \sup_{0 \leq \lambda \leq 1} \big| \Lambda_n (\lambda) - \Lambda_0 (\lambda) \big| = O_\p \bigg( \frac{\log^2 (n)}{n} \frac{\mathrm{tr} (\bar \Gamma)}{\| \delta \|_2^2} \bigg)
\end{align}
and this in turn implies
\begin{align}\label{fürspäter2}
    \sup_{0 \leq \lambda \leq 1} \big| \Lambda_n (\lambda) - \Lambda_0 (\lambda) \big| \| \delta \|^2 = O_\p \bigg(  \frac{\mathrm{tr} (\bar \Gamma)}{p} \frac{\log^2 (n)}{n} \bigg) = o_\p (\sigma_n / \sqrt{n}),
\end{align}
where the last equality follows by \eqref{opinclusion}. This shows the final equality \eqref{lambdalambda} and the proof is finished.

\subsection{ Proof of the remaining statements for the proof of Theorem \ref{thm2.1}  }
Finally, we will provide proofs for the statements \eqref{op1}, \eqref{convergence1},  \eqref{secondterm2}, \eqref{napproxes1}, \eqref{napproxes2} and \eqref{napproxes3} that are required for the proof of Theorem \ref{thm2.1}.


\medskip

\noindent
\textbf{Proof of \eqref{op1}}.
A straightforward calculation gives 
\begin{align*}
    \sum_{\substack{i_1,i_2=1\\|i_1-i_2| > m}}^{\gbr{\lambda \hat k_n}} \negthickspace\negthickspace \mathbbm{1} \{ i_1 \leq k_0 \} =  N_m (\gbr{\lambda \hat k_n})  -  \mathbbm{1} \{ k_0 \leq \gbr{\lambda \hat k_n} \} J_n (\lambda),
\end{align*}
where
\begin{align*}
    J_n (\lambda) := \sum_{i=1}^{\gbr{\lambda \hat k_n} - m} (\gbr{\lambda \hat k_n} - m - i) \cdot \mathbbm{1} \{ i \geq k_0 + 1 \} + \sum_{i=m+1}^{\gbr{\lambda \hat k_n}} (i - m - 1) \cdot \mathbbm{1} \{ i \geq k_0 + 1 \}.
\end{align*}
This yields 
\begin{align}
\label{hd2}
    \begin{aligned}
        &\sup_{0 \leq \lambda \leq 1} \bigg| \frac{1}{N_m (\hat k_n)} \negthickspace \negthickspace \sum_{\substack{i_1,i_2=1\\|i_1-i_2| > m}}^{\gbr{\lambda \hat k_n}} \negthickspace\negthickspace \mathbbm{1} \{ i_1 \leq k_0 \} -  \lambda^2 \bigg|\\
        & ~~~~~ ~~~~~ ~~~~~ ~~~~~ ~~~~~ ~~~~~  ~~~~~ ~~~~~ ~~~~~ \leq \sup_{0 \leq \lambda \leq 1} \bigg| \frac{N_m (\gbr{\lambda \hat k_n})}{N_m (\hat k_n)} - \lambda^2 \bigg| + \sup_{\substack{0 \leq \lambda \leq 1 \\k_0 \leq \gbr{\lambda \hat k_n}}} \bigg| \frac{J_n (\lambda)}{N_m (\hat k_m)} \bigg|
    \end{aligned}
\end{align}
and in the following, we show that both terms are of order $O_\p (m/n)$. 
For the first term, we use the decomposition 
\begin{align}
    \sup_{0 \leq \lambda \leq 1} \bigg| \frac{N_m (\gbr{\lambda \hat k_n})}{N_m (\hat k_n)} - \lambda^2 \bigg| &\leq \sup_{0 \leq \lambda \leq 1} \bigg| \frac{N_m (\gbr{\lambda \hat k_n})}{N_m (\hat k_n)} - \frac{N_m (\gbr{\lambda \hat k_n})}{\hat k_n^2} \bigg| \label{sup1}\\
    & ~~~~~ ~~~~~ + \sup_{0 \leq \lambda \leq 1} \bigg|  \frac{N_m (\gbr{\lambda \hat k_n})}{\hat k_n^2}  - \frac{N_m (\gbr{\lambda \hat k_n} + m + 1)}{\hat k_n^2} \bigg| \label{sup2}\\
    & ~~~~~ ~~~~~ + \sup_{0 \leq \lambda \leq 1} \bigg| \frac{N_m (\gbr{\lambda \hat k_n} + m + 1)}{\hat k_n^2} - \lambda^2 \bigg|. \label{sup3}
\end{align}
By Theorem \ref{cpthm}, we have (note that $\vartheta_0>0)$
\begin{align}\label{taylorconv}
    \frac{1}{\hat \vartheta_n} =
    \frac{1}{\vartheta_0} + o_\p (1),
\end{align}
which gives for the term \eqref{sup1}
\begin{align*} 
    \sup_{0 \leq \lambda \leq 1} \bigg| \frac{N_m (\gbr{\lambda \hat k_n})}{N_m (\hat k_n)} - \frac{N_m (\gbr{\lambda \hat k_n})}{\hat k_n^2} \bigg| &\leq \sup_{0 \leq \lambda \leq 1} \bigg| \frac{N_m(\gbr{\lambda \hat k_n})}{N_m (\hat k_n)} \bigg| \cdot \bigg| \frac{\hat k_n^2 - N_m (\hat k_n)}{\hat k_n^2} \bigg|\\
    &\leq \sup_{0 \leq \lambda \leq 1} \bigg| \frac{2m + 1}{\hat k_n} - \frac{m (m+1)}{\hat k_n^2} \bigg|\\
    &\leq \sup_{0 \leq \lambda \leq 1} \bigg| \frac{2m / n + 1 / n}{\hat \vartheta_n} - \frac{m/n \cdot (m/n+1/n)}{\hat \vartheta_n^2} \bigg|\\
    &= O_\p (m/n).
\end{align*}
For the term in \eqref{sup2} we have by Lemma~\ref{lemNmUpperbound} and \eqref{taylorconv}
\begin{align*}
    \sup_{0 \leq \lambda \leq 1} \bigg|  \frac{N_m (\gbr{\lambda \hat k_n})}{\hat k_n^2}  - \frac{N_m (\gbr{\lambda \hat k_n} + m + 1)}{\hat k_n^2} \bigg|  &\leq \hat k_n^{-2} \max_{\substack{1 \leq a,b \leq n \\ |a-b| \lesssim m}} |N_m (a) - N_m (b)|\\
    &= O_\p (m (m+n) / n^2) = O_\p (m/n),
\end{align*}
because $|\gbr{\lambda \hat k_n} - (\gbr{\lambda \hat k_n} + m  + 1)| \lesssim m$. 
Finally, by the definition of $\gbr{\cdot}$ we have for \eqref{sup3}
\begin{align*}
    \lambda^2 - \frac{\lambda}{\hat k_n} \leq \frac{N_m (\gbr{\lambda \hat k_n} + m +1)}{\hat k_n^2} \leq \lambda^2 + \frac{\lambda}{\hat k_n}.
\end{align*}
This implies that
\begin{align*}
    \sup_{0 \leq \lambda \leq 1} \bigg| \frac{N_m (\gbr{\lambda \hat k_n} + m + 1)}{\hat k_n^2} - \lambda^2 \bigg| = O_\p (n^{-1}),
\end{align*}
which proves
\begin{align}\label{for_later2}
    \sup_{0 \leq \lambda \leq 1} \bigg| \frac{N_m (\gbr{\lambda \hat k_n})}{N_m (\hat k_n)} - \lambda^2 \bigg| = O_\p (m/n).
\end{align}
Consequently, the first term in \eqref{hd2} is of order $O_\p (m/n)$ and it remains to show that the second term is of the same order.

Rewriting the two sums contained in $J_n (\lambda)$, we obtain
\begin{align*}
    J_n(\lambda) &= \sum_{i=1}^{\gbr{\lambda \hat k_n}} (\gbr{\lambda \hat k_n} - 1) \mathbbm{1} \{ i \geq k_0 + 1 \} + \sum_{i = \gbr{\lambda \hat k_n} - m + 1}^{\gbr{\lambda \hat k_n}} (i - \gbr{\lambda \hat k_n} + m) \mathbbm{1} \{ i \geq k_0 + 1 \}\\
    & ~~~~~~ ~~~~~~ ~ + \sum_{i = 1}^m (m - i + 1) \mathbbm{1} \{ i \geq k_0 + 1 \}\\
    & \lesssim \sum_{i = k_0 + 1}^{\gbr{\lambda \hat k_n}} (\gbr{\lambda \hat k_n} -1 ) + m^2\\
    & = (\gbr{\lambda \hat k_n} - k_0) (\gbr{\lambda \hat k_n} -1 ) + m^2\\
    & \leq (\hat k_n - k_0) \hat k_n + m^2.
\end{align*}
Therefore, by Theorem \ref{cpthm}, it follows that 
\begin{align*}
    \sup_{\substack{0 \leq \lambda \leq 1 \\ k_0 \leq \gbr{\lambda \hat k_n}}} \bigg| \frac{J_n(\lambda)}{N_m (\hat k_n )} \bigg| \lesssim \frac{n^2}{N_m (\hat k_n)} \bigg( |\hat \vartheta_n - \vartheta_0|  + \frac{m^2}{n^2} \bigg) = O_\p (m/n),
\end{align*}
because $n^2 / N_m (\hat k_n) = O_\p (1)$. This completes the proof of \eqref{op1}.

\bigskip

\noindent\textbf{Proof of \eqref{convergence1}}.
Note that this equality has to be shown for the case $\hat k_n > k_0$ as the term $K_n$ in \eqref{kn} is multiplied by the indicator $\mathbbm{1} \{ \hat k_n > k_0 \}$ in \eqref{tn23_kn_ln}. By Lemma~\ref{lem1}, we have
\begin{align}
    \sum_{\substack{j_1,j_2=\hat k_n+1\\|j_1-j_2| > m}}^{\hat k_n + \gbr{\lambda (n-\hat k_n)}} \tilde{X}_{j_1}^\top \delta &= \sum_{\substack{j_1,j_2=\hat k_n+1\\|j_1-j_2| > m}}^{\hat k_n + \gbr{\lambda (n-\hat k_n)}} \rho_{j_1}^\top \delta \nonumber \\
    &= \sum_{j = \hat k_n + 1}^{\hat k_n + \gbr{\lambda (n - \hat k_n)} - m } (\hat k_n + \gbr{\lambda (n - \hat k_n)} - m - j) \cdot \rho_j^\top \delta \nonumber \\
    & ~~~~~ + \sum_{j = \hat k_n + m + 1}^{\hat k_n + \gbr{\lambda (n - \hat k_n)}} (j - \hat k_n - m - 1) \cdot \rho_j^\top \delta \nonumber\\
    &=  R_n^{(1)}(\lambda) +  R_n^{(2)} (\lambda) - R_n^{(3)}, \label{hd8}
\end{align}
where
\begin{align}
    R_n^{(1)} (\lambda)&= \sum_{j = \hat k_n + 1}^{\hat k_n + \gbr{\lambda (n - \hat k_n)}} (\gbr{\lambda (n - \hat k_n)} - 2 m - 1) \cdot \rho_j^\top \delta, \label{eintausend}\\
    R_n^{(2)}(\lambda) & = \sum_{j = \hat k_n + \gbr{\lambda (n - \hat k_n)} - m + 1}^{\hat k_n + \gbr{\lambda (n - \hat k_n)}} (j + m - \hat k_n - \gbr{\lambda (n - \hat k_n)} )
    \cdot \rho_j^\top \delta, \label{dreitausend}\\
    R_n^{(3)}  & = \sum_{j = \hat k_n + 1}^{\hat k_n + m} (j - \hat k_n - m - 1) \cdot \rho_j^\top \delta. \label{zweitausend}
\end{align}
Including the additional factors from \eqref{convergence1}, we obtain by Theorem \ref{thm2.0} for the first term
\begin{align*}
    & \frac{\sqrt{n}}{\sigma_n} \cdot \frac{1}{N_m (n - \hat k_n) p } R_n^{(1)}  (\lambda) \\
    & ~~~~~ ~~~~~ ~~~~~ = \frac{\sqrt{n}}{\sigma_n} \cdot \frac{\gbr{\lambda (n - \hat k_n)} - 2m - 1}{N_m (n - \hat k_n) p } \sum_{j = \hat k_n + 1}^{\hat k_n + \gbr{\lambda (n - \hat k_n)} } \rho_j^\top \delta\\
    & ~~~~~ ~~~~~ ~~~~~ = \frac{\gbr{\lambda (n - \hat k_n)} - 2m - 1}{n - \hat k_n - m} \frac{n}{n - \hat k_n - m - 1} \frac{ p^{-1} \sqrt{\delta^\top \Gamma \delta} }{\sigma_n} \\
    & ~~~~~ ~~~~~ ~~~~~ ~~~~~ ~~~~~ \times \bigg\{ \frac{1}{\sqrt{n} \sqrt{\delta^\top \Gamma \delta}} \sum_{j = 1}^{\hat k_n + \gbr{\lambda (n - \hat k_n)} }  \rho_j^\top \delta  -  \frac{1}{\sqrt{n} \sqrt{\delta^\top \Gamma \delta}} \sum_{j=1}^{\hat k_n} \rho_j^\top \delta \bigg\}\\
    & ~~~~~ ~~~~~ ~~~~~ = \frac{\lambda}{1 - \vartheta_0} \frac{ p^{-1} \sqrt{\delta^\top \Gamma \delta}}{\sigma_n}  \bigg( \mathbb{B} \big( (\hat k_n + \gbr{\lambda (n - \hat k_n)}) / n \big) - \mathbb{B} \big( \hat k_n / n \big) \bigg) + o_\p (1) \\
    & ~~~~~ ~~~~~ ~~~~~ = \frac{\lambda}{1 - \vartheta_0} \frac{ p^{-1} \sqrt{\delta^\top \Gamma \delta}}{\sigma_n}  \Big( \mathbb{B} \big( \vartheta_0 + \lambda (1 - \vartheta_0) \big) - \mathbb{B} \big( \vartheta_0 \big) \Big) + o_\p (1)
\end{align*}
uniformly with respect to $\lambda \in [0, 1]$, where we used 
\begin{align*}
    \sup_{0 \leq \lambda \leq 1} \big| (\hat k_n + \gbr{\lambda (n - \hat k_n)})/n - (\vartheta_0 + \lambda (1 - \vartheta_0)) \big| = o_\p (1) ~~~~~ \text{and} ~~~~~ \big| \hat k_n / n - \vartheta_0 \big| = o_\p (1),
\end{align*}
and the uniform continuity of the Brownian motion in the last step.

For the terms $R_n^{(2)} (\lambda) $ in \eqref{dreitausend} and $R_n^{(3)} $ in \eqref{zweitausend}, we apply Proposition \ref{prop_maxmax_m} to $(\rho_j^\top\delta)_{j \in \mathbb{Z}}$ to obtain 
\begin{align*}
     \e \bigg[ \sup_{0 \leq \lambda \leq 1} \big| R_n^{(2)} (\lambda) \big| \bigg] &\leq \e \bigg[ \max_{1 \leq k \leq n} \bigg|\sum_{j = k+1}^{k+m} (j-k) \cdot \rho_{j-m}^\top \delta \bigg| \bigg] \lesssim m^{3/2} \sqrt{n} \bigg( \sum_{h = 0}^{m-1} |\delta^\top \Sigma_h \delta| \bigg)^{1/2}
\end{align*}
and
\begin{align*}
    \e \big[\big| R_n^{(3)} \big| \big]  &\leq \e \bigg[ \max_{1 \leq k \leq n} \bigg|\sum_{j = k+1}^{k+m} (j-k - m - 1) \cdot \rho_j^\top \delta \bigg| \bigg] \lesssim  m^{3/2} \sqrt{n} \bigg( \sum_{h = 0}^{m-1} \big|\delta^\top \Sigma_h \delta \big| \bigg)^{1/2}.
\end{align*}
Including additional factors from \eqref{convergence1} (note that $N_m (n - \hat k_n) \simp n^2$), we conclude
\begin{align*}
    \frac{\sqrt{n}}{n^2 p \sigma_n} \e \bigg[ \sup_{0 \leq \lambda \leq 1} \Big| R_n^{(2)} (\lambda ) - R_{n}^{(3)} \Big| \bigg] \lesssim \frac{m^{3/2}}{n} \bigg( \frac{1}{\delta^\top \Gamma \delta} \sum_{h = 0}^{m-1} \big|\delta^\top \Sigma_h \delta \big| \bigg)^{1/2} = o(1),
\end{align*}
which proves \eqref{convergence1}.

\bigskip

\noindent
\textbf{Proof of \eqref{secondterm2}}.
A straightforward calculation yields 
\begin{align}  \label{toinduce}      &\sum_{\substack{i_1,i_2=1\\|i_1-i_2| > m}}^{\gbr{\lambda \hat k_n}} \big( 
            \mathbbm{1} \{ i_1 \leq k_0 \} \tilde{X}_{i_2}^\top \delta + \mathbbm{1} \{ i_2 \leq k_0 \} \tilde{X}_{i_1}^\top \delta
        \big)\\
        & ~~~~~ ~~~~~ = \sum_{\substack{i_1, i_2 = 1 \\ |i_1 - i_2| > m}}^{k_0 \wedge \gbr{\lambda \hat k_n}} (\eta_{i_1}^\top \delta + \eta_{i_2}^\top \delta) + 2 \cdot \mathbbm{1} \{ k_0 +1\leq \gbr{\lambda \hat k_n} \} \sum_{i= k_0 + 1}^{\gbr{\lambda \hat k_n}} \big ((i - m - 1) \wedge k_0  \big )\cdot \rho_i^\top \delta. \nonumber 
\end{align}
By Lemma~\ref{lem1} and similar arguments as given in the derivation of \eqref{eintausend} -- \eqref{zweitausend}, we obtain for the first term on the right-hand side of \eqref{toinduce}

\begin{align*}
    \sum_{\substack{i_1, i_2 = 1 \\ |i_1 - i_2| > m}}^{k_0 \wedge \gbr{\lambda \hat k_n}} (\eta_{i_1}^\top \delta + \eta_{i_2}^\top \delta) = 2 \big( S_n^{(1)} (\lambda) + S_n^{(2)} (\lambda) + S_n^{(3)} \big),
\end{align*}
where
\begin{align*}
    S_{n}^{(1)} (\lambda) &= \sum_{j = 1}^{k_0 \wedge \gbr{\lambda \hat k_n}} (k_0 \wedge \gbr{\lambda \hat k_n} - 2m - 1) \cdot \eta_{j}^\top \delta,\\
    S_{n}^{(2)} (\lambda) &= \sum_{j = k_0 \wedge \gbr{\lambda \hat k_n} - m + 1}^{k_0 \wedge \gbr{\lambda \hat k_n}} (j + m - k_0 \wedge \gbr{\lambda \hat k_n}) \cdot \eta_{j}^\top \delta, \\
    S_{n}^{(3)} &= \sum_{j = 1}^m (m+1-j) \cdot \eta_{j}^\top \delta.
\end{align*}
For the term $S_n^{(2)}$, we have by Proposition \ref{prop_maxmax_m}
\begin{align*}
    \e \Big[ \sup_{0 \leq \lambda \leq 1} |S_{n}^{(2)} (\lambda)| \Big] \leq \e \bigg[ \max_{1 \leq k \leq n} \Big| \sum_{j = k + 1}^{k +m} (j-k) \cdot \rho_{j-m}^\top \delta \Big| \bigg] &\lesssim m^{3/2} \sqrt{n} \bigg( \sum_{h = 0}^{m-1} |\delta^\top \Sigma_h \delta| \bigg)^{1/2}
\end{align*}
and for $S_n^{(3)}$ Hölder's inequality yields 
\begin{align*}
    \e \big[ \big|S_{n}^{(3)} \big| \big] &\leq \bigg( \sum_{j_1,j_2 = 1}^m (m+1-j_1) (m+1 - j_2) \cdot \delta^\top \Sigma_{j_1 - j_2} \delta \bigg)^{1/2}\\
    &\lesssim m^{3/2} \bigg( \sum_{h = 0}^{m-1} |\delta^\top \Sigma_h \delta| \bigg)^{1/2}.
\end{align*}
Therefore, observing that $N_m (\hat k_n) \simp n^2$ we conclude with condition \ref{assA100} that
\begin{align*}
    \sup_{\lambda \in [0, 1]} \frac{\sqrt{n}}{N_m (\hat k_n) p \sigma_n}  \big| S_{n}^{(2)} (\lambda) + S_{n}^{(3)} \big| &= o_\p (1).
\end{align*}
Turning to the first term, that is, $S_n^{(1)}$ we have by Theorem \ref{thm2.0} that 
\begin{align*}
    \frac{\sqrt{n}}{\sigma_n} \frac{1}{N_m (\hat k_n) p } S_n^{(1)} (\lambda) &= \frac{n}{\hat k_n - m} \frac{k_0 \wedge \gbr{\lambda \hat k_n} - 2m - 1}{\hat k_n - m - 1} \frac{ p^{-1} \sqrt{\delta^\top \Gamma \delta}}{\sigma_n} \frac{1}{\sqrt{n}\sqrt{\delta^\top \Gamma \delta}} \sum_{i=1}^{k_0 \wedge \gbr{\lambda \hat k_n} } \eta_i^\top \delta\\
    & = \frac{\lambda}{\vartheta_0} \frac{ p^{-1} \sqrt{\delta^\top \Gamma \delta}}{\sigma_n} \mathbb{B} \big((k_0 \wedge \gbr{\lambda \hat k_n}) / n \big)  + o_\p (1)\\
    & = \frac{\lambda}{\vartheta_0} \frac{ p^{-1} \sqrt{\delta^\top \Gamma \delta}}{\sigma_n} \mathbb{B} \big(\lambda \vartheta_0 \big)  + o_\p (1)
\end{align*}
uniformly with respect to $\lambda \in [0,1]$, where we have used the uniform continuity of the Brownian motion again with
\begin{align*}
    \sup_{0 \leq \lambda \leq 1} \big| (k_0 \wedge \gbr{\lambda \hat k_n}) / n - \lambda \vartheta_0 \big| = o_\p (1).
\end{align*}
Therefore, it remains to show that the second term in \eqref{toinduce} converges to zero. For this purpose, note that for $k_0 + 1 \leq \gbr{\lambda \hat k_n}$
\begin{align}\label{furtherdecomp}
    \sum_{i = k_0 + 1}^{\gbr{\lambda \hat k_n}} \big ((i-m-1) \wedge k_0  \big ) \cdot \rho_i^\top \delta = \sum_{i = k_0 + 1}^{k_0 + m}  (i-m-1) \cdot \rho_i^\top \delta + k_0 \sum_{i=k_0 + m + 1}^{\gbr{\lambda \hat k_n}} \rho_i^\top \delta.
\end{align}
For the first term in \eqref{furtherdecomp}, we have by Proposition \ref{prop_maxmax_m} (note that $i-m-1 \leq n$) 
\begin{align*}
    \e \bigg[ \max_{1 \leq k \leq n} \sum_{i = k + 1}^{k + m} (i-m - 1) \cdot \rho_i^\top \delta \bigg] \lesssim m \cdot n^{3/2} \bigg( \sum_{h= 0 }^{m-1} | \delta^\top \Sigma_h \delta| \bigg)^{1/2}.
\end{align*}
Therefore, noting that $N_m (\hat k_n) \simp n^2$ and conditions \ref{assA100} and \ref{assA5} we obtain
\begin{align*}
    \frac{\sqrt{n}}{N_m (\hat k_n) p \sigma_n} \bigg| \sum_{i = k_0 + 1}^{k_0 + m} (i-m-1) \cdot \rho_i^\top \delta \bigg| = o(1).
\end{align*}
Finally, for the second term in \eqref{furtherdecomp} we have by Theorem \ref{thm2.0} and the modulus of continuity of the Brownian motion that
\begin{align*}
    \frac{\sqrt{n} k_0}{N_m (\hat k_n) p \sigma_n} \sum_{i=k_0 + m + 1}^{\gbr{\lambda \hat k_n}} \rho_i^\top \delta &= \frac{n}{\hat k_n - m} \frac{k_0}{\hat k_n - m - 1} \frac{p^{-1} \sqrt{\delta^\top \Gamma \delta}}{\sigma_n} \\ & ~~~~~~~~~~~~~~ \times  \bigg( \frac{1}{\sqrt{n} \sqrt{\delta^\top \Gamma \delta}} \sum_{i = 1}^{\gbr{\lambda \hat k_n} } \rho_i^\top \delta - \frac{1}{\sqrt{n} \sqrt{\delta^\top \Gamma \delta}} \sum_{i=1}^{k_0 + m} \rho_i^\top \delta \bigg)\\
    &=\frac{1}{\vartheta_0} \frac{p^{-1} \sqrt{\delta^\top \Gamma \delta}}{\sigma_n} \Big( \mathbb{B} \big( \gbr{\lambda \hat k_n} / n \big) - \mathbb{B} \big((k_0 + m) / n \big) \Big) + o_\p (1)\\
    &= o_\p (1)
\end{align*}
uniformly with respect to $\lambda \in \{ \lambda  \in [0,1] \mid k_0 + 1 \leq \gbr{\lambda \hat k_n} \}$, because
\begin{align*}
    \big| \gbr{\lambda \hat k_n} / n - (k_0 + m) / n \big| \leq \frac{\gbr{\lambda \hat k_n} - k_0}{n} + \frac{m}{n} \leq \frac{\hat k_n - k_0}{n} + \frac{m}{n} = o_\p (1) . 
\end{align*}
This proves \eqref{secondterm2}.
\medskip

\noindent
\textbf{Proof of \eqref{napproxes1}}.
We must show this equality for the case $\gbr{\lambda \hat k_n} > k_0$. In this case, we have by Theorem \ref{cpthm} that 
\begin{align}\label{rateRN}
    |k_0 - \gbr{\lambda \hat k_n}| = \gbr{\lambda \hat k_n} - k_0 \leq \hat k_n - k_0  \lesssim n (\hat \vartheta_n - \vartheta_0) \lesssim 
     r_n := \log^2 (n) \mathrm{tr} (\bar \Gamma) / \| \delta \|_2^2.
\end{align}
This implies by Lemma~\ref{lemNmUpperbound} that we have
\begin{align*}
    \sup_{\substack{0 \leq \lambda \leq 1 \\ \gbr{\lambda \hat k_n} > k_0 }} \bigg| \frac{N_m (k_0)}{N_m (\hat k_n)} - \frac{N_m (\gbr{\lambda \hat k_n})}{N_m (\hat k_n)} \bigg| &\leq \frac{1}{N_n (\hat k_n)} \max_{\substack{1 \leq a,b \leq n \\ |a-b| \lesssim r_n}} | N_m (a) - N_m (b) |\\
    &\lesssim \frac{r_n (n+m)}{N_m (\hat k_n)}\\
    &=O_\p \bigg( \frac{\log^2 (n)}{n} \frac{\mathrm{tr} (\bar \Gamma)}{\| \delta \|_2^2} \bigg),
\end{align*}
since $N_m (\hat k_n) \simp n^2$ and we obtain \eqref{napproxes1}.
\medskip

\noindent
\textbf{Proof of \eqref{napproxes2}}.
We must show this equality for the case $\hat k_n \leq k_0$.
Recall the definition of $\mathcal{L}$ in \eqref{hd5}, then we have for $\lambda \in \mathcal{L}$ that  
\begin{align*}
    |\hat k_n - k_0 + \gbr{\lambda (n-\hat k_n)} - \gbr{\lambda (n-\hat k_n)} | = k_0 - \hat k_n \lesssim n (\hat \vartheta_n - \vartheta_0) \lesssim r_n,
\end{align*}
where $r_n$ was defined in \eqref{rateRN}. By the arguments as in the proof of \eqref{napproxes1}, we obtain
\begin{align*}
    &\sup_{\lambda \in \mathcal{L}} \bigg| \frac{ N_m (\hat k_n - k_0 + \gbr{\lambda(n - \hat k_n)}) }{N_m (n- \hat k_n)} - \frac{ N_m (\gbr{\lambda(n - \hat k_n)}) }{N_m (n- \hat k_n)} \bigg|\\
    & \quad\quad\quad\quad\quad\quad \quad\quad\quad\quad\quad \quad\quad\quad\quad\quad
    \le \frac{1}{N_m (n- \hat k_n)} \max_{\substack{1 \leq a,b \leq n \\ |a-b| \lesssim r_n}} |N_m (a) - N_m (b)| 
\end{align*}
and the remaining part of the proof is now identical to the proof of \eqref{napproxes1}, noting that $N_m (n - \hat k_n) \simp n^2$.
\medskip

\noindent
\textbf{Proof of \eqref{napproxes3}}.
We must show this equality for the case $\hat k_n \leq k_0$, as well. For $\lambda \in \mathcal{L}^C$, we have
\begin{align*}
    |\gbr{\lambda (n - \hat k_n)} | < k_0 - \hat k_n \lesssim n (\hat \vartheta_n - \vartheta_0) \lesssim r_n,
\end{align*}
where $r_n$ was defined in \eqref{rateRN}. This leaves
\begin{align*}    
    \sup_{\lambda \in \mathcal{L}^C} \bigg| \frac{ N_m (\gbr{\lambda(n - \hat k_n)}) }{N_m (n- \hat k_n)}\bigg| \leq \frac{1}{N_m (n- \hat k_n)} \max_{\substack{m + 1 \leq a,b \leq n \\ |a-b| \lesssim r_n}} |N_m (a) - N_m (b)|
\end{align*}
and the remaining part of the proof is now identical to the proof of \eqref{napproxes1} again.

\subsection{Proof of Theorem \ref{testconsistent}}

Before we begin, we state the following result.

\begin{prop}\label{pq_prop}
    Let $\Lambda_1, \Lambda_2 : [0,1] \to [0,1]$ be functions and $p,q : \ell^\infty ([0, 1]) \to \mathbb{R}$ maps defined as
    \begin{align*}
        p(f) &:= \bigg( \int_0^1 \Big(f(\lambda) - \Lambda_1 (\lambda) \cdot f(1) \Big)^2 \mathrm{d} \lambda \bigg)^{1/2},  \\ q(f) & := \bigg( \int_0^1 \big(f(\lambda) - \Lambda_2 (\lambda) \cdot  f(1) \big)^2 \mathrm{d} \lambda \bigg)^{1/2}.
    \end{align*}
    Then,
    \begin{align*}
        |p(f) - q(g)| \leq 2 \sup_{0 \leq \lambda \leq 1} |f(\lambda) - g(\lambda)| + |g(1)| \cdot \sup_{0 \leq \lambda \leq 1} \big| \Lambda_1 (\lambda)  - \Lambda_2 (\lambda) \big|.
    \end{align*}
    In particular, the mappings $p,q : \ell^\infty ([0,1]) \to \mathbb{R}$ are Lipschitz-continuous.
\end{prop}

\begin{proof}
    Let $f,g \in \ell^\infty([0, 1])$. Then, we have by the reverse Minkowski inequality in $L^2$ 
    \begin{align*}
        |p(f) - q(g)| &\leq  \bigg( \int_0^1 \Big( f(\lambda) - g(\lambda) + \Lambda_1 (\lambda) \cdot f(1) - \Lambda_2 (\lambda) \cdot g(1) \Big)^2 \mathrm{d} \lambda \bigg)^{1/2} \\
        &\le  \bigg( \int_0^1 \big( f(\lambda) - g(\lambda) \big)^2 \mathrm{d} \lambda \bigg)^{1/2}   \\
        & ~~~~~ ~~~~~ ~~~~~ + \bigg( \int_0^1 \Big( \Lambda_1 (\lambda) \big(f(1) - g(1) \big) + g(1) \big( \Lambda_1 (\lambda) - \Lambda_2 (\lambda) \big) \Big)^2\mathrm{d} \lambda \bigg)^{1/2}   \\
        &\leq \sup_{0 \leq \lambda \leq 1} |f(\lambda) - g(\lambda)| + |f(1) - g(1)| + |g(1)| \sup_{0 \leq \lambda \leq 1} \big| \Lambda_1 (\lambda)  - \Lambda_2 (\lambda) \big|  \\
        &\leq 2 \sup_{0 \leq \lambda \leq 1} |f(\lambda) - g(\lambda)| + |g(1)| \sup_{0 \leq \lambda \leq 1} \big| \Lambda_1 (\lambda)  - \Lambda_2 (\lambda) \big|,
    \end{align*}
    where we have used that $\Lambda_1 (\lambda) \leq 1$ for any $\lambda \in [0,1]$.
\end{proof}

\noindent
Now, we turn to proving Theorem \eqref{testconsistent} and define 
\begin{align*}
    \tilde V_n := \bigg( \int_0^1 (T_n (\hat k_n, m; \lambda) - \Lambda_0 (\lambda) T_n (\hat k_n, m  ) )^2 \mathrm{d} \nu (\lambda) \bigg)^{1/2},
\end{align*}
where $\Lambda_0$ is defined in \eqref{Lambda_0}. Recall $\Lambda_n$ from \eqref{Lambda_n} and let $p,q$ be the maps from Proposition \ref{pq_prop} with $\Lambda_1 = \Lambda_0$ and $\Lambda_2 = \Lambda_n$. With the notation $u_n (\lambda) = T_n (\hat k_n, m ;\lambda)$, we have $p(u_n) = \tilde V_n$ and $q(u_n) = V_n$, where $V_n$ was defined in \eqref{vn}.
Now Proposition \ref{pq_prop} and Theorem \ref{thm2.1} yield 
\begin{align*}
    \frac{\sqrt{n}}{\sigma_n} |\tilde V_n -  V_n| = \frac{\sqrt{n}}{\sigma_n} |p(u_n) - q (u_n)| &\lesssim \frac{\sqrt{n}}{\sigma_n} |T_n (\hat k_n, m; 1)| \cdot \sup_{0 \leq \lambda \leq 1} |\Lambda_0 (\lambda) - \Lambda_n (\lambda)|\\
    &= \sup_{0 \leq \lambda \leq 1} |\Lambda_0 (\lambda) - \Lambda_n (\lambda)| \bigg( O_\p (1) + \frac{\sqrt{n}}{\sigma_n} \| \delta \|^2 \bigg).
\end{align*}
By \eqref{fürspäter}, \eqref{fürspäter2} and conditions \ref{assA1undA2}, \ref{assA3} and \ref{assA6}, we have 
\begin{align*}
    \frac{\sqrt{n}}{\sigma_n}|\tilde V_n -  V_n| = o_\p (1),
\end{align*}
and therefore
\begin{align}\label{vv_neu}
    \Big | \frac{V_n}{\tilde V_n} - 1 \Big | = \frac{ \frac{\sqrt{n}}{\sigma_n} |V_n - \tilde V_n |}{\frac{\sqrt{n}}{\sigma_n} \tilde V_n} = o_\p (1),
\end{align}
since the denominator converges in distribution. We now consider the map
\begin{align*}
  f \mapsto \frac{f(1)}{\Big( \int_0^1 (f(\lambda) - \Lambda_0 (\lambda) f(1)  )^2 ~ \mathrm{d} \nu (\lambda) \Big)^{1/2}}
\end{align*}
from $\ell^\infty  ([0,1]$ onto $\mathbb{R}$
(for a  given  probability measure $\nu$ on $[0, 1]$), and obtain by Theorem \ref{thm2.1} and the continuous mapping theorem \citep[see][page 67]{vanweak}
\begin{align*}
    \frac{T_n (\hat k_n , m) - \| \delta \|^2}{\tilde V_n} \convd 
    \mathbb{G}, 
\end{align*}
where $\mathbb{G} $ is defined in \eqref{defG}.
Combining this with \eqref{vv_neu} yields
\begin{align}\label{convw}
    \frac{T_n (\hat k_n , m) - \| \delta \|^2}{V_n} \convd 
    \mathbb{G}.
\end{align}
Next, we rearrange the decision rule \eqref{rule} and obtain 
\begin{align}
    \frac{T_n (\hat k_n , m) - \| \delta \|^2}{ V_n} > \frac{\Delta - \| \delta \|^2}{ V_n} + q_{1-\alpha}.
    \label{det112}
\end{align}
First, note that under the null hypothesis, we have $\| \delta \|^2 \leq \Delta $. Therefore, in this case, we obtain for the probability of rejection from \eqref{convw} and \eqref{det112}
$$
    \limsup_{n \to \infty } \p  \big (  {T_n (\hat k_n , m) >    \Delta    } + { V_n}  q_{1-\alpha} \big ) \leq \limsup_{n \to \infty } \p \bigg( \frac{T_n (\hat k_n , m) - \| \delta \|^2}{ V_n}  >   q_{1-\alpha} \bigg) = \alpha ,
$$
which shows that the test \eqref{rule} has asymptotic level $\alpha$. Next, note that it follows from Theorem \ref{thm2.1} that $\sqrt{n} \tilde V_n / \sigma_n \convd ( \int_0^1 \lambda^6 (\mathbb{B} (\lambda) - \lambda \mathbb{B} (1) )^2 \mathrm{d} \nu (\lambda) )^{1/2}$, which is a positive random variable with probability $1$ by Lemma~\ref{lemmasubexp}. Therefore, if $\frac{\sqrt{n}}{\sigma_n}  | \Delta - \| \delta \|^2 | \to 0$, it follows by \eqref{vv_neu} that
\begin{align*}
    \p \bigg(\frac{T_n (\hat k_n , m) - \| \delta \|^2}{ V_n} >q_{1-\alpha} \bigg) \longrightarrow \alpha
\end{align*}
as $n \to \infty$, by \eqref{convw} (note that the distribution function of $\mathbb{G}$ is continuous). By the same argument, we have in the case $\frac{\sqrt{n}}{\sigma_n}  |\Delta - \| \delta \|^2 | \to \infty$ that
\begin{align*}
    \bigg| \frac{\Delta - \| \delta \|^2}{V_n} \bigg|  = \bigg| \frac{\sigma_n^{-1}(\Delta - \| \delta \|^2) / \sqrt{n}}{ \sigma_n^{-1}   V_n / \sqrt{n} } \bigg| \convp  \infty. 
\end{align*}
The statement now follows by considering the cases $\Delta > \| \delta \|^2$ and $\Delta < \| \delta \|^2$ separately.

\subsection{Proof of Theorem \ref{newuniformconsistency}}

Since
\begin{align*}
    \inf_{\delta \in \mathcal{A}_n } \p \big( T_n (\hat k_n, m) > \Delta + V_n q_{1- \alpha} \big)= 1 - \sup_{\delta \in  \mathcal{A}_n} \p \big( T_n (\hat k_n, m) \leq \Delta + V_n q_{1- \alpha} \big),
\end{align*}
it suffices to show that the supremum converges to zero.
For a proof of this statement, first note that for $\delta \in \mathcal{A}_n \cup \mathcal{B}_n$, we have $\mathrm{tr} (\bar \Gamma) = O(\| \delta \|_2^2)$. 
For $\delta \in \mathcal{A}_n$, this result follows from condition \ref{assA3} and the estimate
\begin{align*}
    \frac{\mathrm{tr} (\bar \Gamma)}{\| \delta \|_2^2} < \frac{\mathrm{tr} (\bar \Gamma)}{\| \bar \Gamma \|_F} \frac{\| \delta \|^2 - \Delta}{ \| \delta \|_2^2/p} \leq O(1) ,
\end{align*}
while for $\delta \in \mathcal{B}_n$ it is a a consequence  of the estimate (note that $ 0 < C \leq   \| \delta \|^2  = \| \delta \|_2^2/p$ in this case)
\begin{align*}
    \frac{\mathrm{tr} (\bar \Gamma)}{\| \delta \|_2^2} < \frac{\mathrm{tr} (\bar \Gamma)}{\| \bar \Gamma \|_F} \frac{\Delta - \| \delta \|^2}{ \| \delta \|_2^2/p } \leq \frac{\mathrm{tr} (\bar \Gamma)}{ \| \bar \Gamma \|_F} \frac{\Delta}{C} = O(1). 
\end{align*}
This shows that the assumptions from Theorem \ref{cpthm} are satisfied for $\delta \in \mathcal{A}_n \cup \mathcal{B}_n$.
Recall the definition of $\Lambda_n$ in \eqref{Lambda_n} and the decomposition in \eqref{decomp1}. A careful inspection of the proof of \eqref{tn4_eq} shows that
\begin{align*}
    T_n (\hat k_n, m; \lambda) - \Lambda_n (\lambda) \| \delta \|^2 = \bar T_n (\hat k_n, m; \lambda) +  \mathcal{R}_n (\lambda),
\end{align*}
where the remainder $\mathcal{R}_n (\lambda) = T_n^{(4)} (\lambda) - \Lambda_n (\lambda) \| \delta \|^2$ satisfies
\begin{align*}
    \sup_{0 \leq \lambda \leq 1} | \mathcal{R}_n  (\lambda) | = o_\p (\| \bar \Gamma \|_F / p)
\end{align*}
uniformly with respect to $\delta \in \mathcal{A}_n \cup \mathcal{B}_n$ and $\lambda \in [0,1]$
\begin{align*}
    \bar T_n (\hat k_n, m; \lambda) := T_n^{(1)} (\lambda) + T_n^{(2)} (\lambda) + T_n^{(3)} (\lambda)
\end{align*}
and $T_n^{(1)}, T_n^{(2)}$ and $ T_n^{(3)}$ are defined in \eqref{det2a}, \eqref{det2b} and \eqref{det2c}, respectively.
Moreover, note that
\begin{align*}
    V_n &= \Big ( 
    \int_0^1 \big( T_n (\hat k_n, m; \lambda) - \Lambda_n (\lambda) T_n (\hat k_n, m) \big)^2 \mathrm{d} \nu \Big  )^{1/2}\\
    &= \Big  ( \int_0^1 \big( T_n (\hat k_n, m; \lambda) - \Lambda_n (\lambda) \| \delta \|^2 - \Lambda_n (\lambda) (T_n (\hat k_n, m) - \| \delta \|^2) \big)^2 \mathrm{d} \nu
    \Big  )^{1/2}
    \\
    &\lesssim \sup_{0 \leq \lambda \leq 1} \big| T_n (\hat k_n, m; \lambda ) - \Lambda_n (\lambda) \| \delta \|^2 \big|\\
    &=  \sup_{0 \leq \lambda \leq 1} |\bar T_n (\hat k_n, m ;\lambda)| + \sup_{0 \leq \lambda \leq 1} | \mathcal{R}_n (\lambda)| 
\end{align*}
uniformly with respect to $\delta \in \mathcal{A}_n \cup \mathcal{B}_n$, and clearly we also have
\begin{align*}
    |T_n (\hat k_n, m) - \| \delta \|^2| &\leq \sup_{0 \leq \lambda \leq 1} |T_n(\hat k_n, m; \lambda) - \Lambda_n (\lambda) \|\delta\|^2 | \\
    &= \sup_{0 \leq \lambda \leq 1} | \bar T_n (\hat k_n, m; \lambda) | + \sup_{0 \leq \lambda \leq 1} |  \mathcal{R}_n (\lambda) | .
\end{align*}
Consequently,
\begin{align*}
    \sup_{\delta \in \mathcal{A}_n  } & \p \big( T_n (\hat k_n, m) \leq \Delta + V_n q_{1- \alpha} \big) \\
    &\leq \sup_{\delta \in \mathcal{A}_n  } \p \big( \| \delta \|^2 - \Delta \leq |V_n q_{1- \alpha}|  + |T_n (\hat k_n, m) - \| \delta \|^2| \big) \\
    &\leq \sup_{\delta \in \mathcal{A}_n  } \p \bigg( \| \delta \|^2 - \Delta \leq C \Big \{ \sup_{0 \leq \lambda \leq 1} | \bar T_n (\hat k_n, m; \lambda) | + \sup_{0 \leq \lambda \leq 1} | \mathcal{R}_n (\lambda) |  \Big \}  \bigg) \\
    &\leq \sup_{\delta \in \mathcal{A}_n \cup \mathcal{B}_n} \p \bigg( |\| \delta \|^2 - \Delta| < C \Big \{  \sup_{0 \leq \lambda \leq 1} | \bar T_n (\hat k_n, m; \lambda) | + \sup_{0 \leq \lambda \leq 1} |\mathcal{R}_n (\lambda)| \Big \}   \bigg)
\end{align*}
and analogously
\begin{align*}
    \sup_{\delta \in \mathcal{B}_n  } &\p \big(T_n (\hat k_n, m) > \Delta + V_n q_{1- \alpha} \big)\\
    &\leq \sup_{\delta \in \mathcal{B}_n  } \p \big( |T_n (\hat k_n, m) - \| \delta \|^2| + |V_n q_{1- \alpha}|> \Delta  - \| \delta \|^2 \big) \\
    &\leq \sup_{\delta \in \mathcal{B}_n  } \p \bigg( C \Big \{ \sup_{0 \leq \lambda \leq 1} | \bar T_n (\hat k_n, m; \lambda) | + \sup_{0 \leq \lambda \leq 1} | \mathcal{R}_n (\lambda) | \Big \} > \Delta - \| \delta \|^2  \bigg)\\
    &\leq \sup_{\delta \in \mathcal{A}_n \cup \mathcal{B}_n} \p \bigg( C \Big \{ \sup_{0 \leq \lambda \leq 1} | \bar T_n (\hat k_n, m; \lambda) | + \sup_{0 \leq \lambda \leq 1}| \mathcal{R}_n (\lambda)| \Big \} > | \|\delta\|^2 - \Delta |  \bigg).
\end{align*}
Noting that $o (\| \bar \Gamma \|_F /p) = o (| \| \delta \|^2 -  \Delta |)$ (because $\| \bar \Gamma \|_F /p < |\| \delta \|^2 - \Delta|$ for $\delta \in \mathcal{A}_n \cup \mathcal{B}_n$), it follows that $\sup_{0 \leq \lambda \leq 1} |\mathcal{R}_n (\lambda) | = o_\p (| \| \delta \|^2 -  \Delta |)$,  and therefore we can continue by considering
\begin{align} \nonumber 
    \sup_{\delta \in  \mathcal{A}_n  \cup \mathcal{B}_n }  &\p \Big( \sup_{0 \leq \lambda \leq 1} | \bar T_n (\hat k_n, m; \lambda) | > |\|\delta\|^2 - \Delta|  \Big)\\
    & ~~~~~~~~~~ ~~~~~~~~~~ \begin{aligned}
        &\leq \sup_{\delta \in  \mathcal{A}_n  \cup \mathcal{B}_n } \e \bigg[ \frac{1}{\| \bar \Gamma \|_F / p } \sup_{0 \leq \lambda \leq 1} \big| T_n^{(1)} (\lambda) \big| \bigg]\\
        &\quad\quad + \sup_{\delta \in  \mathcal{A}_n  \cup \mathcal{B}_n } \e \bigg[ \frac{1}{|\|\delta\|^2 - \Delta|} \sup_{0 \leq \lambda \leq 1} \big| T_n^{(2)} (\lambda) + T_n^{(3)} (\lambda) \big| \bigg],
    \end{aligned}
    \label{det1d} 
\end{align}
where we have used Markov's inequality and $|\|\delta\|^2 - \Delta| > \| \bar \Gamma \|_F / p$ for $\delta \in \mathcal{A}_n \cup \mathcal{B}_n$. The first term containing $T_n^{(1)}$ can be rewritten similarly as in the proof of Theorem \ref{thm2.1} recalling the decomposition  \eqref{eq_1}, that is 
\begin{align*}
    T_n^{(1)} (\lambda) 
    &= \frac{N_m (\gbr{\lambda (n-\hat k_n)})}{N_m (n-\hat k_n)} P_n (\lambda) - Q_n (\lambda) + \frac{N_m (\gbr{\lambda \hat k_n})}{N_m (\hat k_n)} R_n (\lambda),
\end{align*}
where $P_n$, $Q_n$ and $R_n$ are defined by \eqref{det1a}, \eqref{det1b} and \eqref{det1c}, respectively.
For the term $Q_n$, we can apply \eqref{comb2} and obtain
\begin{align*}
    \sup_{\delta \in \mathcal{A}_n  \cup \mathcal{B}_n } \e \bigg[ \sup_{0 \leq \lambda \leq 1} \bigg| \frac{1}{\| \bar \Gamma \|_F / p} Q_n (\lambda) \bigg| \bigg]
    &= \sup_{\delta \in \mathcal{A}_n  \cup \mathcal{B}_n } \frac{\sigma_n}{\sqrt{n} \| \bar \Gamma \|_F / p}  \e \bigg[ \sup_{0 \leq \lambda \leq 1} \bigg| \frac{\sqrt{n}}{\sigma_n} Q_n (\lambda) \bigg| \bigg]\\
    & \lesssim \frac{\log^2 (n)}{n} \frac{\mathrm{tr} (\bar \Gamma)}{ \| \bar \Gamma \|_F} = o(1)
\end{align*}
by the definition of $\sigma_n$ and condition \ref{assA3}.

To find an upper bound for $P_n$ and $R_n$ we use Proposition \ref{lem65}, which gives
\begin{align*}
    \e \bigg[ \max_{1 \leq k \leq n} \bigg| \frac{1}{\| \bar \Gamma \|_F} \sum_{\substack{i,j = 1 \\ |i-j| > m}}^k \tilde{X}_i^\top \tilde{X}_j \bigg| \bigg] \lesssim n \log (n).
\end{align*}
Therefore, observing the definition of $P_n$ and $R_n$ in \eqref{det1a} and \eqref{det1c}, respectively, yields the estimates 
\begin{align*}
    \e \bigg[ \sup_{0 \leq \lambda \leq 1} \bigg| \frac{1}{\| \bar \Gamma \|_F / p} P_n (\lambda) \bigg| \bigg] = o (1) ~~~~~ \text{and} ~~~~~ \e \bigg[ \sup_{0 \leq \lambda \leq 1} \bigg| \frac{1}{\| \bar \Gamma \|_F / p} R_n (\lambda) \bigg| \bigg]= o (1)
\end{align*}
uniformly with respect to $\delta \in \mathcal{A}_n \cup \mathcal{B}_n$.  This shows that the first term in \eqref{det1d} is of the order $o(1)$.
 
In order to derive a similar bound for the remaining term in \eqref{det1d} containing the terms $T_n^{(2)}$ and $T_n^{(3)}$ defined in \eqref{det2b} and \eqref{det2c}, we will prove the estimate 
\begin{align*}
    &\e \bigg[ \frac{1}{ p |\|\delta\|^2 - \Delta| } \sup_{0 \leq \lambda \leq 1} \bigg |\sum_{\substack{i_1,i_2=1\\|i_1-i_2| > m}}^{\gbr{\lambda \hat k_n}} \sum_{\substack{j_1,j_2=\hat k_n+1\\|j_1-j_2| > m}}^{\hat k_n + \gbr{\lambda (n-\hat k_n)}} \Big( (\e [X_{i_1}] - \e [X_{j_1}])^\top (\tilde{X}_{i_2} - \tilde{X}_{j_2})\\
    & ~~~~~~~~~~ ~~~~~~ ~~~~~~~~~~ ~~~~~~~~~~ ~~~~~~~~~~~  + (\tilde{X}_{i_1} - \tilde{X}_{j_1})^\top (\e [X_{i_2}] - \e [X_{j_2}]) \Big) \bigg | \bigg] = o(n^4)
\end{align*}
uniformly with respect to $\delta \in \mathcal{A}_n \cup \mathcal{B}_n$.
Using \eqref{exex} and similar arguments as in the proof of Theorem \ref{thm2.1}, this in turn will be implied if
\begin{align}\label{dasletztedaszzist_1}
    \e \bigg[ \sup_{0 \leq \lambda \leq 1} \bigg| \frac{1}{ p |\|\delta\|^2 - \Delta| } \sum_{\substack{i_1,i_2=1\\|i_1-i_2| > m}}^{\gbr{\lambda \hat k_n} } \tilde{X}_{i_1}^\top \delta \bigg|\bigg] = o(n^2)
\end{align}
and
\begin{align*}
    \e \bigg[ \sup_{0 \leq \lambda \leq 1} \bigg| \frac{1}{ p |\|\delta\|^2 - \Delta| } \sum_{\substack{j_1, j_2 = \hat k_n + 1\\|j_1 - j_2| > m}}^{\hat k_n + \gbr{\lambda (n - \hat k_n)} } \tilde{X}_{j_1}^\top \delta \bigg|\bigg] = o(n^2).
\end{align*}
Without loss of generality, we only show the first assertion in \eqref{dasletztedaszzist_1}.
Using Lemma~\ref{lem1}, we obtain the decomposition 
\begin{align*}
    \sum_{\substack{i_1,i_2=1\\|i_1-i_2| > m}}^{\gbr{\lambda \hat k_n} } \tilde{X}_{i_1}^\top \delta = \sum_{j=1}^{\gbr{\lambda \hat k_n} - m} (\gbr{\lambda \hat k_n} - m - j) \tilde X_j^\top \delta + \sum_{j=m+1}^{\gbr{\lambda \hat k_n}} (j - m - 1) \tilde X_j^\top \delta
\end{align*}
and we restrict ourselves to the first term (the estimate corresponding to the second term is obtained by similar arguments). By Proposition 1 in \cite{strongwu} (see also Remark \ref{wu_n2hochd}), we have
\begin{align*}
     \e \bigg[ \sup_{0 \leq \lambda \leq 1} \bigg| \sum_{j=1}^{\gbr{\lambda \hat k_n} - m} (\gbr{\lambda \hat k_n}& - m - j) \tilde X_j^\top \delta \bigg|\bigg]\\
     &\leq \e \bigg[ \max_{1 \leq k \leq n = 2^d} \bigg| \sum_{j=1}^{k} (k - j) \tilde X_j^\top \delta \bigg|\bigg]\\
     &\leq \sum_{r=0}^d \bigg( \sum_{\ell = 1}^{2^{d-r}} \e \bigg[ \bigg( \sum_{j=2^r (\ell -1) + 1}^{2^r \ell} (2^r \ell - j) \tilde X_j^\top \delta \bigg)^2 \bigg] \bigg)^{1/2}.
\end{align*}
Using $\ell \leq 2^{d-r}$ and $\delta^\top \Sigma_{h} \delta \leq \| \Sigma_h \| \| \delta \|_2^2$, the expected value can be rewritten as
\begin{align*}
     \e \bigg[ \bigg( \sum_{j=2^r (\ell -1) + 1}^{2^r \ell} (2^r \ell - j) \tilde X_j^\top \delta \bigg)^2 \bigg] &= \sum_{j_1, j_2 = 2^r (\ell -1 ) + 1}^{2^r \ell} (2^r \ell - j_1) (2^r \ell - j_2 ) \delta^\top \Sigma_{j_2 - j_1} \delta\\
     &\lesssim 2^{2d} \sum_{h = 0}^{2^r - 1} (2^r - h)  \| \Sigma_h \| \| \delta \|^2_2\\
     & =  o(2^{2d+r} \| \bar \Gamma \|_F \| \delta \|_2^2)\\
     &= o(2^{2d+r}  \| \delta \|_2^2 \cdot p | \|\delta \|^2 - \Delta | ),
\end{align*}
where we have used $\sum_{h = 0}^\infty \| \Sigma_h \| = o(\| \bar \Gamma \|_F)$ and the fact $\| \bar \Gamma \|_F /p < | \Delta - \| \delta \|^2 |$ for all for $\delta \in \mathcal{A}_n \cup \mathcal{B}_n$.
This yields 
\begin{align*}
    \e \bigg[  \sup_{0 \leq \lambda \leq 1} \frac{1}{p| \|\delta\|^2 - \Delta |} \bigg| \sum_{j=1}^{\gbr{\lambda \hat k_n} - m}  \negthickspace (\gbr{\lambda \hat k_n} - m - j) \tilde X_j^\top \delta \bigg|\bigg] = o \Big( 2^{3d/2} \| \delta \|_2 / \sqrt{p(|\|\delta\|^2 - \Delta|)} \Big).
\end{align*}
Noting that $|\| \delta \|^2 - \Delta | > 0$ implies
\begin{align*}
    \sup_{\delta \in \mathcal{A}_n \cup \mathcal{B}_n} \frac{\|\delta \|^2}{|\| \delta \|^2 - \Delta|} = O(1),
\end{align*}
which gives 
\begin{align*}
    \sup_{\delta \in \mathcal{A}_n \cup \mathcal{B}_n } \e \bigg[  \sup_{0 \leq \lambda \leq 1} \frac{1}{p| \|\delta\|^2 - \Delta |} \bigg| \sum_{j=1}^{\gbr{\lambda \hat k_n} - m} (\gbr{\lambda \hat k_n} - m - j) \tilde X_j^\top \delta \bigg|\bigg] = o(2^{2d}).
\end{align*}
For a general $n \in \mathbb{N}$, we therefore have
\begin{align*}
    \sup_{\delta \in \mathcal{A}_n \cup \mathcal{B}_n} \e \bigg[ \sup_{0 \leq \lambda \leq 1}   \frac{1}{p |\|\delta\|^2 - \Delta| } \bigg| \sum_{j=1}^{\gbr{\lambda \hat k_n} - m} (\gbr{\lambda \hat k_n} - m - j) \tilde X_j^\top \delta \bigg|\bigg] = o(n^2),
\end{align*}
which proves \eqref{dasletztedaszzist_1}.

\section{Proofs of the results in Section \ref{subsec_theoretical_prop_unif}}

\subsection{Proof of Theorem \ref{snconsistent}}

We begin specializing Theorem \ref{thm2.1} to the statistics
\begin{align}\label{hd13}
    \{ \hat \delta_\ell^2 (\lambda) \} _{\ell \in A} = \{ T_{n, \{ \ell \}} (\hat k_n, m; \lambda) \} _{\ell \in A}
\end{align}
and derive invariance principles with explicit rates for the sets $A=S$ and $A=S^C$, which hold uniformly with respect to these index sets. The proofs of the following two results are given in Sections \ref{sec_proof_d1} and \ref{sec_proof_d2}.

\begin{thm}\label{unif_conv_thm_S}
    Let conditions \ref{ass_shat_1} and \ref{ass_shat_4} - \ref{ass_shat_8} of Assumption \ref{asses_S_equal_Shat} be satisfied. Then, on a possibly richer probability space, for every $\ell \in S$, there exists a Gaussian process $\{ G_{\ell} (\lambda) \}_{\lambda \in [0,1]}$ with the same distribution as $\{ \lambda^3 \mathbb{B} (\lambda) \}_{\lambda \in [0, 1]}$ such that 
    \begin{align*}
        \max_{\ell \in S} \sup_{0 \leq \lambda \leq 1} \big| \sqrt{n} \big( \hat \delta_\ell^2 (\lambda) - \lambda^4 \delta_\ell^2 \big) - \sigma_\ell G_{\ell} (\lambda ) \big| = o_\p \big( n^{-\beta} f(s) \big),
    \end{align*}
    for any $0 \leq \beta < \alpha / 4$, where $\sigma_\ell^2 = 4 \delta_\ell^2 \Gamma_{\ell \ell} / (\vartheta_0 (1-\vartheta_0))$.
\end{thm}

\begin{thm}\label{unif_conv_thm_Sc}
    Let conditions \ref{ass_shat_1} - \ref{ass_shat_8} of Assumption \ref{asses_S_equal_Shat} be satisfied. Then, on a possibly richer probability space, for every $\ell \in S^C$, there exists a centered process $\{ H_\ell (\lambda) \}_{\lambda \in [0,1]}$ with the same distribution as $\{ \lambda^2 (\mathbb{B} (\lambda)^2 - \lambda) \}_{\lambda \in [0,1]}$ such that 
    \begin{align*}
        \max_{\ell \in S^C} \sup_{0 \leq \lambda \leq 1} \big| n \hat \delta_\ell^2 (\lambda) - \tau_{\ell}^2 H_\ell (\lambda) \big| = o_\p \big( n^{-2\beta} f_c (s_c)^2 + n^{-\beta} f_c (s_c) \big),
    \end{align*}
    for any $\beta$ satisfying $0 \leq \beta < \alpha / 4$, where $\tau_\ell^2 = \Gamma_{\ell\ell} / (\vartheta_0 (1-\vartheta_0))$.
\end{thm}

For $\ell \in S$, define the random variables 
\begin{align}\label{def_v_ell}
    v_\ell = \bigg( \int_0^1 \big(G_{\ell} (\lambda) - \lambda^4 G_{\ell} (1) \big)^2 \mathrm{d} \lambda \bigg)^{1/2}   =^d \mathbb{V} := 
    \bigg( \int_0^1 \lambda^6 (\mathbb{B} (\lambda) - \lambda \mathbb{B} (1))^2 \mathrm{d} \lambda \bigg)^{1/2},
\end{align}
where $\{ G_{\ell} (\lambda) \}_{\lambda \in [0,1]}$ is the Gaussian process from Theorem \ref{unif_conv_thm_S}. Moreover, for $\ell \in S^C$ define
\begin{align}\label{def_w_ell}
    \begin{aligned}
        w_\ell &= \bigg( \int_0^1 \big(H_{\ell} (\lambda) - \lambda^4 H_{\ell} (1) \big)^2 \mathrm{d} \lambda \bigg)^{1/2}\\
        &=^d \mathbb{W} := \bigg( \int_0^1 \lambda^4 \big( \mathbb{B}^2 (\lambda) - \lambda - \lambda^2 ( \mathbb{B}^2 (1) - 1 ) \big)^2 \bigg)^{1/2},
    \end{aligned}
\end{align}
where $\{ H_\ell (\lambda) \}_{\lambda  \in [0,1]}$ is the stochastic process from Theorem \ref{unif_conv_thm_Sc}.

The proof of Theorem \ref{snconsistent} is now completed in three steps. First, we study properties of the estimators $\hat \delta_\ell^2$ and $\hat v_\ell$ in \eqref{hat_delta_ell} and \eqref{hat_v_ell}, respectively, if $\ell \in S$ (Lemma~\ref{lemma_max_bound} and \ref{lemma_v_ells}).
Secondly, the corresponding properties for $\ell \in S^C$ are derived in Lemmas \ref{lemma_max_bound_deltanull} and \ref{lemma_w_ells}.
Finally, these results will be combined in the third step, which is the actual proof of Theorem \ref{snconsistent}. The proofs of Lemmas \ref{lemma_max_bound} - \ref{lemma_w_ells} are given in the subsequent section.

\medskip

Define for some $\omega > 1$
\begin{align}\label{InJn}
    I_n (s) = (n^{-\alpha / 8} f(s) \vee 1) \cdot \log^\omega (s), ~~~~~J_n (s) = \max_{\ell \in S} \sigma_\ell + n^{-\alpha / 8} \log (s) f(s),
\end{align}
where the function $f$ and the constant $\alpha$ are introduced in Assumption~\ref{asses_S_equal_Shat}.

\begin{lem}\label{lemma_max_bound}
    Let conditions \ref{ass_shat_1} and \ref{ass_shat_4} - \ref{ass_shat_8} of Assumption \ref{asses_S_equal_Shat} be satisfied. Then, we have
    \begin{align*}
        \p \bigg( \max_{\ell \in S} \bigg| \frac{\hat \delta_\ell^2 - \delta_\ell^2}{v_\ell} \bigg| > \frac{I_n (s)}{\sqrt{n}}  \bigg) = o(1).
    \end{align*}
\end{lem}

\begin{lem}\label{lemma_v_ells}
    Let conditions \ref{ass_shat_1} and \ref{ass_shat_4} - \ref{ass_shat_8} of Assumption \ref{asses_S_equal_Shat} be satisfied. Then, as $n \to \infty$ we have for the quantities $\hat v_\ell$ and $ v_\ell$ defined in \eqref{hat_v_ell} and \eqref{def_v_ell}, respectively and any $0 < \beta < \alpha / 4$
    \begin{enumerate}
        \item [(i)] $\max_{\ell \in S} | \sqrt{n} \hat v_\ell - \sigma_\ell v_\ell | =  o_\p \big( n^{-\beta} f(s) \big)$, 
        \item [(ii)] $\max_{\ell \in S} | \sqrt{n} \hat v_\ell / v_\ell - \sigma_\ell | =  o_\p \big( n^{-\beta}  \log (s) f(s) \big)$.
        \item [(iii)] $\p \big( \min_{\ell \in S}  \frac{v_\ell}{\sqrt{n} \hat v_\ell} \geq J_n (s)^{-1} \big) \to 1$.
    \end{enumerate}
\end{lem}

\begin{lem}\label{lemma_max_bound_deltanull}
    Let Assumption \ref{asses_S_equal_Shat} be satisfied. Then, for any $\omega > 5/2$ we have
    \begin{align*}
       \p \bigg( \max_{\ell \in S^C} \bigg| \frac{\hat \delta_\ell^2}{w_\ell} \bigg| > \frac{ \log^\omega (s_c)}{n} \bigg) = o(1).
    \end{align*}
\end{lem}

\begin{lem}\label{lemma_w_ells}
    Let conditions \ref{ass_shat_1} - \ref{ass_shat_8} and the first part of condition \ref{ass_shat_3} of Assumption \ref{asses_S_equal_Shat} be satisfied. Then, as $n \to \infty$ we have for the quantities $\hat v_\ell$ and $w_\ell$ defined in \eqref{hat_v_ell} and \eqref{def_w_ell}, respectively and any $\alpha / 8 < \beta < \alpha / 4$
    \begin{enumerate}
        \item [(i)] $\max_{\ell \in S^C} | n \hat v_\ell - \tau_{\ell}^2  w_\ell | =  o_\p \big( n^{-\beta} f_c (s_c) \big)$, 
        \item [(ii)] $\max_{\ell \in S^C} | n \hat v_\ell / w_\ell - \tau_{\ell}^2 | =  o_\p \big( n^{-\beta} \log^{2} (s_c) f_c (s_c) \big)$,
        \item [(iii)] $\p \big( \max_{\ell \in S^C} \frac{ \tau_{\ell}^2 w_\ell}{n\hat v_\ell} \leq 2 \big) \to 1$,
    \end{enumerate}
    where $\tau_\ell^2 = \Gamma_{\ell\ell} / (\vartheta_0 (1 - \vartheta_0))$.
\end{lem}

\bigskip
\noindent
With these auxiliary results, we now complete the proof of Theorem \ref{subsec_theoretical_prop_unif}.
Define the random set
\begin{align*}
    \tilde F_n (\delta) := \{  \ell = 1,  \ldots , p \mid 0 < \delta_\ell^2  \leq v_\ell \cdot\tilde \zeta_n \},
\end{align*}
where $\tilde \zeta_n = I_n (s) / \sqrt{n} + J_n (s) \cdot 2\log^{3/2} (p) / \sqrt{n}$. 
Observing the notations of $\zeta_n$, $I_n$ and $J_n$  in \eqref{det30a} and \eqref{InJn}, respectively, yields $\tilde \zeta_n \leq \zeta_n / \sqrt{\log (s)}$ if $\omega < 5/2$.
By part (iii) of Lemma~\ref{lemmasubexp}, we have $\max_{\ell \in S} v_\ell = O_\p \big( \sqrt{\log (s) } \big)$, and by a simple calculation it follows $\p \big( \tilde F_n (\delta) \subseteq F_n (\delta) \big) \to 1$, where $F_n (\delta)$ is defined in \eqref{fndelta}, which implies
\begin{align}\label{tildefndelta}
    \p (\tilde F_n (\delta) = \emptyset) \longrightarrow 1, ~~~~~ \text{ as } n \to \infty.
\end{align}
This result allows us to replace the set $S$ by $S_n := \{ \ell = 1,  \ldots , p \mid \delta_\ell^2 / v_\ell > \tilde \zeta_n  \}$ in the sense that $\p (S_n = S) \to 1$.
To see this, note that $S_n \subset S$ \eqref{tildefndelta} imply
\begin{align*}
    \p (S \subset S_n) = \p \Big( \min_{\ell \in S} \delta_\ell^2 / v_\ell >  \tilde \zeta_n \Big) \geq \p ( \tilde F_n (\delta) = \emptyset) \to 1.
\end{align*}
Therefore, in the following discussion we will work on the event $\{ S = S_n \}$.
By the definition of $\tilde \zeta_n$ in \eqref{det30a} and Lemma~\ref{lemma_max_bound}, we have
\begin{align}
    \nonumber
    \p \big( \exists \ell \in S_n: \hat \delta_\ell^2 < 0 \big)
    &= \p \bigg( \exists \ell \in S_n : \frac{\delta_\ell^2}{v_\ell} < \frac{ \delta_\ell^2 - \hat \delta_\ell^2}{v_\ell} \bigg) \nonumber\\
    &\leq \p \bigg( \tilde \zeta_n < \max_{\ell \in S_n} \bigg| \frac{\hat \delta_\ell^2 - \delta_\ell^2}{v_\ell} \bigg| \bigg) \nonumber\\
    &\leq \p \bigg( \tilde \zeta_n < \max_{\ell \in S_n} \bigg| \frac{\hat \delta_\ell^2 - \delta_\ell^2}{v_\ell} \bigg| \le \frac{I_n (s)}{\sqrt{n}}  \bigg) + o(1) \nonumber\\
    & = o(1). \label{dette2} 
\end{align}
Hence, the probability of the event
\begin{align*}
    \bigg \{ \min_{\ell \in S_n } \hat \delta_\ell^2 \geq  0 \bigg \}
    \cap \bigg\{ \max_{\ell \in S_n} \bigg| \frac{\hat \delta_\ell^2 - \delta_\ell^2}{v_\ell} \bigg| \leq \frac{I_n (s) }{\sqrt{n}} \bigg \}
    \cap \bigg \{  \min_{\ell \in S_n}  \frac{v_\ell}{\sqrt{n} \hat v_\ell} \geq \frac{1}{J_n (s)} \bigg \}
\end{align*}
converges to $1$, by \eqref{dette2}, Lemma~\ref{lemma_max_bound} and part (iii) of Lemma~\ref{lemma_v_ells}. On this event, we have for $\ell \in S_n$
\begin{align*}
    \frac{\hat \delta_\ell^2}{\hat v_\ell}  \geq \sqrt{n} \frac{\hat \delta_\ell^2}{v_\ell} \frac{v_\ell}{\sqrt{n} \hat v_\ell} \geq \frac{ \sqrt{n} }{ J_n(s) } \bigg( \frac{\delta_\ell^2}{v_\ell} - \frac{|\hat \delta_\ell^2 - \delta_\ell^2|}{v_\ell} \bigg) &\geq \frac{\sqrt{n} }{ J_n (s) } \Big (\tilde \zeta_n -  \frac{  I_n (s) }{ \sqrt{n}}  \Big) \geq 2\log^{3/2} (p).
\end{align*}
Therefore, we obtain
\begin{align*}
    \p (S_n \not \subset \hat S_n) &= \p \Big( \min_{\ell \in S_n} \hat \delta_\ell^2 / \hat v_\ell \leq \log^{3/2} (p) \Big)\\
    &\leq \p \Big( 2\log^{3/2} (p) \leq \min_{\ell \in S_n} \hat \delta_\ell^2 / \hat v_\ell \leq \log^{3/2} (p) \Big) + o(1)\\
    &= o(1).
\end{align*}

In order to prove $\p (\hat S_n \subset S) \to 1$, we note that 
\begin{align*}
    \p ( \hat S_n \subset S) = \p ( S^C \subset \hat S_n^C) = 1 - \p ( S^C \not \subset \hat S_n^C) = 1 -  \p \bigg( \exists \ell \in S^C : \frac{\hat \delta_\ell^2}{\hat v_\ell} > \log^{3/2} (p) \bigg).
\end{align*}
By Lemma~\ref{lemma_max_bound_deltanull} and part (iii) of Lemma~\ref{lemma_w_ells}, the probability of the event 
\begin{align*}
    \Big\{ \max_{\ell \in S^C} \frac{\tau_{\ell}^2 w_\ell}{n \hat v_\ell} \leq  2 \Big\} \cap \Big\{ \max_{\ell \in S^C} \bigg| \frac{\hat \delta_\ell^2}{w_\ell} \bigg|\leq \frac{\log^\omega (s_c)}{n} \Big\}
\end{align*}
for $5/2 < \omega < 3$ converges to $1$ (since by condition \ref{ass_shat_3} $n^{-\alpha / 8} f_c (s_c) = o(1)$. On this event, the inequality
\begin{align*}
    \frac{\hat \delta_\ell^2}{\hat v_\ell} \leq \frac{2 \log^\omega (s_c)}{\tau_{\ell}^2}
\end{align*}
holds for all $\ell \in S^C$, which gives
\begin{align*}
    \p \bigg( \exists \ell \in S^C : \frac{\hat \delta_\ell^2}{\hat v_\ell} > \log^{3/2} (p) \bigg) &\leq \p \bigg( \exists \ell \in S^C : \frac{2 \log^\omega (s_c)}{\tau_{\ell}^2} \geq \frac{\hat \delta_\ell^2}{\hat v_\ell} > \log^{3/2} (p) \bigg) + o(1)  \\
    &=  \p \bigg( \exists \ell \in S^C : 2 \log^\omega (s_c) \geq \tau_{\ell}^2 \frac{\hat \delta_\ell^2}{\hat v_\ell} > \tau_{\ell}^2 \log^{3/2} (p)  \bigg)  +  o(1) \\
    & = o(1)  ,
\end{align*}
where the last estimate follows because by \ref{ass_shat_3}, we have $\min_{\ell \in S^C} \tau_{\ell}^2 \ge \log^{3} (s_c) / (\vartheta_0 (1- \vartheta_0) \log^{3/2} (p) )$ and since $5/2 <\omega < 3$.

\subsection{Proofs of Lemmas \ref{lemma_max_bound} - \ref{lemma_w_ells}}

\begin{proof}[Proof of Lemma~\ref{lemma_max_bound}]
    Let $\ell \in S$ and define
    \begin{align*}
        \tilde{\mathbb{G}}_{\ell} = \frac{G_{\ell} (1) }{\big( \int_0^1 (G_{\ell} (\lambda) - \lambda^4 G_{\ell} (1))^2 \mathrm{d} \lambda \big)^{1/2}},
    \end{align*}
    where $\{ G_{\ell} (\lambda) \}_{\lambda \in [0,1]}$ is the Gaussian process from Theorem \ref{unif_conv_thm_S}. Note that $\tilde{\mathbb{G}}_\ell =^d \mathbb{G}$ for each $\ell \in S$, where the random variable $\mathbb{G}$ is defined in \eqref{defG}. We first show that
    \begin{align}\label{wtilde1}
        \max_{\ell \in S} \bigg| \frac{\sqrt{n} (\hat \delta_\ell^2 - \delta_\ell^2)}{v_\ell} - \sigma_\ell \tilde{\mathbb{G}}_{\ell} \bigg| = o_\p \big(n^{-\beta} \log(s) f(s) \big)
    \end{align}
    for any $\alpha / 8 < \beta < \alpha / 4$ and
    \begin{align}\label{wtilde2}
        \max_{\ell \in S} |\tilde{\mathbb{G}}_{\ell}| = o_\p (\log^\omega (s)) ~~~~~ \text{for any } \omega > 1.
    \end{align}
    With these estimates the proof can be easily completed. More precisely, we obtain for $\alpha / 8 < \beta < \alpha /4$
    \begin{align*}
        \p \bigg( \max_{\ell \in S} \Big| \frac{\hat \delta_\ell^2 - \delta_\ell^2}{v_\ell} \Big| > \frac{ I_n (s) }{\sqrt{n}} \bigg) &\leq \p \bigg(  \max_{\ell \in S} \Big| \frac{\sqrt{n}  (\hat \delta_\ell^2 - \delta_\ell^2)}{v_\ell} - \sigma_\ell \tilde {\mathbb{G}}_\ell \Big| +  \max_{\ell \in S} \sigma_\ell |\tilde{\mathbb{G}}_\ell| > I_n (s) \bigg) \\
        &\leq \p \bigg( \max_{\ell \in S} \Big| \frac{\sqrt{n} (\hat \delta_\ell^2 - \delta_\ell^2)}{v_\ell} - \sigma_\ell \tilde{\mathbb{G}}_\ell \Big| > \frac{\log^\omega (s) (n^{-\beta} f(s) \vee 1)}{2} \bigg)\\
        & ~~~~~ ~~~~~ + \p \bigg(  \max_{\ell \in S} |\sigma_\ell \tilde{\mathbb{G}}_\ell| > I_n (s) -  \frac{\log^\omega (s) (n^{-\beta} f(s) \vee 1) }{2} \bigg)\\
        &= \p \Big(  \max_{\ell \in S} |\sigma_\ell \tilde{\mathbb{G}}_\ell| >  \log^\omega (s) (n^{-\beta} f(s) \vee 1) / 2 \Big) + o(1) \\
        &= o(1).
    \end{align*}
    
    It remains to prove \eqref{wtilde1} and \eqref{wtilde2}, and we start with \eqref{wtilde1}. Note that for each $\ell \in S$ the random variables $v_\ell$ have the same distribution as $\mathbb{V}$ in \eqref{def_v_ell}. By part (i) of Lemma~\ref{lemmasubexp}, the random variables $v_\ell^{-1}$ (note that by the same Lemma $v_\ell$ is positive with probability $1$) have subexponential distributions and a standard maximal inequality shows that
    \begin{align}\label{vminusOlogs}
         \max_{\ell \in S} v_\ell^{-1} = O_\p ( \log (s)).
    \end{align}
    Therefore, it follows from Theorem \ref{unif_conv_thm_S} that
    \begin{align*}
        \max_{\ell \in S} \bigg| \frac{\sqrt{n} (\hat \delta_\ell^2 - \delta_\ell^2)}{v_\ell} - \sigma_\ell \tilde{\mathbb{G}}_{\ell} \bigg|  &\leq \max_{\ell \in S} | v_\ell^{-1} | \cdot \max_{\ell \in S} \big| \sqrt{n} (\hat \delta_\ell^2 - \delta_\ell^2) - \sigma_\ell G_{\ell} (1) \big|\\
        &= o_\p \big(n^{-\beta} \log(s) f(s) \big).
    \end{align*}
    
    For \eqref{wtilde2}, we have by a union bound argument and  Lemma~\ref{lemmasubexp}(ii) that
    \begin{align*}
        \p \big( \max_{\ell \in S} |\tilde{\mathbb{G}}_{\ell}| \geq \log^\omega (s) \big) \leq s \cdot \p (|\mathbb{G}| \geq \log^\omega (s)) &\leq 2 s \cdot \exp ( - K_6 \log^\omega (s) )\\
        &= 2 s \cdot\exp(- \log (s))^{K_6 \log^{\omega - 1} (s)}\\
        &= 2 s^{1 - K_6 \log^{\omega - 1} (s)} \\
        &= o(1)
    \end{align*}
    for some  constant $K_6>0$ (note that $\omega > 1$). 
\end{proof}

\begin{proof}[Proof of Lemma~\ref{lemma_v_ells}]~~
    \begin{enumerate}
        \item [(i)] Define the map $p : \ell^\infty ([0,1]) \to \mathbb{R}$ as in Proposition \ref{pq_prop} with the function $\Lambda_1 (\lambda) = \lambda^4$. Then, for $g_1, g_2 \in \ell^\infty ([0,1])$, we have
        \begin{align*}
            |p(g_1) - p(g_2)| \lesssim \sup_{0 \leq \lambda \leq 1} |g_1 (\lambda) - g_2 (\lambda)|.
        \end{align*}
        Define $u_\ell (\lambda) := \sqrt{n} ( \hat \delta^2_\ell (\lambda) - \lambda^4 \delta_\ell^2 )$ and recall from \eqref{hat_v_ell} that $\sqrt{n} \hat v_\ell = p(u_\ell)$ and recall from \eqref{def_v_ell} that $\sigma_\ell v_\ell = p(\sigma_\ell G_\ell)$, where $(G_\ell)_{\ell \in S}$ are the Gaussian processes from Theorem \ref{unif_conv_thm_S}.
        Therefore,
        \begin{align*}
            \max_{\ell \in S} \big| \sqrt{n} \hat v_\ell - \sigma_\ell v_\ell \big| \lesssim \max_{\ell \in S} \sup_{0 \leq \lambda \leq 1} |u_\ell (\lambda) - \sigma_\ell G_\ell (\lambda)| = o_\p ( n^{- \beta} f(s) )
        \end{align*}
        by Theorem \ref{unif_conv_thm_S}, as well.

        \item [(ii)] By \eqref{vminusOlogs} and part (i) of the present Lemma, we obtain 
        \begin{align*}
            \max_{\ell \in S} \bigg| \frac{\sqrt{n} \hat v_\ell }{v_\ell} - \sigma_\ell \bigg|  = \max_{\ell \in S} \bigg| \frac{\sqrt{n} \hat v_\ell - \sigma_\ell v_\ell}{ v_\ell} \bigg| &\leq \max_{\ell \in S} | v_\ell^{-1} | \cdot \max_{\ell \in S} |\sqrt{n} \hat v_\ell  - \sigma_\ell v_\ell |\\
            &= o_\p \big(n^{-\beta} \log (s) f(s) \big).
        \end{align*}

        \item [(iii)]
        Observing the definition of $J_n$ in \eqref{InJn}, part (ii) of the present Lemma yields
        \begin{align*}
            \p \bigg( \min_{\ell \in S} \frac{v_\ell}{\sqrt{n} \hat v_\ell} \geq \frac{1}{J_n (s)} \bigg) &= \p \bigg( \forall \ell \in S: \frac{\sqrt{n} \hat v_\ell}{v_\ell} \leq J_n (s) \bigg)\\
            &\geq \p \bigg( \forall \ell \in S: \bigg| \frac{\sqrt{n} \hat v_\ell}{v_\ell} - \sigma_\ell \bigg| \leq n^{-\beta} \log (s) f(s) \bigg) \to 1.
        \end{align*}
    \end{enumerate}
\end{proof}

\begin{proof}[Proof of Lemma~\ref{lemma_max_bound_deltanull}]
    The proof is similar to the proof of Lemma~\ref{lemma_max_bound}. First, define the random variables
    \begin{align*}
        \tilde{\mathbb{H}}_\ell = \frac{H_\ell (1)}{(\int_0^1 (H_\ell (\lambda) - \lambda^4 H_\ell (1) )^2 \mathrm{d} \lambda )^{1/2}},
    \end{align*}
    where $\{ H_\ell (\lambda) \}_{\ell \in S^C}$ is the stochastic process from Theorem \ref{unif_conv_thm_Sc}. Note that $\tilde{\mathbb{H}}_\ell =^d \mathbb{H}_4$ for $\ell \in S^C$ with $\mathbb{H}_4$ defined in Lemma~\ref{lemmasubsubexp}. Subsequently, we show
    \begin{align}\label{htilde1}
        \max_{\ell \in S^C} \bigg| \frac{n \hat \delta_\ell^2}{w_\ell} - \tau_{\ell}^2 \tilde{\mathbb{H}}_\ell \bigg| = o_\p \big( n^{-\beta} \log^{2} (s_c)  f_c (s_c) \big)
    \end{align}
    and
    \begin{align}\label{htilde2}
        \max_{\ell \in S^C} |\tilde{\mathbb{H}}_\ell| = o_\p (\log^\omega (s_c)) ~~~~~ \text{for any } \omega > 5/2.
    \end{align}
    Then, the assertion follows using similar arguments as in the proof of Lemma~\ref{lemma_max_bound}.

    \medskip

    For a proof of \eqref{htilde1}, recall the definition of $w_\ell$ in \eqref{def_w_ell} and note that part (i) of Lemma~\ref{lemmasubsubexp} and a similar argument as in \eqref{vminusOlogs} yield
    \begin{align*}
        \max_{\ell \in S^C} w_\ell^{-1} = O_\p(\log^2 (s_c))
    \end{align*}
    (note that for each $\ell \in S^C$ the random variables $w_\ell$ have the same distribution as $\mathbb{W}_4$).
    Therefore, it follows from Theorem \ref{unif_conv_thm_Sc}
    \begin{align*}
        \max_{\ell \in S^C} \bigg| \frac{n \hat \delta_\ell^2}{w_\ell} - \tau_{\ell}^2 \tilde{\mathbb{H}}_\ell \bigg| \leq \max_{\ell \in S^C} | w_\ell^{-1} | \cdot \max_{\ell \in S^C} \big| n \hat \delta_\ell^2 - \tau_{\ell}^2 H_\ell (1) \big| = o_\p \big(n^{-\beta} \log^{2} (s_c) f_c (s_c) \big).
    \end{align*}
    
    \medskip

    For the estimate \eqref{htilde2}, we have by a union bound argument and part (ii) of Lemma~\ref{lemmasubsubexp}, for $t$ sufficiently large,
    \begin{align*}
        \p \big( \max_{\ell \in S^C} |\tilde{\mathbb{H}}_{\ell}| \geq \log^\omega (s_c) \big) \leq s_c \cdot \p (|\mathbb{H}_4| \geq \log^\omega (s_c)) \lesssim s_c \cdot \exp ( - D_4 \log^{2 \omega / 5} (s_c) ) = o(1)
    \end{align*}
    for some constant $D_4 > 0$, whenever $\omega > 5/2$.
\end{proof}

\begin{proof}[Proof of Lemma~\ref{lemma_w_ells}]
    The proof of parts (i) and (ii) will be omitted, since the arguments are identical to the proof of parts (i) and (ii) of Lemma~\ref{lemma_v_ells}.
    For part (iii), we note that
    \begin{align*}
        \p \Big( \max_{\ell \in S^C} \frac{\tau_{\ell}^2 w_\ell}{n \hat v_\ell} > 2 \Big) \leq \p \bigg( \max_{\ell \in S^C} \bigg| \frac{\tau_{\ell}^2 w_\ell}{n \hat v_\ell} - 1 \bigg| > 1 \bigg) = \p \big( \exists \ell \in S^C: |\tau_{\ell}^2 w_\ell - n \hat v_\ell| > n \hat v_\ell \big).
    \end{align*}
    Since by assumption $n^{-\alpha/8} f_c (s_c) = o(1)$, we also have $n^{-\beta} f_c(s_c) = o(1)$ for any $\alpha / 8 < \beta < \alpha / 4$ and therefore,
    \begin{align*}
        \p \big( \exists \ell \in S^C: |\tau_{\ell}^2 w_\ell - n \hat v_\ell| > n \hat v_\ell \big) &\leq \p \big( \exists \ell \in S^C: \tau_{\ell}^2 w_\ell > 2n \hat v_\ell \text{ or } \tau_{\ell}^2 w_\ell <  0 \big)\\
        &\leq \p \Big( \exists \ell \in S^C: \tau_{\ell}^2 w_\ell > 2 \tau_{\ell}^2 w_\ell - 2 \max_{\ell \in S^C} | n \hat v_\ell - \tau_{\ell}^2 w_\ell | \Big)\\
        &\leq \p \big( \exists \ell \in S^C: \tau_{\ell}^2 w_\ell > 2 \tau_{\ell}^2 w_\ell - 2 n^{-\beta} f_c (s_c) \big) + o(1)\\
        &= o(1),
    \end{align*}
    by part (i) of the present Lemma.
\end{proof}

\subsection{Proof of Theorem \ref{unif_conv_thm_S}}
\label{sec_proof_d1}

Let $\tilde X_j = X_j - \e [X_j]$. Since the proof mainly follows the same arguments as in the proof of Theorem \ref{thm2.1}, we sketch it only briefly. The proof relies on the following two propositions, which will be proved in Section \ref{sec_aux_d2}.

\subsubsection{Two auxiliary results} 

\begin{prop}\label{prop67}
    Suppose that conditions \ref{ass_shat_4} and \ref{ass_shat_5} of Assumption \ref{asses_S_equal_Shat} and $\sum_{\ell \in S} c_\ell \lesssim n^{1/4} f(s)$ are satisfied. Then, we have
    \begin{align*}
        \max_{\ell \in S} \max_{1 \leq k \leq n} \bigg| \sqrt{n}\sum_{\substack{i,j = 1 \\ |i-j| > m}}^{k} \tilde{X}_{i, \ell} \tilde{X}_{j, \ell} \bigg| = o_\p \big( n^{2-\beta} f(s) \big),
    \end{align*}
    for any $\beta$ with $0 \leq \beta < 1 / 4$.
\end{prop}

\begin{prop}\label{prop_max_unif}
    Assume that condition \ref{ass_shat_4} of Assumption \ref{asses_S_equal_Shat} and $\sum_{\ell \in S} c_\ell \lesssim n^{1/4} f(s)$ is satisfied, and for $n\in\mathbb{N}$ let $(a_{j,n})_{j = 1, \ldots, n}$ denote a given given real-valued sequence. If $\sup_{j = 1,  \ldots , n} |a_{j,n}| = O(1)$, then 
    \begin{align*}
        \max_{\ell \in S} \max_{1 \leq k \leq n} \sqrt{n} \bigg( \sum_{j=1}^k a_{j,n} \tilde{X}_{j, \ell} \bigg)^2 = o_\p \big(n^{2 - \beta} f(s) \big),
    \end{align*}
    for any $\beta$ that satisfies $0 \leq \beta < 1 / 4$.
\end{prop}

\subsubsection{Proof of Theorem \ref{unif_conv_thm_S}}

Recall the notation of $\hat \delta_\ell^2$ in \eqref{hd13} and use the decomposition
\begin{align}\label{decomp_unif_stat}
    \hat \delta_\ell^2 (\lambda) = T_{n, \ell}^{(1)} (\lambda) + T_{n, \ell}^{(2)} (\lambda) + T_{n, \ell}^{(3)} (\lambda) + T_{n, \ell}^{(4)} (\lambda),
\end{align}
for $\ell \in S$, where
\begin{align*}
    T_{n, \ell}^{(1)} (\lambda)  &= 
        \frac{1}{N_m (\hat k_n) N_m (n-\hat k_n)  }  \sum_{\substack{i_1,i_2=1\\|i_1-i_2| > m}}^{\gbr{\lambda \hat k_n}} \sum_{\substack{j_1,j_2=\hat k_n+1\\|j_1-j_2| > m}}^{\hat k_n + \gbr{\lambda (n-\hat k_n)}} (\tilde{X}_{i_1, \ell} - \tilde{X}_{j_1, \ell}) (\tilde{X}_{i_2, \ell} - \tilde{X}_{j_2, \ell}),\\
    T_{n, \ell}^{(2)} (\lambda)  &= 
        \frac{1}{N_m (\hat k_n) N_m (n-\hat k_n) }  \sum_{\substack{i_1,i_2=1\\|i_1-i_2| > m}}^{\gbr{\lambda \hat k_n}} \sum_{\substack{j_1,j_2=\hat k_n+1\\|j_1-j_2| > m}}^{\hat k_n + \gbr{\lambda (n-\hat k_n)}} (\e [X_{i_1, \ell}] - \e [X_{j_1, \ell}]) (\tilde{X}_{i_2, \ell} - \tilde{X}_{j_2, \ell}),\\
    T_{n, \ell}^{(3)} (\lambda)  &= 
        \frac{1}{N_m (\hat k_n) N_m (n-\hat k_n) }  \sum_{\substack{i_1,i_2=1\\|i_1-i_2| > m}}^{\gbr{\lambda \hat k_n}} \sum_{\substack{j_1,j_2=\hat k_n+1\\|j_1-j_2| > m}}^{\hat k_n + \gbr{\lambda (n-\hat k_n)}} (\tilde{X}_{i_1, \ell} - \tilde{X}_{j_1, \ell}) (\e [X_{i_2, \ell}] - \e [X_{j_2, \ell}])
\end{align*}
and
\begin{align*}
    &T_{n, \ell}^{(4)} (\lambda) \\
    &=\frac{1}{N_m (\hat k_n) N_m (n-\hat k_n) }  \sum_{\substack{i_1,i_2=1\\|i_1-i_2| > m}}^{\gbr{\lambda \hat k_n}} \sum_{\substack{j_1,j_2=\hat k_n+1\\|j_1-j_2| > m}}^{\hat k_n + \gbr{\lambda (n-\hat k_n)}}  (\e [X_{i_1, \ell}] - \e [X_{j_1, \ell}])(\e [X_{i_2, \ell}] - \e [X_{j_2, \ell}]).
\end{align*}
The assertion of Theorem \ref{unif_conv_thm_S} now follows from the estimates
\begin{align}
    \max_{\ell \in S} \sup_{0 \leq \lambda \leq 1} \big| \sqrt{n} T_{n,\ell}^{(1)} (\lambda)  \big| &= o_\p \big( n^{-\beta} f(s) \big), \label{tn1_eq_ell} \\
    \max_{\ell \in S} \sup_{0 \leq \lambda \leq 1} \big| \sqrt{n} \big( T_{n,\ell}^{(2)} (\lambda) + T_{n,\ell}^{(3)} (\lambda) \big) -  \sigma_{\ell} G_{\ell} (\lambda) \big| &= o_\p \big( n^{-\beta} f(s) \big), \label{tn23_eq_ell} \\
    \max_{\ell \in S} \sup_{0 \leq \lambda \leq 1} \sqrt{n} \big| T_{n,\ell}^{(4)} (\lambda) - \lambda^4 \delta_\ell^2 \big| &= o_\p ( n^{-\beta} ), \label{tn4_eq_ell}
\end{align}
where $(G_{\ell})_{\ell \in S}$ are centered Gaussian processes with the same covariance function, that is, $(\lambda_1, \lambda_2 ) \mapsto \lambda_1^3 \lambda_2^3 (\lambda_1 \wedge \lambda_2)$.
These estimates will be proved using arguments similar to those used for \eqref{tn1_eq} - \eqref{tn4_eq}.

\medskip

\noindent
\textbf{Proof of \eqref{tn1_eq_ell}.}
We use the representation
\begin{align}
    T_{n,\ell}^{(1)} (\lambda) = \frac{N_m (\gbr{\lambda (n-\hat k_n)})}{N_m (n-\hat k_n)} P_{n, \ell} (\lambda) - Q_{n, \ell} (\lambda) + \frac{N_m (\gbr{\lambda \hat k_n})}{N_m (\hat k_n)} R_{n, \ell} (\lambda), \label{hd15}
\end{align}
where
\begin{align}
    P_{n,\ell} (\lambda) &:= \frac{1}{N_m (\hat k_n)} \sum_{\substack{i_1,i_2=1\\|i_1-i_2| > m}}^{\gbr{\lambda \hat k_n}} \tilde{X}_{i_1, \ell} \tilde{X}_{i_2, \ell},
    \label{hd15a}\\
    Q_{n,\ell} (\lambda) &:= \frac{1}{N_m (\hat k_n) N_m (n-\hat k_n) } \sum_{\substack{i_1,i_2=1\\|i_1-i_2| > m}}^{\gbr{\lambda \hat k_n}} \sum_{\substack{j_1,j_2=\hat k_n+1\\|j_1-j_2| > m}}^{\hat k_n + \gbr{\lambda (n-\hat k_n)}} (\tilde{X}_{j_1, \ell} \tilde{X}_{i_2, \ell} + \tilde{X}_{i_1, \ell} \tilde{X}_{j_2, \ell}),  \label{hd15b} \\
    R_{n,\ell} (\lambda) &:= \frac{1}{N_m (n-\hat k_n) } \sum_{\substack{j_1,j_2=\hat k_n+1\\|j_1-j_2| > m}}^{\hat k_n + \gbr{\lambda (n-\hat k_n)}} \tilde{X}_{j_1, \ell} \tilde{X}_{j_2, \ell}. \label{hd15c}
\end{align}
Proposition \ref{prop67} gives
\begin{align}
    \max_{\ell \in S} \sup_{0 \leq \lambda \leq 1} \big| \sqrt{n} P_{n,\ell} (\lambda) \big| = o_\p \big( n^{-\beta} f(s) \big) ~~~ \text{ and } ~~~ \max_{\ell \in S} \sup_{0 \leq \lambda \leq 1} \big| \sqrt{n} R_{n,\ell} (\lambda) \big| = o_\p \big( n^{-\beta} f(s) \big).
    \label{hd14}
\end{align}
Applying Lemma~\ref{lem1} twice to $Q_{n, \ell} (\lambda)$ 
yields 
\begin{align*}
    &N_m (\hat k_n) N_m (n- \hat k_n) \cdot Q_{n,\ell} (\lambda)\\
    & ~~~~~ = 2 \sum_{i=1}^{\gbr{\lambda \hat k_n} - m} \sum_{j=\hat k_n+1}^{\hat k_n + \gbr{\lambda (n-\hat k_n)} -m} (\gbr{\lambda \hat k_n} - m - i) (\hat k_n + \gbr{\lambda (n-\hat k_n)} - m - j) \cdot \tilde{X}_{i, \ell} \tilde{X}_{j,\ell} \\
    & ~~~~~  ~~~~~  + 2 \sum_{i=1}^{\gbr{\lambda \hat k_n} - m} \sum_{j=\hat k_n+m+1}^{\hat k_n + \gbr{\lambda (n-\hat k_n)}} (\gbr{\lambda \hat k_n} - m - i) (j - \hat k_n - m - 1) \cdot \tilde{X}_{i,\ell} \tilde{X}_{j,\ell} \\
    & ~~~~~  ~~~~~ + 2 \sum_{i=m+1}^{\gbr{\lambda \hat k_n}} \sum_{j=\hat k_n+1}^{\hat k_n + \gbr{\lambda (n-\hat k_n)} - m} (i-m-1) (\hat k_n + \gbr{\lambda (n-\hat k_n)} - m - j) \cdot \tilde{X}_{i, \ell} \tilde{X}_{j,\ell} \\
    & ~~~~~  ~~~~~ + 2 \sum_{i=m+1}^{\gbr{\lambda \hat k_n}} \sum_{j=\hat k_n+m+1}^{\hat k_n + \gbr{\lambda (n-\hat k_n)}} (i-m-1) (j - \hat k_n - m - 1) \cdot \tilde{X}_{i,\ell} \tilde{X}_{j, \ell}.
\end{align*}
With Cauchy-Schwarz' and Young’s inequality and Proposition \ref{prop_max_unif} it follows that
\begin{align*}
    \max_{\ell \in S} \sup_{0 \leq \lambda \leq 1} \big| \sqrt{n} Q_{n, \ell} (\lambda)  \big| = o_\p \big( n^{-\beta} f(s) \big),
\end{align*}
where $0 \leq \beta < \alpha / 4$. Combining this estimate with \eqref{hd14} and \eqref{hd15}  yields \eqref{tn1_eq_ell}.

\medskip

\noindent
\textbf{Proof of \eqref{tn23_eq_ell}.}
As in the proof of \eqref{tn23_eq}, we only consider the case $\hat k_n > k_0$ (the other case can be treated in a similar way) and rewrite $T_{n, \ell}^{(2)} (\lambda) + T_{n,\ell}^{(3)} (\lambda) = (T_{n, \ell}^{(2)} (\lambda) + T_{n,\ell}^{(3)} (\lambda)) \mathbbm{1} \{ \hat k_n > k_0 \}$ as
\begin{align}\label{det40a}
    \big(T_{n, \ell}^{(2)} (\lambda) + T_{n,\ell}^{(3)} (\lambda) \big) \mathbbm{1} \{ \hat k_n > k_0 \} = \big( K_{n, \ell} (\lambda) - L_{n,\ell} (\lambda) \big) \mathbbm{1} \{ \hat k_n > k_0 \}
\end{align}
with
\begin{align*}
    K_{n, \ell} (\lambda) &:= \frac{1}{N_m (\hat k_n) N_m (n-\hat k_n) }\sum_{\substack{i_1,i_2=1\\|i_1-i_2| > m}}^{\gbr{\lambda \hat k_n}}   \sum_{\substack{j_1,j_2=\hat k_n+1\\|j_1-j_2| > m}}^{\hat k_n + \gbr{\lambda (n-\hat k_n)}} \big( \mathbbm{1} \{ i_1 \leq k_0 \} \tilde{X}_{j_2, \ell} \delta_\ell \\
    & \quad\quad\quad\quad\quad\quad\quad\quad\quad\quad\quad\quad\quad\quad\quad\quad\quad\quad\quad\quad\quad\quad + \mathbbm{1} \{ i_2 \leq k_0 \} \tilde{X}_{j_1, \ell} \delta_\ell \big),\\
    L_{n, \ell} (\lambda) &:= \frac{N_m (\gbr{\lambda (n - \hat k_n)})}{N_m (n-\hat k_n)} \frac{1}{N_m (\hat k_n)}  \sum_{\substack{i_1,i_2=1\\|i_1-i_2| > m}}^{\gbr{\lambda \hat k_n}} \big( \mathbbm{1} \{ i_1 \leq k_0 \} \tilde{X}_{i_2, \ell} \delta_\ell + \mathbbm{1} \{ i_2 \leq k_0 \} \tilde{X}_{i_1, \ell} \delta_\ell \big).
\end{align*}
We start by considering $K_{n,\ell} (\lambda)$, which can be rewritten analogously to \eqref{kn} as
\begin{align*}
    K_{n,\ell} (\lambda) = \frac{2}{N_m (\hat k_n)} \sum_{ \substack{i_1, i_2 = 1 \\ |i_1 - i_2| > m} }^{\gbr{\lambda \hat k_n}} \mathbbm{1} \{ i_1 \leq k_0 \} \cdot \frac{1}{N_m (n-\hat k_n)} \sum_{\substack{j_1, j_2 = \hat k_n + 1 \\ |j_1 - j_2| > m}}^{\hat k_n + \gbr{\lambda (n- \hat k_n)}} \tilde X_{j_2, \ell} \delta_\ell.
\end{align*}
We will show in Section \ref{sec_aux_d2} that
\begin{align}\label{kn_ell_1}
    \begin{aligned}
        \frac{\sqrt{n}}{N_m (n - \hat k_n)} &\sum_{\substack{j_1, j_2 = \hat k_n + 1 \\ |j_1 - j_2| > m}}^{\hat k_n + \gbr{\lambda (n- \hat k_n)}} \tilde X_{j_2, \ell}  \\
        &= \frac{\lambda}{1 - \vartheta_0} \Big( \mathbb{B} \big( \Gamma_{\ell \ell} (\vartheta_0 + \lambda (1 - \vartheta_0)) \big) - \mathbb{B} \big( \Gamma_{\ell \ell} \vartheta_0 \big) \Big) + o_\p \big( n^{-\beta} f(s) \big)
    \end{aligned}
\end{align}
uniformly with respect to $\ell \in S$ and $\lambda \in [0,1]$. Therefore, by \eqref{kn_ell_1}, \eqref{op1} and the first part of condition \ref{ass_shat_1}, we obtain
\begin{align*}
    \max_{\ell \in S} &\sup_{\lambda \in [0,1]} \big| \sqrt{n} K_{n,\ell} (\lambda) - \delta_\ell E_{\ell} (\lambda) \big| = o_\p \big( n^{-\beta} f(s) \big),
\end{align*}
where
\begin{align*}
    E_{\ell} (\lambda) = \frac{2 \lambda^3}{1 - \vartheta_0} \Big( \mathbb{B} \big( \Gamma_{\ell \ell}(\vartheta_0 + \lambda (1 - \vartheta_0) ) \big) - \mathbb{B} \big( \vartheta_0 \Gamma_{\ell \ell} \big) \Big).
\end{align*}
To derive a corresponding result for the term $L_{n, \ell} (\lambda)$, we use
\begin{align}\label{ln_ell_1}
    \frac{\sqrt{n}}{N_m (\hat k_n)}  \sum_{\substack{i_1,i_2=1\\|i_1-i_2| > m}}^{\gbr{\lambda \hat k_n}} \big( \mathbbm{1} \{ i_1 \leq k_0 \} \tilde{X}_{i_2, \ell} + \mathbbm{1} \{ i_2 \leq k_0 \} \tilde{X}_{i_1, \ell} \big) = \frac{\lambda}{\vartheta_0} \mathbb{B} \big( \lambda \vartheta_0 \Gamma_{\ell \ell} \big) + o_\p \big( n^{-\beta} f(s) \big),
\end{align}
which holds uniformly in $\ell \in S$ and $\lambda \in [0,1]$ and will also be proved in Section \ref{sec_aux_d2}. Combining this approximation with \eqref{secondterm1} and the first part of condition \ref{ass_shat_1}, we obtain
\begin{align*}
    \max_{\ell \in S} \sup_{0 \leq \lambda \leq 1} \big| \sqrt{n} L_{n,\ell} (\lambda) - \delta_\ell D_{\ell} (\lambda) \big|,
\end{align*}
where
\begin{align*}
    D_\ell (\lambda) = \frac{ 2\lambda^3 }{\vartheta_0}  \mathbb{B} \big( \lambda \vartheta_0 \Gamma_{\ell\ell} \big).
\end{align*}
Therefore, by observing the decomposition \eqref{det40a}, we obtain
\begin{align*}
    \max_{\ell \in S} \sup_{0 \leq \lambda \leq 1} \big| \sqrt{n}  \big( T_{n, \ell}^{(2)} (\lambda) + T_{n,\ell}^{(3)} (\lambda) \big) \mathbbm{1} \{ \hat k_n > k_0 \}  - \delta_\ell \big( E_\ell (\lambda) - D_\ell (\lambda) \big) \big| = o_\p \big(n^{-\beta} f(s) \big).
\end{align*}
A simple calculation shows that the covariance function of $\{\delta_\ell ( E_\ell (\lambda) - D_\ell (\lambda) ) \}_{\lambda \in [0,1]}$ is given by $(\lambda_1, \lambda_2) \mapsto \sigma_\ell^2 \cdot \lambda_1^3 \lambda_2^3 (\lambda_1 \wedge \lambda_2) $. This yields the assertion \eqref{tn23_eq} with $G_\ell (\lambda) = \sigma_\ell^{-1} \delta_\ell \big( E_\ell (\lambda) - D_\ell (\lambda) \big)$.

\medskip

\noindent 
\textbf{Proof of \eqref{tn4_eq_ell}.}
We decompose $T_{n, \ell}^{(4)} (\lambda)$ as
\begin{align*}
    T_{n, \ell}^{(4)} (\lambda) &= \delta_\ell^2 \mathbbm{1} \{ \hat k_n > k_0 \}  \frac{ N_m (\gbr{\lambda (n - \hat k_n)}) }{N_m (n- \hat k_n)} C_n^{(1)} (\lambda) + \delta_\ell^2 \mathbbm{1} \{ \hat k_n \leq k_0 \} \frac{ N_m (\gbr{\lambda \hat k_n}) }{N_m (\hat k_n)}  C_n^{(2)} (\lambda),
\end{align*}
where the constants $C_n^{(1)}$ and $C_n^{(2)}$  have been defined in \eqref{gnhn}.
By the triangle inequality, it follows that
\begin{align*}
    \max_{\ell \in S} &\sup_{0 \leq \lambda \leq 1} \sqrt{n} \big| T_{n, \ell}^{(4)} (\lambda) - \lambda^4 \delta_\ell^2 \big|\\
    &\leq \max_{\ell \in S} \sup_{0 \leq \lambda \leq 1} \sqrt{n} \big| T_{n, \ell}^{(4)} (\lambda) - \Lambda_n (\lambda) \delta_\ell^2 \big| + \max_{\ell \in S} \sup_{0 \leq \lambda \leq 1} \sqrt{n} \delta_\ell^2 \big| \Lambda_n (\lambda) - \lambda^4 \big| \\
    &\leq \max_{\ell \in S} \delta_\ell^2 \cdot \bigg( O_\p \bigg( \frac{\mathrm{tr} (\bar \Gamma)}{\| \delta \|_2^2} \frac{\log^2 (n)}{\sqrt{n}} \bigg) + O_\p \bigg( \frac{m}{\sqrt{n}} \bigg) \bigg)\\
    &= o_\p (n^{- \beta}),
\end{align*}
where the second inequality follows from the  definition of $\Lambda_n$ in  \eqref{Lambda_n}, \eqref{gnhnlambda} and \eqref{for_later2}, while the last estimate is a consequence of conditions \ref{ass_shat_1} and \ref{ass_shat_7}.

\subsubsection{Proofs of auxiliary results}
\label{sec_aux_d2}

In this section, we will provide the proofs to Propositions \ref{prop67} and \ref{prop_max_unif}, as well as the proofs for the statements \eqref{kn_ell_1} and \eqref{ln_ell_1}.

\begin{proof}[Proof of Proposition \ref{prop67}]
    The proof follows along the lines of the proof of Proposition \ref{lem65} using the estimate (Proposition \ref{maxmax} for $n=2^d$)
    \begin{align*}
        &\e \bigg[  \max_{\ell \in S} \max_{1 \leq k \leq n} \bigg| \sqrt{n} \sum_{\substack{i,j = 1 \\ |i-j| > m}}^k  \tilde{X}_{j, \ell} \tilde{X}_{j,\ell} \bigg| \bigg] 
        \lesssim   \tilde R_{n1} + \tilde R_{n  2},
    \end{align*}
    where
    \begin{align*} 
        \tilde R_{n1} & = 2^{d/2} \sum_{r=0}^d \bigg( \sum_{\ell \in S}
        \sum_{u=1}^{2^{d-r}} \sum_{h_1, h_2 = m + 1}^{2^r (u - 1) - 1} \sum_{j_1, j_2 = 2^r(u - 1) + 1}^{2^r u}  \e [\tilde{X}_{j_1, \ell} \tilde{X}_{j_1 - h_1, \ell} \tilde{X}_{j_2, \ell} \tilde{X}_{j_2 - h_2, \ell}]
        \bigg)^{1/2}\\
        \tilde R_{n2} & = 2^{d / 2} \sum_{r=0}^d \bigg( \sum_{\ell \in S}
            \sum_{u=1}^{2^{d-r}} \sum_{h_1, h_2 = 2^r(u - 1)}^{2^r u - 1} \sum_{j_1 = 1}^{2^r u - h_1} \sum_{j_2 = 1}^{2^r u - h_2} \e [\tilde{X}_{j_1, \ell} \tilde{X}_{j_1 + h_1, \ell} \tilde{X}_{j_2, \ell} \tilde{X}_{j_2 + h_2, \ell}]
        \bigg)^{1/2}.
    \end{align*}
    The terms $ \tilde R_{n1}$ and $\tilde R_{n2}$ can be estimated in a similar way as the terms $ R_{n1} $ and $ R_{n2}$ in \eqref{term1} and \eqref{term2}.
    As both are treated similarly, we restrict ourselves to the $\tilde R_{n1}$. Using stationarity, we have
    \begin{align*}
        &\sum_{h_1, h_2 = m+1}^{2^r u - 1} \sum_{j_1, j_2 = 2^r(u - 1) + 1}^{2^r u}  \e [\tilde{X}_{j_1, \ell} \tilde{X}_{j_1 - h_1, \ell} \tilde{X}_{j_2, \ell} \tilde{X}_{j_2 - h_2, \ell}]\\
        &~~ \lesssim 2^r \sum_{h_1, h_2 = m+1}^{2^r u - 1} \sum_{|k| \leq 2^r - 1} \big|\e [ \tilde{X}_{0, \ell} \tilde{X}_{h_1, \ell} \tilde{X}_{k + h_1, \ell} \tilde{X}_{k + h_1 - h_2, \ell}] \big|\\
        &~~\lesssim 2^r \sum_{h_1, h_2 = m+1}^{2^r u - 1} \sum_{k=0}^{2^r - 1} \Big( \big|\e [ \tilde{X}_{0, \ell} \tilde{X}_{h_1, \ell} \tilde{X}_{k + h_1, \ell} \tilde{X}_{k + h_1 - h_2, \ell}] \big| \\
        & ~~~~~~~~~~~~~~~~~~~~~~~~~~~~~~~~~~~~~~~~~~ ~~~~~~~~~~~~ + \big|\e [ \tilde{X}_{0, \ell} \tilde{X}_{h_1, \ell} \tilde{X}_{h_1 - k, \ell} \tilde{X}_{h_1 - h_2 - k, \ell}] \big| \Big)\\
        &~~= 2^r \sum_{h_1, h_2 = m+1}^{2^r u - 1} \sum_{k=0}^{2^r - 1} \big|\e [ \tilde{X}_{0, \ell} \tilde{X}_{h_1, \ell} \tilde{X}_{h_1 - k, \ell} \tilde{X}_{h_1 - h_2 - k, \ell}] \big|,
    \end{align*}
    which gives
    \begin{align*}
        \tilde R_{n1} \lesssim 2^{d/2} \sum_{r=0}^d 2^{r/2} \bigg( \sum_{\ell \in S}
        \sum_{u=1}^{2^{d-r}} \sum_{h_1, h_2 = m+1}^{2^r u - 1} \sum_{k=0}^{2^r - 1} \big|\e [ \tilde{X}_{0, \ell} \tilde{X}_{h_1, \ell} \tilde{X}_{h_1 - k, \ell} \tilde{X}_{h_1 - h_2 - k, \ell}] \big|
        \bigg)^{1/2}.
    \end{align*}
    Elementary properties of cumulants yield
    \begin{align}\label{hd12}
        \begin{aligned}
            \e [\tilde{X}_{0, \ell} \tilde{X}_{h_1, \ell} \tilde{X}_{h_1 - k, \ell} \tilde{X}_{h_1 - h_2 - k, \ell}]
            &= \Sigma_{h_1, \ell \ell} \Sigma_{h_2, \ell \ell} + \Sigma_{k, \ell \ell} \Sigma_{k + h_2 - h_1, \ell \ell} + \Sigma_{k + h_2, \ell \ell} \Sigma_{k - h_1, \ell \ell}\\
            & ~~~~~  + \mathrm{cum} (\tilde{X}_{0, \ell}, \tilde{X}_{h_1, \ell}, \tilde{X}_{h_1 - k, \ell}, \tilde{X}_{h_1 - h_2 - k, \ell})
        \end{aligned}
    \end{align}
    and we now decompose the argument of the square root in $\tilde R_{n1}$ according to this decomposition and estimate each term separately. 

    The same arguments as given in the proof of Proposition \ref{lem65} give for the sum corresponding to the first term in \eqref{hd12}
    \begin{align*}
        2^{d/2} \sum_{r=0}^d 2^{r/2} \bigg( \sum_{\ell \in S}
        \sum_{u=1}^{2^{d-r}} \sum_{h_1, h_2 = m + 1}^{2^r (u - 1) - 1} \sum_{k = 0}^{2^r - 1} &~ |\Sigma_{h_1, \ell \ell} \Sigma_{h_2, \ell \ell} |
        \bigg)^{1/2}\\
        &\lesssim 2^{d/2} \sum_{r=0}^d 2^{r} \bigg(  \sum_{\ell \in S} \sum_{u = 1}^{2^{d-r}} \bigg( \sum_{h = m+1}^{2^r (u-1) - 1} | \Sigma_{h, \ell \ell} | \bigg)^2
        \bigg)^{1/2}\\
        &\lesssim 2^{d} \sum_{r = 0}^d 2^{r/2} \bigg( \sum_{\ell \in S} c_\ell^2 \bigg)^{1/2} \\
        &\lesssim 2^{7d / 4} f(s)\\
        &= o \big( n^{2 - \beta} f(s) \big)
    \end{align*}
    for any $0 \leq \beta < 1 / 4$ by condition \ref{ass_shat_4}.

    For the sum corresponding to the second term in the decomposition \eqref{hd12} we have, using condition \ref{ass_shat_4} again,
    \begin{align*}
        2^{d/2} \sum_{r=0}^d 2^{r/2} \bigg( \sum_{\ell \in S}
        \sum_{u=1}^{2^{d-r}} \sum_{h_1, h_2 = m + 1}^{2^r (u - 1) - 1} \sum_{k = 0}^{2^r - 1}  & ~|\Sigma_{k, \ell \ell} \Sigma_{k + h_2 - h_1, \ell \ell} |
        \bigg)^{1/2}\\
        & \lesssim 2^{3d/2} \sum_{r=0}^d \bigg(  \sum_{\ell \in S} \sum_{k = 0}^{2^r - 1} |\Sigma_{k, \ell} |  \sum_{|h| \leq 2^d}  |\Sigma_{k + h, \ell \ell} |
        \bigg)^{1/2} \\
        & \lesssim 2^{3d/2} \sum_{r=0}^d \bigg(   \sum_{\ell \in S} c_\ell \sum_{k = 0}^{2^r - 1} |\Sigma_{k,\ell} | \gamma_k
        \bigg)^{1/2} \\
        & \lesssim 2^{3d/2} \sum_{r=0}^d \bigg(   \sum_{\ell \in S} c_\ell^2 \sum_{k = 0}^{2^r - 1}\gamma_k^2
        \bigg)^{1/2} \\
        &\lesssim d 2^{7d / 4} \cdot f(s)\\
        &= o(n^{2-\beta}f(s))
    \end{align*}
    for any $0 \le \beta < 1 / 4$.

    Next, for the decomposition corresponding to the third term in \eqref{hd12}, note that for any $k,N \in \mathbb{N}$
    \begin{align*}
        \sum_{h = 0}^{N} |\Sigma_{k-h, \ell\ell} | \leq \sum_{|h| \leq N} |\Sigma_{k+h, \ell\ell} | \lesssim \sum_{h = 0}^N |\Sigma_{k+h, \ell\ell} |.
    \end{align*}
    Therefore, we obtain for the sum corresponding to the third term in \eqref{hd12}
    \begin{align*}
        2^{d/2} \sum_{r=0}^d 2^{r/2} \bigg( \sum_{\ell \in S}
        \sum_{u=1}^{2^{d-r}} \sum_{h_1, h_2 = m + 1}^{2^r (u - 1) - 1} \sum_{k = 0}^{2^r - 1}  & ~|\Sigma_{k + h_2, \ell \ell} \Sigma_{k - h_1, \ell \ell} |
        \bigg)^{1/2}\\
        &\lesssim 2^{d} \sum_{r=0}^d \bigg( \sum_{\ell \in S} 
         \sum_{k = 0}^{2^r - 1}  \bigg( \sum_{h = 0}^{2^d} |\Sigma_{k + h, \ell \ell} | \bigg)^2
        \bigg)^{1/2} \\
        & \lesssim 2^{d} \sum_{r=0}^d \bigg( \sum_{\ell \in S} c_\ell^2
         \sum_{k = 0}^{2^r - 1}  \gamma_k^2
        \bigg)^{1/2}\\
        & \lesssim d \cdot 2^{5 d / 4} f(s) \\
        & = o \big( n^{2-\beta} f(s) \big)
    \end{align*}
    for any $0 \le \beta < 1 / 4$ by condition \ref{ass_shat_4}.

    Finally, we have for the sum corresponding to the cumulant in \eqref{hd12}
    \begin{align*}
        2^{d/2} \sum_{r=0}^d 2^{r/2} \bigg( \sum_{\ell \in S} 
        \sum_{u=1}^{2^{d-r}} & ~\sum_{h_1, h_2 = m + 1}^{2^r (u - 1) - 1} \sum_{k = 0}^{2^r - 1}  |\mathrm{cum} (\tilde{X}_{0, \ell}, \tilde{X}_{h_1, \ell}, \tilde{X}_{h_1 - k, \ell}, \tilde{X}_{h_1 - h_2 - k, \ell})|
        \bigg)^{1/2}\\
        &\lesssim 2^{d/2} \sum_{r=0}^d 2^{r/2} \bigg( \sum_{\ell \in S} c_\ell^2 \sum_{u=1}^{2^{d-r}}
         \sum_{h_1, h_2 = m + 1}^{2^r (u - 1) - 1} \sum_{k = 0}^{2^r - 1} \rho^{h_1 - (h_1 - h_2 - k)}
        \bigg)^{1/2} \\
        & \lesssim 2^{3d/2} d \bigg( \sum_{\ell \in S} c_\ell^2 \bigg)^{1/2}\\
        & = o \big(n^{2 - \beta} f(s) \big)
    \end{align*}
    for any $0 \le \beta < 1 / 4$, where we have used condition \ref{ass_shat_5} in the first inequality.

    Combining the above estimates yields $\tilde R_{n1} = o \big( n^{2-\beta} f(s) \big)$ for any $0 \le \beta < 1 / 4$. By similar arguments, it can be shown that also $\tilde R_{n2} = o \big( n^{2-\beta} f(s) \big)$, which completes the proof of Proposition \ref{prop67}.
\end{proof}

\begin{proof}[Proof of Proposition \ref{prop_max_unif}]
We only consider the case $n = 2^d$ and obtain by similar arguments as given in  the proof of Proposition \ref{max2} and Proposition \ref{maxmax}
    \begin{align*}
        &\e \bigg[ \max_{\ell \in S} \max_{1 \leq k \leq 2^d} \sqrt{n} \bigg( \sum_{j=1}^k  a_{j,n} (X_{j, \ell} - \e [X_{j, \ell}]) \bigg)^2 \bigg]^{1/2}\\
        & ~~~~~ ~~~~~ ~~~~~ ~~~~~ \leq 2^{d/4} \e \bigg[ \bigg( \max_{\ell \in S} \max_{1 \leq k \leq 2^d}
        \Big | \sum_{j=1}^k  a_{j,n}  (X_{j, \ell} - \e [X_{j, \ell}])  \Big | \bigg)^2 \bigg]^{1/2}\\
        & ~~~~~ ~~~~~ ~~~~~ ~~~~~ \leq 2^{d/4} \sum_{r=0}^d \bigg( \sum_{u=1}^{2^{d-r}} \sum_{\ell \in S}  \e \bigg[ \bigg( \sum_{j = 2^r(u-1) + 1}^{2^r u} a_{j,n} \tilde{X}_{j,\ell} \bigg)^2 \bigg] \bigg)^{1/2}\\
        & ~~~~~ ~~~~~ ~~~~~ ~~~~~ \leq 2^{d/4} \sup_{j = 1,  \ldots , n} |a_{j,n}| \sum_{r=0}^d 2^{r/2} \bigg( \sum_{u=1}^{2^{d-r}} \sum_{\ell \in S} \sum_{h = 0}^{2^r - 1} |\Sigma_{h, \ell \ell}| \bigg)^{1/2} \\
        & ~~~~~ ~~~~~ ~~~~~ ~~~~~ \lesssim  d 2^{3d/4} \bigg( \sum_{\ell \in S} c_\ell
        \bigg)^{1/2}\\
        &~~~~~ ~~~~~ ~~~~~ ~~~~~ \lesssim o \big(n^{1-\beta / 2} f(s)^{1/2} \big)
    \end{align*}
    by condition \ref{ass_shat_4} for any $0 \leq \beta < 1 / 4$, which proves the proposition.
\end{proof}

\noindent \textbf{Proof of \eqref{kn_ell_1}.}
As in \eqref{hd8}, we use the decomposition
\begin{align*}
    \sum_{\substack{j_1, j_2 = \hat k_n + 1 \\ |j_1 - j_2| > m}}^{\hat k_n + \gbr{\lambda (n- \hat k_n)}} \tilde X_{j_2, \ell} =  R_{n, \ell}^{(1)} (\lambda) + R_{n, \ell}^{(2)} (\lambda) - R_{n, \ell}^{(3)},
\end{align*}
where
\begin{align*}
    R_{n, \ell}^{(1)} (\lambda)&= \sum_{j = \hat k_n + 1}^{\hat k_n + \gbr{\lambda (n - \hat k_n)}} (\gbr{\lambda (n - \hat k_n)} - 2 m - 1) \cdot \tilde X_{j,\ell}\\
    R_{n, \ell}^{(2)}(\lambda) & = \sum_{j = \hat k_n + \gbr{\lambda (n - \hat k_n)} - m + 1}^{\hat k_n + \gbr{\lambda (n - \hat k_n)}} (j + m - \hat k_n - \gbr{\lambda (n - \hat k_n)} )
    \cdot \tilde X_{j, \ell} \\
    R_{n, \ell}^{(3)}  & = \sum_{j = \hat k_n + 1}^{\hat k_n + m} (j - \hat k_n - m - 1) \cdot \tilde X_{j,\ell}.
\end{align*}
By Proposition \ref{prop_maxmax_m}, we have for the terms $R_{n,\ell}^{(2)}$ and $R_{n,\ell}^{(3)}$
\begin{align*}
    \e \Big[ \max_{\ell \in S} \sup_{0 \leq \lambda \leq 1} \big| R_{n,\ell}^{(2)} (\lambda) \big| \Big] &\leq \e \bigg[ \max_{\ell \in S} \max_{1 \leq k \leq n} \bigg| \sum_{j = k - m + 1}^k (j+m - k) \cdot\tilde X_{j,\ell} \bigg| \bigg] \\
    &= \e \bigg[ \max_{\ell \in S} \max_{1 \leq k \leq n} \bigg| \sum_{j = k + 1}^{k+m} (j - k) \cdot \tilde X_{j-m, \ell} \bigg| \bigg] \\
    &\lesssim m^{3/2} \sqrt{n} \bigg( \sum_{\ell \in S} \sum_{h = 0}^{m-1} |\Sigma_{h, \ell\ell} | \bigg)^{1/2}
\end{align*}
and
\begin{align*}
    \e \Big[ \max_{\ell \in S} \big| R_{n,\ell}^{(3)} \big| \Big] &\leq \e \bigg[ \max_{\ell \in S} \max_{1 \leq k \leq n} \bigg| \sum_{j = k + 1}^{k+m} (j-k-m-1) \cdot \tilde X_{j,\ell} \bigg| \bigg] \\
    &\lesssim m^{3/2} \sqrt{n} \bigg( \sum_{\ell \in S} \sum_{h = 0}^{m-1} |\Sigma_{h, \ell\ell} | \bigg)^{1/2}.
\end{align*}
Therefore, using conditions \ref{ass_shat_4} and \ref{ass_shat_6} we get (note that $N_m (n - \hat k_n) \simp n^2$)
\begin{align*}
    \e \bigg[ \max_{\ell \in S} \sup_{0 \leq \lambda \leq 1} \frac{\sqrt{n}}{n^2} \big| R_{n,\ell}^{(2)} (\lambda) - R_{n,\ell}^{(3)} \big| \bigg] &\lesssim \frac{m^{3/2}}{n}  \bigg( \sum_{\ell \in S} \sum_{h = 0}^{m-1} |\Sigma_{h,\ell\ell}| \bigg)^{1/2}\\
    &\lesssim  \frac{m^{3/2}}{n}  \bigg( \sum_{\ell \in S} c_\ell \bigg)^{1/2}\\
    &= O \bigg( \Big( \frac{m^2}{n} \Big)^{3/4}n^{-1/8} \sqrt{f(s)} \bigg)\\
    &= o \big(n^{-\beta} f(s)^{1/2} \big).
\end{align*}

We now turn to $R_{n,\ell}^{(1)} (\lambda)$ and obtain by Theorem \ref{thm2.0_unif_neu}
\begin{align*}
    & \frac{\sqrt{n}}{N_m (n - \hat k_n)} R_{n, \ell}^{(1)} (\lambda)\\
    & ~~~~~ = \frac{n}{n - \hat k_n - m} \frac{\gbr{\lambda (n - \hat k_n)} - 2m - 1}{n - \hat k_n - m - 1} \bigg( \frac{1}{\sqrt{n}} \sum_{j = 1}^{\hat k_n + \gbr{\lambda (n - \hat k_n)}} \tilde X_{j,\ell} - \frac{1}{\sqrt{n}} \sum_{j = 1}^{\hat k_n} \tilde X_{j,\ell} \bigg)\\
    & ~~~~~ = \frac{\lambda}{1 - \vartheta_0} \Big( \mathbb{B} \big( \Gamma_{\ell \ell} (\hat k_n + \gbr{\lambda(n - \hat k_n)}) / n \big) - \mathbb{B} \big( \Gamma_{\ell \ell} \hat k_n / n \big) \Big) + o_\p (n^{-\beta} f(s))
\end{align*}
uniformly in $\ell \in S$ and $\lambda \in [0,1]$. Note that we have used
\begin{align}\label{uniformbound_bm}
    \max_{\ell \in S^C} \sup_{0 \leq \lambda \leq 1} |\mathbb{B} \big( \lambda \Gamma_{\ell\ell} \big)| \leq \sup_{0 \leq \lambda \leq C} |\mathbb{B} (\lambda) | = O_\p (1),
\end{align}
which holds, because by the second part of condition \ref{ass_shat_1} there exists $C > 0$ such that $\max_{\ell\in S} \Gamma_{\ell\ell} \leq C$ and the supremum is taken over the interval $[0, C]$ which is independent of $n$.

From Theorem \ref{cpthm} and conditions \ref{ass_shat_1} and \ref{ass_shat_7} it follows for $i=1,2$
\begin{align}
    \begin{aligned}\label{rates_logn}
        \max_{\ell \in S} \sup_{0 \leq \lambda \leq 1} \Gamma_{\ell\ell} \big| \gbr{\lambda \hat k_n} / n - \lambda \vartheta_0 \big| = O_\p (\log^2 (n) / n) & = o_\p (n^{-\alpha / 2}), \\
        \max_{\ell \in S}  \sup_{0 \leq \lambda \leq 1} \Gamma_{\ell\ell}  \big| (\hat k_n + \gbr{\lambda (n - \hat k_n)}) / n - (\vartheta_0 + \lambda (1- \vartheta_0) ) \big| & = o_\p (n^{-\alpha / 2}),
    \end{aligned}
\end{align}
because $\alpha < 1$. We therefore conclude with Lemma~\ref{double_unif_cont_BM} for any $\beta < \alpha / 4$
\begin{align*}
    \frac{\sqrt{n}}{N_m (n - \hat k_n)} R_{n, \ell}^{(1)} (\lambda) &= \frac{\lambda}{1 - \vartheta_0} \Big( \mathbb{B} \big( \Gamma_{\ell \ell} (\vartheta_0 + \lambda (1 - \vartheta_0)) \big) - \mathbb{B} \big( \Gamma_{\ell \ell} \vartheta_0 \big) \Big) + o_\p (n^{-\beta} f(s))
\end{align*}
and this estimate holds uniformly in $\ell \in S$ and $\lambda \in [0,1]$. This proves \eqref{kn_ell_1}.

\medskip

\noindent \textbf{Proof of \eqref{ln_ell_1}.} 
As in \eqref{toinduce}, we have
\begin{align*}
    \sum_{\substack{i_1,i_2=1\\|i_1-i_2| > m}}^{\gbr{\lambda \hat k_n}} &\big(\mathbbm{1} \{ i_1 \leq k_0 \} \tilde{X}_{i_2, \ell} + \mathbbm{1} \{ i_2 \leq k_0 \} \tilde{X}_{i_1, \ell} \big)\\
    &= \sum_{\substack{i_1,i_2=1\\|i_1-i_2| > m}}^{k_0 \wedge \gbr{\lambda \hat k_n}} (\tilde{X}_{i_1, \ell} + \tilde{X}_{i_2, \ell}) + 2 \cdot \mathbbm{1} \{ k_0 \leq \gbr{\lambda \hat k_n} \} \sum_{i=k_0 + 1}^{\gbr{\lambda \hat k_n}} (i - m - 1) \wedge k_0 \cdot  \tilde X_{i, \ell}.
\end{align*}
For the first term, we use the decomposition
\begin{align*}
    \sum_{\substack{i_1,i_2=1\\|i_1-i_2| > m}}^{k_0 \wedge \gbr{\lambda \hat k_n}} (\tilde{X}_{i_1, \ell} + \tilde{X}_{i_2, \ell}) = 2 \big( S_{n, \ell}^{(1)} (\lambda) + S_{n, \ell}^{(2)} (\lambda) + S_{n, \ell}^{(3)} \big),
\end{align*}
where
\begin{align*}
    S_{n, \ell}^{(1)} (\lambda) &= \sum_{j = 1}^{k_0 \wedge \gbr{\lambda \hat k_n}} (k_0 \wedge \gbr{\lambda \hat k_n} - 2m - 1) \cdot \tilde X_{j,\ell},\\
    S_{n,\ell}^{(2)} (\lambda) &= \sum_{j = k_0 \wedge \gbr{\lambda \hat k_n} - m + 1}^{k_0 \wedge \gbr{\lambda \hat k_n}} (j + m - k_0 \wedge \gbr{\lambda \hat k_n}) \cdot \tilde X_{j,\ell}, \\
    S_{n,\ell}^{(3)} &= \sum_{j = 1}^m (m+1-j) \cdot \tilde X_{j,\ell}.
\end{align*}
By the same arguments used in the proof of \eqref{secondterm2} and condition \ref{ass_shat_4}, we obtain
\begin{align*}
    \max_{\ell \in S} \sup_{\lambda \in [0,1]} \frac{\sqrt{n}}{N_m (\hat k_n)} \big| S_{n,\ell}^{(2)} (\lambda) + S_{n,\ell}^{(3)} \big| &\lesssim \frac{n}{\hat k_n - m} \frac{m^{3/2}}{\hat k_n - m - 1} \bigg( \sum_{\ell \in S} c_\ell \bigg)^{1/2}\\
    & = o_\p \big( n^{-\beta} \sqrt{f(s)} \big).
\end{align*}
Turning to the term $S_n^{(1)}$ we have by Theorem \ref{thm2.0_unif_neu} and Lemma~\ref{double_unif_cont_BM}
\begin{align*}
    \frac{\sqrt{n}}{N_m (\hat k_n)}  S_{n,\ell}^{(1)} (\lambda) &= \frac{n}{\hat k_n - m} \frac{k_0 \wedge \gbr{\lambda \hat k_n} - 2m - 1}{\hat k_n - m - 1} \frac{1}{\sqrt{n}} \sum_{i=1}^{k_0 \wedge \gbr{\lambda \hat k_n}} \tilde X_{i, \ell}\\
    &= \frac{\lambda}{\vartheta_0} \mathbb{B} \big( \lambda \vartheta_0 \Gamma_{\ell\ell} \big) + o_\p \big( n^{-\beta} f(s) \big).
\end{align*}
By the same arguments used in the proof of \eqref{secondterm2}, we have
\begin{align*}
    \max_{\ell \in S} \sup_{\substack{0 \leq \lambda \leq 1 \\ k_0 + 1 \leq \gbr{\lambda \hat k_m}}} \bigg| \sum_{i = k_0 + 1}^{\gbr{\lambda \hat k_n}} (i-m-1)\wedge k_0 \cdot \tilde X_{i,\ell} \bigg| = o_\p (n^{-\beta} f(s)),
\end{align*}
which finishes the proof.

\subsection{Proof of Theorem \ref{unif_conv_thm_Sc}}
\label{sec_proof_d2}

As $T_{n, \ell}^{(2)} (\lambda) = T_{n, \ell}^{(3)} (\lambda) = T_{n, \ell}^{(4)} (\lambda)  \equiv 0$ for indices $\ell \in S^C$,  we obtain from the decomposition \eqref{decomp_unif_stat},
\begin{align}\label{hd15neu}
    \hat \delta_\ell^2 ( \lambda ) = T_{n, \ell}^{(1)} (\lambda) = \frac{N_m (\gbr{\lambda (n-\hat k_n)})}{N_m (n-\hat k_n)} P_{n, \ell} (\lambda) - Q_{n, \ell} (\lambda) + \frac{N_m (\gbr{\lambda \hat k_n})}{N_m (\hat k_n)} R_{n, \ell} (\lambda),
\end{align}
where $P_{n, \ell}, Q_{n, \ell}, R_{n, \ell}$ are defined in \eqref{hd15a} -- \eqref{hd15c}. Note that \eqref{hd15neu} coincides with \eqref{hd15}, but is now considered for $\ell \in S^C$. With the notation
\begin{align*}
    H_{1, \ell} (\lambda) := \frac{ \mathbb{B} \big(\lambda \vartheta_0 \Gamma_{ \ell \ell} \big) }{\vartheta_0}, ~~~~ H_{2, \ell} (\lambda) := \frac{ \mathbb{B}  \big( \Gamma_{\ell \ell}(\vartheta_0 + \lambda(1-\vartheta_0)) \big) - \mathbb{B} \big( \vartheta_0 \Gamma_{\ell \ell} \big) }{1 - \vartheta_0}, 
\end{align*}
$\tau_{\ell}^2 = \Gamma_{\ell\ell} / (\vartheta_0 (1-\vartheta_0))$ and $r_n = n^{-\beta} f_c (s_c)$, we will show below that
\begin{align}
    \max_{\ell \in S^C} \sup_{\lambda \in [0, 1]} \big| n P_{n, \ell} (\lambda) -  (H_{1, \ell }^2 (\lambda) - \lambda \Gamma_{\ell\ell} / \vartheta_0) \big| = o_\p  \big(r_n^2 + r_n\big),\label{ph1}\\
    \max_{\ell \in S^C} \sup_{\lambda \in [0, 1]} \big| n Q_{n, \ell} (\lambda) -  2\lambda^2 H_{1, \ell } (\lambda) H_{2, \ell } (\lambda) \big| = o_\p \big( r_n^2 + r_n \big),\label{qh1h2}\\
    \max_{\ell \in S^C} \sup_{\lambda \in [0, 1]} \big| n R_{n, \ell} (\lambda) -  (H_{2, \ell }^2 (\lambda) - \lambda \Gamma_{\ell\ell} / (1-\vartheta_0) ) \big| = o_\p \big( r_n^2 + r_n\big).\label{rh2}
\end{align}
With these approximations, we obtain from \eqref{hd15neu}
\begin{align*}
    \max_{\ell \in S^C} \sup_{\lambda \in [0, 1]} \big| n T_{n, \ell}^{(1)} (\lambda) - \tau_{\ell}^2 H_\ell (\lambda)  \big| = o_\p \big( r_n^2 + r_n \big),
\end{align*}
where
\begin{align*}
    H_\ell (\lambda) &:= \frac{ \lambda^2 }{\tau_{\ell}^2} \bigg( H_{1, \ell}^2 (\lambda) - 2 H_{1, \ell} (\lambda) H_{2, \ell} (\lambda) + H_{2, \ell}^2 (\lambda) - \lambda \bigg( \frac{\Gamma_{\ell\ell}}{\vartheta_0} + \frac{\Gamma_{\ell\ell}}{1 - \vartheta_0} \bigg) \bigg)\\
    &~= \frac{ \lambda^2 }{\tau_{\ell}^2} \big( (H_{1, \ell} (\lambda) - H_{2, \ell} (\lambda))^2 - \lambda \tau_{\ell}^2 \big)\\
    &~= \lambda^2 \bigg( \bigg( \frac{ H_{1, \ell} (\lambda) - H_{2, \ell} (\lambda) }{\tau_{\ell}}  \bigg)^{2}  - \lambda \bigg),
\end{align*}
because $ \tau_{\ell}^2 = \Gamma_{\ell\ell} / \vartheta_0 + \Gamma_{\ell\ell} / (1-\vartheta_0)$.
The statement will then follow, because $$\Big\{ \tau_{\ell}^{-1} \big( H_{1, \ell} (\lambda) - H_{2, \ell} (\lambda) \big) \Big\}_{\lambda \in [0,1] } =^d \{\mathbb{B} (\lambda) \}_{\lambda \in [0,1]}$$ for any $\ell \in S^C$.

\medskip

\noindent\textbf{Proofs of \eqref{ph1} and \eqref{rh2}.}
We enlarge the sum to form a square
\begin{align*}
    P_{n, \ell} (\lambda) &= \frac{1}{N_m (\hat k_n)} \sum_{\substack{i_1,i_2=1\\|i_1-i_2| > m}}^{\gbr{\lambda \hat k_n}} \tilde{X}_{i_1, \ell} \tilde{X}_{i_2, \ell}\\
    &=\frac{{\mathbbm{1} \{ \gbr{\lambda \hat k_n}  \geq m + 2 \} }}{N_m (\hat k_n)} \bigg\{ \bigg( \sum_{i = 1}^{\gbr{\lambda \hat k_n}} \tilde X_{i, \ell} \bigg)^2 - \sum_{|h| \leq m} \sum_{i = 1}^{\gbr{\lambda \hat k_n} - |h|} \tilde X_{i, \ell} \tilde X_{i + |h|, \ell} \bigg\}
\end{align*}
(note that by definition $P_{n,\ell} (\lambda) = 0$ whenever $\gbr{\lambda \hat k_n} \leq m + 1$).
Hence,
\begin{align*}
    n P_{n, \ell} (\lambda) &= \frac{n^2 \mathbbm{1} \{ \gbr{\lambda \hat k_n}  \geq m + 2 \} }{N_m (\hat k_n)} \bigg\{ \bigg( \frac{1}{\sqrt{n}} \sum_{i = 1}^{\gbr{\lambda \hat k_n}} \tilde X_{i, \ell} \bigg)^2 - \frac{1}{n} \sum_{|h| \leq m} \sum_{i = 1}^{\gbr{\lambda \hat k_n} - |h|} \tilde X_{i, \ell} \tilde X_{i + |h|, \ell} \bigg\}.
\end{align*}
By the same arguments as used before,
\begin{align*}
    \bigg| \frac{n^2}{N_m (\hat k_n)} - \frac{1}{\vartheta_0^2} \bigg| = o_\p (n^{-\beta}) ~~~~ \text{ and } ~~~~ \bigg| \frac{n^2}{N_m (n - \hat k_n)} - \frac{1}{(1-\vartheta_0)^2} \bigg| = o_\p (n^{-\beta}) 
\end{align*}
and similar arguments as given in the Proof of Theorem \ref{unif_conv_thm_S} show 
\begin{align*}
    \max_{\ell \in S^C} \sup_{\lambda \in [0, 1]} \bigg| \frac{1}{\sqrt{n}} \sum_{i = 1}^{\gbr{\lambda \hat k_n}} \tilde X_{i, \ell}  - \mathbb{B} \big(\lambda \vartheta_0 \Gamma_{\ell\ell} \big) \bigg| = o_\p \big( r_n \big),
\end{align*}
and we will prove below that
\begin{align}\label{stilltoprove}
    \max_{\ell \in S^C} \sup_{\substack{\lambda \in [0,1] \\ \gbr{\lambda \hat k_n} \geq m + 2}} \bigg| \frac{1}{n} \sum_{|h| \leq m} \sum_{i = 1}^{\gbr{\lambda \hat k_n} - |h|} \tilde X_{i, \ell} \tilde X_{i + |h|, \ell} - \lambda \vartheta_0 \Gamma_{ \ell \ell}  \bigg| = o_\p \big( r_n \big).
\end{align}
With these approximations it follows, observing that $\p (\gbr{\lambda \hat k_n} \leq  m + 2 ) =o(1)$ by condition \ref{ass_shat_6},
\begin{align*}
    n P_{n, \ell} (\lambda) &= \frac{n^2}{N_m (\hat k_n)}  \Big( \mathbb{B}^2 \big(\lambda \vartheta_0 \Gamma_{ \ell \ell} \big) - \lambda \vartheta_0 \Gamma_{ \ell \ell} \Big) + o_\p (r_n^2 + r_n)\\
    &= \vartheta_0^{-2}  \Big( \mathbb{B}^2 \big(\lambda \vartheta_0 \Gamma_{ \ell \ell} \big) - \lambda \vartheta_0 \Gamma_{ \ell \ell} \Big) + o_\p (r_n^2 + r_n)\\
    &= H_{1, \ell}^2 (\lambda) - \lambda \frac{\Gamma_{\ell\ell}}{\vartheta_0}  + o_\p (r_n^2 + r_n )
\end{align*}
uniformly in $\ell \in S^C$ and $\lambda \in [0,1]$, which proves \eqref{ph1}. Note that we also used \eqref{uniformbound_bm} in this argument.

By the same arguments, we also have
\begin{align*}
    n R_{n, \ell} (\lambda) &= \frac{n^2}{N_m (n - \hat k_n)} \Big\{ \Big( \mathbb{B}  \big( \Gamma_{\ell \ell} (\vartheta_0 + \lambda(1-\vartheta_0)) \big) - \mathbb{B} \big( \vartheta_0 \Gamma_{\ell \ell} \big)  \Big)^2 - \lambda (1-\vartheta_0) \Gamma_{\ell \ell} \Big\}\\
    & ~~~~~ ~~~~~ ~~~~~ + o_\p \big( r_n^2 + r_n \big)\\
    &= H_{2, \ell}^2 (\lambda) - \lambda \frac{\Gamma_{\ell\ell}}{1 - \vartheta_0}  + o_\p \big( r_n^2 + r_n \big)
\end{align*} 
uniformly in $\ell \in S^C$ and $\lambda \in [0,1]$, which proves \eqref{rh2}.
\medskip

\noindent\textbf{Proof of \eqref{qh1h2}.}
We use Lemma~\ref{lem1} to obtain representation
\begin{align*}
    N_m (\hat k_n) & N_m (n- \hat k_n) \cdot Q_{n,\ell} (\lambda)\\
    &= 2\bigg( \sum_{j = 1}^{\gbr{\lambda \hat k_n} - m}  (\gbr{\lambda \hat k_n} - m - j) \cdot \tilde X_{j, \ell} + \sum_{j = m+1}^{\gbr{\lambda \hat k_n}} (j - m - 1) \cdot \tilde X_{j, \ell} \bigg)\\
    &~~~~~~ ~~~~~~ ~~~~~~  \times \bigg( \sum_{j = \hat k_n + 1}^{\hat k_n + \gbr{\lambda (n - \hat k_n)} - m}  (\hat k_n + \gbr{\lambda (n - \hat k_n)} - m - j) \cdot \tilde X_{j, \ell} \\
    & ~~~~~~ ~~~~~~ ~~~~~~~~~~~~~~~~~~~~~~~~~~~~~~+ \sum_{j = \hat k_n + m + 1}^{\hat k_n + \gbr{\lambda (n - \hat k_n)}} (j - \hat k_n - m - 1) \cdot \tilde X_{j, \ell}  \bigg)\\
    &= 2\Big( A_{n, \ell}^{(1)} (\lambda) + A_{n, \ell}^{(2)} (\lambda)  + A_{n, \ell}^{(3)} \Big) \Big( B_{n, \ell}^{(1)} (\lambda) +B_{n, \ell}^{(2)} (\lambda) - B_{n, \ell}^{(3)}  \Big),
\end{align*}
where
\begin{align*}
    A_{n, \ell}^{(1)} (\lambda) &= \sum_{j = 1}^{\gbr{\lambda \hat k_n}} (\gbr{\lambda \hat k_n} - 2m - 1) \cdot \tilde X_{j, \ell},\\
    A_{n, \ell}^{(2)} (\lambda) &= \sum_{j = \gbr{\lambda \hat k_n} - m + 1}^{\gbr{\lambda \hat k_n}} (j + m - \gbr{\lambda \hat k_n}) \cdot \tilde X_{j, \ell} , \\
    A_{n, \ell}^{(3)} &= \sum_{j = 1}^m (m+1-j) \cdot \tilde X_{j, \ell}
\end{align*}
and 
\begin{align*}
    B_{n, \ell}^{(1)} (\lambda) &= \sum_{j=\hat k_n + 1}^{\hat k_n + \gbr{\lambda (n - \hat k_n)}} (\gbr{\lambda (n - \hat k_n)} - 2m - 1) \cdot \tilde X_{j, \ell},\\
    B_{n, \ell}^{(2)} (\lambda) &= \sum_{j = \hat k_n + \gbr{\lambda (n - \hat k_n)} -m +1}^{\hat k_n + \gbr{\lambda (n - \hat k_n)}} (j + m - \hat k_n - \gbr{\lambda (n - \hat k_n)}) \cdot \tilde X_{j, \ell} , \\
    B_{n, \ell}^{(3)} &= \sum_{j=\hat k_n + 1}^{\hat k_n + m} (j - \hat k_n - m - 1) \cdot \tilde X_{j,\ell}.
\end{align*}
By Proposition \ref{prop_maxmax_m} and condition \ref{ass_shat_4}, we get that
\begin{align*}
    \e \Big[ \max_{\ell \in S^C} \sup_{0 \leq \lambda \leq 1} |A_{n, \ell}^{(2)} (\lambda)| \Big] &\leq \e \bigg[ \max_{\ell \in S^C} \max_{1 \leq k \leq n} \Big| \sum_{j = k + 1}^{k +m} (j-k) \cdot \tilde X_{j-m, \ell} \Big| \bigg]\\
    &\lesssim m^{3/2} \sqrt{n} \bigg( \sum_{\ell \in S^C} c_\ell \bigg)^{1/2}\\
    &\lesssim m^{3/2} \sqrt{n^{1 + \alpha / 8} f_c (s_c)}
\end{align*}
and
\begin{align*}
    \e \Big[ \max_{\ell \in S^C}  |A_{n, \ell}^{(3)} | \Big] \leq \e \bigg[ \max_{\ell \in S^C} \max_{1 \leq k \leq n} \Big| \sum_{j = k + 1}^{k +m} (m+1-j) \cdot \tilde X_{j, \ell} \Big| \bigg] &\lesssim m^{3/2} \sqrt{n} \bigg( \sum_{\ell \in S^C} c_\ell \bigg)^{1/2}\\
    &\lesssim m^{3/2} \sqrt{n^{1 + \alpha / 8} f_c (s_c)}.
\end{align*}
Therefore, noting that $N_m (\hat k_n) \simp n^2$ and using that $m = O(n^{1/2 - \alpha / 4})$ by \ref{ass_shat_6}, we conclude 
\begin{align}\label{an2bis3}
    \max_{\ell \in S^C} \sup_{\lambda \in [0, 1]} \frac{\sqrt{n}}{N_m (\hat k_n)}  \Big|A_{n,\ell}^{(2)} (\lambda) + A_{n,\ell}^{(3)}\Big| &= o_\p \big( n^{-\beta} \sqrt{f_c (s_c)} \big) = o_\p (r_n).
\end{align}
By the same arguments, one also finds
\begin{align}\label{bn2bis3}
    \max_{\ell \in S^C} \sup_{\lambda \in [0, 1]}  \frac{ \sqrt{n} }{N_m (n - \hat k_n)} \Big|B_{n,\ell}^{(2)} (\lambda) - B_{n,\ell}^{(3)}\Big| &= o_\p \big( n^{-\beta} \sqrt{f_c (s_c)} \big) = o_\p (r_n).
\end{align}

Next, we turn to $A_{n,\ell}^{(1)}$ and $B_{n, \ell}^{(1)}$, where we have for the first term and any $\ell \in S^C$
\begin{align}
    \frac{\sqrt{n}}{N_m (\hat k_n)} A_{n, \ell}^{(1)} (\lambda) &= \frac{n}{\hat k_n - m} \frac{\gbr{\lambda \hat k_n} - 2m - 1}{\hat k_n - m - 1} \frac{1}{\sqrt{n}} \sum_{j = 1}^{\gbr{\lambda \hat k_n}} \tilde X_{j, \ell}\nonumber\\
    &= \frac{\lambda}{\vartheta_0} \mathbb{B} \big( \Gamma_{\ell\ell} \gbr{\lambda \hat k_n} / n \big) + o_\p \big( r_n  \big)\nonumber\\
    &= \frac{\lambda}{\vartheta_0} \mathbb{B} \big( \Gamma_{\ell\ell} \lambda \vartheta_0 \big) + o_\p \big( r_n  \big) \nonumber\\
    &= \lambda H_{1, \ell} (\lambda) + o_\p (r_n), \label{an1_res}
\end{align}
where in the second to last line we have used Lemma~\ref{double_unif_cont_BM} together with \eqref{rates_logn}. Using the same arguments, we also find for any $\ell \in S^C$
\begin{align}
    &\frac{\sqrt{n}}{N_m (n - \hat k_n)} B_{n, \ell}^{(1)} (\lambda) \nonumber\\
    & ~~~~~ = \frac{n}{n - \hat k_n - m} \frac{\gbr{\lambda (n - \hat k_n)} - 2m - 1}{n - \hat k_n - m - 1} \bigg( \frac{1}{\sqrt{n}} \sum_{j = 1}^{\hat k_n + \gbr{\lambda (n - \hat k_n)}} \tilde X_{j,\ell} - \frac{1}{\sqrt{n}} \sum_{j = 1}^{\hat k_n} \tilde X_{j,\ell} \bigg) \nonumber\\
    & ~~~~~ = \frac{\lambda}{1 - \vartheta_0} \Big( \mathbb{B} \big( \Gamma_{\ell \ell} (\hat k_n + \gbr{\lambda(n - \hat k_n)}) / n \big) - \mathbb{B} \big( \Gamma_{\ell \ell} \hat k_n / n \big) \Big) + o_\p (r_n)\nonumber\\
    & ~~~~~ = \frac{\lambda}{1 - \vartheta_0} \Big( \mathbb{B}  \big( \Gamma_{\ell \ell}(\vartheta_0 + \lambda(1-\vartheta_0)) \big) - \mathbb{B} \big( \vartheta_0 \Gamma_{\ell \ell} \big) \Big) + o_\p (r_n ) \nonumber\\
    & ~~~~~ = \lambda H_{2, \ell} (\lambda) + o_\p (r_n ).  \label{bn1_res}
\end{align}

Combining \eqref{uniformbound_bm}, \eqref{an2bis3}, \eqref{bn2bis3}, \eqref{an1_res} and \eqref{bn1_res}, we obtain
\begin{align*}
    nQ_{n, \ell} (\lambda) = 2 \lambda^2 H_{1,\ell} (\lambda) H_{2, \ell} (\lambda) + o_\p (r_n^2 + r_n).
\end{align*}
\medskip

\noindent\textbf{Proof of \eqref{stilltoprove}.} 
\noindent Similar arguments as used before yield
\begin{align*}
    \max_{\ell \in S^C} \sup_{\lambda \in [0,1]} \Gamma_{\ell\ell} | \lambda \vartheta_0 - \gbr{\lambda \hat  k_n } / n| = o_\p (n^{-2\beta}),
\end{align*}
which gives
\begin{align}
    \nonumber
    &\max_{\ell \in S^C} \sup_{\substack{\lambda \in [0,1] \\ \gbr{\lambda \hat k_n}  \geq m + 2}}  \bigg| \frac{1}{n } \sum_{|h| \leq m} \sum_{i = 1}^{\gbr{\lambda \hat k_n} - |h|} \tilde X_{i, \ell} \tilde X_{i + |h|, \ell} - \lambda \vartheta_0 \Gamma_{ \ell \ell} \bigg|\\
    \nonumber 
    & ~~~~~ ~~~~~ \leq \max_{\ell \in S^C} \sup_{\substack{\lambda \in [0,1] \\ \gbr{\lambda \hat k_n}  \geq m + 2}} 
    \frac{1}{n} \bigg| \sum_{|h| \leq m} \sum_{i = 1}^{\gbr{\lambda \hat k_n} - |h|} \tilde X_{i, \ell} \tilde X_{i + |h|, \ell} - \Gamma_{\ell \ell} \gbr{\lambda \hat k_n} \bigg| + o_\p (n^{-2\beta})\\
    \nonumber 
    & ~~~~~ ~~~~~ \leq \max_{\ell \in S^C} \sup_{\substack{\lambda \in [0,1] \\ \gbr{\lambda \hat k_n} \geq m + 2}}  \frac{1}{n} \bigg| \sum_{|h| \leq m} \bigg( \sum_{i = 1}^{\gbr{\lambda \hat k_n} - |h|} \tilde X_{i, \ell} \tilde X_{i + |h|, \ell} - \Sigma_{h, \ell \ell} (\gbr{\lambda \hat k_n} - |h| )\bigg) \bigg| \\
    \nonumber 
    & ~~~~~ ~~~~~ ~~~~~    + \max_{\ell \in S^C} \frac{1}{n} \bigg| \sum_{|h| \leq m} |h| \cdot \Sigma_{h, \ell \ell} \bigg| + \max_{\ell \in S^C} \bigg| \sum_{|h| > m} \Sigma_{h, \ell \ell} \bigg|  + o_\p (n^{-2\beta})\\
    \nonumber 
    & ~~~~~ ~~~~~ \lesssim \max_{\ell \in S^C} \max_{m + 2 \leq k \leq n} \frac{1}{n} \bigg| \sum_{h=0}^m  \sum_{i = 1}^{k - h} (\tilde X_{i, \ell} \tilde X_{i + h, \ell} - \Sigma_{h, \ell \ell} ) \bigg) \bigg| \\
    \label{det50a}
    & ~~~~~ ~~~~~ ~~~~~    + \max_{\ell \in S^C} \frac{1}{n} \bigg| \sum_{|h| \leq m} |h| \cdot \Sigma_{h, \ell \ell} \bigg| + \max_{\ell \in S^C} \bigg| \sum_{|h| > m} \Sigma_{h, \ell \ell} \bigg|  + o_\p (n^{-2\beta})
\end{align}
for any $0 < \beta < \alpha / 4$.
By condition \ref{ass_shat_2}, we obtain
\begin{align*}
    \max_{\ell \in S^C} \bigg| \sum_{|h| > m} \Sigma_{h, \ell \ell} \bigg| = O (n^{-\alpha / 4}) = o(n^{-\beta})
\end{align*}
for any $0 \leq \beta < \alpha / 4$ and using conditions \ref{ass_shat_4} and \ref{ass_shat_6} yields
\begin{align*}
    \max_{\ell \in S^C} \bigg| \frac{1}{n}  \sum_{|h| \leq m} |h| \cdot \Sigma_{h, \ell \ell} \bigg| &\lesssim \frac{m}{n}  \max_{\ell \in S^C} \sum_{h = 0}^m |\Sigma_{h, \ell \ell} |\\
    &\leq \frac{m}{n} \max_{\ell \in S^C} c_\ell \leq \frac{m}{n} \sum_{\ell \in S^C}  c_\ell = o \big(n^{-\beta} f_c (s_c) \big) = o(r_n).
\end{align*}
for any $0 \leq \beta < \alpha /4$. It remains to show a corresponding estimate for the first term in \eqref{det50a}, that is
\begin{align*}
    \max_{\ell \in S^C} \max_{m+2 \leq k \leq n} \frac{1}{n}\bigg|   \sum_{h = 0}^m \sum_{i = 1}^{k - h} 
    Y_{i,h,\ell}
    \bigg| = o_\p (r_n),
\end{align*}
where $Y_{i,h,\ell} = \tilde X_{i, \ell} \tilde X_{i + h, \ell} - \Sigma_{h, \ell \ell}$.
We have by Proposition \ref{maxmax} (using a telescopic sum argument again) that
\begin{align*}
    \max_{\ell \in S^C} \max_{1 \leq k \leq n} \bigg| \frac{1}{n} \sum_{h=0}^m  \sum_{i = 1}^{k - h} Y_{i, h, \ell} \bigg|
    & \leq 2^{-d} \sum_{r = 0}^d \bigg( \sum_{\ell \in S^C} \sum_{u = 1}^{2^{d-r}} \e \bigg[ \bigg( \sum_{h=0}^m \sum_{i = 2^r(u-1) + 1 - h}^{2^r u - h} Y_{i,h, \ell} \bigg)^2 \bigg] \bigg)^{1/2},
\end{align*}
where the value of the inner sum is defined as $0$ whenever $k-h< 1$. 
Moreover, a straightforward calculation using stationarity of $Y_{i,h,\ell}$ and shifting the two inner sums yields
\begin{align*}
    \e \bigg[ \bigg( \sum_{h=0}^m &\sum_{i = 2^r(u-1) + 1 - h}^{2^r u - h} Y_{i,h, \ell} \bigg)^2 \bigg]\\
    &= \sum_{h_1,h_2 = 0}^m \sum_{i_1 = 2^r (u-1) + 1 - h_1}^{2^r u -h_1}  \sum_{i_2 = 2^r (u-1) + 1 - h_2}^{2^r u - h_2} \e[Y_{i_1, h_1, \ell} Y_{i_2, h_2, \ell} ] \\
    &= \sum_{h_1,h_2 = 0}^m \sum_{i_1, i_2 = 2^r (u-1) + 1}^{2^r u} \e[Y_{0, h_1, \ell} Y_{i_2 - i_1 + h_1 - h_2, h_2, \ell} ]\\
    &=  \sum_{h_1,h_2 = 0}^m \sum_{|k| \leq 2^r - 1} (2^r - |k|) \e[Y_{0, h_1, \ell} Y_{k + h_1 - h_2, h_2, \ell} ]\\
    &\leq 2^r \sum_{h_1,h_2 = 0}^m \bigg( \sum_{k=0}^{2^r - 1} \big| \e[Y_{0, h_1, \ell} Y_{k + h_1 - h_2, h_2, \ell} ] \big| + \sum_{k=1}^{2^r - 1} \big| \e[Y_{0, h_1, \ell} Y_{h_1 - h_2 - k, h_2, \ell} ] \big| \bigg)\\
    & \lesssim 2^r \sum_{h_1,h_2 = 0}^m \sum_{k=0}^{2^r - 1} \e[Y_{0, h_1, \ell} Y_{h_1 - h_2 - k, h_2, \ell} ],
\end{align*}
where we have used the stationarity of $Y_{i,h,\ell}$ in the last step. Hence, the top expression reduces to
\begin{align*}
    \max_{\ell \in S^C} \max_{1 \leq k \leq n} \bigg| \frac{1}{n} \sum_{h=0}^m  \sum_{i = 1}^{k - h} Y_{i, h, \ell} \bigg| \lesssim 2^{-d / 2} \sum_{r = 0}^d \bigg( \sum_{\ell \in S^C} \sum_{h_1,h_2 = 0}^m \sum_{k = 0}^{2^r - 1} \big| \e[Y_{0, h_1, \ell} Y_{h_1 - h_2 - k, h_2, \ell} ] \big| \bigg)^{1/2}.
\end{align*}
By the definition of $Y_{i,h,\ell}$ and simple properties of cumulants, we obtain
\begin{align}\label{det50b}
    \begin{split}
        \e[Y_{0, h_1, \ell} Y_{h_1 - h_2 - k, h_2, \ell} ] &=  \Sigma_{k,\ell \ell} \Sigma_{k + h_2 - h_1, \ell \ell} + \Sigma_{k - h_1, \ell \ell} \Sigma_{k + h_2, \ell \ell}  \\
        & ~~~~~ + \mathrm{cum} ( \tilde X_{0, \ell}, \tilde X_{h_1, \ell}, \tilde X_{h_1 - k, \ell}, \tilde X_{h_1 - h_2 - k, \ell}) ,
    \end{split}
\end{align}
and we consider the sums corresponding to each term in this decomposition separately. Starting with the first sum, by condition \ref{ass_shat_4}, we get that
\begin{align*}
    2 ^{-d/2} \sum_{r = 0}^d \bigg( \sum_{\ell \in S^C} \sum_{h_1, h_2 = 0}^m \sum_{k = 0}^{2^r - 1} &|\Sigma_{k, \ell \ell} \Sigma_{k + h_2 - h_1, \ell \ell} | \bigg)^{1/2}\\
    &= 2^{-d/2} \sum_{r = 0}^{d} \bigg(  \sum_{\ell \in S^C} \sum_{k = 0}^{2^r - 1} | \Sigma_{k, \ell \ell} |  \sum_{|h| \leq m - 1} (m - |h|) | \Sigma_{k+h, \ell \ell} | \bigg)^{1/2}\\
    &\lesssim 2^{-d/2} \sqrt{m} \sum_{r = 0}^{d} \bigg(  \sum_{\ell \in S^C} \sum_{k = 0}^{2^r - 1}  | \Sigma_{k, \ell \ell} |  \sum_{|h| \leq m - 1} |\Sigma_{k+h, \ell \ell} | \bigg)^{1/2} \\
    &\lesssim 2^{-d / 2} \sqrt{m} \sum_{r = 0}^{d} \bigg\{ \sum_{\ell \in S^C} \sum_{k = 0}^{2^r - 1} |\Sigma_{k, \ell \ell}| \cdot c_\ell  \gamma_k  \bigg\}^{1/2}\\
    &\lesssim 2^{-d / 2} \sqrt{m} \sum_{r = 0}^{d} \bigg\{ \sum_{\ell \in S^C} c_\ell^2 \sum_{k = 0}^{2^r - 1}  \gamma_k^2   \bigg\}^{1/2}\\
    &\lesssim 2^{-d/2} d \sqrt{m}  \bigg( \sum_{\ell \in S^C} c_\ell^2 \bigg)^{1/2}\\
    &\lesssim \frac{\log (n)}{\sqrt[4]{n}} \frac{\sqrt{m}}{\sqrt[4]{n}} n^{\alpha / 8} f_c (s_c)\\
    &= o(r_n)
\end{align*}
for any $0 \leq \beta < \alpha / 4 < 1/4 $, since $\sqrt{m} / \sqrt[4]{n} = O(n^{- \alpha / 8})$ by condition \ref{ass_shat_6}.

For the sum corresponding to the second term in \eqref{det50b} we have
\begin{align*}
    \sum_{h=0}^m |\Sigma_{k-h, \ell \ell }| \leq \sum_{|h|\leq m} |\Sigma_{k+h, \ell \ell }| \lesssim \sum_{h = 0}^m |\Sigma_{k+h, \ell \ell }|,
\end{align*}
and we obtain by condition \ref{ass_shat_4}
\begin{align*}
    2 ^{-d/2} \sum_{r = 0}^d \bigg( \sum_{\ell \in S^C} \sum_{h_1, h_2 = 0}^m \sum_{k = 0}^{2^r - 1} &| \Sigma_{k - h_1, \ell \ell} \Sigma_{k +  h_2, \ell \ell} | \bigg)^{1/2}\\
    &\leq 2^{-d / 2} \sum_{r=0}^d \bigg\{ 
        \sum_{\ell \in S^C} \sum_{k = 0}^{2^r - 1} \sum_{h_1 = 0}^m |\Sigma_{k + h_1, \ell \ell} | \sum_{h_2 =0}^m |\Sigma_{k-h_2, \ell \ell}|
    \bigg\}^{1/2}\\
    &\lesssim 2^{-d / 2} \sum_{r=0}^d \bigg\{ 
        \sum_{\ell \in S^C} \sum_{k = 0}^{2^r - 1} \bigg( \sum_{h = 0}^m |\Sigma_{k+h, \ell \ell} | \bigg)^2
    \bigg\}^{1/2}\\
    &\lesssim 2^{-d / 2} \sum_{r=0}^d \bigg\{ 
        \sum_{\ell \in S^C} c_\ell^2 \sum_{k = 0}^{2^r - 1} \gamma_k^2 
    \bigg\}^{1/2}\\
    &\lesssim \frac{\log (n)}{\sqrt[4]{n}} \frac{\sqrt{m}}{\sqrt[4]{n}} \frac{1 }{ \sqrt{m} } \cdot n^{\alpha / 8} f_c (s_c)\\
    &= o(r_n)
\end{align*}
similarly as above.

For the sum corresponding to the third term in \eqref{det50b} it follows from condition \ref{ass_shat_5} that  
\begin{align*}
    2 ^{-d/2} \sum_{r = 0}^d  \bigg( \sum_{\ell \in S^C} \sum_{k = 0}^{2^r - 1} \sum_{h_1, h_2 = 0}^m  &| \mathrm{cum} (\tilde X_{0, \ell}, \tilde X_{h_1, \ell}, \tilde X_{h_1 - k, \ell}, \tilde X_{h_1 - h_2 - k, \ell}) |  \bigg)^{1/2}\\
    &\lesssim  2^{-d/2} \sum_{r = 0}^d \bigg( \sum_{\ell \in S^C} c_\ell^2 \sum_{k = 0}^{2^r - 1} \sum_{h_1, h_2 = 0}^m  \rho^{h_1 - (h_1 - h_2 - k)} \bigg)^{1/2}\\
    &\lesssim 2^{-d / 2} d \sqrt{m} \bigg( \sum_{\ell \in S^C} c_\ell^2 \bigg)^{1/2}\\
    &\lesssim \frac{\log (n)}{\sqrt[4]{n}} \frac{\sqrt{m}}{\sqrt[4]{n}} \cdot n^{\alpha / 8} f_c (s_c)\\
    & = o (r_n),
\end{align*}
which completes the proof of \eqref{stilltoprove}.

\subsection{Proof of Theorem \ref{adaptive_consistency_thm}}

By Theorem \ref{snconsistent}, $\p ({ \hat S_n \neq S }) = o(1)$, which yields
\begin{align*}
    \p (\hat T_{n, \hat S_n} > \Delta + q_{1-\alpha} \hat V_{n, \hat S_n}) = \p (\hat T_{n, S} > \Delta + q_{1-\alpha} \hat V_{n, S}, S = \hat S_n) + o(1). 
\end{align*}
The assertion follows by the same arguments as given in the proof of Theorem \ref{testconsistent}.


%% file: main.bbl
\begin{thebibliography}{}

\bibitem[Aue and Horváth, 2013]{auehor2013}
Aue, A. and Horváth, L. (2013).
\newblock Structural breaks in time series.
\newblock {\em J. Time Series Anal.}, 34(1):1--16.

\bibitem[Aue et~al., 2009]{aue4moments}
Aue, A., Hörmann, S., Horváth, L., and Reimherr, M. (2009).
\newblock Break detection in the covariance structure of multivariate time
  series models.
\newblock {\em Ann. Statist.}, 37(6B):4046--4087.

\bibitem[Barigozzi et~al., 2018]{Barigozzietal2018}
Barigozzi, M., Cho, H., and Fryzlewicz, P. (2018).
\newblock Simultaneous multiple change-point and factor analysis for
  high-dimensional time series.
\newblock {\em J. Econometrics}, 206(1):187--225.

\bibitem[Basu and Michailidis, 2016]{basu2016regularized}
Basu, S. and Michailidis, G. (2016).
\newblock Regularized estimation in sparse high-dimensional time series models.
\newblock {\em Biometrika}, 103(4):891--938.

\bibitem[Berger and Delampady, 1987]{berger1987}
Berger, J.~O. and Delampady, M. (1987).
\newblock Testing precise hypotheses.
\newblock {\em Statist. Sci.}, 2(3):317 -- 335.

\bibitem[Billingsley, 1995]{billingsley1995probability}
Billingsley, P. (1995).
\newblock {\em Probability and Measure}.
\newblock Wiley Series in Probability and Statistics. Wiley.

\bibitem[Brillinger, 2001]{brillinger}
Brillinger, D.~R. (2001).
\newblock {\em Time Series: Data Analysis and Theory}.
\newblock Number~36 in Classics in Applied Mathematics. SIAM, Philadelphia, PA.

\bibitem[Chen and and, 2019]{ChenWu2019}
Chen, L. and and, W. B.~W. (2019).
\newblock Testing for trends in high-dimensional time series.
\newblock {\em J. Amer. Statist. Assoc.}, 114(526):869--881.

\bibitem[Chen et~al., 2022a]{chengwangwu}
Chen, L., Wang, W., and Wu, W.~B. (2022a).
\newblock Inference of breakpoints in high-dimensional time series.
\newblock {\em J. Amer. Statist. Assoc.}, 117(540):1951--1963.

\bibitem[Chen and Qin, 2010]{chenqin}
Chen, S.~X. and Qin, Y.-L. (2010).
\newblock A two-sample test for high-dimensional data with applications to
  gene-set testing.
\newblock {\em Ann. Statist.}, 38(2):808 -- 835.

\bibitem[Chen et~al., 2013]{chen2013covariance}
Chen, X., Xu, M., and Wu, W.~B. (2013).
\newblock Covariance and precision matrix estimation for high-dimensional time
  series.
\newblock {\em Ann. Statist.}, 41(6):2994--3021.

\bibitem[Chen et~al., 2022b]{CWS22}
Chen, Y., Wang, T., and Samworth, R.~J. (2022b).
\newblock High-dimensional, multiscale online changepoint detection.
\newblock {\em J. R. Stat. Soc. Ser. B. Stat. Methodol.}, 84:234--266.

\bibitem[Cho, 2016]{haerancho}
Cho, H. (2016).
\newblock Change-point detection in panel data via double cusum statistic.
\newblock {\em Electron. J. Stat.}, 10(2):2000 -- 2038.

\bibitem[Cho and Fryzlewicz, 2015]{chofryz2015}
Cho, H. and Fryzlewicz, P. (2015).
\newblock Multiple-change-point detection for high dimensional time series via
  sparsified binary segmentation.
\newblock {\em J. R. Stat. Soc. Ser. B. Stat. Methodol.}, 77(2):475--507.

\bibitem[Cho and Kirch, 2024]{cho:kirch:2024}
Cho, H. and Kirch, C. (2024).
\newblock Data segmentation algorithms: Univariate mean change and beyond.
\newblock {\em Econometrics Stat.}, 30:76--95.

\bibitem[Cho et~al., 2023]{choetal2023}
Cho, H., Maeng, H., Eckley, I.~A., and Fearnhead, P. (2023).
\newblock High-dimensional time series segmentation via factor-adjusted vector
  autoregressive modelling.
\newblock {\em J. Amer. Statist. Assoc.}, pages 1--28.

\bibitem[Chu and Chen, 2019]{chuchen2019}
Chu, L. and Chen, H. (2019).
\newblock Asymptotic distribution-free change-point detection for multivariate
  and non-euclidean data.
\newblock {\em Ann. Statist.}, 47(1):382 -- 414.

\bibitem[Council, 2013]{NationalResearchCouncil.2013}
Council, N.~R. (2013).
\newblock {\em Frontiers in massive data analysis}.
\newblock The National Academies Press, Washington, D.C.

\bibitem[Dedieu et~al., 2021]{Dedieu.2021}
Dedieu, A., L{\'a}zaro-Gredilla, M., and George, D. (2021).
\newblock Sample-efficient l0-l2 constrained structure learning of sparse
  {Ising} models.
\newblock {\em Proc. AAAI Conf. Artif. Intell.}, 35(8):7193--7200.

\bibitem[Dette and G{\"o}smann, 2018]{dettegös}
Dette, H. and G{\"o}smann, J. (2018).
\newblock Relevant change points in high dimensional time series.
\newblock {\em Electron. J. Stat.}, 12(2):2578 -- 2636.

\bibitem[Dette et~al., 2020]{dettekokot}
Dette, H., Kokot, K., and Volgushev, S. (2020).
\newblock Testing relevant hypotheses in functional time series via
  self-normalization.
\newblock {\em J. R. Stat. Soc. Ser. B. Stat. Methodol.}, 82(3):629--660.

\bibitem[Enikeeva and Harchaoui, 2019]{faridazaid}
Enikeeva, F. and Harchaoui, Z. (2019).
\newblock High-dimensional change-point detection under sparse alternatives.
\newblock {\em Ann. Statist.}, 47(4):2051 -- 2079.

\bibitem[Fan et~al., 2015]{powerenhancement}
Fan, J., Liao, Y., and Yao, J. (2015).
\newblock Power enhancement in high-dimensional cross-sectional tests.
\newblock {\em Econometrica}, 83(4):1497--1541.

\bibitem[Fan and Mackey, 2017]{fanmackey}
Fan, Z. and Mackey, L. (2017).
\newblock Empirical {Bayesian} analysis of simultaneous changepoints in
  multiple data sequences.
\newblock {\em Ann. of Appl. Statist.}, 11(4):2200 -- 2221.

\bibitem[Fischer and Nappo, 2009]{fischer}
Fischer, M. and Nappo, G. (2009).
\newblock On the moments of the modulus of continuity of {Itô} processes.
\newblock {\em Stoch. Anal. and Appl.}, 28(1):103--122.

\bibitem[Gall, 2016]{LeGall2016}
Gall, J.~L. (2016).
\newblock {\em Brownian Motion, Martingales, and Stochastic Calculus}, volume
  274 of {\em Graduate Texts in Mathematics}.
\newblock Springer, Cham, Switzerland.

\bibitem[Gao et~al., 2003]{laplacetransform}
Gao, F., Hannig, J., Lee, T.-Y., and Torcaso, F. (2003).
\newblock {Laplace} transforms via {Hadamard} factorization.
\newblock {\em Electronic Journal of Probability}, 8:1--20.

\bibitem[Gut, 1992]{gut}
Gut, A. (1992).
\newblock The weak law of large numbers for arrays.
\newblock {\em Statistics \& Probability Letters}, 14(1):49--52.

\bibitem[Hall and Heyde, 1980]{hallheyde}
Hall, P. and Heyde, C.\, C. (1980).
\newblock {\em Martingale Limit Theory and Its Application}.
\newblock Probability and Mathematical Statistics. Academic Press, New York.

\bibitem[Hariz et~al., 2007]{harizcp}
Hariz, S.~B., Wylie, J.~J., and Zhang, Q. (2007).
\newblock Optimal rate of convergence for nonparametric change-point estimators
  for nonstationary sequences.
\newblock {\em Ann. Statist.}, 35(4):1802--1826.

\bibitem[Horvath et~al., 2021]{Horvathetal2021}
Horvath, L., Liu, Z., Rice, G., and Zhao, Y. (2021).
\newblock {Detecting common breaks in the means of high dimensional
  cross-dependent panels}.
\newblock {\em Econom. J.}, 25(2):362--383.

\bibitem[Horváth and Hušková, 2012]{horvathhuskova}
Horváth, L. and Hušková, M. (2012).
\newblock Change-point detection in panel data.
\newblock {\em J. Time Series Anal.}, 33(4):631--648.

\bibitem[Hoyer, 2004]{hoyer}
Hoyer, P.~O. (2004).
\newblock Non-negative matrix factorization with sparseness constraints.
\newblock {\em J. Mach. Learn. Res.}, 5:1457–1469.

\bibitem[Ifantis and Siafarikas, 1991]{Ifantis1991}
Ifantis, E.~K. and Siafarikas, P.~D. (1991).
\newblock Bounds for modified {Bessel} functions.
\newblock {\em Rendiconti del Circolo Matematico di Palermo Series 2},
  40(3):347--356.

\bibitem[Jandhyala et~al., 2013]{jandhyala:2013}
Jandhyala, V., Fotopoulos, S., MacNeill, I., and Liu, P. (2013).
\newblock Inference for single and multiple change-points in time series.
\newblock {\em J. Time Series Anal.}, 34(4):423--446.

\bibitem[Jirak, 2015]{jirak2015}
Jirak, M. (2015).
\newblock Uniform change point tests in high dimension.
\newblock {\em Ann. Statist.}, 43(6):2451 -- 2483.

\bibitem[Kirch et~al., 2015]{kirchetal2015}
Kirch, C., Muhsal, B., and Ombao, H. (2015).
\newblock Detection of changes in multivariate time series with application to
  eeg data.
\newblock {\em J. Amer. Statist. Assoc.}, 110(511):1197--1216.

\bibitem[Liang et~al., 2013]{Linagetal2013}
Liang, Y., Liu, C., Luan, X.-Z., Leung, K.-S., Chan, T.-M., Xu, Z.-B., and
  Zhang, H. (2013).
\newblock Sparse logistic regression with a l1/2 penalty for gene selection in
  cancer classification.
\newblock {\em BMC Bioinformatics}, 14(1):198.

\bibitem[Liu et~al., 2022]{liuzhangliu}
Liu, B., Zhang, X., and Liu, Y. (2022).
\newblock High dimensional change point inference: Recent developments and
  extensions.
\newblock {\em J. Multivar. Anal.}, 188:104833.
\newblock 50th Anniversary Jubilee Edition.

\bibitem[Liu et~al., 2020]{liuetal2020}
Liu, B., Zhou, C., Zhang, X., and Liu, Y. (2020).
\newblock {A Unified Data-Adaptive Framework for High Dimensional Change Point
  Detection}.
\newblock {\em J. R. Stat. Soc. Ser. B. Stat. Methodol.}, 82(4):933--963.

\bibitem[Liu et~al., 2021]{liuetal2021}
Liu, H., Gao, C., and Samworth, R.~J. (2021).
\newblock Minimax rates in sparse, high-dimensional change point detection.
\newblock {\em Ann. Statist.}, 49(2):1081 -- 1112.

\bibitem[Lobato, 2001]{lobato}
Lobato, I.~N. (2001).
\newblock Testing that a dependent process is uncorrelated.
\newblock {\em J. Amer. Statist. Assoc.}, 96(455):1066--1076.

\bibitem[Lopes, 2013]{pmlr-v28-lopes13}
Lopes, M. (2013).
\newblock Estimating unknown sparsity in compressed sensing.
\newblock In Dasgupta, S. and McAllester, D., editors, {\em Proceedings of the
  30th International Conference on Machine Learning}, volume~28 of {\em
  Proceedings of Machine Learning Research}, pages 217--225, Atlanta, Georgia,
  USA. PMLR.

\bibitem[Mainardi, 2010]{mainardifrancesco}
Mainardi, F. (2010).
\newblock {\em Fractional Calculus and Waves in Linear Viscoelasticity}.
\newblock Imperial College Press, London.

\bibitem[Niu et~al., 2016]{Niuetal2016}
Niu, Y.~S., Hao, N., and Zhang, H. (2016).
\newblock Multiple change-point detection: A selective overview.
\newblock {\em Statist. Sci.}, 31(4):611--623.

\bibitem[Ombao et~al., 2005]{ombaoetal2015}
Ombao, H., von Sachs, R., and Guo, W. (2005).
\newblock Slex analysis of multivariate nonstationary time series.
\newblock {\em J. Amer. Statist. Assoc.}, 100(470):519--531.

\bibitem[Page, 1954]{page1}
Page, E.~S. (1954).
\newblock Continuous inspection schemes.
\newblock {\em Biometrika}, 41(1/2):100--115.

\bibitem[Preuss et~al., 2015]{puchpreudet2015}
Preuss, P., Puchstein, R., and Dette, H. (2015).
\newblock Detection of multiple structural breaks in multivariate time series.
\newblock {\em J. Amer. Statist. Assoc.}, 110(510):654--668.

\bibitem[Scott, 1973]{scott1973central}
Scott, D.\, J. (1973).
\newblock Central limit theorems for martingales and for processes with
  stationary increments using a {Skorokhod} representation approach.
\newblock {\em Advances in Applied Probability}, 5(1):119--137.

\bibitem[Shao and Zhang, 2010]{shaozhang2010}
Shao, X. and Zhang, X. (2010).
\newblock Testing for change points in time series.
\newblock {\em J. Amer. Statist. Assoc.}, 105(491):1228--1240.

\bibitem[Strassen, 1967]{strassen}
Strassen, V. (1967).
\newblock Almost sure behavior of sums of independent random variables and
  martingales.
\newblock In {\em Proceedings of the Fifth Berkeley Symposium on Math. Statist.
  and Probab.}, volume~2, pages 315--343, Berkeley, CA. University of
  California Press.

\bibitem[Sundararajan and Pourahmadi, 2018]{SunPour2018}
Sundararajan, R.~R. and Pourahmadi, M. (2018).
\newblock Nonparametric change point detection in multivariate piecewise
  stationary time series.
\newblock {\em J. Nonparametr. Stat.}, 30(4):926--956.

\bibitem[Tolmatz, 2002]{tolmatz2002distribution}
Tolmatz, L. (2002).
\newblock On the distribution of the square integral of the {Brownian} bridge.
\newblock {\em Ann. Probab.}, 30(1):253--269.

\bibitem[Truong et~al., 2020]{truongetal2020}
Truong, C., Oudre, L., and Vayatis, N. (2020).
\newblock Selective review of offline change point detection methods.
\newblock {\em Signal Processing}, 167:107299.

\bibitem[Tukey, 1991]{tukey1991}
Tukey, J.~W. (1991).
\newblock The philosophy of multiple comparisons.
\newblock {\em Statist. Sci.}, 6(1):100--116.

\bibitem[van~der Vaart and Wellner, 1996]{vanweak}
van~der Vaart, A.~W. and Wellner, J.~A. (1996).
\newblock {\em Weak Convergence and Empirical Processes: With Applications to
  Statistics}.
\newblock Springer Series in Statistics. Springer, New York.

\bibitem[Wang et~al., 2018]{wangzouyin2018}
Wang, G., Zou, C., and Yin, G. (2018).
\newblock Change-point detection in multinomial data with a large number of
  categories.
\newblock {\em Ann. Statist.}, 46(5):2020 -- 2044.

\bibitem[Wang and Shao, 2020]{wangshao}
Wang, R. and Shao, X. (2020).
\newblock Hypothesis testing for high-dimensional time series via
  self-normalization.
\newblock {\em Ann. Statist.}, 48(5):2728--2758.

\bibitem[Wang et~al., 2022]{wangvolgushev}
Wang, R., Zhu, C., Volgushev, S., and Shao, X. (2022).
\newblock Inference for change points in high-dimensional data via
  selfnormalization.
\newblock {\em Ann. Statist.}, 50(2):781--806.

\bibitem[Wang and Samworth, 2017]{wangtengyao}
Wang, T. and Samworth, R.~J. (2017).
\newblock {High Dimensional Change Point Estimation via Sparse Projection}.
\newblock {\em J. R. Stat. Soc. Ser. B. Stat. Methodol.}, 80(1):57--83.

\bibitem[Wu, 2007]{strongwu}
Wu, W.~B. (2007).
\newblock Strong invariance principles for dependent random variables.
\newblock {\em Ann. Probab.}, 35(6):2294--2320.

\bibitem[Xu et~al., 2021]{XU2021486}
Xu, Y., Narayan, A., Tran, H., and Webster, C.~G. (2021).
\newblock Analysis of the ratio of $\ell^1$ and $\ell^2$ norms in compressed
  sensing.
\newblock {\em Appl. Comput. Harmon. Anal.}, 55:486--511.

\bibitem[Yu and Chen, 2020]{yuchen2020}
Yu, M. and Chen, X. (2020).
\newblock {Finite Sample Change Point Inference and Identification for
  High-Dimensional Mean Vectors}.
\newblock {\em J. R. Stat. Soc. Ser. B. Stat. Methodol.}, 83(2):247--270.

\bibitem[Zhang et~al., 2022]{zhangetal2022}
Zhang, Y., Wang, R., and Shao, X. (2022).
\newblock Adaptive inference for change points in high-dimensional data.
\newblock {\em J. Amer. Statist. Assoc.}, 117(540):1751--1762.

\bibitem[Zhao et~al., 2022]{ZifengJianShao2022}
Zhao, Z., Jiang, F., and Shao, X. (2022).
\newblock {Segmenting Time Series via Self-Normalisation}.
\newblock {\em J. R. Stat. Soc. Ser. B. Stat. Methodol.}, 84(5):1699--1725.

\end{thebibliography}
